\documentclass[10pt]{article}
\RequirePackage[left=28mm,right=28mm,top=35mm,bottom=30mm]{geometry}
\usepackage[utf8]{inputenc}
\usepackage[english]{babel}
\usepackage[]{amsmath}
\usepackage{amssymb}
\usepackage{amsthm}
\usepackage[shortlabels]{enumitem} 
\usepackage[square,sort,comma,numbers]{natbib}
\usepackage{afterpage}
\usepackage{graphicx} 
\usepackage{tikz-cd} 
\usepackage{stmaryrd} 
\usepackage{url}
\usepackage{hyperref}
\usepackage{multicol}
\usepackage{xcolor}
\usepackage[toc, title]{appendix}

\DeclareFontFamily{U}{mathx}{\hyphenchar\font45}
\DeclareFontShape{U}{mathx}{m}{n}{
	<5> <6> <7> <8> <9> <10>
	<10.95> <12> <14.4> <17.28> <20.74> <24.88>
	mathx10
}{}
\DeclareSymbolFont{mathx}{U}{mathx}{m}{n}
\DeclareFontSubstitution{U}{mathx}{m}{n}
\DeclareMathAccent{\widecheck}{0}{mathx}{"71}
\DeclareMathAccent{\wideparen}{0}{mathx}{"75}

\tikzcdset{scale cd/.style={every label/.append style={scale=#1},
		cells={nodes={scale=#1}}}}
	
	\usepackage{mathtools}

\newcommand{\End}{\operatorname{End}}
\newcommand{\Hom}{\operatorname{Hom}}
\newcommand{\gldim}{\operatorname{gldim}}

\newcommand{\Ext}{\operatorname{Ext}}
\newcommand{\Tor}{\operatorname{Tor}}
\newcommand{\add}{\!\operatorname{add}}
\newcommand{\pdim}{\operatorname{pdim}}
\newcommand{\m}{\!\operatorname{-mod}} 
\newcommand{\M}{\!\operatorname{-Mod}} 
\newcommand{\proj}{\!\operatorname{-proj}}

\newcommand{\St}{\Delta}
\newcommand{\Cs}{\nabla}
\renewcommand{\L}{\Lambda}
\renewcommand{\l}{\lambda}
\newcommand{\pri}{\mathfrak{p}} 
\newcommand{\id}{\operatorname{id}}
\newcommand{\mi}{\mathfrak{m}} 

\newcommand{\Spec}{\operatorname{Spec}}
\newcommand{\rank}{\operatorname{rank}} 

\newcommand{\Proj}{\!\operatorname{-Proj}}
\newcommand{\MaxSpec}{\operatorname{MaxSpec}} 
\newcommand{\Stsim}{\tilde{\St}}
 
\newcommand{\Ann}{\operatorname{Ann}}
\newcommand{\coker}{\operatorname{coker}}
\newcommand{\sumSt}{\underset{\l\in\L}{\bigoplus}}

\newcommand{\characteristic}{\operatorname{char}}
\newcommand{\Cssim}{\tilde{\Cs}}
\newcommand{\domdim}{\operatorname{domdim}}
\newcommand{\R}{\operatorname{R}}
\newcommand{\flatdim}{\operatorname{flatdim}}
\newcommand{\height}{\operatorname{ht}}
\newcommand{\HN}{\operatorname{HNdim}}
\newcommand{\qcharacteristic}{\operatorname{-}\characteristic} 
\newcommand{\coheight}{\operatorname{coht}}

\newtheorem{numberingthm}{Theorem}[subsection] 
\theoremstyle{definition}
\newtheorem{Def}[numberingthm]{Definition}
\newtheorem{example}[numberingthm]{Example}

\theoremstyle{plain}
\newtheorem{Prop}[numberingthm]{Proposition}
\newtheorem{Theorem}[numberingthm]{Theorem}
\newtheorem{Cor}[numberingthm]{Corollary}
\newtheorem{Lemma}[numberingthm]{Lemma}
\newtheorem{Remark}[numberingthm]{Remark}

\newenvironment{Example}
{\pushQED{\qed}\example}
{\popQED\endexample}

\theoremstyle{remark}
\newtheorem{Observation}[numberingthm]{Observation}

\theoremstyle{empty}
\newtheorem*{thmintroduction}{Theorem}
\providecommand{\keywords}[1]
{\scriptsize
	\textbf{\textit{Keywords:}} #1 \normalsize \hfill
}
\providecommand{\msc}[1]
{\scriptsize
	\textbf{\textit{2020 Mathematics Subject Classification:}} #1 \normalsize \hfill
}

\title{On Noetherian algebras, Schur functors and Hemmer-Nakano dimensions}
\author{Tiago Cruz} 
\date{}

\newcommand{\Address}{{
		\bigskip
		\footnotesize
		
		TIAGO CRUZ,\par \textsc{Institute of Algebra and Number Theory}\par \textsc{University of Stuttgart,}\par \textsc{Pfaffenwaldring 57, 70569 Stuttgart, Germany,}\par\nopagebreak
		\textit{E-mail address}, T.~Cruz: \texttt{tiago.cruz@mathematik.uni-stuttgart.de}
		}}
		
		\allowdisplaybreaks

\begin{document}

\maketitle

\begin{abstract}
	 Important connections in representation theory arise from resolving a finite-dimensional algebra by an endomorphism algebra of a generator-cogenerator with finite global dimension; for instance, Auslander's correspondence, classical Schur--Weyl duality and Soergel's Struktursatz. Here, the module category of the resolution and the module category of the algebra being resolved are linked via an exact functor known as the Schur functor.
	
In this paper, we investigate how to measure the quality of the connection between module categories of (projective) Noetherian algebras, $B$, and module categories of endomorphism algebras of generator-relative cogenerators over $B$ which are split quasi-hereditary Noetherian algebras. In particular, we are interested in finding, if it exists, the highest degree $n$ so that the endomorphism algebra of a generator-cogenerator provides an $n$-faithful cover, in the sense of Rouquier, of $B$. 
The degree $n$ is known as the Hemmer-Nakano dimension of the standard modules. 

We prove that, in general, the Hemmer-Nakano dimension of standard modules with respect to a Schur functor from a split highest weight category over a field to the module category of a finite-dimensional algebra $B$ is bounded above by the number of non-isomorphic simple modules of $B$.

We establish methods for reducing computations of Hemmer-Nakano dimensions in the integral setup to computations of Hemmer-Nakano dimensions over finite-dimensional algebras, and vice-versa. In addition, we extend the framework to study Hemmer-Nakano dimensions of arbitrary resolving subcategories.
In this setup, we find that the relative dominant dimension over (projective) Noetherian algebras is an important tool in the computation of these degrees, extending the previous work of Fang and Koenig. 
In particular, this theory allows us to derive results for Schur algebras and the BGG category $\mathcal{O}$ in the integral setup from the finite-dimensional case. More precisely, we use the relative dominant dimension of Schur algebras to completely determine the Hemmer-Nakano dimension of standard modules with respect to Schur functors between module categories of Schur algebras over regular Noetherian rings and module categories of group algebras of symmetric groups over regular Noetherian rings.

We exhibit several structural properties of deformations of the blocks of the Bernstein-Gelfand-Gelfand category $\mathcal{O}$ establishing an integral version of Soergel's Struktursatz. We show that deformations of the combinatorial Soergel's functor have better homological properties than the classical one.

\end{abstract}

 \newpage
\tableofcontents

\section{Introduction}\label{Introduction}

A common theme in representation theory is the study of a module category through the lens of another module category having nicer properties. 
A successful example of this approach is the study of a finite-dimensional algebra by one of its quasi-hereditary covers. Quasi-hereditary covers appear quite frequently in algebraic Lie theory (for example as an instance of classical Schur--Weyl duality \cite{zbMATH03708660} and Soergel's Struktursatz \cite{zbMATH00005018}) and in several abstract results of representation theory, like  Auslander's correspondence \cite{zbMATH03517355} and finiteness of representation dimension \cite{zbMATH01849919}. In \cite{Dlab1989}, Dlab and Ringel using a construction developed in \cite{zbMATH03517355} by Auslander 
have shown that every finite-dimensional algebra admits a quasi-hereditary cover in the sense of Rouquier (see \cite{Rouquier2008}). In particular, all finite-dimensional algebras can be resolved by finite-dimensional algebras of finite global dimension.

\paragraph*{General setup} Given a finite-dimensional algebra $A$ and a finitely generated projective $A$-module $P$, we say that $(A, P)$ is a (resp. split) quasi-hereditary cover of $B$ if $A$ is a (resp. split) quasi-hereditary algebra (with respect to some ordering on the simple modules), $B$ is the endomorphism algebra of $P$, and the exact functor $F:=\Hom_A(P, -)\colon A\m\rightarrow B\m$ restricts to a fully faithful functor on the full subcategory of finitely generated projective $A$-modules. The functor $F$ is known in the literature as Schur functor and it can be used to completely determine the simple $B$-modules knowing the simple $A$-modules \citep[6.2]{zbMATH03708660}. 
To transfer cohomological information from the quasi-hereditary cover $(A, P)$ to $B$ through the Schur functor, the quasi-hereditary cover should possess stronger properties. In particular, we can distinguish quasi-hereditary covers by the properties that the Schur functor exhibits on standard modules. Following \cite{Rouquier2008}, a quasi-hereditary cover is called $n$-faithful if the exact functor $\Hom_A(P, -)\colon A\m\rightarrow B\m$ identifies extensions groups $$\Ext_A^i(M, N)\simeq \Ext^i_{B}(\Hom_A(P, M), \Hom_A(P, N))$$ for all integers $i$ ranging from $0$ to $n$ and all modules $M$ and $N$ having a filtration by standard $A$-modules. 
We give a meaning also to the term $(-1)$-faithful cover (see Definition \ref{faithfulcoverdef}). 
 In \cite{Rouquier2008}, we can see that covers with large enough quality (that is $n$-faithful covers with $n$ large enough) are in some sense unique, in particular, $1$-faithful (split quasi-hereditary) covers under some mild assumptions are unique. Nowadays, the optimal value $n$ making a cover being $n$-faithful is known as the Hemmer-Nakano dimension of the subcategory whose modules admit a filtration by standard modules, coined in \cite{zbMATH05871076}.  
 
\paragraph*{Examples in this setup} Schur algebras $S(d, d)$ together with their faithful projective-injective module, $V^{\otimes d}$, are a classical example of a (split) quasi-hereditary cover of group algebras of the symmetric group of $d$ letters $S_d$. The block algebras of the Bernstein-Gelfand-Gelfand category $\mathcal{O}$, together with their projective-injective module, form (split) quasi-hereditary covers of subalgebras of coinvariant algebras. The former connection is a consequence of Schur--Weyl duality \cite{zbMATH03708660} while the latter follows from Soergel's Struktursatz \cite{zbMATH00005018}.
 One more classical example of quasi-hereditary covers is the class of Auslander algebras. Auslander's correspondence can be viewed as an instance of cover theory. In fact, it assigns to each representation-finite finite-dimensional algebra $B$ a quasi-hereditary cover of $B$ known as Auslander algebra. These three classes of examples have actually more in common: they are instances of quasi-hereditary covers that can be realised as endomorphism algebras of generator-cogenerators. Here, generator-cogenerators, are modules whose additive closure contains all injective and all projective modules. In view of the Morita-Tachikawa correspondence all such covers are exactly the covers formed by algebras having dominant dimension at least two (see for example \cite{zbMATH03248955}). In such a case, the computation of Hemmer-Nakano dimensions relies on the computation of dominant dimensions following \cite{zbMATH03248955} and \cite{zbMATH05871076}.

 In \cite{CRUZ2022410}, the author introduced a generalisation of dominant dimension suitable for the integral setup. We wonder to what extent the computations of Hemmer-Nakano dimensions in the integral setup rely on computations of relative dominant dimension in the sense of \cite{CRUZ2022410}.
 Our focus in this paper is to construct more tools to compute Hemmer-Nakano dimensions and in particular to develop techniques to be able to compute Hemmer-Nakano dimensions in the integral setup like deformation techniques (see Sections \ref{A-covers under change of ground ring} and \ref{Hemmer--Nakano dimension with respect to covers constructed from relative injective modules}). There are very strong reasons to care about the integral case in this setup. To understand why this problem is relevant, we shall have a closer look to what values the Hemmer-Nakano dimension takes in the classical examples.

\paragraph*{Schur algebras and symmetric groups} Over the complex numbers, Schur in modern terminology used in \cite{zbMATH02662157} the cover $(S(d, d), V^{\otimes d})$ and more precisely the Schur functor (associated with this cover) to connect the polynomial representation theory of $\operatorname{GL}_d(\mathbb{C})$ with the complex representation theory of $S_d$. In this case, the Schur functor is an equivalence of categories. In positive characteristic, this cover is no longer a particular case of an equivalence of categories. Its quality in the positive characteristic case started to attract attention in \cite{hemmer_nakano_2004}. 
In \cite{hemmer_nakano_2004}, Hemmer and Nakano established that when the underlying field has characteristic $p>3$, $(S(d, d), V^{\otimes d})$ is a $(p-3)$-faithful cover. In particular, they proved that the Schur functor induces an exact equivalence between the full subcategory of modules having a Weyl filtration and the full subcategory of the module category of the group algebra of the symmetric group whose modules admit a filtration by dual Specht modules if the characteristic is bigger than three.
Later, in \cite{zbMATH05871076}, it was reproved that $(S(d, d), V^{\otimes d})$ is a $(p-3)$-faithful cover using dominant dimension when the underlying field has characteristic $p>0$.  There, it was established that this degree is optimal, that is, $(S(d, d), V^{\otimes d})$ is not a $(p-2)$-faithful cover when the underlying field has characteristic $p>0$. Unfortunately, the Schur functor in characteristics two or three no longer identifies the subcategory of modules having a Weyl filtration with the subcategory of modules having a filtration by dual Specht modules.

\paragraph*{BGG category $\mathcal{O}$ and subalgebras of coinvariant algebras}The situation for the BGG category $\mathcal{O}$ of a complex semi-simple Lie algebra seems to be even worse. In fact, all standard modules, also known as Verma modules, are sent to the same module under Soergel's combinatorial functor. Here Soergel's combinatorial functor is the Schur functor in Soergel's Struktursatz. So, these covers formed by the block algebras of the BGG category $\mathcal{O}$ are not even $0$-faithful (see \cite{zbMATH05278765}). 

\paragraph*{Integral setup} It turns out that this is not the end of the story for the study of these covers. The concept  of cover can be defined over any commutative Noetherian ring and Rouquier's framework is also suitable for Noetherian algebras. Moreover, \citep[Proposition 4.42]{Rouquier2008} provides evidence that quasi-hereditary covers might behave better in the integral setup. This claim is also supported by the work developed in \cite{zbMATH00971625}. Naively, we could think that increasing the global dimension of the ground ring could allow more extensions groups to be identified, improving the situation overall.

All previous examples mentioned above can be studied using dominant dimension.
In the integral setup, projective-injective modules rarely exist, so the classical dominant dimension cannot be used directly here. The generalisation of dominant dimension introduced in \cite{CRUZ2022410} fixes this obstacle. A major difference is that this new relative dominant dimension over projective Noetherian algebras is not characterised in terms of Ext groups but instead by Tor groups. This fact combined with \citep[Theorem 6.13]{CRUZ2022410} provides further evidence that covers might behave better in the integral setup. So, to continue this story, it is fundamental to understand the connections between Hemmer-Nakano dimensions and this recent concept of relative dominant dimension.

\paragraph*{List of questions} Given this context, the following questions are crucial.
\textit{\begin{enumerate}[(1)]
	\item Under what conditions are integral covers better than finite-dimensional ones?  
	\item Can we lift the covers mentioned above to covers over commutative regular rings with higher quality, resulting in higher values of Hemmer-Nakano dimensions?
	\item Does an increase in the Krull dimension of the ground ring result in a cover with higher quality?
	Conversely, is the cover tensored with a quotient ring always worse than the original cover?
\end{enumerate}}

In \cite{chen_fang_kerner_koenig_yamagata_2021}, the concept of rigidity dimension was introduced to measure the quality of the best resolution of a finite-dimensional algebra by an endomorphism algebra of a generator-cogenerator. In particular, such a dimension aims to give an upper bound to the quality of any resolution of a finite-dimensional algebra via endomorphism algebra of generator-cogenerators with finite global dimension.  In general, the finiteness of the rigidity dimension is an open problem. In our setup, we can replace the algebras of finite global dimension with stronger assumptions like being quasi-hereditary  algebras or even split quasi-hereditary algebras. This situation raises the following question:
\textit{\begin{enumerate}[(1)] \setcounter{enumi}{3}
	\item  Can the Hemmer-Nakano dimension of the subcategory of modules admitting a filtration by standard modules with respect to a (split) quasi-hereditary cover of $B$ be controlled solely by invariants of $B$?
\end{enumerate}}

\paragraph*{Contributions} The aim of this paper is to advance our knowledge on how to compute Hemmer-Nakano dimensions in the integral setup by giving answers to these questions and providing a generalisation of the Hemmer-Nakano theorem for Schur algebras (and $q$-Schur algebras) over regular rings. In addition, we study deformations of covers of subalgebras of coinvariant algebras having higher quality than their finite-dimensional counterparts. That is, these new covers have  higher levels of faithfulness - see Definition \ref{faithfulcoverdef}- in the sense of Rouquier.
We answer Question (4) for split quasi-hereditary covers in Theorem \ref{rigiditypartI}. Theorem \ref{rigiditypartI} states that if $(A, P)$ is an $n$-faithful (split quasi-hereditary) cover of a finite-dimensional algebra $B$ with infinite global dimension then $n$ must be smaller than or equal to the number of non-isomorphic classes of simple $B$-modules.
This result gives finiteness to resolving algebras by quasi-hereditary covers, in contrast with the rigidity dimension whose finiteness relies on homological conjectures.
Our answer to Question (1) is given mainly in Theorem \ref{improvingcoverwithspectrum}. This result allows us to determine the degrees in which a Schur functor identifies extension groups between modules belonging to a given resolving subcategory and their images under the Schur functor covering in this way more general situations than $n$-faithful covers. 
\begin{thmintroduction}[Theorem \ref{improvingcoverwithspectrum} and Corollary \ref{coverheightprimeideal} for $n$-faithful covers]
	Let $R$ be a regular local commutative Noetherian ring with quotient field $K$. Suppose that $(A, P)$ is a $0$-faithful cover of $B$. Let $i\geq 0$. Then $(A, P)$ is an $(i+1)$-faithful cover of $B$ if and only if the  following conditions are satisfied:
	\begin{enumerate}[(i)]
		\item $(K\otimes_R A, K\otimes_R P)$ is an $(i+1)$-faithful cover of $K\otimes_R B$.
		\item For each prime ideal $\mathfrak{p}$ of height one, $(R/\mathfrak{p}\otimes_R A, R/\mathfrak{p}\otimes_R P)$ is an $i$-faithful cover of $R/\mathfrak{p}\otimes_R B$.
	\end{enumerate}
\end{thmintroduction}
This means that the computation of the Hemmer-Nakano dimension depends on the spectrum of the ground ring $R$ and knowing the behaviour of a cover on residue fields is not enough, in contrast to the relative dominant dimension. Our approach evaluates the Schur functor mainly on resolving subcategories that behave nicely under change of ground ring like the subcategory of projective modules and the subcategory of modules having a finite filtration by direct summands of direct sums of standard modules.
Condition (ii) might be dropped if the Krull dimension is just one, or if we already know that $(A, P)$ is an $i$-faithful cover. The difference between the global dimension of the rings $R$ and $R/\pri$ is just one whenever $\pri$ has height one. Therefore, Condition (ii) guarantees that if $(A, P)$ (resp. $(A/\mi A, P/\mi P)$ with $\mi$ the maximal ideal of $R$) is an $i$-faithful cover (resp. $j$-faithful cover ) of $B$ (resp. $B/\mi B$) then $i$ is greater than or equal to $j$ (see also Propositions \ref{faithfulcoversresiduepart3} and \ref{faithfulcoverresiduefieldnext}) and their difference is not greater than the global dimension of the coefficient ring $R$. In short, a cover $(A, P)$ has better properties than their finite-dimensional counterpart $(A/\mi A, P/\mi P)$ if tensoring with the quotient field produces a cover with better properties.

Question (2) has a positive answer for both Schur algebras and blocks of the BGG category $\mathcal{O}$ of a semi-simple Lie algebra. \paragraph*{Applications for Schur algebras} For Schur algebras we obtain the following:

\begin{thmintroduction}[\ref{schuralgebraHn}, ~\ref{Hemmenakanodimofprojectives}]
	Let $R$ be a local regular commutative Noetherian ring and $d\in \mathbb{N}$. Define $i:=\inf\{k\in \mathbb{N} \ | \ (k+1)\cdot 1_R\notin R^\times, \ k<d \}\in \mathbb{N}\cup \{+\infty\}$, where $R^\times$ denotes the set of invertible elements of $R$.
	If $R$ is a local ring of equal characteristic, then $(S_R(d, d), V^{\otimes d})$ is an $(i-2)$-faithful cover of $RS_d$.
	If $R$ is a local ring of unequal characteristic, then $(S_R(d, d), V^{\otimes d})$ is an $(i-1)$-faithful cover of $RS_d$. Both of these values are maximal.
\end{thmintroduction}This result gives a positive answer to Question (2) for the case of Schur algebras. Moreover, this result says that the cover of the symmetric group formed by the Schur algebra behaves exactly like in the finite-dimensional case if the coefficient ring is a local regular ring containing a field as a subring, that is, a local ring of equal characteristic. 
In particular, it follows that the Schur functor over the localization of the integers away from 2 restricts to an exact equivalence between the category whose modules admit Weyl filtration and the category whose modules admit a dual Specht filtration improving, therefore, the characteristic three case. 
In Theorems \ref{HNqschuralgebraI} and \ref{HNqschuralgebraII}, we obtain an analogue result for $q$-Schur algebras.

\paragraph*{Applications for the BGG category $\mathcal{O}$} For the BGG category $\mathcal{O}$ the improvement of going integrally is more dramatic.
\begin{thmintroduction}[\ref{Oissplitqh}, ~\ref{integralstruktursatz}, ~\ref{HNforcategoryO}]
		Fix a natural number $t$. Let $R$ be the localization of $\mathbb{C}[X_1, \ldots, X_t]$ at the maximal ideal $(X_1, \ldots, X_t)$. Denote by $\mi$ the unique maximal ideal of $R$. Let $\mathcal{O}_{\mathcal{D}}$ be a block of the BGG category $\mathcal{O}$ of a complex semi-simple Lie algebra $\mathfrak{g}$. For any natural number $s$ being smaller than or equal to the minimum between the dimension of the Cartan subalgebra of $\mathfrak{g}$ and $t$, there exists an $R$-algebra $A_{\mathcal{D}_s}$ which is projective and finitely generated as $R$-module so that there exists an exact equivalence of categories between  $\mathcal{O}_{\mathcal{D}}$ and $R/\mi\otimes_R A_{\mathcal{D}_s}\m$. Furthermore, $A_{\mathcal{D}_s}$ is a split quasi-hereditary algebra over $R$ and  there exists a projective $A_{\mathcal{D}_s}$-module $P$ so that the following assertions hold:
		\begin{enumerate}[(i)]
			\item  $(A_{\mathcal{D}_s}, P)$ is a relative gendo-symmetric $R$-algebra;
			\item   $(A_{\mathcal{D}_s}, P)$ is an $(s-1)$-faithful cover of a deformation of a subalgebra of a coinvariant algebra.
		\end{enumerate}
\end{thmintroduction}
Surprisingly, in the integral setup if $\mathfrak{g}\neq \mathfrak{sl}_2$ this result says that we can construct a deformation of  Soergel's combinatorial functor that actually restricts to an exact equivalence between the category of modules having a finite filtration by integral Verma modules and the category of modules having a finite filtration by the image of integral Verma modules by such a functor.

This result about deformations of blocks of the BGG category $\mathcal{O}$ answers Question (3). It turns out that the Hemmer-Nakano dimension is not fully determined by the Krull dimension and the relative dominant dimension alone, and without further assumptions simply increasing the Krull dimension of the ground ring does not cause an increase of the Hemmer-Nakano dimension.
These results also clarify that using extension groups to determine relative dominant dimension as it was introduced in \cite{CRUZ2022410} is not precise over rings with higher Krull dimension and so using Tor groups is essential for that purpose (see Remark \ref{dominantdimensionusingExtclarification}).

\paragraph*{Strategy} The quality of covers is a local property, so it is enough to consider the cases when the ground ring is a local regular ring.
The proofs of Theorem \ref{improvingcoverwithspectrum} and Corollary \ref{coverheightprimeideal} make use of a version of the universal coefficient Theorem to yield information about vanishing of the right derived functor of the right adjoint of a Schur functor by tensoring a suitable cochain of projective $R$-modules with $R/x$, for some element $x$ of a regular sequence of the regular local ring $R$.  We shall now give a brief overview of our approach to (1). Assume that we are provided with a finite-dimensional algebra $B$ and with a generator-cogenerator $M$ over $B$. Morita-Tachikawa correspondence states that the endomorphism algebra of $M$ over $B$, $A$, has dominant dimension at least two. To find an integral version of this resolution is then the same as finding a projective Noetherian algebra $A_R$ (over a commutative Noetherian ring or just a local commutative Noetherian $R$) with relative dominant dimension at least two in the sense of \cite{CRUZ2022410} and a projective module $M_R$ over $A_R$ so that $A$ can be recovered as $A_R/\mi A_R$ and $M$ as $M_R/\mi M_R$ for every maximal ideal $\mi$ of $R$, respectively. In particular, the endomorphism algebra  of $M_R$ over $A_R$, which we denote by $B_R$, is an integral version of $B$. Here projective Noetherian algebra means an algebra whose regular module is finitely generated and projective as a module over the ground ring.
By the relative Morita-Tachikawa correspondence (see \citep[Theorem 4.1]{CRUZ2022410}), $M_R$
is a generator and relative cogenerator over $B_R$ so that $A_R$ is the endomorphism algebra of $M_R$ over $B_R$ having a base change property. Now using Propositions \ref{faithfulcoversresiduepart3} and \ref{faithfulcoverresiduefieldnext} or relative dominant dimension in the form of \citep[Theorem 6.13]{CRUZ2022410} we see that the connection between the module categories of $A_R$ and $B_R$ is no less strong than the connection between the module categories of $A$ and $B$. Such connection is then obtained by measuring relative dominant dimensions over $A_R$ and over $Q(R/\pri)\otimes_R A_R$ for all quotient fields of quotients by prime ideals $\pri$ of $R$.

\paragraph*{Structure of the paper} This paper is organised as follows:
\newline
In Subsection \ref{Some basic facts}, we collect some basic results on change of rings and some elementary facts involving tensor products. In Subsections \ref{Split quasi-hereditary algebras} and \ref{Covers} we recall the definition and some properties of split quasi-hereditary algebras and covers, respectively. In \ref{Relative dominant dimension over Noetherian algebras}, we bring back the concept of relative dominant dimension over Noetherian algebras introduced by the author in \cite{CRUZ2022410}. 

In Section \ref{faithful covers}, we generalize the concept of faithful covers to what we call $\mathcal{A}$-covers, where $\mathcal{A}$ represents a resolving subcategory of the module category. We give also their basic properties. 

In Section \ref{Upper bounds for the quality of an A-cover}, we explore the question of finiteness of the quality of  an $\mathcal{A}$-cover, generalising the concept of equivalence of covers to tackle the uniqueness of covers. In particular, in Corollary \ref{equivalenceofoneuniqueness} we give an alternative proof for the uniqueness of 1-faithful covers under mild assumptions. In Theorem \ref{rigiditypartI}, we prove that if the degrees, in which a Schur functor preserves extensions groups between standard modules and their images under the Schur functor, are greater than the number of non-isomorphic classes of simple modules of one of the algebras involved then the Schur functor is actually an equivalence of categories.

In Section \ref{A-covers under change of ground ring}, we study the behaviour of $\mathcal{A}$-covers under change of ground rings. Here we have to restrict our attention to resolving subcategories that remain resolving under change of ground ring. To do that, in Definition \ref{wellbehavedresolving}, we introduce what we call well behaved resolving subcategories of a module category. In this section, we explore how we can increase and decrease the quality of a cover by changing the ground ring. In particular, Propositions \ref{arbitraryfaithfulcoverflattwo}, \ref{faithfulcoversresiduepart3}, Corollary \ref{coverheightprimeideal} and Theorem \ref{improvingcoverwithspectrum} say that the quality of a cover over a regular Noetherian ring gets determined by knowing the effect that tensoring the cover with the quotient field of $R/\pri$ causes, running $\pri$ over all the prime ideals of the ground ring $R$. 

In Section \ref{Hemmer--Nakano dimension with respect to covers constructed from relative injective modules}, we combine the tools of relative dominant dimension with cover theory to give lower and upper bounds for Hemmer-Nakano dimensions of well behaved resolving subcategories (with respect to covers formed by a relative QF3-algebra over a regular ring with relative dominant dimension at least two) in terms of relative dominant dimension of modules belonging to the resolving subcategory under study and of the Krull dimension of the ground ring. Based on the work of \cite{zbMATH05871076}, in Subsections \ref{Hemmer-Nakano dimension of}, \ref{Hemmer-Nakano dimension of f}, \ref{Hemmer-Nakano dimension of contravariantly finite resolving subcategories}, we show that inside the resolving categories that we are interested in there are tilting modules whose relative dominant dimension can be used to control the Hemmer-Nakano dimension of the subcategory under study.

In Subsection \ref{Classical Schur algebras}, we study the quality of the Schur functor from the module category over a Schur algebra over a regular Noetherian ring to the module category of the group algebra of a symmetric group over a regular Noetherian ring. In Subsubsection \ref{Uniqueness of covers for RSd}, we discuss the problem of the existence of better covers for symmetric groups than the one formed by the Schur algebra. In Subsection \ref{qSchuralgebras}, we study the quality of the cover of the Iwahori-Hecke algebra formed by the $q$-Schur algebra and $V^{\otimes d}$. This answer requires the introduction of the concept of a partially quantum divisible ring.  In Subsection \ref{Deformations of the BGG category O}, we study deformations of Bernstein-Gelfand-Gelfand category $\mathcal{O}$, over commutative rings, introduced in \cite{zbMATH03747378}. In \ref{noetherianalgebracategoryOblock}, we present for any block $\mathcal{O}_{\mathcal{D}}$ of the BGG category $\mathcal{O}$, a projective Noetherian algebra whose module category is a deformation of the block $\mathcal{O}_{\mathcal{D}}$. In Theorem \ref{Oissplitqh}, we show that this algebra is split quasi-hereditary with integral Verma modules as standard modules. In Theorem \ref{Oiscellular}, we prove that this algebra is also cellular over any local regular ring which is a $\mathcal{Q}$-algebra. In Theorem \ref{dominantdimensionO}, we compute its relative dominant dimension. Such a result is then applied to establish in Theorem \ref{integralstruktursatz} an integral version of Soergel's Struktursatz.
In Theorem \ref{HNforcategoryO}, we address the quality of the respective integral version of   Soergel's combinatorial functor.


Further applications of the work here developed will appear in forthcoming work. For instance, Theorem \ref{HNforcategoryO} will be used to deduce that the blocks of the BGG category $\mathcal{O}$ of complex semi-simple Lie algebra are Ringel self-dual. Another application of the framework here developed will be the existence of Hemmer-Nakano type results involving generalisations of Auslander algebras which are not necessarily quasi-hereditary.

\section{Preliminaries}\label{Preeliminaries}
In this section, we give the notation and the objects to be used in the current paper.

Throughout this paper, we assume that $R$ is a Noetherian commutative ring with identity and $A$ is a projective Noetherian $R$-algebra, unless stated otherwise. By a \textbf{projective Noetherian $R$-algebra} we mean an $R$-algebra $A$ so that $A$ is finitely generated and projective as $R$-module. 
We call $A$ a \textbf{free Noetherian} $R$-algebra if $A$ is a Noetherian $R$-algebra so that $A$ is free of finite rank as $R$-module. By $A\m$ we denote the category of all finitely generated (left) $A$-modules.
Given $M\in A\m$, $\add_A M$ denotes the full subcategory of $A\m$ whose modules are direct summands of a direct sum of copies of $M$. We will write $\add M$ when there is no ambiguity on the ambient algebra. We will denote by $\id_M$  the identity map on $M\in A\m$. With $\End_A(M)$ we denote the endomorphism algebra over an $A$-module $M$. We denote by $A\proj$ the subcategory $\add_A A$ of $A\m$. We will denote by $A^{op}$ the opposite algebra of $A$ and by $D_R$ the standard duality functor $\Hom_R(-, R)\colon A\m\rightarrow A^{op}\m$. We will just write $D$ when there is no ambiguity on the ground ring. A module $M\in A\m$ is known as \textbf{generator} (resp. \textbf{$(A, R)$-cogenerator}) if $A\in \add_A M$ (resp. $DA\in \add_A M$). By a \textbf{progenerator} we mean a generator that is also projective. By an $(A, R)$-exact sequence we mean an exact sequence of $A$-modules which splits as a sequence of $R$-modules. A map $f\in \Hom_A(M, N)$ is called an \textbf{$(A, R)$-monomorphism} if the sequence $0\rightarrow M\xrightarrow{f} N$ is $(A, R)$-exact. A module $M\in A\m\cap R\proj$ is \textbf{$(A, R)$-injective} if and only if $M\in \add DA$.

Given $\mathcal{C}$ a full subcategory of $A\m$, we denote by $\widecheck{\mathcal{C}}$ the full subcategory of $A\m$ whose modules $M$ fit into an exact sequence of the form $0\rightarrow M\rightarrow X_0 \rightarrow \cdots \rightarrow X_s\rightarrow 0$, where all $X_i\in \mathcal{C}$. 

We write $\pdim_A M$ to denote the projective dimension of $M\in A\m$ and we write $\gldim A$ to denote the global dimension of $A$. By $\dim R$ we mean the Krull dimension of $R$. We will denote by $\Spec R$ the set of prime ideals of $R$ and by $\MaxSpec R$ the set of maximal ideals of $R$. Given $\pri\in \Spec R$, we denote by $\height(\pri)$ (resp. $\coheight(\pri))$ the height (resp. coheight) of $\pri$ and by $M_\pri$ (resp. $R_\pri$) the localization of $M\in R\m$ (resp. $R_\pri$) at $\pri$. Given $f\in \Hom_A(M, N)$ we write $f_\pri$ to denote the localization of $f$ at $\pri\in \Spec R$. By a \textbf{regular ring} we mean a commutative Noetherian ring $R$ for which for every $\pri\in \Spec R$ the commutative Noetherian local ring $R_\pri$ has finite global dimension. In such a case, $\dim R=\gldim R$ (see for example \citep[Theorem 8.62]{Rotman2009a}). 
By $R^\times$ we denote the set of invertible elements of $R$. 
For each $\mi\in \MaxSpec R$, we will denote by $R(\mi)$ the residue field $R/\mi\simeq R_\mi/\mi_\mi$. For each $\mi\in \MaxSpec R$ and $M\in A\m$ we will write $M(\mi)$ to denote $R(\mi)\otimes_R M\in A(\mi)\m$. Further, for each $\mi\in \MaxSpec R$, we will use $D_{(\mi)}$ to denote the standard duality $D_{R(\mi)}$.
Given a subset $S$ of the extended natural numbers $\mathbb{N}\cup \{0, +\infty\}$, we will denote by $\inf S$ the infimum of the subset $S$ in the poset of the extended natural numbers.

\subsection{Some basic facts}\label{Some basic facts}

	Some facts to keep in mind about localization are the following:  $M=0$ if and only if $M_\mi=0$ for every $\mi\in \MaxSpec R$ and localization is an exact functor. In particular, localization commutes with $\Ext$ and $\Tor$ functors (see for example \citep[Proposition 3.3.10]{Rotman2009a}).

\begin{Lemma}\label{homtensorarbitraryring}
	Let $f\colon R\rightarrow S$ be a surjective $R$-algebra homomorphism.  Let $A$ be an $R$-algebra.
	If $M$ and $N$ are $A$-modules, then
	$\Hom_{S\otimes_R A}(S\otimes_R M, S\otimes_R N)\simeq \Hom_{A}(S\otimes_R M, S\otimes_R N)\simeq \Hom_A(M, S\otimes_R N)$. 
\end{Lemma}
\begin{proof}
	Let $\phi\in \Hom_{S\otimes_R A}(S\otimes_R M, S\otimes_R N)$. Then, for any $a\in A, s\otimes_R m\in S\otimes_R M$,
	\begin{align}
		\phi(a (s\otimes m))=\phi(s\otimes am)=\phi((1_S\otimes a)(s\otimes m))=(1_S\otimes a)\phi(s\otimes m)=a\phi(s\otimes a).
	\end{align}Thus, $\phi\in \Hom_A(S\otimes_R M, S\otimes_R N)$. Now consider $\phi\in \Hom_A(S\otimes_R M, S\otimes_R N)$. For any $a\in A$, $m\in M$, $s', s\in S$ we have
	\begin{align}
		\phi(s'\otimes a s\otimes m)&=\phi(s's\otimes am)=\phi(a(s's\otimes m))=a\phi(f(r')s\otimes m)=a \phi(r'f(1_R)s\otimes m)\\&= r'a\phi(1_Ss\otimes m)=(1_S\otimes r'a)\phi(s\otimes m)=(f(1_Rr')\otimes a)\phi(s\otimes m)=s'\otimes a\phi(s\otimes m),
	\end{align}
	for some $ r'\in R$. Hence, $\phi\in \Hom_{S\otimes_R A}(S\otimes_R M, S\otimes_R N)$. Therefore, the first isomorphism is established.
	Let $\phi\in \Hom_A(M, S\otimes_R N)$. We extend $\phi$ to a map $\phi'\in \Hom_A(S\otimes_R M, S\otimes_R N)$ by imposing ${\phi'(s\otimes m)}=s\phi(m)$. Let $\phi\in \Hom_A(S\otimes_R M, S\otimes_R N)$, we restrict it to $\phi_|\in \Hom_A(M, S\otimes_R N)$ by defining $\phi_|(m)=\phi(1_S\otimes m)$, $m\in M$.
	Using these two correspondences, we obtain the second isomorphism.
\end{proof}

\begin{Lemma}\label{zerodivisormono}
	Let $x$ be a non-zero divisor of $R$. The following assertions hold.
	\begin{enumerate}[(i)]
		\item Let $M\in R\proj$. Then the $R$-homomorphism $\delta\colon M\rightarrow M, \ m\mapsto xm$ is a monomorphism.
		\item Let $M\in A\m\cap R\proj$. 
		The map ${\End_A(M)\otimes_R R/Rx\rightarrow \Hom_A(M, R/Rx\otimes_R M)}$, given by ${f\otimes r+Rx}\mapsto (m\mapsto r+Rx\otimes_R f(m))$, is a monomorphism.
	\end{enumerate} 
\end{Lemma}
\begin{proof}
	Let $m\in M$ such that $xm=0$. Since $M$ is projective over $R$, there exists a natural number $n$ and $K\in R\m$ such that $R^n\simeq M\bigoplus K$. So, there exists $\alpha_i\in R$ satisfying $m=\sum_{i} \alpha_i e_i$, where $\{e_i: i=1, \ldots, n\}$ is an $R$-basis for $R^n$. Therefore, $x\alpha_i=0$ for all $i=1, \ldots, n$. Since $x$ is a non-zero divisor, $\alpha_i=0$ for all $i$. Hence, $m=0$. Thus, $i)$ follows.
	
	Assume that $\displaystyle 
	0=\delta(\sum_i f_i\otimes_R r_i+Rx)=\delta(\sum_i r_if_i\otimes_R 1+Rx)=\delta(f\otimes_R 1+Rx)
	$ for some $f\in \End_A(M)$. In particular, $1+Rx\otimes_R f(m)=0$ for all $m\in M$. Since $R/Rx\otimes_R M\simeq M/xM,$ it follows that $f(m)\in RxM=xM$ for all $m\in M$. We claim that $f=xg$ for some $g\in \End_A(M)$.
	By assumption, there is for every $m\in M$, $y_{m}\in M$ satisfying $f(m)=xy_m$. 
	Note that, any $b\in A$ and $m, m_1, m_2\in M$
	\begin{align}
		xy_{bm} = f(bm)=bf(m)=b(xy_m)\implies x(y_{bm}-by_m)=0 \text{ and}\\
		xy_{m_1+m_2}=f(m_1+m_2)=f(m_1)+f(m_2)=xy_{m_1}+xy_{m_2}.
	\end{align}By $i)$, $y_{bm}-by_m=0$ and $y_{m_1+m_2}=y_{m_1}+y_{m_2}$. Thus, $g\colon M\rightarrow M$, given by $g(m):=y_m$ for every $m\in M$, is a well defined element of $\End_A(M)$ satisfying $f=xg$.
	
	Hence, $f\otimes_R (1+Rx)=xg\otimes 1+Rx=g\otimes x+Rx=0$. So, $\delta$ is a monomorphism.
\end{proof}

\begin{Lemma}\label{tensoronsecondcomponetofhom}
	Let $M\in A\proj$. Then the map \mbox{$\varsigma_{M, N, U}\colon \Hom_A(M, N)\otimes_R U\rightarrow \Hom_A(M, N\otimes_R U)$}, given by $ \ g\otimes u\mapsto g(-)\otimes u$, is an $R$-isomorphism.
\end{Lemma}
\begin{proof}
	Consider $M=A$. The following diagram is commutative.
	\begin{center}
		\begin{tikzcd}
			\Hom_A(A, N)\otimes_R U \arrow[r, "\varsigma_{A, N, U}"] \arrow[d]& \Hom_A(M, N\otimes_R U)\arrow[d]\\
			N\otimes_R U\arrow[r, equal] &N\otimes_R U
		\end{tikzcd}
	\end{center} 
	Both columns are isomorphisms, thus $\varsigma_{A, N, U}$ is an isomorphism. Since this map is compatible with direct sums, it follows that $\varsigma_{M, N, U}$ is an isomorphism for every $M\in A\proj$ and any $N\in A\M$, $U\in R\M$.
\end{proof}

\subsubsection{Filtrations and exact categories}

A category $\mathcal{A}$ is called a \textbf{(Quillen) exact category} if $\mathcal{A}$ is a full subcategory closed under extensions of some abelian category $\mathcal{C}$.

A functor $F\colon \mathcal{A}\rightarrow \mathcal{B}$ between exact categories is called \textbf{exact} provided that $F$ preserves exact sequences. Let $B$ be a projective Noetherian $R$-algebra.
Let $\mathcal{A}$ be a full subcategory of $A\m$ closed under extensions and closed under direct summands and let $\mathcal{B}$ be a full subcategory of  $B\m$ closed under extensions and closed under direct summands. In this paper, we will call $F$ an \textbf{exact equivalence (commuting with $R$)} or just \textbf{exact equivalence} if $F$ is an exact functor, an equivalence of categories and it satisfies $F(M\otimes_R U)\simeq FM\otimes_R U$ for all $U\in R\proj$. Observe that if $R$ is a field then the latter just follows from $F$ preserving direct summands and direct sums.
Moreover, it follows from Lemma \ref{tensoronsecondcomponetofhom} that all exact equivalences  of exact categories that descend from equivalences of categories of module categories  do have this extra property. 

A category $\chi$  is said to be \textbf{resolving} of a category $\mathcal{A}$ if it is closed under extensions, direct summands and kernels of epimorphisms, and it contains all projective objects of $\mathcal{A}$. 

Given a set of finitely generated $A$-modules, $\Theta$, we denote by $\mathcal{F}(\Theta)$ the full subcategory of $A\m$ whose modules $M$ admit a finite filtration
$
			0= M_{n+1}\subset M_n\subset \cdots \subset M_1=M$,  such that $ M_i/M_{i+1}\in \Theta.
$

Given an exact functor $F\colon \mathcal{A}\rightarrow \mathcal{B}$ and suppose that $\mathcal{F}(\Theta)\subset \mathcal{A}$, then we write $F\Theta$ to denote the set of objects ${\{F\theta\colon \theta\in \Theta \}}$. 

\subsection{Split quasi-hereditary algebras}\label{Split quasi-hereditary algebras}

Quasi-hereditary algebras over fields were introduced in \cite{MR961165} to model the representation theory of complex semi-simple Lie algebras and algebraic groups. For the study of quasi-hereditary algebras over fields, we refer to \citep{PS88, MR961165, zbMATH00140218}, \citep[A]{MR1284468}. In particular, over algebraically closed fields, every quasi-hereditary algebra is split quasi-hereditary. In \citep{CLINE1990126}, the concept of quasi-hereditary algebra and split quasi-hereditary algebra was extended to Noetherian algebras. In \cite{Rouquier2008}, the study of split quasi-hereditary algebras over commutative Noetherian rings (not necessarily local) was developed using a module theoretical approach exploiting "integral" standard modules.

\begin{Def}\label{qhdef}
	Given a projective Noetherian $R$-algebra $A$ and a collection of finitely generated left $A$-modules $\{\St(\l)\colon \l\in \L\}$ indexed by a  poset $\L$, we say that $(A, \{\Delta(\lambda)_{\lambda\in \Lambda}\})$ is a \textbf{split quasi-hereditary $R$-algebra} if the following conditions hold:
	\begin{enumerate}[(i)]
				\item The modules $\St(\l)\in A\m$ are projective over $R$.
		\item Given $\l, \mu\in \L$, if $\Hom_A(\St(\l), \St(\mu))\neq 0$, then $\l\leq\mu$.
		\item $\End_A(\St(\l))\simeq R$, for all $\l\in\L$.
		\item Given $\l\in\L$, there is $P(\l)\in A\proj$ and an exact sequence $0\rightarrow C(\l)\rightarrow P(\l)\rightarrow \St(\l)\rightarrow 0$ such that $C(\l)$ has a finite filtration by modules of the form $\St(\mu)\otimes_R U_\mu$ with $U_\mu\in R\proj$ and $\mu>\l$. 
		\item $P=\sumSt P(\l)$ is a progenerator for $A\m$. 
	\end{enumerate}
\end{Def}
Under these conditions, we also say that $(A\m, \{\Delta(\lambda)_{\lambda\in \Lambda}\})$ is a \textbf{split highest weight category.}

Here we can write the condition (v) instead of condition (4) in \citep[Definition 4.11]{Rouquier2008} by exploiting that a module $P$ is a progenerator if and only if for any $M\in A\m$ there exists a surjective map $X\rightarrow M$ with $X\in \add_A P$. Using (iv) any non-zero map $P(\l)\rightarrow M$ induces a non-zero map $\St(\mu)\rightarrow M$ for some standard module appearing in the filtration of $P(\l)$. Taking the direct sum of all generators of $\Hom_A(\sumSt P(\l), A)$ and using the previous fact, we can deduce our claim.

  Results on split quasi-hereditary algebras can be found for example in \citep{CLINE1990126, Rouquier2008, cruz2021cellular}. See also \citep{appendix}. We shall recall some facts about these algebras.
  
 Let $(A, \{\Delta(\lambda)_{\lambda\in \Lambda}\})$ be a split quasi-hereditary algebra. 
 We will write $\St$ to denote the set ${\{\St(\l)\colon \l\in \L \}}$ and $\Stsim$ to denote the set ${\{\St(\l)\otimes_R U_\l\colon \l\in \L, U_\l\in R\proj \}}$.  By $\Stsim_{\mu>\l}$, we will denote the set \linebreak${\{\St(\mu)\otimes_R U_\mu\colon \mu\in \L, \ \mu>\l, U_\mu\in R\proj \}}$. Analogously, we define the set $\Stsim_{\mu<\l}$.
The subcategory $\mathcal{F}(\Stsim)$ contains all finitely generated projective $A$-modules, and moreover it is a resolving subcategory of $A\m\cap R\proj$.
Since $\Ext_A^{i>0}(\St(\l), \St(\mu))\neq 0$ only if $\l<\mu$, filtrations of the modules of $\mathcal{F}(\Stsim)$ can be rearranged into a filtration of the form $0=M_{n+1}\subset M_n\subset \cdots\subset M_1=M$ with $ M_i/M_{i+1}\simeq \St_i\otimes_R U_i, $ for some $U_i\in R\proj,$ where $\L\rightarrow \{1, \ldots, n\}$, $\l\mapsto i_\l$, is an increasing bijection with $\St_{i_\l}:=\St(\l)$. 

The opposite algebra of a split quasi-hereditary algebra has again a split quasi-hereditary structure, therefore there exists for each $\l\in \L$ a costandard module $\Cs(\l)$ making $(A^{op}, \{D\Cs(\lambda)_{\lambda\in \Lambda}\})$ a split quasi-hereditary algebra. Fixing $\Cssim=\{\Cs(\l)\otimes_R U_\l\colon \l\in \L, \ U_\l\in R\proj \}$, we obtain \begin{align}
	\mathcal{F}(\Cssim)&=\{M\in A\m\cap R\proj\colon \Ext_A^1(X, M)=0, \ \forall X\in \mathcal{F}(\Stsim) \}\\&=\{M\in A\m\cap R\proj\colon \Ext_A^{i>0}(X, M)=0,  \ \forall X\in \mathcal{F}(\Stsim) \}.
\end{align}
An important aspect of split quasi-hereditary algebras is the existence of a \textbf{characteristic tilting module} $T=\bigoplus_{\l\in \L} T(\l)$ which is a full generalized tilting module satisfying $\add T=\mathcal{F}(\Stsim)\cap \mathcal{F}(\Cssim)$ and $\mathcal{F}(\Stsim)=\widecheck{\add T}$. In particular, there exist exact sequences of the form
$0\rightarrow \St(\l)\rightarrow T(\l)\rightarrow X(\l)\rightarrow 0$, with $X\in \mathcal{F}(\Stsim_{\mu<\l})$, $\l\in \L$ and $T(\l)\in \mathcal{F}(\Cssim)$.

\begin{Def}
	Given two split highest weight categories $(A\m, \{\Delta(\lambda)_{\lambda\in \Lambda}\})$, $(B\m, \{\Omega(\theta)_{\theta\in \Theta}\})$, we say that a functor $F\colon A\m\rightarrow B\m$ is an equivalence of highest weight categories if it satisfies the following:
	\begin{enumerate}[(i)]
		\item $F$ is an equivalence of categories.
		\item There is a bijection of posets $\Phi\colon \L\rightarrow \Theta$ and for each $\l\in \L$ there exists an invertible $R$-module $U_\l$ satisfying $F\St(\l)\simeq \Omega(\Phi(\l))\otimes_R U_\l$.
	\end{enumerate}
\end{Def}
By an \textbf{invertible} $R$-module we mean an $R$-module $M$ satisfying $M_\mi\simeq R_\mi$ for all $\mi\in \MaxSpec R$. We denote by $Pic(R)$ the set of isomorphism classes of invertible $R$-modules.

\begin{Theorem}\label{standardsgivesthewholealgebrapre}
	Two split highest weight categories $(A\m, \{\Delta(\lambda)_{\lambda\in \Lambda}\})$ and $(B\m, \{\Omega(\theta)_{\theta\in \Theta}\})$ are equivalent as split highest weight categories if and only if there exists an exact equivalence between $\mathcal{F}(\Stsim)$ and $\mathcal{F}(\tilde{\Omega})$.
\end{Theorem}
\begin{proof}
	If $R$ is a field we refer to \citep{zbMATH00140218}. For the general case, we refer to \citep[Theorem 6.1]{appendix}.
\end{proof}

\subsection{Covers}\label{Covers}

Let $P$ be a finitely generated projective $A$-module and $B$ the endomorphism algebra $\End_A(P)^{op}$. The functor $F:=\Hom_A(P, -)\colon A\M\rightarrow B\M$ is known as the \textbf{Schur functor}. Since $P\in A\proj$, the Schur functor is also isomorphic to the functor $FA\otimes_A -\colon A\M\rightarrow B\M$ (see for example Lemma 4.2.5 of \cite{Zimmerman}). Observe that the module $FA$ is a generator when viewed as a left $B$-module. We will denote by $G$ the right adjoint functor of $F$,
$G=\Hom_B(FA, -)\colon B\M\rightarrow A\M$. Since $R$ is a commutative Noetherian ring, the functors $\Hom_A(P, -)\colon A\m\rightarrow B\m$ and $\Hom_B(FA, -)\colon B\m\rightarrow A\m$ are well defined and form an adjoint pair as a consequence of the Tensor-Hom adjunction. In fact, for any $X\in A\m$, $FX\subset \Hom_R(P, X)$ which is finitely generated over $R$, in particular, $FX$ is finitely generated over $B$.

The unit of the adjunction $F\dashv G$ is the natural transformation $\eta\colon \id_{A\m}\rightarrow  G\circ F$ such that  for any module $N\in A\m$, the $A$-homomorphism \begin{align*}
	\eta_N\colon N \rightarrow \Hom_B(FA, \Hom_A(P, N)) \text{ is given by } \eta(n)(f)(p)=f(p)n, \ n\in N, \ f\in FA, p\in P.
\end{align*}

The counit of the adjunction $F\dashv G$ is the natural transformation $\varepsilon\colon F\circ G\rightarrow \id_{B\m}$ such that for any module $M\in B\m$, the $B$-homomorphism is given by the following commutative diagram
\begin{center}
	\begin{tikzcd}
		FA\otimes_A \Hom_B(FA, M)\arrow[r, "\varepsilon'_M"] \arrow[d, "\simeq"] & M\arrow[d, equal]\\
		\Hom_A(P, \Hom_B(FA, M))\arrow[r, "\varepsilon_M"]& M
	\end{tikzcd} where $\varepsilon_M'\colon FA\otimes_A \Hom_B(FA, M)\rightarrow M$ is given by $\varepsilon_M'(f\otimes g)=g(f), f\otimes g\in FA\otimes_A \Hom_B(FA, M).$
\end{center}

\begin{Prop}\label{fullfaithfulG}
	$G$ is fully faithful and $\varepsilon_M$ is a $B$-isomorphism for any $M\in B\m$.
\end{Prop}
\begin{proof}
	For any $M\in B\m$, there are canonical isomorphisms by Tensor-Hom adjunction,
	\begin{align*}
		\Hom_A(P, \Hom_B(FA, M))\simeq \Hom_B(FA\otimes_A P, M)\simeq \Hom_B(B, M)\simeq M.
	\end{align*}So, the counit $\varepsilon_M$ is a $B$-isomorphism for any $M\in B\m$. This fact implies that $G$ is full and faithful.
\end{proof}

In \citep{Rouquier2008}, the concept of cover was introduced. The pair $(A, P)$ is a \textbf{cover} of $B$ if the functor $F=\Hom_A(P, -)\colon A\m\rightarrow B\m$ is fully faithful on $A\proj$. 
Covers can also be characterized using double centralizer properties. To see that, we can observe that since the functors $F$ and $G$ are additive Proposition \ref{fullfaithfulG} implies the following result:

\begin{Lemma}\citep[Lemma 4.32]{Rouquier2008}\label{whenunitisiso}
	Let $M\in A\m$. The following assertions are equivalent.\begin{enumerate}[(a)]
		\item The unit $\eta_M\colon M\rightarrow GFM$ is an isomorphism.
		\item $F$ induces a bijection of abelian groups $\Hom_A(N, M)\rightarrow \Hom_B(FN, FM)$, $f\mapsto Ff$ for every $N\in A\m$.
		\item $F$ induces an isomorphism of $A$-modules $\Hom_A(A, M)\rightarrow \Hom_B(FA, FM)$, $f\mapsto Ff$.
		\item $M$ is a direct summand of a module in the image of $G$.
	\end{enumerate}
\end{Lemma}

Again since $F$ and $G$ are additive functors, and the counit $\eta_A$ gives the canonical homomorphism of $R$-algebras $A\rightarrow \End_B(FA)^{op}$, Lemma \ref{whenunitisiso} implies the following characterization of covers.

\begin{Prop}\citep[Proposition 4.33]{Rouquier2008} The following assertions are equivalent. \label{dcpropertycover}
	\begin{enumerate}[(i)]
		\item The canonical map of algebras $A\rightarrow \End_B(FA)^{op}$, given by $a\mapsto (f\mapsto f(-)a)$, $a\in A, f\in FA$, is an isomorphism of $R$-algebras, that is, there exists a double centralizer property on $FA$.
		\item For all $M\in A\proj$, the unit $\eta_M\colon M\rightarrow GFM$ is an isomorphism of $A$-modules.
		\item $(A, P)$ is a cover of $B$.
	\end{enumerate}
\end{Prop}

In particular, covers provide an abstract setup to deal with double centralizer properties on projective modules. Since $FA$ is a generator over $B$, if $(A, P)$ is a cover of $B$ we see that Proposition \ref{fullfaithfulG} is a special case of the Gabriel-Popescu theorem.

By a \textbf{split quasi-hereditary cover of $B$} we mean a cover $(A, P)$ of $B$ such that $(A, \{\Delta(\lambda)_{\lambda\in \Lambda}\})$  is a split quasi-hereditary algebra for some collection of standard modules. By a \textbf{quasi-hereditary cover of $B$} we mean a cover $(A, P)$ of $B$ such that $A$ is a  quasi-hereditary algebra.

Every finite-dimensional algebra over a field has a quasi-hereditary cover (see \cite{Dlab1989}). 

Moreover, there are a couple of general methods to construct quasi-hereditary covers. In fact, Iyama gave another construction of quasi-hereditary covers in \cite{zbMATH01849919, zbMATH02070262} to establish Iyama’s finiteness theorem. This construction has better properties than the construction established in \cite{Dlab1989}. In particular, we have the following result.

\begin{Theorem}Let $k$ be a field.
	Let $B$ be a finite-dimensional $k$-algebra. Then $B$ has a (not necessarily split) quasi-hereditary cover $(A, P)$ whose associated Schur functor $\Hom_A(P, -)$ is also faithful on $\mathcal{F}(\St)$.
\end{Theorem}
\begin{proof}
	By Theorem 5(2) of \cite{zbMATH05697051}, there is a left strongly quasi-hereditary algebra $A$ and an idempotent $e$ of $A$ such that $eA$ is a generator-cogenerator of $eAe=B$ and $A=\End_{eAe}(eA)$. Therefore, $(A, Ae)$ is a cover of $eAe=B$. Now, since $A$ is left strongly quasi-hereditary for a certain poset $\L$, there are for each $\l\in \L$,
	exact sequences
	\begin{align}
		0\rightarrow X(\l)\rightarrow P(\l)\rightarrow \St(\l)\rightarrow 0
	\end{align}with both $X(\l)$ and $P(\l)$ projective $A$-modules. Let $F=\Hom_A(Ae, -)$ and $G$ its right adjoint. Since, $(A, Ae)$ is a cover of $B$, $\eta_X$ is an isomorphism for every $X\in A\proj$. By the commutativity of the diagram 
	\begin{equation}
		\begin{tikzcd}
			0\arrow[r] & X(\l) \arrow[r] \arrow[d, "\eta_{X(\l)}"] & P(\l)\arrow[r] \arrow[d, "\eta_{P(\l)}"] & \St(\l)\arrow[r] \arrow[d, "\eta_{\St(\l)}"] & 0\\
			0\arrow[r] & GFX(\l)\arrow[r]& GFP(\l) \arrow[r] & GF\St(\l)
		\end{tikzcd}, 	\end{equation} and the Snake Lemma we deduce that $\eta_{\St(\l)}$ is a monomorphism for every $\l\in \L$.
\end{proof} 

Not every (split) quasi-hereditary cover has this extra property. For instance, in Example 14 of \cite{cruz2021characterisation}   it was constructed a split quasi-hereditary cover $(A, P)$ of an algebra isomorphic to the $2\times 2$ lower triangular matrix algebra where the standard modules of $A$ are labelled by $3>2>1$ and $\eta_{\St(1)}$ is the zero map. Also, we should mention that not all projective Noetherian algebras have split quasi-hereditary covers.
	\begin{Cor}Let $C_3$ be the abelian group of order 3. Let $\mathbb{Z}_7$ be the localization of $\mathbb{Z}$ at $7\mathbb{Z}$.
	The group algebra $\mathbb{Z}_7 C_3$ over $\mathbb{Z}_7$ does not have a split quasi-hereditary cover.\label{noteverygroupalgebrahassplitqhcover}
	Moreover, the group algebra $\mathbb{Z}C_3$ does not have a split quasi-hereditary cover.
\end{Cor}
\begin{proof}
	In \citep{zbMATH03382538}, it was shown that the ring $\mathbb{Z}_7 C_3$ is not semi-perfect. By \citep[Theorem 3.4.1]{cruz2021cellular}, every split quasi-hereditary algebra over $\mathbb{Z}_7$ is semi-perfect.  If $(A, P)$ is a split quasi-hereditary cover of $\mathbb{Z}_7C_3$ then Proposition 3.14 and Theorem 2.12 of \citep{MR3025306} yield that $P$ is a direct sum of $A$-modules with local endomorphism rings. Proposition 3.14 of \citep{MR3025306} would then imply that $\mathbb{Z}_7C_3=\End_A(P)^{op}$ is semi-perfect, which is a contradiction.
	Since any split quasi-hereditary cover remains a split quasi-hereditary cover under localization we obtain that $\mathbb{Z}C_3$ cannot have a split quasi-hereditary cover.
\end{proof}

Our focus will not be on studying the existence of (split) quasi-hereditary covers but instead on how to measure their quality, in the cases that we know that they exist.
Before we discuss properties and how to measure the quality of a split quasi-hereditary cover we need the following technical results.

The faithful part of Lemma \ref{whenunitisiso} is as follows.
\begin{Lemma}\label{considerationseta}
	Let $(A, P)$ be a cover of $B$.  Let $M\in A\m$. Then
the map $\eta_{M}$ is monomorphism if and only if $\Hom_A(N, M)\rightarrow \Hom_B(FN, FM)$ is injective for any $N\in A\m$ if and only if ${\Hom_A(A, M)\rightarrow \Hom_B(FA, FM)}$ is injective.
\end{Lemma}
\begin{proof}
	Assume that $\Hom_A(N, M)\rightarrow \Hom_B(FN, FM)$ is injective for any $N\in A\m$. In particular, \linebreak\mbox{$\Hom_A(A, M)\rightarrow \Hom_B(FA, FM)$} is injective. Let $m\in M$ such that $\eta_M(m)=0$.  Consider $f_m\in \Hom_A(A, M)$, given by $f_m(1_A)=m$. Then $Ff_m=\eta_M(m)=0$. Thus, $f_m=0$ and $m=0$. So, $\eta_M$ is a monomorphism. Now assume that $\eta_M$ is a monomorphism. Let $f\in \Hom_A(N, M)$ satisfying $Ff=0$. Then $
		\eta_M\circ f=GFf\circ \eta_N=0\implies f=0.$
	Thus, the statement follows.
\end{proof}
The following result clarifies the behaviour of the unit under a change of ground ring.

\begin{Lemma}\label{unitmono}
	Let $M\in A\m\cap R\proj$. If the unit $\eta_{M(\mi)}$ is a monomorphism for every maximal ideal $\mi$ in $R$, then the unit $\eta_M$ is $(A, R)$-monomorphism. If, in addition, $DM\otimes_A P\otimes_B\Hom_A(P, A)\in R\proj$ and  $\eta_M$ is $(A, R)$-monomorphism, then $\eta_{M(\mi)}$ is a monomorphism for every maximal ideal $\mi$ in $R$.
\end{Lemma}
\begin{proof}
	Let $\lambda_M\colon DM\otimes_A P\otimes_B \Hom_A(P, A)\rightarrow DM$, given by $\lambda_M(f\otimes p\otimes g)=fg(p)$ for \linebreak\mbox{$f\otimes p\otimes g\in DM\otimes_A P\otimes_B \Hom_A(P, A)$.}
	There is a commutative diagram
	\begin{equation}
		\begin{tikzcd}
			DDM \arrow[rr, "D\lambda_M"] & & D(DM\otimes_A P\otimes_B \Hom_A(P, A))\\
			&	& D(D\Hom_A(P, M)\otimes_B\Hom_A(P, A))\arrow[u, "D(\iota\otimes \Hom_A(P{,}A))"]\\
			M\arrow[rr, "\eta_M"]\arrow[uu, "w_M"]& &\Hom_B(\Hom_A(P, A), \Hom_A(P, M))\arrow[u, "\kappa"] \label{eqmsf20}
		\end{tikzcd} 
	\end{equation}where the isomorphism maps $\kappa$ and $\iota$ are according to \citep[Proposition 2.1]{CRUZ2022410}. 
	In fact, for $m\in M,$ \linebreak\mbox{$f\otimes p\otimes g\in DM\otimes_A P\otimes_B \Hom_A(P, A)$,}
	\begin{align*}
		D(\iota\otimes \Hom_A(P, A))\circ \kappa \circ \eta_M(m)(f\otimes p\otimes g)&= \kappa(\eta_M(m))(\iota\otimes \Hom_A(P, A))(f\otimes p \otimes g)\\
		&= \kappa(\eta_M(m))(\iota(f\otimes p)\otimes g)=\iota(f\otimes p)(\eta_M(m)(g))\\&=f(\eta_M(m)(g)(p))=f(g(p)m)\\
		D\lambda_M\circ w_M(m)(f\otimes p\otimes g)&=w_M(m)(\lambda_M(m)(f\otimes p \otimes g))=w_M(m)(f\cdot g(p))\\&=(f\cdot g(p))(m)=f(g(p)m).
	\end{align*}
	By assumption, $\eta_{M_{(\mi)}}$ is a monomorphism for every maximal ideal $\mi$ in $R$. According to the commutative diagram (\ref{eqmsf20}), $D_{(\mi)}\lambda_{M(\mi)}=\Hom_{R(\mi)}(\lambda_{M(\mi)}, R(\mi))$ is a monomorphism for every maximal ideal $\mi$ in $R$. Hence, $\lambda_{M(\mi)}$ is surjective for every maximal ideal in $R$. In view of the commutative diagram
	\begin{equation}
		\begin{tikzcd}
			DM\otimes_A P\otimes_B \Hom_A(P, A)(\mi)\arrow[r, "\lambda_M(\mi)"]\arrow[d, "\simeq"] & DM(\mi)\arrow[dd, "\simeq"]\\
			DM(\mi)\otimes_{A(\mi)}P(\mi)\otimes_{B(\mi)}\Hom_A(P, A)(\mi)\arrow[d, "\simeq"] &\\
			D_{(\mi)}M(\mi)\otimes_{A(\mi)}P(\mi)\otimes_{B(\mi)}\Hom_{A(\mi)}(P(\mi), A(\mi))\arrow[r, "\lambda_{M(\mi)}"]& D_{(\mi)}M(\mi)
		\end{tikzcd}, \label{eqmsf21}
	\end{equation}
	$\lambda_M(\mi)$ is surjective for every maximal ideal in $R$. By Nakayama's Lemma, $\lambda_M$ is surjective. As $DM\in R\proj$, $\lambda_M$ splits over $R$, so there is an $R$-homomorphism $t$ such that $t\circ D\lambda_M=\id_M$. Thus,
	\begin{align*}
		w_M^{-1}\circ t\circ D(\iota\otimes \Hom_A(P, A))\circ \kappa\circ \eta_M=w_M^{-1}\circ t \circ D\lambda_M\circ w_M=w_M^{-1}\circ w_M=\id_M.
	\end{align*}Hence, $\eta_M$ is an $(A, R)$-monomorphism.
	
	Conversely, assume that $\eta_M$ is an $(A, R)$-monomorphism and ${DM\otimes_A P\otimes_B \Hom_A(P, A)}\in R\proj$. In view of diagram (\ref{eqmsf20}), $D\lambda_M$ is an $(A, R)$-monomorphism. Then $DD\lambda_M$ is surjective. 
	
	As ${DM\otimes_A P\otimes_B \Hom_A(P, A)}\in R\proj$, the map $w_{DM\otimes_A P\otimes_B \Hom_A(P, A)}$ is an isomorphism and consequently, $\lambda_M$ is surjective. Applying the right exact functor $R(\mi)\otimes_R -$, we obtain by  diagram (\ref{eqmsf21}) that $\lambda_{M(\mi)}$ is surjective for every maximal ideal $\mi$ in $R$. By the first diagram, it follows that $\eta_{M(\mi)}$ is a monomorphism for every maximal ideal $\mi$ in $R$.
\end{proof}

Higher versions of Lemma \ref{whenunitisiso} and Lemma \ref{considerationseta} are encoded in the long exact sequence obtained by Grothendieck's Spectral sequence applied to the Schur functor $F$.

\begin{Lemma}\label{spectralsequencegrothendieck}
	Let $M\in A\m$.
	Suppose that either $\R^iG(FM)=0$ for $1\leq i\leq q$ or $q=0$. Then, for any $X\in A\m$, there are isomorphisms $\Ext_A^i(X, GFM)\simeq \Ext_B^i(FX, FM), \ 0\leq i\leq q$ and an exact sequence \begin{multline*}
		0\rightarrow \Ext_A^{q+1}(X, GFM)\rightarrow \Ext_B^{q+1}(FX, FM)\rightarrow \Hom_A(X, \R^{q+1}G(FM))\rightarrow \Ext_A^{q+2}(X, GFM)\\\rightarrow \Ext_B^{q+2}(FX, FM).
	\end{multline*}
\end{Lemma}
\begin{proof}See for example \citep[Proposition 3.1]{zbMATH05278765} and \citep[2.2]{Doty2004}. The case $q=0$ is obtained by applying the dual version of \citep[Lemma A.3]{CRUZ2022410} together with an example of a  Grothendieck's Spectral sequence, $E_2^{i, j}=\Ext_A^i(X, \Ext_B^j(FA, FM))$.
\end{proof}

\begin{Cor} \label{Kunnethdeformationresult}
	Let $P^{\bullet}$ be a flat cochain complex of $R$-modules $0\rightarrow P_0\rightarrow P_1\rightarrow \cdots$. Let $M$ be an $R$-module with flat $R$-dimension at most one. Then, for each integer $n\geq 0$, there exists an exact sequence
	\begin{align}
		0\rightarrow H^n(P^{\bullet})\otimes_R M\rightarrow H^n(P^{\bullet}\otimes_RM)\rightarrow \Tor_1^R(H^{n+1}(P^{\bullet}), M)\rightarrow 0.
	\end{align}
\end{Cor}
\begin{proof}
	See for example \citep[III, Lemma 2.1.2 (Universal coefficient theorem)]{zbMATH01527053}.
\end{proof}

The classification of simple $eAe$-modules in terms of Schur functors for a given finite-dimensional algebra $A$ over a field goes back to the work by Green and T. Martins \citep[Theorem 6.2g]{zbMATH03708660}.

\begin{Theorem}
	\label{simplemoduleseAe} Let $A$ be a finite-dimensional algebra over a field $k$.
	Suppose $\{V_\lambda\colon \lambda\in \Lambda \}$ is a full set of simple modules in $A\m$, indexed by a set $\Lambda$. Let \mbox{$\Lambda'=\{\lambda\in\Lambda\colon eV_{\lambda}\neq 0 \}.$} Then $\{eV_{\lambda}\colon \lambda\in \Lambda' \}$ is a full set of simple modules in $eAe\m$. 
	The simple $A/AeA$-modules are exactly the simple $A$-modules $S$ with $eS=0$.
\end{Theorem}

\subsection{Relative dominant dimension over Noetherian algebras}\label{Relative dominant dimension over Noetherian algebras}

In \cite{CRUZ2022410}, the author introduced the concept of relative dominant dimension for projective Noetherian algebras.

	\begin{Def}\label{relativedominantdef}
	Let $M\in A\m$. We say that $M$ has \textbf{relative dominant dimension at least $t\in \mathbb{N}$}  if there exists an $(A, R)$-exact sequence of 
	finitely generated left $A$-modules
	\begin{align}
		0\rightarrow M\rightarrow I_1\rightarrow\cdots \rightarrow I_t \label{eq1}
	\end{align}with $I_i$ being both $A$-projective and $(A, R)$-injective. If $M$ admits no such $(A, R)$-exact sequence, then we say that $M$ has relative dominant dimension zero. Otherwise, the relative dominant dimension of $M$ is the supremum of the set of all values $t$ such that an $(A, R)$-exact sequence of the form \ref{eq1} exists.
	We denote by $\domdim_{(A, R)} M$ the relative dominant dimension of $M$. We write $\domdim_A M$ whenever $R$ is a field and consequently when $A$ is a finite-dimensional algebra over a field.
\end{Def}

We will denote by $\domdim (A, R)$ the relative dominant dimension of the regular $A$-module $A$.
A left $A$-module is called \textbf{$(A, R)$-strongly faithful} if there exists an $(A, R)$-monomorphism $A\hookrightarrow M^t$ for some $t>0$. We say that a left $A$-module is  an \textbf{$(A, R)$-injective-strongly faithful module} if it is simultaneously $(A, R)$-injective and $(A, R)$-strongly faithful.
We call a triple $(A, P, V)$ a \textbf{relative QF3 $R$-algebra, or just RQF3 algebra} provided $P$ is an
$A$-projective $(A, R)$-injective-strongly faithful left $A$-module and $V$ is an
$A$-projective $(A, R)$-injective-strongly faithful right $A$-module.

The following characterisation of relative dominant dimension, known as Mueller's theorem in the classical case, will be crucial to our purposes.

\begin{Theorem}\label{Mullertheorem}
	Let $(A, P, V)$ be an RQF3 algebra and $B$ be the endomorphism algebra $\End_A(V)$.
	
	For any $X\in A\m\cap R\proj$, the following assertions hold.
	\begin{enumerate}[(a)]
		\item $\domdim_{(A, R)} X\geq 1$ if and only if the canonical map
		$\Phi_X\colon \Hom_{A^{op}}(V, DX)\otimes_B V\rightarrow DX$, given by $f\otimes v\mapsto f(v)$, is an epimorphism.
		\item If $\domdim_{(A, R)} X\geq 1$, then $\alpha_X\colon X\rightarrow \Hom_B(V, V\otimes_A X)$, given by $x\mapsto (v\mapsto v\otimes x)$, is an $(A,
		R)$-monomorphism. The converse holds if $\Hom_{A^{op}}(V, DX)\otimes_B V\in R\proj$.
		\item	For any natural number $n\geq 2$, the following assertions are equivalent:
		\begin{enumerate}[(i)]
		\item $\domdim_{(A, R)} X\geq n.$
		\item $\Phi_X\colon \Hom_A(V, DX)\otimes_B V\rightarrow DX$ is an isomorphism
and $\Tor_i^B(\Hom_{A^{op}}(V, DX), V)=0,$  \mbox{$1\leq i\leq n-2$.}
		\end{enumerate}
	\item Let $n$ be an integer greater than 1. If $\domdim_{(A, R)} X\geq n$, then $\alpha_X\colon X\rightarrow \Hom_B(V, V\otimes_A X)$ is an isomorphism and
	$\Ext_B^i(V, V\otimes_A X)=0$ for $1\leq i\leq n-2$.
	\item  Assume that $R$ is a commutative Noetherian regular ring and let $n$ be a natural number.  If \mbox{$\alpha_X\colon X\rightarrow \Hom_B(V, V\otimes_A X)$} is an isomorphism and $\Ext_B^i(V, V\otimes_A X)=0$ for every \mbox{$1\leq i\leq n-2$}, then $\domdim_{(A, R)} X\geq n-\dim R$.
	\end{enumerate}
\end{Theorem}
\begin{proof}
	Assertions (a) and (b) are contained in \citep[Proposition 3.23]{CRUZ2022410}, (c) and (d) follow from  \citep[Proposition 3.23, Theorem 5.2]{CRUZ2022410} while (e) is \citep[Proposition 5.5]{CRUZ2022410}.
\end{proof}

The following technical lemma is useful to relate covers with relative dominant dimension.

\begin{Lemma}\label{unitschurfunctoranddominant}
	Let $(A, P, V)$ be an RQF3 algebra and $B$ be the endomorphism algebra $\End_A(V)$.
	Let $F$ be the Schur functor  $\Hom_A(\Hom_{A^{op}}(V, A), -)\colon A\m \rightarrow B\m$.
	
	For any $X\in A\m$, there exists an isomorphism $\beta_X\in \Hom_A(\Hom_B(V,
	V\otimes_A X), \Hom_B(FA, FX))$ making the following diagram commutative: 
	\begin{center}
		\begin{tikzcd}
			X\arrow[d, equal]\arrow[r, "\alpha_X"] & \Hom_B(V, V\otimes_A X) \arrow[d,
			"\beta_X"] \\
			X\arrow[r, "\eta_X"] & \Hom_B(FA, FX)
		\end{tikzcd}
	\end{center}
\end{Lemma}
\begin{proof}
	Denote by $w_V$ the map $V\rightarrow \Hom_A(\Hom_{A^{op}}(V, A), A)$, given by
	$w(v)(f)=f(v)$. Since $V$ is a projective $A$-module, this map is an
	$(\End_{A^{op}}(V), A)$-bimodule isomorphism.
	
	Fix $\psi_X\colon \Hom_A(\Hom_{A^{op}}(V, A), A)\otimes_A X\rightarrow
	\Hom_A(\Hom_{A^{op}}(V, A), X)$ the canonical isomorphism given by $\Hom_{A^{op}}(V, A)$ being a projective $A$-module (see for example Lemma 4.2.5 of \cite{Zimmerman}). Then
	define $\beta_X=\Hom_B(FA, \psi_X\circ w_V\otimes \id_X)\circ \Hom_B(w_V^{-1},
	V\otimes_A X)$. Let $x\in X$. Then
	\begin{align}
		\Hom_B(FA, \psi_X\circ w_V\otimes \id_X)\circ \Hom_B(w_V^{-1}, V\otimes_A
		X)(\alpha_X(x))&=\Hom_B(FA, \psi_X\circ w_V\otimes \id_X)(\alpha_X(x)\circ
		w_V^{-1}) \nonumber\\
		&= \psi_X\circ w_V\otimes \id_X \circ \alpha_X(x)\circ w_V^{-1}.
	\end{align} For $v\in V, \ f\in \Hom_{A^{op}}(V, A)$,
	\begin{align*}
		\psi_X\circ w_V\otimes \id_X\circ \alpha_X(x)(v)(f)= \psi_X\circ w_V\otimes
		\id_X(v\otimes x)(f)=\psi_X(w_V(v)\otimes x)(f)=w_V(v)(f)x=f(v)x.
	\end{align*} 
	On the other hand,
	$
	\eta_X(x)\circ w_V(v)(f)=w_V(v)(f)x=f(v)x.
	$
	Therefore, composing with $w_V^{-1}$ on both sides we conclude
	\begin{align}
		\Hom_B(FA, \psi_X\circ w_V\otimes \id_X)\circ \Hom_B(w_V^{-1}, V\otimes_A
		X)(\alpha_X(x)) = \eta_X(x), \ x\in X. \tag*{\qedhere}
	\end{align}
\end{proof}

\begin{Prop}
	Let $(A, P, V)$ be an RQF3 algebra over a commutative Noetherian ring $R$. If \linebreak$\domdim (A, R)\geq 2$, then $(A, \Hom_A(V, A))$ is a cover of $B:=\End_{A}(V)$.\label{dominantgeqtwodcp}
\end{Prop}
\begin{proof}
	Since $\domdim(A, R)\geq 2$ then $\alpha_A\colon A\rightarrow \Hom_B(V, V)$ is an isomorphism by Theorem \ref{Mullertheorem}. By Lemma \ref{dcpropertycover}, it follows that $(A, \Hom_{A^{op}}(V, A))$ is a cover of $B=\End_A(\Hom_{A^{op}}(V, A))^{op}\simeq \End_A(V)$.
\end{proof}

\begin{Remark}\label{gendocovers}
	It is essential to consider the projective $\Hom_{A^{op}}(V, A)$ instead of $P$. Indeed, in \citep[Example 15]{cruz2021characterisation}, we see that there are examples of algebras with dominant dimension two with a projective-injective-faithful module $P$ but the pair $(A, P)$ fails to be a cover of $\End_A(P)^{op}$.
\end{Remark}

Given Remark \ref{gendocovers}, we could ask in what situations $(A, P)$ is a cover of $\End_A(P)^{op}$ for a given RQF3 algebra $(A, P, V)$. 
It turns out that this property (over finite-dimensional algebras) characterizes Morita algebras, that is, endomorphism algebras of generators over self-injective algebras.

\begin{Theorem}\citep[Theorem 1]{cruz2021characterisation} Let $k$ be a field. \label{gendocharaccovers}
	Let $(A, P, V)$ be a QF3 $k$-algebra. Then $(A, P)$ is a cover of $\End_A(P)^{op}$ if and only if $A$ is a Morita algebra.
\end{Theorem}

Recall that a pair $(A, P)$ is a \textbf{relative Morita algebra} over a commutative Noetherian ring if $(A, P, DP)$ is an RQF3 algebra so that $\domdim (A, R)\geq 2$ and $\add DA\otimes_A P=\add P$. 
\begin{Remark}
	It remains true that if $(A, P)$ is a relative Morita algebra over a commutative Noetherian ring $R$, then $(A, P)$ is a cover of $B$.
\end{Remark}

\section{Introduction to $\mathcal{A}$-covers}\label{faithful covers}

In this section, we give the setup to measure the quality of a cover based on the approach of \citep{Rouquier2008}. In particular, we introduce the concept of an $\mathcal{A}$-cover $(A, P)$ for an arbitrary resolving subcategory $\mathcal{A}$ of $A\m$. Under this concept, faithful split quasi-hereditary covers are exactly $\mathcal{F}(\Stsim)$-covers.
Allowing arbitrary resolving subcategories offers the possibility to assign a quality to a cover which is not necessarily quasi-hereditary and to distinguish the covers which are not even $(-1)$-faithful.

\begin{Def}\label{faithfulcoverdef}
	Let $A$ be a projective Noetherian $R$-algebra. Let $\mathcal{A}$ be a resolving subcategory of $A\m$. Let $B=\End_A(P)^{op}$ for some $P\in A\proj$ and let $i\geq 0$. We say that the pair $(A, P)$ is an \textbf{$i$-$\mathcal{A}$ cover} of $B$ if the Schur functor $F=\Hom_A(P, -)$ induces isomorphisms 
	\begin{align*}
		\Ext_A^j(M, N)\rightarrow \Ext_B^j(FM, FN), \quad \forall M, N\in \mathcal{A}, \ j\leq i.
	\end{align*}\par
	We say that $(A, P)$ is a \textbf{$i$-cover }of $B$ if $(A, P)$ is an $i$-$(A\proj)$ cover of $B$. \par
	We say that $(A, P)$ is an \textbf{$(-1)$-$\mathcal{A}$ cover} of $B$ if $(A, P)$ is a cover of $B$ and $F$ induces monomorphisms 
	\begin{align*}
		\Hom_A(M, N)\rightarrow\Hom_B(FM, FN), \quad \forall M, N\in \mathcal{A}.
	\end{align*}
We say that $(A, P)$ is an $(+\infty)$-$\mathcal{A}$ cover of $B$ if $(A, P)$ is an $i$-$\mathcal{A}$ cover of $B$ for all $i\geq 0$.
\end{Def}

Let $(A, \{\Delta(\lambda)_{\lambda\in \Lambda}\})$ be a split quasi-hereditary algebra. As we have mentioned $\mathcal{F}(\Stsim)$ is a resolving subcategory of $A\m\cap R\proj$ (see for example \citep[Theorem 4.1]{appendix}). 
In particular, we can see that the concept of an $i$-$\mathcal{F}(\Stsim)$ cover coincides with the concept of $i$-faithful cover introduced in \citep[Definition 4.37]{Rouquier2008}.

\begin{Remark}
	In our notation, a $0$-cover is a cover in the usual sense.
\end{Remark}

We can see that Lemma \ref{spectralsequencegrothendieck} gives that a cover $(A, P)$ of $B$ is a $(-1)$-$\mathcal{A}$ cover of $B$ if and only if the restriction of $\Hom_A(P, -)$ to $\mathcal{A}$ is faithful if and only if $\eta_M$ is a monomorphism for every $M\in \mathcal{A}$.

\begin{Prop}\label{zeroAcover}
	The following assertions are equivalent.
	\begin{enumerate}[(a)]
		\item $(A, P)$ is a $0$-$\mathcal{A}$ cover; that is, the restriction of $F=\Hom_A(P, -)$ to $\mathcal{A}$ is full and faithful.
		\item $\eta_M$ is an isomorphism for all $M\in \mathcal{A}$.
		\item Every module of $\mathcal{A}$ is in the image of the functor $G=\Hom_B(FA, -)$.
	\end{enumerate}
\end{Prop}
\begin{proof}
	$(a)\implies (b)$. Since $\mathcal{A}$ is resolving of $A\m\cap R\proj$, $A\in \mathcal{A}$. By $a)$, \linebreak$\Hom_A(A, M)\rightarrow \Hom_B(FA, FM)$ is an isomorphism. By Lemma \ref{considerationseta}, $\eta_M$ is an isomorphism.
	
	$(b)\implies (c)$. Let $M\in \mathcal{A}$. By assumption, $\eta_M$ is an isomorphism. Hence, $M\simeq G(FM)$. 
	
	$(c)\implies (b)$. Let $M\in \mathcal{A}$. There exists $N\in B\m$ such that $GN\simeq M$. Since $\id_{GN}=G\varepsilon_N\circ \eta_{GN}$ and $\varepsilon_N$ is an isomorphism according to Proposition \ref{fullfaithfulG}, it follows that $\eta_{GN}$ is an isomorphism. Let $\alpha\colon M\rightarrow GN$ be an isomorphism. As $\eta_M$ is the composition of the isomorphisms $GF\alpha^{-1}\circ \eta_{GN}\circ \alpha$, it is an isomorphism.
	
	$(b)\implies (a)$. By Lemma \ref{whenunitisiso}, $\Hom_A(M, N)\rightarrow \Hom_B(FN, FM)$ for every $N, M\in \mathcal{A}$.
\end{proof}

\begin{Prop}\label{arbitraryAcover}
	Let $(A, P)$ be a $0$-$\mathcal{A}$ cover of $B$. Let $i\geq 1$. The following assertions are equivalent.
	\begin{enumerate}[(a)]
		\item $(A, P)$ is an $i$-$\mathcal{A}$ cover of $B$.
		\item For all $M\in \mathcal{A}$, we have $\R^jG(FM)=0$, $1\leq j\leq i$.
	\end{enumerate}
\end{Prop}
\begin{proof}
	$(a)\implies (b)$. Let $M\in \mathcal{A}$. Let $1\leq j\leq i$. Then
	\begin{align}
		\R^jG(FM)=\R^j \Hom_B(FA, -) (FM)=\Ext_B^j(FA, FM)\simeq \Ext_A^j(A, M)=0.
	\end{align}
	
	$(b)\implies (a)$. Let $M\in \mathcal{A}$. By assumption, $\R^jG(FM)=0$, $1\leq j\leq i$. By Lemma \ref{spectralsequencegrothendieck}, \linebreak\mbox{$\Ext_A^j(X, GFM)\simeq \Ext_B^j(FX, FM)$,} $0\leq j\leq i$ for any $X\in A\m$. Since $(A, P)$ is a $0$-$\mathcal{A}$ cover of $B$, $\eta_M\colon M\rightarrow GFM$ is an $A$-isomorphism, and thus we have
	\begin{align}
		\Ext_A^j(X, M)\simeq \Ext_B^j(FX, FM), \quad 0\leq j\leq i, \ \forall X\in A\m.
	\end{align}The choice of $M\in \mathcal{A}$ is arbitrary, hence $(a)$ follows.
\end{proof}

\subsection{Faithful (split quasi-hereditary) covers}

The most prominent example of $\mathcal{A}$-covers is the faithful split quasi-hereditary covers. We shall now recall what Propositions \ref{zeroAcover} and \ref{arbitraryAcover} translate into.

\begin{Prop}\label{m1faithfulcover}
	Let $(A, P)$ be a cover of $B$ and $(A, \{\Delta(\lambda)_{\lambda\in \Lambda}\})$ be a split quasi-hereditary algebra over a commutative Noetherian ring. The following assertions are equivalent.
	\begin{enumerate}[(i)]
		\item $(A, P)$ is a $(-1)$-faithful split quasi-hereditary cover of $B$; that is the restriction of $F=\Hom_A(P, -)$ to $\mathcal{F}(\Stsim)$ is faithful. 
		\item $\eta_{\sumSt \St(\l)}$ is a monomorphism.
		\item $\eta_{\St(\l)}$ is a monomorphism for all $\l\in \L$.
		\item $\eta_{M}$ is a monomorphism for all $M\in \mathcal{F}(\Stsim)$.
		\item $\eta_T$ is a monomorphism for all (partial) tilting modules $T$.
		\item Every module of $\mathcal{F}(\Stsim)$ can be embedded into some module in the image of the functor $G=\Hom_B(FA, -)$. 
	\end{enumerate}
\end{Prop}
\begin{proof}
Implications	$(ii)\implies (iii)$ and 	$(iv)\implies (vi)$ are clear.

	$(i)\implies (ii)$. $A\in \mathcal{F}(\Stsim)$ and clearly $\sumSt \St(\l)\in \mathcal{F}(\Stsim)$. 
	
	In view of $(i)$, $\Hom_A(A, \sumSt \St(\l))\rightarrow \Hom_B(FA, F\sumSt \St(\l))$ is injective. By Lemma \ref{considerationseta}, $\eta_{\sumSt \St(\l)}$ is a monomorphism.

	$(iii)\implies (iv)$.  By induction on the size of a filtration of $M\in \mathcal{F}(\Stsim)$ and using the Snake Lemma, it follows that $\eta_M$ is a monomorphism for all $M\in \mathcal{F}(\Stsim)$.

	$(vi)\implies (v)$. Every (partial) tilting module belongs to $\mathcal{F}(\Stsim)\cap \mathcal{F}(\Cssim)$. In particular, it belongs to $\mathcal{F}(\Stsim)$. Thus, given a (partial) tilting $T$, there exists a monomorphism $\alpha\colon T\rightarrow GN$ for some $N\in B\m$. Since $\id_{GN}=G\varepsilon_N\circ \eta_{GN}$ and $\varepsilon_N$ is an isomorphism according to Proposition \ref{fullfaithfulG}, it follows that $\eta_{GN}$ is an isomorphism. Now, $GF\alpha\circ \eta_T=\eta_{GN}\circ \alpha$ is a monomorphism. Thus, $\eta_T$ is a monomorphism.

	$(v)\implies (iv)$. Let $M\in \mathcal{F}(\Stsim)$. Since $\mathcal{F}(\Stsim)=\widecheck{\add_A T}$ and $\mathcal{F}(\Stsim)$ is closed under kernels of epimorphisms, there exists a  (partial) tilting module $T$, $N\in \mathcal{F}(\Stsim)$ and an exact sequence $
	0\rightarrow M\rightarrow T \rightarrow N\rightarrow 0.
	$ Applying $G\circ F=\Hom_B(FA, F-)$ (left exact functor) yields the following commutative diagram with exact rows
	\begin{center}
		\begin{tikzcd}
			0\arrow[r]& M\arrow[r]\arrow[d, "\eta_M"]& T\arrow[r]\arrow[d, "\eta_T"] & N\arrow[r]\arrow[d, "\eta_N"]& 0\\
			0\arrow[r]& GFM\arrow[r] & GFT\arrow[r]& GFN &
		\end{tikzcd}.
	\end{center}
	By assumption, $\eta_T$ is a monomorphism. By Snake Lemma, $\eta_M$ is a monomorphism.
	
	$(i)\Leftrightarrow (iv)$. By Lemma \ref{considerationseta}, $\eta_M$ is monomorphism for every $M\in \mathcal{F}(\Stsim)$ if and only if the functor $F_{|_{\mathcal{F}(\Stsim)}}$ is faithful.
\end{proof}

\begin{Prop}\label{zerofaithfulcover}
	The following assertions are equivalent.
	\begin{enumerate}[(i)]
		\item $(A, P)$ is a $0$-faithful split quasi-hereditary cover of $B$; that is, the restriction of $F=\Hom_A(P, -)$ to $\mathcal{F}(\Stsim)$ is full and faithful.
		\item $\eta_{\sumSt \St(\l)}$ is an isomorphism.
		\item $\eta_{\St(\l)}$ is an isomorphism for all $\l\in \L$.
		\item $\eta_{M}$ is an isomorphism for all $M\in \mathcal{F}(\Stsim)$.
		\item Every module of $\mathcal{F}(\Stsim)$ is in the image of the functor $G=\Hom_B(FA, -)$.
		\item $\eta_T$ is an isomorphism for all (partial) tilting modules $T$.
		\item Let $T$ be a characteristic tilting module. Every module of $\add T$ is in the image of the functor \linebreak\mbox{$G=\Hom_B(FA, -)$}.
	\end{enumerate}
\end{Prop}
\begin{proof}
	See for example \citep[Proposition 4.40]{Rouquier2008}.
	This is an application of Proposition \ref{zeroAcover} using analogous arguments to Proposition \ref{m1faithfulcover}.
\end{proof}

If we know that $(A, P)$ is a split quasi-hereditary cover of $B$, testing that it is a $0$-faithful (split quasi-hereditary) cover amounts to check that the counit is an epimorphism on standard modules.

\begin{Prop}\label{coverisfaithfulwhenepi}
	Let $(A, P)$ be a cover of $B$. Then $(A, P)$ is a $0$-faithful split quasi-hereditary cover of $B$ if and only if $\eta_{\St(\l)}$ is an epimorphism for all $\l\in \L$.
\end{Prop}
\begin{proof}
	By Proposition \ref{zerofaithfulcover}, one implication is clear.
	
	Assume that $\eta_{\St(\l)}$ is an epimorphism for all $\l\in \L$. We claim that $\eta_M$ is an epimorphism for all $M\in \mathcal{F}(\Stsim)$. We will prove it by induction on the size of filtration of $M$, $t$.
	If $t=1$, then $M\simeq \St(\l)\otimes_R U_\l$ for some $\l$ and $U_\l\in R\proj$. So, $\St(\l)\otimes_R U_\l\in \add_A \St(\l)$ It follows that $\eta_{\St(\l)\otimes_R U_\l}$ is surjective. Assume $t>1$. 
	There is a commutative diagram with exact rows
	\begin{center}
		\begin{tikzcd}
			0\arrow[r] & M'\arrow[r] \arrow[d, "\eta_{M'}"]&M \arrow[r] \arrow[d, "\eta_M"] & \St(\mu)\otimes_R U_\mu \arrow[r]\arrow[d, "\eta_{\St(\mu)\otimes_R U_\mu}"]& 0\\
			0 \arrow[r]& GFM'\arrow[r] &GFM\arrow[r] & GF(\St(\mu)\otimes_R U_\mu) &
		\end{tikzcd}.
	\end{center}
	By induction, $\eta_{M'}$ is an epimorphism. By Snake Lemma, $\eta_M$ is an epimorphism. Now consider the commutative diagram
	\begin{center}
		\begin{tikzcd}
			0\arrow[r] & K(\l)\arrow[r] \arrow[d, "\eta_{K(\l)}"]&P(\l) \arrow[r] \arrow[d, "\eta_{P(\l)}"] & \St(\l) \arrow[r]\arrow[d, "\eta_{\St(\l)}"]& 0\\
			0 \arrow[r]& GFK(\l)\arrow[r] &GFP(\l)\arrow[r] & GF\St(\l) &
		\end{tikzcd}.
	\end{center} Since $K(\l)\in \mathcal{F}(\Stsim)$, $\eta_{K(\l)}$ is an epimorphism. By Snake Lemma, there is an exact sequence
	\begin{align}
		0=\ker \eta_{P(\l)}\rightarrow \ker \eta_{\St(\l)}\rightarrow \coker \eta_{K(\l)}=0.
	\end{align}It follows that $ \eta_{\St(\l)}$ is also a monomorphism, and thus $ \eta_{\St(\l)}$ is an isomorphism for every $\l\in \L$. By Proposition \ref{zerofaithfulcover}, the result follows.
\end{proof}

\begin{Prop}\label{onefaithfulcover}
	Let $(A, P)$ be a $0$-faithful split quasi-hereditary cover of $B$. The following assertions are equivalent.
	\begin{enumerate}[(a)]
		\item $(A, P)$ is a $1$-faithful split quasi-hereditary cover of $B$.
		\item $F=\Hom_A(P, -)$ restricts to an exact equivalence of categories $\mathcal{F}(\Stsim)\rightarrow\mathcal{F}(F\Stsim)$ with inverse the exact functor $G_{|_{\mathcal{F}(F\Stsim)}}=\Hom_B(FA, -)_{|_{\mathcal{F}(F\Stsim)}}$.
		\item For all $M\in \mathcal{F}_A(\Stsim)$, we have $\R^1G(FM)=0$.
	\end{enumerate}
\end{Prop}
\begin{proof}
	See for example \citep[Proposition 4.41]{Rouquier2008}.
\end{proof}

\begin{Prop}\label{arbitraryfaithfulcover}
	Let $(A, P)$ be a $0$-faithful split quasi-hereditary cover of $B$. Let $i\geq 1$. The following assertions are equivalent.
	\begin{enumerate}[(a)]
		\item $(A, P)$ is an $i$-faithful split quasi-hereditary cover of $B$.
		\item For all $M\in \mathcal{F}(\Stsim)$, we have $\R^jG(FM)=0$, $1\leq j\leq i$.
		\item For all $\l\in\L$, we have $\R^jG(F\St(\l))=0$, $1\leq j\leq i$.
	\end{enumerate}
\end{Prop}
\begin{proof}
	$(a)\Leftrightarrow (b)$ is given by Proposition \ref{arbitraryAcover}. The implication $(b)\implies (c)$ is also clear.
	
	Assume that $(c)$ holds. Let $M\in \mathcal{F}(\Stsim)$. There is a filtration
	\begin{align}
		0=M_{n+1}\subset M_n\subset M_{n-1}\subset \cdots \subset M_1=M, \quad M_i/M_{i+1}\simeq \St_i\otimes_R U_i, \quad 1\leq i\leq n.
	\end{align}
	We claim that $\R^j G(FM_t)=0$, for $t=1, \ldots, n$, $1\leq j\leq i$. We will prove it by induction on $n-t+1$. Assume that $n-t+1=1$. Let $1\leq j\leq i$. Then $\R^jG(FM_t)=\R^jG(F(\St_t\otimes_R U_t))$ is an $R$-summand of $\R^jG(F\St_t)^s=0$ for some $s>0$ since $U_t\in R\proj$. Thus, $ \R^jG(FM_t)=0$. Moreover, $\R^jG(F(\St_i)\otimes_R U_i)=0$ for every $i=1, \ldots, n$. Assume that the claim holds for $s>t$ for some $n\geq t>1$. 
	Consider the exact sequence
$
		0\rightarrow M_{t+1}\rightarrow M_t\rightarrow \St_t\otimes_R U_t\rightarrow 0.
$ Applying the left exact functor $G\circ F$ yields the exact sequence
	\begin{align}
		\R^jG(FM_{t+1})\rightarrow \R^jG(FM_t)\rightarrow \R^jG(F\St_t\otimes_R U_t)=0.
	\end{align}By induction, $\R^jG(FM_{t+1})=0$, hence $\R^jG(FM_t)=0$. Therefore, $(b)$ follows.
\end{proof}

Hence, the quality of a $0$-faithful split quasi-hereditary cover is given by the value $$n(\mathcal{F}(\Stsim))=\sup\{i\in \mathbb{N}_0\colon \R^jG(F\St(\l))=0, \ \l\in\L, \ 1 \leq j\leq i \}.$$ 

This motivates the following definition:

\begin{Def}\label{Hemmernakanodimension}
	Let $(A, P)$ be a cover of $B=\End_A(P)^{op}$. 
	Let $\mathcal{A}$ be a resolving subcategory of $A\m\cap R\proj$. The\textbf{ Hemmer--Nakano dimension} of $\mathcal{A}$ (with respect to $P$) is the maximal number $n$ such that $(A, P)$ is an $n$-$\mathcal{A}$ cover of $B$. We will denote it by $\HN_F \mathcal{A}$, where $F$ denotes the functor $\Hom_A(P, -)$. 
\end{Def}

We also say the Hemmer--Nakano dimension of $\mathcal{A}$ (with respect to the functor $F=\Hom_A(P, -)$).
When there is no confusion about the functor $F$, we will just call $\HN_F \mathcal{A}$ the Hemmer-Nakano dimension of $\mathcal{A}$. If $(A, P)$ is not a $(-1)$-$\mathcal{A}$ cover of $B$, then we say that the Hemmer-Nakano dimension of $\mathcal{A}$ (with respect to $P$) is $-\infty$.

\subsection{Elementary properties of $\mathcal{A}$-covers}

$0$-$\mathcal{A}$ covers can help us understand the indecomposable objects in $B\m$ using the indecomposable modules of $\mathcal{A}$.

\begin{Prop}
	Let $(A, P)$ be a $0$-$\mathcal{A}$ cover of $B$ for some resolving subcategory $\mathcal{A}$ of $A\m$.
	Then the Schur functor $F=\Hom_A(P, -)$ preserves the indecomposable objects of $\mathcal{A}$.
\end{Prop}
\begin{proof}
	Let $M\in \mathcal{A}$ be an indecomposable module. Assume that we can write $FM\simeq X_1\oplus X_2$. Then \begin{align}
		M\simeq GFM\simeq GX_1\oplus GX_2.
	\end{align}So, either $GX_1=0$ or $GX_2=0$. Since $FA$ is a $B$-generator, there must exist a non-zero epimorphism $FA^t\rightarrow X_1$ for some $t>0$ if $X_1$ is non-zero. So, if $X_1\neq 0$, then $GX_1\neq 0$. Thus, $FM$ is indecomposable.
\end{proof}

The study of $i$-$\mathcal{A}$ covers of finite-dimensional algebras over a field with $i\geq 0$ can be reduced to covers coming from idempotents.

\begin{Prop}\label{idempotentsAcovers}
	Let $R$ be a field and let $i\geq 0$ be an integer. Let $\mathcal{A}$ be a resolving subcategory of $A\m$.
	If $(A, P)$ is an $i$-$\mathcal{A}$ cover of $B$, then there exists an idempotent $e\in A$ such that $(A, Ae)$ is an $i$-$\mathcal{A}$ cover of $eAe$.
\end{Prop}
\begin{proof}
There exists an idempotent $e\in A$ such that $(A, Ae)$ is a cover of $eAe$ and $eAe$ is Morita equivalent to $B$ (see for example \citep[Proposition 9]{cruz2021characterisation}). Denote by $H$ the equivalence of categories ${B\m\rightarrow eAe\m}$. For $M, N\in \mathcal{A}$,
	\begin{align}
		\Ext_A^j(M, N)\simeq \Ext_B^j(FM, FN)\simeq \Ext_{eAe}^j(HFM, HFN), \ j\leq i.
	\end{align}It remains to show that $HFM\simeq \Hom_A(Ae, M)$ for every $M\in \mathcal{A}$. But, this isomorphism holds since $(A, P)$ is a $0$-$\mathcal{A}$ cover of $B$. Thus,  $(A, Ae)$ is an $i$-$\mathcal{A}$ cover of $eAe$.
\end{proof}

We should remark that the exact equivalence in Proposition \ref{onefaithfulcover} does not make the image of the resolving subcategory $\mathcal{A}$ under the Schur functor a resolving subcategory in $B\m\cap R\proj$. In fact, this only occurs when the Schur functor is an equivalence of categories.

\begin{Prop}\label{resolvingsubcategoryonecover}
	Let $\mathcal{A}$ be a resolving subcategory of $A\m$. Let $(A, P)$ be a $1$-$\mathcal{A}$ cover of $B$. Assume that $\{FM\colon M\in \mathcal{A} \}$ is a resolving subcategory of $B\m$. Then $F=\Hom_A(P, -)$ is an exact equivalence.
\end{Prop}
\begin{proof}
	Consider the projective $B$-presentation \begin{align}
		\delta\colon 0\rightarrow K\rightarrow Q\rightarrow FA\rightarrow 0.
	\end{align}By projectivization, $Q=FX $ for some $X\in \add P$, and consequently $X\in \mathcal{A}$. Because $\{FM\colon M\in \mathcal{A} \}$ is a resolving subcategory, there exists $N\in \mathcal{A}$ such that $K\simeq FN$. Hence, 
	\begin{align}
		\delta\in \Ext_B^1(FA, K)\simeq \Ext_A^1(A, N)=0.
	\end{align}Therefore, $FA$ is a $B$-summand of $Q$. Thus, $FA\in B\proj$. As we have seen before, since $(A, P)$ is a cover of $B$, this implies that $F$ is an exact equivalence.
\end{proof}
Observe that $\{FM\colon M\in \mathcal{A} \}$ being a resolving subcategory of $B\m$ is not a sufficient condition for $F$ to be an equivalence of categories. As an example, one may think about the split quasi-hereditary cover $(S_{\mathbb{F}_2}(2, 2), \mathbb{F}_2^{\otimes 2})$ of $\mathbb{F}_2 S_2$.

\section{Upper bounds for the quality of an $\mathcal{A}$-cover}\label{Upper bounds for the quality of an A-cover}

As Proposition \ref{onefaithfulcover} suggests, the better the quality of a cover $(A, P)$ of $B$ the better the connection between $A\m$ and $B\m$. In this section, this slogan is made precise when $A$ is of finite global dimension.
\subsection{$\mathcal{F}(\St)$}\label{mathcalFSt}
For finite-dimensional algebras over fields, there is an upper bound for the level of faithfulness of a split quasi-hereditary cover. To realize that we need the concept of length in a poset.

For any $\l\in \L$, we define the \textbf{length of $\l\in \L$} to be the length $t$ of the longest chain \linebreak${\l=x_0<x_1<\ldots<x_t}$ in $\L$ and denote it by $d(\L, \l)$. Denote by $d(\L)$ to be the maximum value of $d(\L, \l)$ over all $\l\in \L$. 
Note that $\l\in \L$ is maximal if and only if $d(\L, \l)=0$ and $d(\L)$ is bounded by $|\L|$. Also, if $z>\l$, then $d(\L, \l)\geq $ $d(\L, z)+1$.

\begin{Theorem}
	\label{rigiditypartI}
	Let $R$ be a field and let $(A, \{\Delta(\lambda)_{\lambda\in \Lambda}\})$ be a split quasi-hereditary algebra over $R$. Let \mbox{$\{S(\l)\colon \l\in \L\}$} be a complete set of non-isomorphic simple $A$-modules.
	
	We shall denote by $\L^*$ the set $\{\l\in \L\colon FS(\l)\neq 0\}$. 
	
	If $(A, P)$ is a split quasi-hereditary $(d(\L^*)+1)$-faithful cover of $B\m$, then the Schur functor induces by restriction to $A\proj$ the functor $F_{|_{A\proj}}\colon A\proj \rightarrow B\proj$.
	Moreover, $F$ is an equivalence of categories.
\end{Theorem}
\begin{proof}  Since $R$ is a field we can assume, without loss of generality, that there exists an idempotent $e\in A$ such that $P=Ae$ and $B=eAe$.
	By Theorem \ref{simplemoduleseAe}, the simple $B$-modules can be written in the form $FS$ where $S$ is a simple $A$-module. Moreover, the simple $B$-modules are indexed by $\L^*$. The set $\L^*$ is again a poset where its partial order is the one induced by the poset $\L$.
	
	Consider $M$ a finitely generated projective $A$-module. We want to show that $FM$ is a projective $B$-module. It is enough to show that $\textrm{Ext}_B^1(FM, S)=0$ for all simple $B$-modules $S$.

	We claim that $\textrm{Ext}_B^{j}(FM, FS(\l))=0, \ 1\leq j\leq d(\L^*, \l)+1$, $\l\in \L^*$.
	
	We shall proceed by induction on $n(\l)=d(\L^*)-d(\L^*, \l)$, $\l\in \L^*$. Assume that $n(\l)=0$. Then $\l$ is minimal in $\L^*$.
	Assume that $\l$ is also minimal in the poset $\L$. Then $\St(\l)=S(\l)$. Hence, $F\St(\l)=FS(\l)$. Now, assume that $\l$ is not minimal in $\L$.  Consider the short exact sequence
	\begin{align}
		0\rightarrow X\rightarrow \St(\l)\rightarrow S(\l)\rightarrow 0, \label{eqfc43}
	\end{align} where $X$ has a composition series with composition factors $S(\mu)$ satisfying $\mu<\l$. The minimality of $\l$ in $\L^*$ implies that $FS(\mu)=0$ for $\mu<\l$, $\mu\in \L$. By induction on the length of the composition series of $X$ it follows that $FX=0$. Applying the functor $F$ to the short exact sequence (\ref{eqfc43}) yields $F\St(\l)\simeq FS(\l)$. 
	
	Therefore,
	\[
	\Ext^j_B(FM, FS(\l))=\Ext_B^j(FM, F\St(\l))\simeq \Ext_A^j(M, \St(\l))=0, \quad 1\leq j\leq d(\L^*)+1=d(\L^*, \l)+1.
	\]
	The last isomorphism follows from the fact that $(A, P)$ is a $(d(\L^*)+1)$-faithful cover of $B$.
	
	Assume that there exists a positive integer $k$ such that the claim holds for all $\l\in \L^*$ satisfying $n(\l)<k$. Let $\l\in \L^*$ such that $n(\l)=k$. 
	Consider again the short exact sequence (\ref{eqfc43}). Let $S(\mu)$ be a composition factor of $X$. Hence, $\mu<\l$. If $\mu\notin \L^*$, then $FS(\mu)=0$. Otherwise, $d(\L^*, \mu)\geq d(\L^*, \l)+1$ and  \[
	n(\mu)=d(\L^*)-d(\L^*, \mu)\leq d(\L^*)-d(\L^*, \l)-1=k-1<k\]By induction,
	$\Ext_B^j(FM, FS(\mu))=0$, $1\leq j\leq d(\L^*, \l)+2.$ By induction on the length of the composition series of $FX$, we obtain $\Ext_B^j(FM, FX)=0$, $1\leq j\leq d(\L^*, \l)+2.$ Now, applying the functor $\Hom_B(FM, -)\circ F$ to (\ref{eqfc43}) yields the long exact sequence
	\begin{align}
		0=\Ext_B^j(FM, F\St(\l))\rightarrow \Ext_B^j(FM, FS(\l))\rightarrow \Ext_B^{j+1}(FM, FX)=0, \ 1\leq j\leq d(\L^*, \l)+1.
	\end{align}This completes the proof of our claim. In particular, $\Ext_B^1(FM, FS(\l))=0$ for all $\l\in \L^*$. So, $FM$ is projective over $B$. By projectivization, since the Schur functor is written in the form $F=\Hom_A(P, -)$,  $B\proj$ is equivalent to $\add(P)$. Thus, by projectivization, the functor $F_{|_{A\proj}}\colon A\proj \rightarrow B\proj$ is essentially surjective. As by definition of cover, the functor $F_{|_{A\proj}}$ is fully and faithful it follows that the functor $F_{|_{A\proj}}\colon A\proj\rightarrow B\proj$ is an equivalence of categories. 
	
	So, for any  finitely generated projective $A$-module $M$, we obtain $FM=\Hom_A(P, M)\cong \Hom_A(P, P')$ for some $P'\in \textrm{add}(P)$. By applying the adjoint functor $G$ we get that $M\simeq GFM\simeq GFP'\simeq P'$. So, $A\in  \add(P)$, which means that $P$ is a progenerator. Hence, by Morita theory (see for example \citep[Proposition 4.2.4]{Zimmerman}), $F$ is an equivalence of categories. 
\end{proof}

Observe that $d(\L^*)+1\leq |\L^*|-1+1=|\L^*|$ which is exactly the number of non-isomorphic classes of simple $B$-modules. We have therefore proved that the number of simple $B$-modules is an upper bound for the level of faithfulness of a split quasi-hereditary cover of $B$.

\subsection{$A\proj$}

We can also give upper bounds for $A\proj$-covers. To do that, we will use another example of resolving subcategories. Let $i\geq 0$ be an integer. Let $\mathcal{P}^i$ be the full subcategory of $A\m$ whose modules have projective dimension over $A$ less than or equal to $i$. The category $\mathcal{P}^i$ is a resolving subcategory of $A\m\cap R\proj$. For $i=0$, $\mathcal{P}^i$ is exactly $A\proj$. For $i=\gldim A$, $\mathcal{P}^i=A\m$.

\begin{Theorem}\label{resolvingprojdimensioncover} Let $A$ be a projective Noetherian $R$-algebra.
	Let $i, j\geq 0$ be integers. If $(A, P)$ is an $i$-$\mathcal{P}^j$ cover of $B=\End_A(P)^{op}$, then $(A, P)$ is an $(i-1)$-$\mathcal{P}^{j+1}$ cover of $B$. 
\end{Theorem}
\begin{proof}
	Let $X$ be a module with projective dimension at most $j+1$. We can consider a projective presentation over $A$ for $X$ 
	\begin{align}
		0\rightarrow Q\rightarrow P\rightarrow X\rightarrow 0, \label{eqfc44}
	\end{align}such that $Q\in \mathcal{P}^j$ and $P\in A\proj$.
	Consider the following commutative diagram
	\begin{equation}
		\begin{tikzcd}
			0\arrow[r] &Q\arrow[r] \arrow[d, "\eta_Q"]& P\arrow[r] \arrow[d, "\eta_P"] & X\arrow[r] \arrow[d, "\eta_X"] &0\\
			0\arrow[r] & GFQ \arrow[r] & GFP\arrow[r] & GFX \arrow[r] & \R^1G(FQ) 
		\end{tikzcd}.
	\end{equation}
	Due to $i\geq 0$ and $Q, P\in \mathcal{P}^j$, $\eta_Q$ and $\eta_P$ are $A$-isomorphisms. By Snake Lemma, $\eta_X$ is a monomorphism. So, $(A, P)$ is an $(-1)$-$\mathcal{P}^{j+1}$ cover of $B$. If $i\geq 1$, then $\R^1G(FQ) =0$. In such a case, the Snake Lemma implies that $\eta_X$ is an isomorphism. So, the claim holds for $i=1$. Assume now that $i\geq 2$. Applying $GF$ to (\ref{eqfc44}) yields the long exact sequence
	\begin{align}
		0=\R^l G(FP)\rightarrow \R^lG(FX)\rightarrow \R^{l+1}G(FQ)=0, \quad 1\leq l\leq i-1.
	\end{align} Thus, $(A, P)$ is an $(i-1)$-$\mathcal{P}^{j+1}$ cover of $B$. 
\end{proof}

An immediate consequence of this result is the following bound on $A\proj$-covers.

\begin{Cor}\label{resolvingprojdimensioncoverfinitess}
	Let $(A, P)$ be a $(\gldim A)$-$(A\proj)$ cover of $B$. Then \mbox{$\Hom_A(P, -)\colon A\m\rightarrow B\m$} is an equivalence of categories.
\end{Cor}
\begin{proof}
	Using induction on Theorem \ref{resolvingprojdimensioncover}, we obtain that $(A, P)$ is a $0$-$\mathcal{P}^{\gldim A}$ cover of $B$. Moreover, $(A, P)$ is a $0$-$(A\m)$ cover of $B$. Thus, $\eta_M$ is an isomorphism for every $M\in A\m$. This means that the functor $\Hom_A(P, -)\colon A\m\rightarrow B\m$ is full and faithful. Because  $(A, P)$ is a cover of $B$, the left adjoint of $\Hom_A(P, -)$ is also full and faithful. Therefore, $\Hom_A(P, -)$ is an equivalence of categories.
\end{proof}

	\subsection{Uniqueness of $\mathcal{A}$-covers}
	
	We will now introduce a concept of equivalence of covers that generalizes the concept of equivalent highest weight covers of \citep{Rouquier2008}.
	
	\begin{Def}\label{uniquenessAcovers}
		Let $A, A', B, B'$ be projective Noetherian $R$-algebras and $\mathcal{A}$ and $\mathcal{A}'$ be resolving subcategories of $A\m\cap R\proj$ and $A'\m\cap R\proj$, respectively.
		
		Assume that $(A, P)$ is a $0$-$\mathcal{A}$ cover of $B$ and $(A', P')$ is a $0$-$\mathcal{A}'$ cover of $B'$. We say that the $\mathcal{A}$-covers $(A, P)$ and $(A', P')$ are \textbf{equivalent} if there is an equivalence of categories $H\colon A\m\rightarrow A'\m$, which restricts to an exact equivalence $\mathcal{A}\rightarrow \mathcal{A}'$, and an equivalence of categories $L\colon B\m\rightarrow B'\m$ making the following diagram commutative:
		
		\begin{center}
			\begin{tikzcd}
				A\m \arrow[rr, "{\Hom_A(P, -)}"] \arrow[d, swap, "H"] && B\m \arrow[d, "L"]\\
				A'\m \arrow[rr,swap, "{\Hom_{A'}(P', -)}"] && B'\m  
			\end{tikzcd}.
		\end{center}
		We say that two covers $(A, P)$ and $(A', P')$ are \textbf{isomorphic} if they are equivalent with $L$ being the restriction of scalars functor $B\m\rightarrow B'\m$ along an isomorphism of $R$-algebras $B'\simeq B$.
	\end{Def}
	
	With this concept, Proposition \ref{idempotentsAcovers} says that if $(A, P)$ is an $i$-$\mathcal{A}$ cover of a finite-dimensional algebra (over a field) $B$ then there exists an idempotent $e\in A$ such that the cover $(A, P)$ is equivalent with $(A, Ae)$.
	The first observation to make is that equivalent covers have the same level of faithfulness.
	
	\begin{Prop}\label{uniquenesssamelevel}
		Let $(A, P)$ be a $0$-$\mathcal{A}$ cover of $B$ and let $(A', P')$ be a $0$-$\mathcal{A}'$ cover of $B'$. Assume that the covers $(A, P)$ and $(A', P')$ are equivalent. If $(A, P)$ is an $i$-$\mathcal{A}$ cover of $B$, then $(A', P')$ is an $i$-$\mathcal{A}'$ cover of $B'$.
	\end{Prop}
	\begin{proof}
		Denote the functor $\Hom_{B'}(F'A', -)$ by $G'$. Let $M\in \mathcal{A}'$. Then, for $1\leq j\leq i$,
		\begin{multline*}
			\R^j G'(F'M)=\Ext_{B'}^j(F'A', F'M)=\Ext_B^j(F'HQ, F'HX)\\=\Ext_{B'}^j(LFQ, LFX)=\Ext_B^j(FQ, FX) =\Ext^j_A(Q, X)=0,
		\end{multline*}for some $X\in \mathcal{A}$ and $Q\in A\proj$. By Proposition \ref{arbitraryAcover}, the result follows.
	\end{proof}

	Rouquier defined equivalence of split quasi-hereditary covers in the following way.
	\begin{Def}\label{uniquenessfaithfuldef}
		Two split quasi-hereditary covers $(A, P)$ and $(A', P')$ are \textbf{equivalent in the sense of Rouquier} if there is an equivalence of highest weight categories $A\m\xrightarrow{\simeq} A'\m$ making the following diagram commutative:
		\begin{center}
			\begin{tikzcd}
				A\m \arrow[dr, "{Hom_A(P, -)}"] \arrow[dd, swap, "\simeq"] &\\
				& B\m\\
				A'\m \arrow[ur,swap, "{Hom_{A'}(P', -)}"] & 
			\end{tikzcd}.
		\end{center}
	\end{Def}
	
	We will show next that this definition is a particular case of our definition of isomorphic covers by fixing $\mathcal{A}=\mathcal{F}(\Stsim)$, $\mathcal{A}'=\mathcal{F}(\Stsim')$, and $L$ the identity functor.
	Moreover, the notion of isomorphic covers for the resolving subcategory $\mathcal{F}(\Stsim)$ is equivalent to the equivalence of covers of Definition  \ref{uniquenessfaithfuldef}.

	\begin{Prop}\label{equivalenceofnotionsuniqueness}
		Let $(A, \{\Delta(\lambda)_{\lambda\in \Lambda}\})$ , $(A', \{\Delta'(\lambda')_{\lambda'\in \Lambda'}\})$  be split quasi-hereditary algebras over a commutative Noetherian ring $R$.
		Let $(A, P)$ and $(A', B')$ be split quasi-hereditary covers of $B$. The covers $(A, P)$ and $(A', P')$ are equivalent with $L=\id_{B\m}$ in the sense of Definition  \ref{uniquenessfaithfuldef} if and only if they are isomorphic in the sense of Definition \ref{uniquenessAcovers} with respect to the resolving subcategories $\mathcal{F}(\Stsim)$ and $\mathcal{F}(\Stsim')$.
	\end{Prop}
	\begin{proof}
		Assume that $(A, P)$ and $(A', P')$ are equivalent in the sense of Definition \ref{uniquenessfaithfuldef}. Let \linebreak${H\colon A\m\rightarrow A'\m}$ be a highest weight category equivalence such that $\Hom_{A'}(P', -)\circ H=\Hom_A(P, -)$.  Since $H$ is an equivalence of highest weight categories, there is a bijection $\phi\colon \L\rightarrow \L'$ satisfying $H\St(\l)=\St'(\phi(\l))\otimes_R U_\l$. As $H$ is exact and $H\St(\l)\in \mathcal{F}(\Stsim')$, the restriction functor $H\colon \mathcal{F}(\Stsim)\rightarrow \mathcal{F}(\Stsim')$ is well defined and it is fully faithful and exact. As $U_\l\in Pic(R)$ there is $F_\l$ such that $F_\l\otimes_R U_\l\simeq R$, thus $\St(\l')=H\St(\phi^{-1}(\l'))\otimes_R F_{\l'}=H(\St(\phi^{-1}(\l'))\otimes_R F_{\l'})$. Let $M\in \mathcal{F}(\Stsim')$. By induction on the filtration of $M$, we deduce that $M$ is in the image of $H_{|_{\mathcal{F}(\Stsim)}}$. Therefore, $H$ restricts to an exact equivalence $\mathcal{F}(\Stsim)\rightarrow \mathcal{F}(\Stsim')$ and $B\simeq \End_A(P)^{op}\simeq \End_{A'}(P')^{op}$. Hence,  the covers $(A, P)$ and $(A', P')$ are isomorphic in the sense of Definition \ref{uniquenessAcovers}. 
		
		Conversely, assume that the covers $(A, P)$ and $(A', P')$ are isomorphic in the sense of Definition \ref{uniquenessAcovers} with respect to the resolving subcategories $\mathcal{F}(\Stsim)$ and $\mathcal{F}(\Stsim')$. Let $I$ be the quasi-inverse of $H$.
		Then $HA$ is a $B$-progenerator and for any $Y\in B\m$, \begin{align}
			\Hom_B(HA, Y)\simeq \Hom_A(IHA, IY)\simeq \Hom_A(A, IY)\simeq IY.
		\end{align}In particular, $I$ commutes with tensor products of projective $R$-modules, that is, for $Y\in B\m$ and $X\in R\proj$, $I(Y\otimes_R X)=IY\otimes_R X$.
		In view of the proof of Theorem 6.1 of \cite{appendix}, there is a bijection $\phi\colon \L\rightarrow \L'$ and \begin{align}
			\Hom_B(HA, \St'(\l'))=\St(\phi^{-1}(\l'))\otimes_R U_{\l'}=\St(\l)\otimes_R U_{\l'}, \ U_{\l'}\in Pic(R).
		\end{align}Moreover, as $A$-modules,
		\begin{align}
			I(\St'(\l')\otimes_R F_{\l'})&\simeq I\St'(\l')\otimes_R F_{\l'}\simeq \Hom_A(A, I\St'(\l'))\otimes_R F_{\l'}\simeq \Hom_A(IHA, I\St'(\l'))\otimes_R F_{\l'}\\&\simeq \Hom_B(HA, \St'(\l'))\otimes_R F_{\l'}
			\simeq \St(\l)\otimes_R U_{\l'} \otimes_R F_{\l'} \simeq  \St(\l).
		\end{align}
		Thus,
		$
			H\St(\l)\simeq HI(\St'(\l')\otimes_R F_{\l'})\simeq \St'(\phi(\l))\otimes_R F_{\l},
		$ with $F_{\l'}=F_{\l}$. Thus, $H$ is an equivalence of highest weight categories, and it follows that $(A, P)$ and $(A', P')$ are equivalent in the sense of Definition \ref{uniquenessfaithfuldef}.
	\end{proof}
	The main reason to make the difference between isomorphic covers and equivalent covers is that intuitively the covers constructed in Proposition \ref{idempotentsAcovers} should be equivalent although they are not isomorphic. So, from now on, we will use only the concepts in the sense of Definition \ref{uniquenessAcovers}.
	
	In \citep[Proposition 4.45]{Rouquier2008} it is stated that when $B$ admits a $1$-faithful split quasi-hereditary cover $(A, P)$, for every object $M\in \mathcal{B}:=\mathcal{F}(\Hom_A(P, \Stsim))$ there exists a unique pair $(Y(M), M)$ together with a surjection $\pi_M$ so that $Y(M)$ is relative projective to $\mathcal{B}$ and $\ker \pi_M\in \mathcal{B}$.
	However, in \citep[Proposition 4.45]{Rouquier2008} the assumptions to guarantee the uniqueness of the pair $(Y(M), M)$ have not been stated explicitly. When the ground ring $R$ is assumed to be a local ring, an appropriate assumption is to require $Y(M)$ to be indecomposable when $M$ is indecomposable. This assumption is used in the context of Hecke algebras in \citep[Subsection 4.2]{zbMATH05707949}. In a general context, this assumption does, however, restrict the situation. Indeed, even the standard modules are not indecomposable if $R$ is not local. 
	
	In the following, we first give an example showing that \citep[Proposition 4.45]{Rouquier2008} fails without more restrictive assumptions. Subsequently, we then give an alternative proof of \citep[Corollary 4.46]{Rouquier2008}, avoiding to use \citep[Proposition 4.45]{Rouquier2008} and avoiding to impose additional assumptions, by using instead our Theorem \ref{standardsgivesthewholealgebrapre}.
	
	
	\begin{Example}
		Let $A$	the path algebra of the quiver 
			\begin{tikzcd}
				2\arrow[r, "\alpha", shift left=0.75ex] & 1 \arrow[l, "\beta",  shift left=0.75ex]
			\end{tikzcd} 
	 modulo the ideal generated by $\alpha\beta$. Pick the partial order $2>1$ and $\St(2)=P(2)$, $\St(1)=1$. Trivially, $(A, {}_A A)$ is a 1-faithful split quasi-hereditary cover of $A$. The exact sequence
		\begin{align}
			0\rightarrow \St(2)\oplus \St(2)\rightarrow \St(2)\oplus P(1)\rightarrow \St(1)\rightarrow 0
		\end{align} also satisfies the conditions required for the pair $(Y(\St(1)), p_{\St(1)})$. 
	\end{Example}

	We will now give an alternative proof of \citep[Corollary 4.46]{Rouquier2008}.
	\begin{Cor}\citep[Corollary 4.46]{Rouquier2008}\label{equivalenceofoneuniqueness}
		Let $(A, P)$ and $(A', P')$ be two 1-faithful split quasi-hereditary covers of $B$. Let $F=\Hom_A(P, -)$ and let $F'=\Hom_{A'}(P', -)$. 
		Assume that there exists an exact equivalence $L\colon B\m\rightarrow B\m$ which restricts to an exact equivalence \begin{align}
			\mathcal{F}(F\Stsim)\rightarrow  \mathcal{F}(F'\Stsim'). \label{fceq148}
		\end{align} Then $(A, P)$ and $(A', P')$ are equivalent as  faithful split quasi-hereditary covers of $B$. If, in addition, the given bijection $\phi\colon \L\rightarrow \L'$ associated with the equivalence of categories $H\colon A\m\rightarrow A'\m$ satisfies
		\begin{align*}
			F\St(\l)=F'\St'(\phi(\l))\otimes_R U_\l, \forall \l\in \L.
		\end{align*} where $H\St(\l)\simeq \St'(\phi(\l))\otimes_R U_\l$, $U_\l\in Pic(R)$,
		then $(A, P)$ and $(A', P')$ are isomorphic as split quasi-hereditary covers of $B$.
	\end{Cor}
	\begin{proof}
		By Proposition \ref{onefaithfulcover}, there exists an exact equivalence $$\mathcal{F}(\Stsim)\xrightarrow{F} \mathcal{F}(F\Stsim) \xrightarrow{L} \mathcal{F}(F'\Stsim') \xrightarrow{G'} \mathcal{F}(\Stsim').$$
		By Theorem \ref{standardsgivesthewholealgebrapre}, $A$ and $A'$ are equivalent as split quasi-hereditary algebras.
		Furthermore, the equivalence of categories is given by $H=\Hom_A(GL^{-1}F'A', -)\colon A\m\rightarrow A'\m$.  Thus, for every $X\in A\proj$,
		\begin{align}
			HX&\simeq \Hom_A(GL^{-1}F'A', X)\simeq \Hom_B(FGL^{-1}F'A', FX)\simeq \Hom_B(L^{-1}F'A', FX) \\
			&\simeq \Hom_B(F'A', LFX)\simeq \Hom_{A'}(A', G'LFX)\simeq G'LFX.
		\end{align} Therefore, $F'HX\simeq LFX$ for every $X\in A\proj$. Since all functors involved are exact we conclude that $F'H=LF$.

		Assume, in addition, 	 
		$	F\St(\l)=F'\St'(\phi(\l))\otimes_R U_\l, \forall \l\in \L
		$ and $H\St(\l)\simeq \St'(\phi(\l))\otimes_R U_\l$, $U_\l\in Pic(R)$ for some bijection $\phi\colon \L\rightarrow \L'$. 
		Using induction on the filtrations of $M\in A\proj$ we obtain $FM\simeq F'HM$ for every $M\in A\proj$. For any $M\in A\m$, applying $F$ and $F'H$ to one of the projective resolutions of $M$, it follows that $F'HM\simeq FM$. 
	\end{proof}

\section{$\mathcal{A}$-covers under change of ground ring}\label{A-covers under change of ground ring}
We shall now see how  $\mathcal{A}$-covers behave under change of ground ring. Here we need to impose constraints to the resolving subcategories $\mathcal{A}$ we want to work with. As a first step, note that $A\m\cap R\proj$ is a resolving subcategory of $A\m$. So, we will restrict our attention to resolving subcategories of $A\m\cap R\proj$ since exact sequences in this category remain exact under change of ground ring. However, this is not sufficient, so we are interested in resolving subcategories which behave well under change of ground ring in the following sense.

\begin{Def}\label{wellbehavedresolving}
	We will call $\mathcal{R}_A$ a \textbf{well behaved resolving subcategory} of $A\m\cap R\proj$ if it is a resolving subcategory of $A\m\cap R\proj$ and the following properties are satisfied:
	\begin{enumerate}
		\item For any commutative $R$-algebra $S$ which is a Noetherian ring, there is a resolving subcategory $ \mathcal{R}_A(S\otimes_R A)$  of  $S\otimes_R A\m\cap S\proj$ and there is a well defined functor $H\colon \mathcal{R}_A\rightarrow \mathcal{R}_A(S\otimes_R A)$, given by ${M\mapsto S\otimes_R M}$, satisfying \mbox{$\langle H\mathcal{R}_A \rangle=\mathcal{R}_A(S\otimes_R A)$,} where $\langle H\mathcal{R}_A \rangle$ denotes the smallest subcategory of ${S\otimes_RA\m\cap S\proj}$ containing $H\mathcal{R}_A$ closed under direct summands and extensions.
		\item $M\in \mathcal{R}_A$ if and only if $M_\mi\in  \mathcal{R}_A(A_\mi)$ for every maximal ideal $\mi$ in $R$.
		\item $M\in \mathcal{R}_A$ if and only if $M(\mi)\in  \mathcal{R}_A(A(\mi))$ for every maximal ideal $\mi$ in $R$ and $M\in R\proj$.
	\end{enumerate} 
\end{Def}
It follows that $\mathcal{R}_A(A)=R_A$.
From now on we will consider $\mathcal{R}_A$ to be a well behaved resolving subcategory of $A\m\cap R\proj$. Here are some examples of well behaved resolving subcategories.

\begin{Prop}\label{exampleswellbehavedres}
	Let $A$ be a projective Noetherian $R$-algebra. The following assertions hold.
	\begin{enumerate}[(I)]
		\item $A\proj$ is a well behaved resolving subcategory of $A\m\cap R\proj$.
		\item Let $(A, \{\Delta(\lambda)_{\lambda\in \Lambda}\})$ be a split quasi-hereditary algebra. Then $\mathcal{F}(\Stsim)$ is a well behaved resolving subcategory of $A\m\cap R\proj$.
	\end{enumerate}
\end{Prop}
\begin{proof}
	Clearly, $A\proj$ is a resolving subcategory of $A\m\cap R\proj$. Condition \ref{wellbehavedresolving}.2 follows from the identity: $\Ext_A^1(M, N)_\mi\simeq \Ext_{A_\mi}^1(M_\mi, N_\mi)$ for every maximal ideal $\mi$ of $R$.  Condition \ref{wellbehavedresolving}.3 follows by \citep[Lemma 3.3.2]{CLINE1990126}. Let $M\in A\proj$. Then $A^t\simeq M\bigoplus K$ for some $t>0$ and some module $K$. Hence, \mbox{$(S\otimes_R A)^t\simeq S\otimes_R M\bigoplus S\otimes_R K$.} So, $S\otimes_RM\in S\otimes_RA\proj$. Thus, the functor $H$ is well defined. Let  $X\in S\otimes_R A\proj$. Hence, $S\otimes_R A^s\simeq (S\otimes_R A)^s\simeq X\bigoplus K$ for some $s>0$ and $S\otimes_R A^s\in \langle H(A\proj)\rangle$, thus $X\in\langle H(A\proj)\rangle$.  So, $(I)$ holds.
	
	By Theorem 4.1 of \cite{appendix}, $\mathcal{F}(\Stsim)$ is a resolving subcategory of $A\m\cap R\proj$. Recall that
	$0=\Ext_A^{i>0}(M, N)$ if and only if $\Ext_{A_\mi}^{i>0}(M_\mi, N_\mi)=0$ for every maximal ideal of $R$.
	By  \citep[Lemma 4.3, Proposition 5.1]{appendix}, Condition \ref{wellbehavedresolving}.2 follows. Condition \ref{wellbehavedresolving}.3 follows from Proposition 5.7 of \cite{appendix}. Since the exact sequences arising from a filtration of $M\in \mathcal{F}(\Stsim)$ are $(A, R)$-exact, applying the tensor product \mbox{$S\otimes_R -$} preserves the filtration and hence $S\otimes_RM\in \mathcal{F}(S\otimes \Stsim)$. So, the functor $H$ is well defined. Let \mbox{$X\in \mathcal{F}(S\otimes_R \Stsim)$.} Then there is a filtration \begin{align}
		0=X_{n+1}\subset X_n\subset \cdots \subset X_1=X, \ X_i/X_{i+1}\simeq S\otimes_R\St_i\otimes_S U_i, 1\leq i \leq n.
	\end{align}
	We will proceed by induction to prove that each $X_i$ belongs to $\langle H\mathcal{F}(\Stsim) \rangle$. For $i=n$, $X_n=S\otimes_R \St_n \otimes_S U_n$ is an $S\otimes_R A$-summand of $(S\otimes_R \St_n)^s\simeq S\otimes_R \St_n^s\in \langle H\mathcal{F}(\Stsim)\rangle$ for some $s>0$. Since $\langle H\mathcal{F}(\Stsim)\rangle$ is closed under direct summands, $X_n\in \langle H\mathcal{F}(\Stsim)\rangle$. 
	Assume that we have proven the result for $X_s$ for $s>i$ for some $i$. Consider the exact sequence
	\begin{align}
		0\rightarrow X_{i+1}\rightarrow X_i\rightarrow S\otimes_R \St_i\otimes_S U_i\rightarrow 0.
	\end{align} By induction, $X_{i+1}\in \langle H\mathcal{F}(\Stsim)\rangle$, and since it is closed under extensions, $X_i\in \langle H\mathcal{F}(\Stsim)\rangle$. Thus, $(II)$ holds.
\end{proof}

As before we will separate the cases -1 and 0 and consider them first.

\begin{Prop}\label{arbitraryfaithfulcoverflat}
	Let $i\in \{-1, 0\}$. Let $\mathcal{R}_A$ be a well behaved resolving subcategory of $A\m\cap R\proj$. The following assertions are equivalent.
	\begin{enumerate}[(a)]
		\item  $(A, P)$ is an $i$-$\mathcal{R}_A$ cover of $B$.
		\item  $(S\otimes_R A, S\otimes_R P)$ is an $i$-$\mathcal{R}_A(S\otimes_R A)$ cover of $S\otimes_R B$ for any commutative flat $R$-algebra $S$ which is a Noetherian ring.
		\item $(A_\mi, P_\mi)$ is an $i$-$\mathcal{R}_A(A_\mi)$ cover of $B_\mi$ for every maximal ideal $\mi$ of $R$.
	\end{enumerate}
\end{Prop}
\begin{proof}
	Let $M\in A\m\cap R\proj$ and $S$ be a commutative flat $R$-algebra which is a Noetherian ring. Consider the following diagram
	\begin{equation}
		\begin{tikzcd}
			S\otimes_R M\arrow[dd, equal] \arrow[r, "S\otimes_R \eta_M"]& S\otimes_R \Hom_B(\Hom_A(P, A), \Hom_A(P, M))\arrow[d, "\omega_{\Hom_A(P {,}A) {,} \Hom_A(P {,}M)}"]\\
			& \Hom_{S\otimes_R B}(S\otimes_R \Hom_A(P, A), S\otimes_R \Hom_A(P, M))\\
			S\otimes_RA \arrow[r, "\eta_{S\otimes_R M}"] & \Hom_{S\otimes_R B}(\Hom_{S\otimes_R A}(S\otimes_R P, S\otimes_R A),\Hom_{S\otimes_R A}(S\otimes_R P, S\otimes_R M) ) \arrow[u, "\omega_{P {,} M}^{-1}\circ (-)\circ \omega_{P {,} A}"]
		\end{tikzcd}. \label{eqfc34}
	\end{equation}
	The maps $\omega$ are the canonical isomorphisms giving the identity involving $\operatorname{Hom}$ and tensor product with a flat module (see for example \citep[Proposition 3.3.10]{Rotman2009a}).
	This is a commutative diagram. In fact, for every $s, s', s''\in S$, $m\in M$, $g\in \Hom_A(P, A)$, $p\in P$, we have
	\begin{align*}
		\begin{aligned}
			\omega_{P, M}\circ 	\omega_{P, M}^{-1}\circ (-)\circ \omega_{P, A}\circ  \eta_{S\otimes_R M}(s\otimes m)(s'\otimes g)(s''\otimes p)&=\eta_{S\otimes_R M}(s\otimes m)\omega_{P, A}(s'\otimes g)(s''\otimes p)\\&=\omega_{P, A}(s'\otimes g)(s''\otimes p)s\otimes m=ss's''\otimes g(p)m
		\end{aligned}\\
		\begin{aligned}
			&\omega_{P, M}\circ \omega_{\Hom_A(P {,}A), \Hom_A(P, M)}\circ S\otimes_R \eta_M(s\otimes m)(s'\otimes g)(s''\otimes p)=\\& \omega_{P, M}\omega_{\Hom_A(P {,}A), \Hom_A(P, M)}(s\otimes \eta_M(m))(s'\otimes g)(s''\otimes p)= \omega_{P, M}(ss'\otimes \eta_M(m)(g))(s''\otimes p)=ss's''\otimes g(p)m.
	\end{aligned}	\end{align*}
	
	The implication $(b)\implies (c)$ is clear. 
	By (\ref{eqfc34}), it follows that if $S\otimes_R \eta_A$ is an isomorphism then $\eta_{S\otimes_R A}$ is an isomorphism. Thus, since $S$ is flat if $\eta_A$ is an isomorphism, then $\eta_{S\otimes_RA}$ is an isomorphism. On the other hand, if $\eta_{A_\mi}$ is an isomorphism for all $\mi\in \MaxSpec R$, then (\ref{eqfc34}) gives that $(\eta_A)_\mi$ is an isomorphism for all $\mi\in \MaxSpec R$ which in turn implies that $\eta_A$ is an isomorphism. This proves the result for $\mathcal{R}_A=A\proj$.
	
	Assume that $(a)$ holds. By the previous case, $(S\otimes_R A, S\otimes_R P)$ is a cover of $S\otimes_R B$. Let $M\in \mathcal{R}_A$. By assumption, $\eta_M$ is a monomorphism in case $i=-1$ or it is an isomorphism in case $i=0$. Applying the exact functor $S\otimes_R -$, $S\otimes_R \eta_M$ is a monomorphism (resp. an isomorphism) if $i=-1$ (resp. $i=0$)
	In view of the diagram (\ref{eqfc34}), $\eta_{S\otimes_R M}$ is a monomorphism (resp. an isomorphism)  if $i=-1$ (resp. $i=0$).  According to the additivity of $G\circ F$ and Snake Lemma, $\eta_N$ is a monomorphism (resp. an isomorphism)  if $i=-1$ (resp. $i=0$) for any $N\in \langle H\mathcal{R}_A \rangle=\mathcal{R}_A(S\otimes_R A)$. By Proposition \ref{zeroAcover} and Lemma \ref{considerationseta}, $(b)$ follows. 
	
	Assume that $(c)$ holds. By the first case, $( A,  P)$ is a cover of $ B$.
	Let $M\in \mathcal{R}_A$. By assumption, $\eta_{M_\mi}$ is a monomorphism (resp. an isomorphism) in case $i=-1$ (resp. $i=0$) for every maximal ideal $\mi$ in $R$. According to the diagram (\ref{eqfc34}), $(\eta_M)_\mi$ is a monomorphism (resp. an isomorphism) in case $i=-1$ (resp. $i=0$) for every maximal ideal $\mi$ in $R$. Therefore, $\eta_M$ is a monomorphism (resp. isomorphism) in case $i=-1$ (resp. $i=0$).  Thus, $(a)$ follows.
\end{proof}

To simplify the notation, whenever $(S\otimes_R A, S\otimes_R P)$ is a cover of $S\otimes_R B$ we will write $F_S$ to denote the Schur functor associated with this cover and we will write $G_S$ to denote its right adjoint. When $S=R(\mi)$ with $\mi\in \MaxSpec R$ we will write instead $F_{(\mi)}$ and $G_{(\mi)}$.

\begin{Prop}\label{arbitraryfaithfulcoverflattwo}
	Let $\mathcal{R}_A$ be a well behaved resolving subcategory of $A\m\cap R\proj$.	Let $(A, P)$ be a  $0$-$\mathcal{R}_A$ cover of $B$. For $i\geq 1$, the following assertions are equivalent.
	\begin{enumerate}[(a)]
		\item $(A, P)$ is an $i$-$\mathcal{R}_A$ cover of $B$.
		\item $(S\otimes_R A, S\otimes_R P)$ is an $i$-$\mathcal{R}_A(S\otimes_R A)$ cover of $S\otimes_R B$ for any commutative flat $R$-algebra $S$ which is Noetherian ring.
		\item $(A_\pri, P_\pri)$ is an $i$-$\mathcal{R}_A(A_\pri)$ cover of $B_\pri$ for every prime ideal $\pri$ of $R$.
		\item $(A_\mi, P_\mi)$ is an $i$-$\mathcal{R}_A(A_\mi)$ cover of $B_\mi$ for every maximal ideal $\mi$ of $R$.
	\end{enumerate}
\end{Prop}
\begin{proof}
	$(a)\implies (b)$. By Proposition \ref{arbitraryfaithfulcoverflat}, $(S\otimes_R A, S\otimes_R P)$ is a $0$-$\mathcal{R}_A(S\otimes_R A)$ cover of $S\otimes_R B$ for any commutative flat $R$-algebra $S$ which is a Noetherian ring. Let $M\in \mathcal{R}_A$. Let $1\leq j\leq i$. Then
	\begin{align*}
		\R^jG_S(F_S(S\otimes_R M))&=\Ext_{S\otimes_R B}^j(\Hom_{S\otimes_R A}(S\otimes_R P, S\otimes_R A), \Hom_{S\otimes_R A}(S\otimes_R P, S\otimes_R M))\\&=\Ext_{S\otimes_R B}^j(S\otimes_R \Hom_A(P, A), S\otimes_R \Hom_A(P, M))\\&=S\otimes_R \Ext_B^j(\Hom_A(P, A), \Hom_A(P, M))=S\otimes_R \R^jG(FM)=0
	\end{align*}
	Using long exact sequences coming from the derived functors $\R^j G_S$ and since it commutes with direct summands we obtain that $\R^jG_S(F_S(N))=0$ for all $N\in \langle H\mathcal{R}_A \rangle=\mathcal{R}_A(S\otimes_R A)$. By Proposition \ref{arbitraryAcover}, $(b)$ follows. The implications $(b)\implies (c)\implies (d)$ are clear.
	
	Assume that $(d)$ holds. By Proposition \ref{arbitraryfaithfulcoverflat}, $(A, P)$ is a $0$-$\mathcal{R}_A$ cover of $B$. Let $M\in \mathcal{R}_A$. Let $1\leq j\leq i$. We have, for every maximal ideal $m$ in $R$,
	\begin{align}
		\R^jG(FM)_\mi&=\Ext_{B}^j(\Hom_{A}(P, A), \Hom_{A}(P, M))_\mi\simeq \Ext_{B_\mi}^j(\Hom_{A_\mi}(P_\mi, A_\mi), \Hom_{A_\mi}(P_\mi, M_\mi))\\&=\R^jG_\mi(F_\mi M_\mi)=0,
	\end{align}since $M_\mi\in \mathcal{R}_A(A_\mi)$.
	Therefore, $\R^jG(FM)=0$. By Proposition \ref{arbitraryAcover}, $(A, P)$ is an $i$-$\mathcal{R}_A$ cover of $B$.
\end{proof}

After seeing that $\mathcal{R}_A$-covers behave nicely under extension of scalars to a flat commutative ring, our next step is to see to what extent can the study of $\mathcal{R}_A$-covers be reduced to the study of covers of finite-dimensional algebras.

\begin{Prop}\label{faithfulcoverresiduefieldnext}
	Let $R$ be a regular (commutative Noetherian) ring. Let $\mathcal{R}_A$ be a well behaved resolving subcategory of $A\m\cap R\proj$.  Let $P\in A\m\cap R\proj$ and $i\in \{-1, 0\}$. If  $(A(\mi), P(\mi))$ is an $i$-$\mathcal{R}_A(A(\mi))$ cover of $B(\mi)$ for every maximal ideal $\mi$ of $R$, then $(A, P)$ is an $i$-$\mathcal{R}_A$ cover of $B$.
\end{Prop}
\begin{proof}
	As $P\in R\proj$ and $P(\mi)\in A(\mi)\proj$ for every $\mi\in \MaxSpec R$ we obtain that $P\in A\proj$. 
	By Proposition \ref{arbitraryfaithfulcoverflattwo}, we can assume without loss of generality that $R$ is a regular local commutative Noetherian ring with unique maximal ideal $\mi$. In particular, we can assume that $R$ is an integral domain.
	
	The strategy, inspired by \citep[Proposition 4.36]{Rouquier2008} is for any $M\in \mathcal{R}_A$ to compare $\eta_M(\mi)$ with $R/Rx\otimes_R \eta_M$, where $x$ is a non-zero divisor of $R$, which in turn can be compared with $\eta_{R/Rx\otimes_R M}$. Using such comparisons, the injectivity and bijectivity of the latter map can then be used to infer the injectivity or bijectivity of $\eta_M$.
	
	Let $x\in \mi/\mi^2$ be a non-zero element. Then $Q:=R/Rx$ is a regular commutative Noetherian local ring with Krull dimension $\dim(Q)=\dim R-1$ (see for example \citep[Theorem 8.62, Proposition 8.56]{Rotman2009a}). Also, $\mi_Q:=\mi/Rx$ is the unique maximal ideal of $R$ and $R(\mi)\simeq R/Rx/\mi/Rx\simeq  Q(\mi_Q)$. 
	
	Let $M\in \mathcal{R}_A$. By definition, $Q\otimes_R M\in \mathcal{R}_A(Q\otimes_RA)$.
	
	We will denote by $\mu_{Q, M}$ the following composition of canonical maps
	\begin{center}
		\begin{tikzcd}[column sep=tiny]
\hspace*{-2cm}	Q\otimes_R M \arrow[r, "\eta_{Q\otimes_RM}"near start, shorten >=8ex, above right] &[-2.5em] \hspace*{-1.2cm} \Hom_{Q\otimes_RB}(\Hom_{Q\otimes_R A}(Q\otimes_RP, Q\otimes_R A), \Hom_{Q\otimes_R A}(Q\otimes_RP, Q\otimes_RM))\arrow[d] \\       \Hom_B(\Hom_A(P, A),Q\otimes_R \Hom_A(P, M) ) & \Hom_{Q\otimes_R B}(Q\otimes_R \Hom_A(P, A), Q\otimes_R \Hom_A(P, M)) \arrow[l] 
		\end{tikzcd}. \label{map4}
	\end{center} Explicitly, we have $\mu_{Q, M}(q\otimes m)=(f\mapsto q\otimes \eta_M(m)(f)),$ \mbox{$q\otimes m\in Q\otimes_R M$} (see for example Lemma \ref{homtensorarbitraryring} and \citep[Proposition 2.4]{CRUZ2022410}.)  In particular, if $\eta_{Q\otimes_R M}$  is an isomorphism, then so it is $\mu_{Q, M}$ since $P\in A\proj$.
	We have a commutative triangle
	\begin{center}
		\begin{tikzcd}
			Q\otimes_R M\arrow[rr, "\mu_{Q {,} M}"] \arrow[dr, swap," Q\otimes_R\eta_M"]& &\Hom_B(\Hom_A(P, A),Q \otimes_R\Hom_A(P, M) )\\
			& Q\otimes_R\Hom_B(\Hom_A(P, A), \Hom_A(P, M)) \arrow[ur, hookrightarrow, "\delta"] &
		\end{tikzcd} 
	\end{center}with a monomorphism $\delta$ given by Lemma \ref{zerodivisormono}. It follows that if $\eta_{Q\otimes_R M}$ is bijective then $\delta$ must be also surjective and so, in such a case, $\delta$ must be bijective. This implies that if $Q\otimes_R \eta_M$ is an isomorphism whenever $\eta_{Q\otimes_R M}$. 
	Denote the canonical surjective map $Q\rightarrow Q/\mi/Rx=R(\mi)$ by $\pi$. Now using the commutative diagram
	\begin{center}
		\begin{tikzcd}
			Q \otimes_R M\arrow[r, "Q\otimes_R \eta_M"] \arrow[d, "M\otimes_R \pi"] & Q\otimes_R \Hom_B(\Hom_A(P, A), \Hom_A(P, M)) \arrow[d, "\pi\otimes_R \Hom_B(\Hom_A(P{,} A){,} \Hom_A(P{,} M))"]\\
			R(\mi) \otimes_R	M \arrow[r, "\eta_M(\mi)"] &R(\mi) \otimes_R  \Hom_B(\Hom_A(P, A), \Hom_A(P, M))
		\end{tikzcd},
	\end{center}
	It follows that $\eta_M(\mi)\circ M\otimes_R \pi$ is surjective whenever $\eta_{Q\otimes_R M}$ is an isomorphism. In particular, $\eta_M(\mi)$ is surjective whenever $\eta_{Q\otimes_R M}$ is an isomorphism.
	
	Using this observation, we will proceed by induction on the Krull dimension of $R$ to prove the case $i=0$.
	If $\dim R=0$ then $R$ is a field, and so there is nothing to prove. If $\dim R>1$ we can pick $x\in \mi/\mi^2$ making that $Q:=R/Rx$ has Krull dimension equal to $\dim R -1$. By assumption, $((Q\otimes_RA)(\mi_Q), (Q\otimes_R P(\mi_Q))$ is a $0$-$\mathcal{R}_A((Q\otimes_RA)(\mi_Q))$ cover of $Q\otimes_R B(\mi_Q)$. By induction, \linebreak\mbox{$(Q\otimes_R A, Q\otimes_RP)$} is a $0$-$\mathcal{R}_A(Q\otimes_R A)$ cover of $Q\otimes_RB$. By the previous discussion, we obtain that $\eta_M(\mi)$ is surjective for every $M\in \mathcal{R}_A$. By Nakayama's Lemma, $\eta_M$ is surjective for every $M\in \mathcal{R}_A$. 
	Also, by assumption, $\eta_{M(\mi)}$ is injective, for every $M\in \mathcal{R}_A$.
	By Lemma \ref{unitmono}, it follows that $\eta_M$ is an isomorphism for every $M\in \mathcal{R}_A$ and consequently the case $i=0$ is finished by Proposition \ref{zeroAcover}.
	Assume now that $i=-1$. By the previous case, $(A, P)$ is a cover of $B$. 	Let $M\in \mathcal{R}_A$. By definition, $M(\mi)\in \mathcal{R}_A(A(\mi))$ for every maximal ideal $\mi$ in $R$ and $M\in R\proj$. By Lemma \ref{considerationseta}, $\eta_{M(\mi)}$ is a monomorphism. By Lemma \ref{unitmono}, $\eta_M$ is a monomorphism. Applying again Lemma \ref{considerationseta}, it follows that $(A, P)$ is a $(-1)$-$\mathcal{R}_A$ cover of $B$.
\end{proof}

\begin{Prop}\label{faithfulcoversresiduepart3}
	Let $R$ be a regular (commutative Noetherian) ring.  Let $\mathcal{R}_A$ be a well behaved resolving subcategory of $A\m\cap R\proj$. Let $P\in A\m\cap R\proj$. Let $i\geq 1$. If $(A(\mi), P(\mi))$ is an $i$-$\mathcal{R}_A(A(\mi))$ cover of $B(\mi)$ for every maximal ideal $\mi$ of $R$, then $(A, P)$ is an $i$-$\mathcal{R}_A$ cover of $B$.
\end{Prop}
\begin{proof}
	$(A(\mi), P(\mi))$ is a $0$-$\mathcal{R}_A(A(\mi))$ cover of $B(\mi)$ for every maximal ideal $\mi$ of $R$. By Proposition \ref{faithfulcoverresiduefieldnext}, $(A, P)$ is a $0$-$\mathcal{R}_A$ cover of $B$. We can assume, without loss of generality, that $R$ is a local regular ring. Let $M\in \mathcal{R}_A$.  Let \begin{align}
		\Hom_A(P, A)^{\bullet}\colon\cdots\rightarrow Q_1\rightarrow Q_0\rightarrow 0
	\end{align} be a deleted complex chain obtained by deleting $\Hom_A(P, A)$ from a projective $B$-resolution of $\Hom_A(P, A)$.  Consider the cochain complex $P^{\bullet}=\Hom_B(\Hom_A(P, A)^{\bullet}, \Hom_A(P, M))$. Note that each module in $\Hom_A(P, A)^\bullet$ is projective over $B$, so that each module in $P^{\bullet}$ belongs to $\add_R \Hom_A(P, M)$. In particular, each module in $P^{\bullet}$ is projective over $R$. 
	
	We claim that $\R^jG(FM)=0$, $1\leq j\leq i$. We shall prove it by induction on $\dim R$.
	
	If $\dim R=0$, there is nothing to show. Assume that $\dim R>0$. Let $x\in \mi/\mi^2$. Then $\dim(R/Rx)=\dim R-1$. $\mi/Rx$ is the unique maximal ideal of $R/Rx$ and $R/Rx/\mi/Rx\simeq R/\mi$ as $R$-modules. Hence, for every $X\in A\m$,
	\begin{align}
		R/Rx\otimes_R X(\mi/Rx)=R/Rx/\mi/Rx\otimes_{R/Rx} R/Rx\otimes_R X\simeq R/\mi\otimes_R X=X(\mi).
	\end{align}
	Thus,  $(R/Rx\otimes_R A(\mi/Rx), R/Rx\otimes_R P(\mi/Rx))=(A(\mi), P(\mi))$ is an $i$-$\mathcal{R}_A(A(\mi/Rx))$ cover of \linebreak$R/Rx\otimes_R B(\mi/Rx)$. Denote by $F_x \dashv G_x$ the adjoint functors associated with this cover. Therefore, $R^jG_x(F_x(R/Rx\otimes_R M))=0$ for $1\leq j\leq i$.
	
	Observe that $\pdim_R R/Rx\leq 1$. By Corollary \ref{Kunnethdeformationresult}, for each $1\leq j\leq i$, there exists an exact sequence
	\begin{align}
		0\rightarrow R/Rx\otimes_R H^j(P^{\bullet})\rightarrow H^j(R/Rx\otimes_R P^{\bullet}).
	\end{align}
	Note that, $ H^j(R/Rx\otimes_R P^{\bullet})=R^jG_x(F_x(R/Rx\otimes_R M))=0$ and $H^j(P^{\bullet})=\Ext_B^j(FA, FM)$ for each $1\leq j\leq i$.
	
	Consider the surjective map $R/Rx\rightarrow R/\mi$ induced by the canonical map $R\rightarrow R/\mi$. Applying, for each $1\leq j\leq i$, $\Ext_B^j(FA, FM)\otimes_R -$ yields that $\Ext_B^j(FA, FM)(\mi)=0$, $1\leq j\leq i$. So, we conclude that $\Ext_B^j(FA, FM)=0$ for $1\leq j\leq i$. 
\end{proof}

We say that a cover $(A', P')$ is \textbf{truncated from} $(A, P)$ if there exists an ideal $J$ of $A$ such that $(A', P')$ and $(A/J, P/JP)$ are two isomorphic covers.
We shall now see that under some conditions truncating a cover, the quality of the cover drops at most by one.

\begin{Theorem} \label{truncationcovers} Let $R$ be a commutative Noetherian ring. 
	Let $I$ be an ideal of $R$ such that $I\in R\proj$ and $i\geq 0$. Let $\mathcal{R}_A$ be a well behaved resolving subcategory of $A\m\cap R\proj$. Let $(A, P)$ be an $i$-$\mathcal{R}_A$ cover of $B$.
	Then $(R/I\otimes_R A, R/I\otimes_R P)$ is an $(i-1)$-$\mathcal{R}_A(R/I\otimes_R A)$ cover of $R/I\otimes_R B$.
\end{Theorem}
\begin{proof}
	Denote by $Q$ the commutative ring $R/I$. Consider the exact sequence
	\begin{align}
		0\rightarrow I\rightarrow R\rightarrow Q\rightarrow 0 \label{eqfc60}.
	\end{align} This exact sequence induces the fully faithful functor $H\colon Q\otimes_R A\m\rightarrow A\m$. Moreover, for every $M\in Q\otimes_R A\m$,
	$Q\otimes_R HM \simeq HM/IHM=HM=M$. Hence, it is enough, to show that if $\eta_{Q\otimes_R M}^Q$ is an isomorphism (resp. a monomorphism) for every $M\in \mathcal{R}_A$, then $(Q\otimes_R A, Q\otimes_R P)$ is a $0$-$\mathcal{R}_A(Q\otimes_RA)$ (resp. $(-1)$-$\mathcal{R}_A(Q\otimes_RA)$) cover of $Q\otimes_R B$.
	
	Here, $\eta^Q$ denotes the unit associated with the adjunction \begin{align}
		F_Q:=\Hom_{Q\otimes_R A}(Q\otimes_R P, -) \dashv \Hom_{Q\otimes_R B}(F_Q(Q\otimes_R A), -):=G_Q.	\end{align}
	First, we will show that for every $M\in \mathcal{R}_A$ we can relate $\eta_M$ with $\eta^Q_{Q\otimes_R M}$.
	
	Applying $-\otimes_R M$ and $GF(-\otimes_R M)$ to (\ref{eqfc60})  yields the commutative diagram
	\begin{equation}
		\begin{tikzcd}
			0 \arrow[r] & I\otimes_R M\arrow[r] \arrow[d, "\eta_{I\otimes_R  M}"] & M \arrow[r] \arrow[d, "\eta_M"] & R/I\otimes_R M \arrow[d, "\eta_{Q\otimes_R M}"] \arrow[r] & 0\\
			0 \arrow[r] & GF(I\otimes_R M) \arrow[r] & GFM \arrow[r] & GF(R/I\otimes_R M)\arrow[r] & \Ext_B^1(FA, F(I\otimes_R M))
		\end{tikzcd}\label{eqfc62}
	\end{equation} with exact rows. Since $I\in R\proj$, $I\otimes_R M\in \add_A M$. Thus, $I\otimes_R M\in \mathcal{R}_A(A)$. Hence, $\eta_{I\otimes_R M}$ and $\eta_M$ are isomorphisms. By Snake Lemma, $\eta_{Q\otimes_R M}$ is a monomorphism. If $\Ext_B^1(FA, F(I\otimes_R M))=0$, then $\eta_{Q\otimes_R M}$ is an isomorphism.
	
	On the other hand, there are isomorphisms $\delta$ and $\psi$ making the following diagram commutative:
	\begin{equation}
		\begin{tikzcd}
			Q\otimes_R M \arrow[r, "\eta_{Q\otimes_R M}"] \arrow[d, "\mu"] & GF(Q\otimes_R M)\\
			\Hom_{Q\otimes_R B}(Q\otimes_R FA, Q\otimes_R FM)\arrow[r, "\delta"] & \Hom_B(FA, Q\otimes_R FM) \arrow[u, "\Hom_B(FA{,} \psi)"] 
		\end{tikzcd}. \label{eqfc63}
	\end{equation}
	
	By Lemma \ref{homtensorarbitraryring}, $\delta$ is an isomorphism. By Lemma \ref{tensoronsecondcomponetofhom}, $\psi$ is an isomorphism. We define $\mu$ to be the $Q\otimes_RA$-homomorphism that maps $m\otimes q$ to the map $(q_1\otimes f\mapsto qq_1\otimes f(-)m)$. We claim that (\ref{eqfc63}) is commutative. Let $m\in M, \ q\in Q, \ g\in FA, \ p\in P$. Then
	\begin{align}
		\Hom_B(FA, \psi)\delta\mu(q\otimes m)(g)(p)&= \psi(\delta \mu(q\otimes m))(g)(p)= \psi(\mu(q\otimes m)(1_Q\otimes g))(p) \\&=\psi(q\otimes g(-)m)(p)=q\otimes g(p)m
		\\&=\eta_{Q\otimes_R M}(q\otimes m)(g)(p).
	\end{align}
	Finally, we shall relate $\mu$ with $\eta_{Q\otimes_R M}^Q$. There exists a commutative diagram
	\begin{equation}
		\begin{tikzcd}
			Q\otimes_R M\arrow[r, "\mu"] \arrow[d, "\eta_{Q\otimes_R M}^Q"] & \Hom_{Q\otimes_R B}(Q\otimes_R FA, Q\otimes_R FM)\arrow[d, "\Hom_B(Q\otimes_R FA{,} \varphi_M)"] \\
			\Hom_{Q\otimes_R B}(F_Q (Q\otimes_R A), F_Q(Q\otimes_R M)) \arrow[r, "\Hom_{Q\otimes_R B}(\varphi_A {,} F_Q(Q\otimes_R M))", outer sep=0.75ex] & \Hom_{Q\otimes_R B}(Q\otimes_R FA, F_Q(Q\otimes_R M))
		\end{tikzcd},\label{eqfc67}
	\end{equation}where $\varphi_X$, $X\in A\m$, is the canonical isomorphism $Q\otimes_R FX\rightarrow F_Q Q\otimes_R X$.
	In fact, for any $m\in M, p\in P, f\in FA, q_1, q_2, q_3\in Q$
	\begin{align*}
		\begin{aligned}
			\hspace*{-1em}\Hom_{Q\otimes_R B}(\varphi_A, F_Q(Q\otimes_R M))\eta_{Q\otimes_R M}^Q(q_1\otimes m)(q_2\otimes f)(q_3\otimes p)&= \eta_{Q\otimes_R M}^Q(q_1\otimes m)(\varphi_A(q_2\otimes f))(q_3\otimes p)
			\\ &= \varphi_A(q_2\otimes f)(q_3\otimes p)(q_1\otimes m)\\&=q_2q_3\otimes f(p) q_1\otimes m\\&= q_1q_2q_3\otimes f(p)m\end{aligned}\\ \begin{aligned}
			\Hom_B(Q\otimes_R FA, \varphi_M)\mu(q_1\otimes m)(q_2\otimes f)(q_3\otimes p)&= \varphi_M (\mu(q_1\otimes m)(q_2\otimes f))(q_3\otimes p)\\&= \varphi_M(q_1q_2\otimes f(-)m)(q_3\otimes p)
			\\&=q_1q_2q_3\otimes f(-)m(p)\\&=q_1q_2q_3\otimes f(p)m.\nonumber 
	\end{aligned}\end{align*}
	Thus, using commutative diagrams (\ref{eqfc63}) and (\ref{eqfc67}) we deduce that $\eta_{Q\otimes_R M}^Q$ is a monomorphism. Further, $\eta_{Q\otimes_R M}^Q$ is an isomorphism if $\R^1G(FM)=0$. So, the result follows for $i\in \{0, 1\}$. Assume that $i\geq 1$. Then $(Q\otimes_R A, Q\otimes_R P)$ is a $0$-$\mathcal{R}_A(Q\otimes_R A)$ cover of $Q\otimes_R B$. The exact sequence (\ref{eqfc60}) yields that $\flatdim_R Q\leq 1$. Let $FA^{\bullet}$ be a deleted projective $B$-resolution of $FA$ and $M\in \mathcal{R}_A$. By Corollary \ref{Kunnethdeformationresult}, for each $n\geq 0$, there exists an exact sequence
	\begin{align}
		H^n(\Hom_B(FA^{\bullet}, FM))\otimes_R Q\hookrightarrow H^n(\Hom_B(FA^{\bullet}, FM)\otimes_R Q)\twoheadrightarrow \Tor_1^R(H^{n+1}(\Hom_B(FA^{\bullet}, FM)), Q). \nonumber
	\end{align}
	Notice that $H^n(\Hom_B(FA^{\bullet}, FM))=\Ext_B^n(FA, FM)=0, \ 1\leq n\leq i$. Hence, \begin{align}
		0&=H^n(\Hom_B(FA^{\bullet}, FM)\otimes_RQ)=H^n(\Hom_{Q\otimes_R B}(Q\otimes_R FA^{\bullet}, Q\otimes_R FM))\\&=H^n(\Hom_{Q\otimes_R B}(F_Q(Q\otimes_R A)^{\bullet}, F_Q Q\otimes_R M) )=\Ext_{Q\otimes_R B}^n(F_QQ\otimes_R A, F_Q Q\otimes_R M), \ 1\leq n\leq i-1. \nonumber
	\end{align} It follows that $(Q\otimes_R A, Q\otimes_R P)$ is an $(i-1)$-$\mathcal{R}_A(Q\otimes_R A)$ cover of $Q\otimes_R B$.
\end{proof}

We can describe Theorem \ref{truncationcovers} not just for projective ideals of $R$ but also for prime ideals of $R$ in case $R$ is a commutative Noetherian regular local ring.

\begin{Cor}\label{coverheightprimeideal}
	Let $R$ be a commutative Noetherian regular local ring. Let $\mathcal{R}_A$ be a well behaved resolving subcategory of $A\m\cap R\proj$. Let $(A, P)$ be an $i$-$\mathcal{R}_A$ cover of $B$ for some integer $i\geq 0$. Then \linebreak $(R/\mathfrak{p}\otimes_R A, R/\mathfrak{p}\otimes_R P)$ is an $(i-\height(\mathfrak{p}))$-$\mathcal{R}_A(R/\mathfrak{p}\otimes_R A)$ cover of $R/\mathfrak{p}\otimes_R B$ for every prime ideal $\mathfrak{p}$ of $R$ with $\height(\mathfrak{p})\leq i+1$.
\end{Cor}
\begin{proof}
	Let $\mathfrak{p}$ be a prime ideal of $R$. Suppose that, for $n=\height(\mathfrak{p})$, \begin{align}
		0=\mathfrak{p}_0\subset \mathfrak{p}_1\subset \cdots \subset \mathfrak{p}_n=\mathfrak{p}
	\end{align} is the largest chain of distinct prime ideals that are contained in $\mathfrak{p}$. We will proceed by induction on $n=\height(\mathfrak{p})$.
	
	If $n=0$, there is nothing to show. Assume that $n>0$. By construction, $\height(\mathfrak{p}_{n-1})=\height(\mathfrak{p})-1=n-1$, or even, $\height(\mathfrak{p}/\mathfrak{p}_{n-1})=1$. By induction, $(R/\mathfrak{p}_{n-1}\otimes_R A, R/\mathfrak{p}_{n-1}\otimes_R P)$ is an $(i-\height(\mathfrak{p}_{n-1}))$-$\mathcal{R}_A(R/\mathfrak{p}_{n-1}\otimes_R A)$ cover of $R/\mathfrak{p}_{n-1}\otimes_R B$. On the other hand, $R/\mathfrak{p}_{n-1}$ is a local regular ring. Hence, $R/\mathfrak{p}_{n-1}$ is a unique factorization domain. Therefore, every prime ideal of height one is principal. So, $\mathfrak{p}/\mathfrak{p}_{n-1}=R/\mathfrak{p}_{n-1}x\in R/\mathfrak{p}_{n-1}\proj$ for some $x\in R/\mathfrak{p}_{n-1}$. Note that, $i-\height(\mathfrak{p}_{n-1})=i-\height(\mathfrak{p})+1\geq i-i-1+1=0$.  By Theorem \ref{truncationcovers},  \begin{align*}
		(R/\mathfrak{p}\otimes_R A, R/\mathfrak{p}\otimes_R P)=(R/\mathfrak{p}_{n-1}/\mathfrak{p}/\mathfrak{p}_{n-1}\otimes_{R/\mathfrak{p}_{n-1}} R/\mathfrak{p}_{n-1}\otimes_R A, R/\mathfrak{p}_{n-1}/\mathfrak{p}/\mathfrak{p}_{n-1}\otimes_{R/\mathfrak{p}_{n-1}} R/\mathfrak{p}_{n-1}\otimes_R P)
	\end{align*}is an $i-\height(\mathfrak{p})$-$\mathcal{R}_A(R/\mathfrak{p}\otimes_R A)$ cover of $R/\mathfrak{p}\otimes_R B$.
\end{proof}

Now, we shall see that under some conditions we can obtain a reciprocal statement of Theorem \ref{truncationcovers}. Furthermore, we want to establish, similar to Rouquier's work, 
that by increasing the Krull dimension of the ground ring we can create covers with higher quality.

\begin{Theorem}\label{deformationpartthhre}
	Let $R$ be a commutative Noetherian regular ring with Krull dimension at least one. Let $i\geq 0$. Let $\mathcal{R}_A$ be a well behaved resolving subcategory of $A\m\cap R\proj$. Let $P\in A\m\cap R\proj$. Assume that $(K\otimes_R A, K\otimes_RP)$ is an $(i+1)$-$\mathcal{R}_A(K\otimes_RA)$ cover of $K\otimes_R B$ for some Noetherian commutative flat $R$-algebra $K$. 
	
	If $(A(\mi), P(\mi))$ is an $i$-$\mathcal{R}_A(A(\mi))$ cover of $B(\mi)$ for every maximal ideal $\mi$ of $R$, then $(A, P)$ is a $(1+i)$-$\mathcal{R}_A$ cover of $B$.
\end{Theorem}
\begin{proof}
	We can assume, without loss of generality, that $R$ is a regular local commutative Noetherian ring. By Proposition \ref{faithfulcoversresiduepart3},  $(A, P)$ is an $i$-$\mathcal{R}_A$ cover of $B$. Let $M\in \mathcal{R}_A$. It is enough to show that $\R^{i+1}G(FM)=0$. Hence, we want to show that the annihilator of $\R^{i+1}G(FM)$ is $R$.  Assume, by contradiction, that $\Ann_R \R^{i+1}G(FM) =0$. In particular, $\R^{i+1}G(FM)$ is a faithful $R$-module. Thus, there exists an exact sequence
	\begin{align}
		0 \rightarrow R\rightarrow \bigoplus_{I} \R^{i+1}G(FM), 
	\end{align}for some set (possibly infinite) $I$. Since $K$ is flat over $R$ we obtain a monomorphism \linebreak\mbox{$K\rightarrow  \bigoplus_{I}  K\otimes_R \R^{i+1}G(FM)$.} On the other hand, as $K\otimes_R M\in \mathcal{R}_A(K\otimes_R A)$,
	\begin{align*}
		K\otimes_R \R^{i+1}G(FM) \simeq \Ext_{K\otimes_R B}^{i+1}(K\otimes_R FA, K\otimes_R FM)\simeq
		\Ext_{K\otimes_R B}^{i+1}(F_K (K\otimes_R A), F_K(K\otimes_R M))=0.
	\end{align*}Here, $F_K$ denotes the functor $\Hom_{K\otimes_R A}(K\otimes_R P, -)$.
	This would imply that $K=0$. Hence, $\R^{i+1}G(FM)$ cannot be $R$-faithful. Moreover, there exists a non-zero divisor  $x\in R$ such that \begin{align}
		\R^{i+1}G(FM)[x]:=\{y\in \R^{i+1}G(FM)\colon xy=0 \} = \R^{i+1}G(FM). \label{eqfc86}
	\end{align} Observe that if $x_1x_2 y=0$, then $x_2y\in \R^{i+1}G(FM)[x_1]$ where $y\in \R^{i+1}G(FM)$ and $x_1$ and $x_2$ belong to the unique maximal ideal $\mi$. Thus, we can assume without loss of generality, that the element $x$ given in  (\ref{eqfc86}) belongs to $\mi\backslash \mi^2$.
	Furthermore, $\mi/Rx$ is the unique maximal ideal of $R/Rx$ so that \mbox{$(R/Rx\otimes_R A(\mi/Rx), R/Rx\otimes_R P (\mi/Rx))$}$=(A(\mi), P(\mi))$ is an $i$-$\mathcal{R}_A(A(\mi))$ cover of $B(\mi)$. Therefore, $(R/Rx\otimes_R A, R/Rx\otimes_R P)$ is an $i$-$\mathcal{R}_A(R/Rx\otimes_R A)$ cover of $R/Rx\otimes_R B$. Denote by $F_x$ and $G_x$, with $F_x \dashv G_x$, the adjoint functors associated with this cover. Let \begin{align}
		\Hom_A(P, A)^{\bullet}\colon\cdots\rightarrow Q_1\rightarrow Q_0\rightarrow 0
	\end{align}  be a deleted complex chain obtained by deleting $\Hom_A(P, A)$ from a projective $B$-resolution of $\Hom_A(P, A)$.

	Observe that	$\pdim_R R/Rx \leq 1$, so applying $P^{\bullet}:=\Hom_B(\Hom_A(P, A)^{\bullet}, \Hom_A(P, M))$ on Corollary \ref{Kunnethdeformationresult} yields  exact sequences
	\begin{align}
		0\rightarrow R/Rx\otimes_RH^n( P^\bullet) \rightarrow H^n(R/Rx\otimes_R P^\bullet) \rightarrow \Tor_1^R(H^{n+1}(P^\bullet), R/Rx)\rightarrow 0, \ \forall n\geq 0. \label{eqfc64}
	\end{align} 
	
	First, assume  that $i>0$. 
	Then $H^i(R/Rx\otimes_R P^\bullet)= \R^iG_x (F_x R/Rx\otimes_R M)=0$. So,
	\begin{align}
		\R^{i+1}G(FM)=\R^{i+1}G(FM)[x]=\Tor_1^R(H^{i+1}(P^\bullet), R/Rx)=0.
	\end{align}
	
	Now, assume that $i=0$. We need to proceed by induction on the Krull dimension of $R$. If $\dim R=1$, then $Rx=\mi$. As $R/\mi$ is a field and \begin{align}
		R/\mi\otimes_R H^0(P^\bullet)= R/\mi\otimes_R GFM\simeq M(\mi)\simeq G_x F_x (M(\mi))= H^0(R/\mi\otimes_R P^\bullet)
	\end{align} the exact sequence (\ref{eqfc64}) yields that 
	\begin{align}
		\R^{1}G(FM)=\R^{1}G(FM)[x]=\Tor_1^R(H^{1}(P^\bullet), R/\mi)=0.
	\end{align} Assume that the result holds for all rings with Krull dimension less than $t$. Let $R$ have Krull dimension $t$. The Krull dimension of $R/Rx$ is $t-1$. By induction, $\R^1G_x(F_x R/Rx\otimes_R M)=0$. The exact sequence (\ref{eqfc64}) implies that $0=R/Rx\otimes_R H^1(P^\bullet)=R/Rx\otimes_R \R^1G(FM)$. Applying the functor $\R^1G(FM)\otimes_R -$ on the surjective map $R/Rx\rightarrow R/\mi$ we get that $\R^1G(FM)(\mi)=0$. Thus, $\R^1G(FM)=0$. This completes the proof. 
\end{proof}

We remark that Proposition 4.42 of \citep{Rouquier2008} is a particular case of Theorem \ref{deformationpartthhre} by fixing $\mathcal{R}_A=\mathcal{F}(\Stsim)$ and $i=1$. To illustrate, recall that for a flat $R$-algebra  $K$ with $\gldim K\otimes_R A=0$, every module in $K\otimes_R A\m$ is projective over $K\otimes_R A$. So, $\mathcal{R}_A(K\otimes_R A)=K\otimes_R A\m$.  By Proposition \ref{arbitraryfaithfulcoverflattwo}, $(K\otimes_R A, K\otimes_R P)$ is a 0-$\mathcal{R}_A(K\otimes_R A)$ cover of $K\otimes_R B$. By Lemma \ref{whenunitisiso}, the functor ${\Hom_{K\otimes_R A}(K\otimes_R P, -)\colon K\otimes_R A\m\rightarrow K\otimes_R B\m}$ is full and faithful. Consequently, it is an equivalence of categories.

It is now natural to ask, knowing Theorem \ref{deformationpartthhre}, how large can be the difference between the level of faithfulness of the covers (if they exist) $(A, P)$ and $(A(\mi), P(\mi))$ for $\mi\in \MaxSpec R$.

The answer to this question in theoretical terms depends heavily on the spectrum of the ground ring.

\subsection{Quality of a cover and the spectrum of the ground ring}
In the same spirit of Theorem \ref{deformationpartthhre}, we can obtain a converse statement for Corollary \ref{coverheightprimeideal}.

\begin{Theorem}\label{improvingcoverwithspectrum}
	Let $R$ be a regular local commutative Noetherian ring with quotient field $K$. Suppose that $(A, P)$ is a $0$-$\mathcal{R}_A$ cover of $B$ for some well behaved resolving subcategory $\mathcal{R}_A$ of $A\m\cap R\proj$. Let $i\geq 0$.
	Assume that the following conditions hold:
	\begin{enumerate}[(i)]
		\item $(K\otimes_R A, K\otimes_R P)$ is an $(i+1)$-$\mathcal{R}_A(K\otimes_R A)$ cover of $K\otimes_R B$.
		\item For each prime ideal $\mathfrak{p}$ of height one, $(R/\mathfrak{p}\otimes_R A, R/\mathfrak{p}\otimes_R P)$ is an $i$-$\mathcal{R}_A(R/\mathfrak{p}\otimes_R A)$ cover of $R/\mathfrak{p}\otimes_R B$.
	\end{enumerate}
	Then $(A, P)$ is an $(i+1)$-$\mathcal{R}_A$ cover of $B$.
\end{Theorem}
\begin{proof}
	Let $1\leq j\leq i+1$ and let $M\in \mathcal{R}_A$. Denote by $F_K$ and $G_K$ the adjoint functors associated with the cover $(K\otimes_R A, K\otimes_R P)$ and denote by $F_\mathfrak{p}$ and $G_\mathfrak{p}$ the adjoint functors associated with the cover $(R/\mathfrak{p}\otimes_R A, R/\mathfrak{p}\otimes_R P)$, for each prime ideal $\mathfrak{p}$ of $R$. Assumption $(i)$ implies that
	\begin{align}
		K\otimes_R \R^jG(FM)\simeq \R^jG_K(F_K K\otimes_R M)=0.
	\end{align} Hence, $\R^jG(FM)$ cannot be $R$-faithful. Moreover, for each $1 \leq j\leq i+1$, there exists a non-zero divisor $x_j\in \mi/\mi^2$ such that \begin{align}
		\R^{j}G (FM)[x_j]=\R^jG(FM),
	\end{align}where $\mi$ is the unique maximal ideal of $R$. Since $x_j\in \mi/\mi^2$, $R/Rx_j$ is an integral domain of Krull dimension $\dim R-1$. So, $Rx_j$ is a prime ideal of height one. Fixing $P^{\bullet}:=\Hom_B(\Hom_A(P, A)^{\bullet}, \Hom_A(P, M))$ on Corollary \ref{Kunnethdeformationresult}  we get  exact sequences
	\begin{align}
		0\rightarrow R/Rx_j\otimes_RH^n( P^\bullet) \rightarrow H^n(R/Rx_j\otimes_R P^\bullet) \rightarrow \Tor_1^R(H^{n+1}(P^\bullet), R/Rx_j)\rightarrow 0, \ \forall n\geq 0. \label{eqfc103}
	\end{align}  Using now assumption $(ii)$ it follows that $H^{j-1}(R/Rx_j\otimes_R P^\bullet) =0$ for $i\geq j>1$. So, $\R^jG(FM)=0$ for $2\leq j\leq i+1$. The case $j=1$ requires a little more work. For each $x\in \mi/\mi^2$, consider the exact sequence
	\begin{align}
		0\rightarrow R\rightarrow R\rightarrow R/Rx\rightarrow 0,
	\end{align}where the first map is multiplication by $x$. Since $M\in R\proj$, we get an exact sequence
	\begin{align}
		0\rightarrow M\rightarrow M\rightarrow M/xM\rightarrow 0,
	\end{align}where the first map is multiplication by $x$. Denote by $\pi$ the projection $M\rightarrow M/xM$. Applying $\Hom_B(FA, F-)$ yields a long exact sequence
	\begin{align}
		GFM\rightarrow \Hom_B(FA, F(M/xM))\rightarrow \R^1G(FM)\rightarrow \R^1G(FM). \label{eqfc104}
	\end{align} By Lemma \ref{homtensorarbitraryring} and Lemma \ref{tensoronsecondcomponetofhom}, there exists a commutative diagram
	\begin{equation}
		\begin{tikzcd}
			GFM\arrow[r, "\Hom_B(FA{,} F\pi)", outer sep=0.75ex, swap] & \Hom_B(FA, F(M/xM)) \arrow[r, "\simeq"] & G_{Rx}F_{Rx}(M/xM)\\
			M\arrow[u, "\eta_M"] \arrow[rr, "\pi"] & & M/xM\arrow[u, "\eta^{Rx}_{M/xM}", "\simeq"']
		\end{tikzcd}.
	\end{equation}
	Thus, $\Hom_B(FA{,} F\pi)$ is surjective. By exactness of (\ref{eqfc104}), the map $\R^1G(FM)\rightarrow \R^1G(FM)$ is injective. Since this map is given by multiplication by $x$, its kernel is $\R^1G(FM)[x]=0$. As discussed before, $0=\R^1G(FM)[x]=\R^1G(FM)$. Thus, the result follows.
\end{proof}

Observe that the arguments used in the proof of Theorems \ref{deformationpartthhre} and \ref{faithfulcoversresiduepart3} remain valid if we are interested only in a given module $M\in \mathcal{R}_A$. Hence, the following corollary follows.

\begin{Cor}\label{deformationmodulesisolated}
	Let $(S\otimes_R A, S\otimes_R P)$ be a cover of $S\otimes_R B$ for any commutative  $R$-algebra $S$ which is a Noetherian ring and let $M\in$ $A\m\cap R\proj$. Assume that the following conditions hold.
	\begin{enumerate}
		\item The unit $\eta_M\colon M\rightarrow GFM$ is an isomorphism.
		\item $\Ext_{B(\mi)}^j(FA (\mi), FM(\mi) )=0$ for every maximal ideal $\mi$ of $R$, where $1\leq j\leq i$ for some $i\geq 0$.
	\end{enumerate}
	Then $\Ext_B^{j}(FA, FM)=0$ for all $1\leq j\leq i$. If, in addition, $\dim R\leq 1$ and there exists a Noetherian commutative $R$-algebra $K$ such that $\Ext_{K\otimes_R B}^{i+1}(K\otimes_R FA, K\otimes_R FM)=0$, then $\Ext_B^{j}(FA, FM)=0$ for all $1\leq j\leq i+1$.
\end{Cor}

\section{Hemmer--Nakano dimension with respect to covers constructed from relative injective modules}\label{Hemmer--Nakano dimension with respect to covers constructed from relative injective modules}

As we saw in Proposition \ref{dominantgeqtwodcp}, for an RQF3 algebra $(A, P, V)$ so that $\domdim (A, R)\geq 2$, \linebreak$(A, \Hom_{A^{op}}(V, A))$ is a cover of $\End_A(V)$. In the following, we will see how we can use relative dominant dimension to measure the quality of this cover.

\begin{Theorem}\label{boundrelationcoversdom}
	Let $(A, P, V)$ be an RQF3 algebra over a commutative Noetherian regular ring $R$ with \linebreak$\domdim (A, R)\geq 2$. Suppose that $\mathcal{R}_A$ is a well behaved resolving subcategory of $A\m\cap R\proj$. Let $$n=\inf \{\domdim_{(A, R)} M\colon M\in \mathcal{R}_A \}\in \mathbb{Z}_{\geq 0}\cup \{+\infty\}.$$ Then $(A, \Hom_{A^{op}}(V, A))$ is an $(n-2)$-$\mathcal{R}_A$ cover of $\End_A(V)$. Moreover, the Hemmer--Nakano dimension of $\mathcal{R}_A$ (with respect to $\Hom_{A^{op}}(V, A)$) is less than or equal to $n+\dim R -2$.
\end{Theorem}
\begin{proof}
	By Proposition \ref{dominantgeqtwodcp}, $(A, \Hom_{A^{op}}(V, A))$ is a cover of $\End_A(V)$. Observe that since $V$ is projective as $A$-module, for every $i\in \mathbb{N}$ there exists isomorphisms
	\begin{multline}
		\Ext_B^ i(V, V\otimes_A M)\simeq \Ext_B^i(\Hom_A(\Hom_{A^{op}}(V, A), A), \Hom_A(\Hom_{A^{op}}(V, A), A)\otimes_A M)\\
		\simeq \Ext_B^i(\Hom_A(\Hom_{A^{op}}(V, A), A), \Hom_A(\Hom_{A^{op}}(V, A), M))=\Ext_B^i(FA, FM)=\R^iG(FM). \label{eq59}
	\end{multline} Assume $n=0$. By contradiction, assume that $(A, \Hom_{A^{op}}(V, A))$ is a $(\dim R-1)$-$\mathcal{R}_A$ cover of $\End_A(V)$. If $\dim R=0$, then every localization of $R$ at a maximal ideal is a field. In particular, $\eta_{M_\mi}$ is a monomorphism for every maximal ideal $\mi$ in $R$ for every $M\in \mathcal{R}_A$. As $R_\mi$ is a field, in view of Lemma \ref{unitschurfunctoranddominant}, $\domdim_{A_\mi} M_\mi\geq 1$ for every maximal ideal $\mi$ in $R$. By  \citep[Proposition 6.8]{CRUZ2022410}, $\domdim_{(A, R)} M\geq 1$ for every $M\in \mathcal{R}_A$. This is a contradiction with $n$ being zero. If $\dim R\geq 1$, then $\eta_M$ is an isomorphism for every $M\in \mathcal{R}_A$ by Proposition \ref{zeroAcover}. By Theorem \ref{Mullertheorem} and Lemma \ref{unitschurfunctoranddominant}, $\domdim_{(A, R)} M\geq 1$ for every $M\in \mathcal{R}_A$ which contradicts our assumption on $n$. 
	
	Now assume that $n=1$. For every $M\in \mathcal{R}_A$, $\domdim_{(A, R)} M\geq 1$. Hence, for every maximal ideal $\mi$ of $R$, $\domdim_{A(\mi)} M(\mi)\geq 1$. By Lemma \ref{unitschurfunctoranddominant} and Theorem \ref{Mullertheorem}, $\eta_{M(\mi)}$ is a monomorphism for every maximal ideal $\mi$ in $R$. By Lemma \ref{unitmono}, $\eta_M$ is an $(A,R)$-monomorphism for every $M\in \mathcal{R}_A$. By Lemma \ref{considerationseta}, we obtain that $(A, \Hom_{A^{op}}(V, A))$ is a $(-1)$-$\mathcal{R}_A$ cover of $\End_A(V)$. By contradiction, assume that $(A, \Hom_{A^{op}}(V, A))$ is a $\dim R$-$\mathcal{R}_A$ cover of $\End_A(V)$. Then, in particular, $\eta_M$ is an isomorphism and $\R^iG(FM)=0$, $1\leq i\leq \dim R$, for every $M\in \mathcal{R}_A$. By Lemma \ref{unitschurfunctoranddominant}, $\alpha_M$ is an isomorphism for every $M\in \mathcal{R}_A$. By Theorem \ref{Mullertheorem} and (\ref{eq59}), $\domdim_{(A, R)} M\geq \dim R+2-\dim R=2$ for every $M\in \mathcal{R}_A$ which contradicts our assumption of $n$.

	Finally, assume that $n\geq 2$. By Theorem \ref{Mullertheorem} and (\ref{eq59}), $\alpha_M$ is an isomorphism for every $M\in \mathcal{R}_A$ and
	\begin{align}
		0&=\Ext_B^ i(V, V\otimes_A M)\simeq \R^iG(FM), \ 1\leq i\leq n-2.\nonumber
	\end{align}
	Hence, $(A, \Hom_{A^{op}}(V,A))$ is an $(n-2)$-$\mathcal{R}_A$ cover of $B$. If $n=+\infty$, then we are done. If $n<+\infty$, using again Theorem \ref{Mullertheorem}, we see that $(A, \Hom_{A^{op}}(V, A))$ cannot be an $(n+3-\dim R)$-$\mathcal{R}_A$ cover of $B$.
\end{proof}

For algebras admitting additional properties and with Krull dimension greater than one, we can improve the lower bound.

\begin{Theorem}\label{boundcoverimprovementdom}
	Let $R$ be a commutative Noetherian regular ring with Krull dimension at least one. Let $(A, P, V)$ be an RQF3 algebra over a commutative Noetherian ring $R$ with $\domdim (A, R)\geq 2$.  Let $\mathcal{R}_A$ be a well behaved resolving subcategory of $A\m\cap R\proj$. 
	Let $n=\inf \{\domdim_{(A, R)} M\colon M\in \mathcal{R}_A \}\in \mathbb{Z}_{\geq 0}\cup \{+\infty\}$ and suppose that $n\geq 2$.
	Assume that $\inf \{\domdim_{(K\otimes_R A, K)} N\colon N\in \mathcal{R}_A(K\otimes_RA) \}\geq n+1$  for some Noetherian commutative flat $R$-algebra $K$. 
	
	Then $(A, \Hom_{A^{op}}(V, A))$ is an $(n-1)$-$\mathcal{R}_A$ cover of $\End_A(V)$. Moreover, if $\dim R=1$ the Hemmer--Nakano dimension of $\mathcal{R}_A$ (with respect to $\Hom_{A^{op}}(V, A)$) is $n-1$.
\end{Theorem}
\begin{proof}
	Let $B$ denote the endomorphism algebra $\End_A(V)$. Let $\mi$ be a maximal ideal in $R$. Fix $F_{(\mi)}$ the functor \mbox{$\Hom_{A(\mi)}(\Hom_{A(\mi)^{op}}(V(\mi), A(\mi)), -)$} and $G_{(\mi)}$ its right adjoint. Since every module in $\mathcal{R}_A(A(\mi))$ is constructed as a direct summand or by extensions of modules  $R(\mi)\otimes_R M$ for $M\in \mathcal{R}_A$, it is enough to check that $\eta_{M(\mi)}$ is an isomorphism and $\R^i G_{(\mi)}F_{(\mi)}(M(\mi))=0$, $1\leq i\leq n-2$ for every $M\in \mathcal{R}_A$ to deduce that $(A(\mi), \Hom_{A(\mi)^{op}}(V(\mi), A(\mi)))$ is an $(n-2)$-$\mathcal{R}_A$ cover of $B(\mi)$. Let $M\in \mathcal{R}_A$. Then
	\begin{align}
		\domdim_{A(\mi)} M(\mi)\geq \domdim_{(A, R)} M\geq 2.
	\end{align}By Theorem \ref{Mullertheorem},
	\begin{align}
		0=\Ext_{B(\mi)}^i(V(\mi), V(\mi)\otimes_{A(\mi)} M(\mi))=\Ext_{B(\mi)}^i(F_{(\mi)}A(\mi), F_{(\mi)}(M(\mi))), \ 1\leq i\leq n-2.
	\end{align}In the same way, $(K\otimes_R A, K\otimes_R \Hom_{A^{op}}(V, A))$ is an $(n-1)$-$\mathcal{R}_A(K\otimes_R A)$ cover of $\End_{K\otimes_R A}(K\otimes_R V)$.
	By Theorem \ref{deformationpartthhre}, $(A, \Hom_{A^{op}}(V, A))$ is an $(n-1)$-$\mathcal{R}_A$ cover of $\End_A(V)$.
\end{proof}

If the ground ring is an integral domain, then we can use its quotient field to take the role of $K$. Even better using the quotient field we can improve Theorem \ref{boundcoverimprovementdom} to include the case $n=1$.

\begin{Theorem}\label{boundcoverimprovementdomquotientfield}
	Let $R$ be a commutative Noetherian regular integral domain with Krull dimension at least one and with quotient field $K$. Let $(A, P, V)$ be an RQF3 algebra over a commutative Noetherian ring $R$ with $\domdim (A, R)\geq 2$.  Let $\mathcal{R}_A$ be a well behaved resolving subcategory of $A\m\cap R\proj$. 
	Let $$n=\inf \{\domdim_{(A, R)} M\colon M\in \mathcal{R}_A\}\in \mathbb{Z}_{\geq 0}\cup \{+\infty\}.$$
	Assume that $n\geq 1$ and \mbox{$\inf \{\domdim_{(K\otimes_R A, K)} N\colon N\in \mathcal{R}_A(K\otimes_RA) \}\geq n+1$}. Then $(A, \Hom_{A^{op}}(V, A))$ is an \mbox{$(n-1)$-$\mathcal{R}_A$} cover of $\End_A(V)$. Moreover, if $\dim R=1$ the Hemmer--Nakano dimension of $\mathcal{R}_A$ (with respect to $\Hom_{A^{op}}(V, A)$) is $n-1$. 
\end{Theorem}
\begin{proof}
	The case $n\geq 2$ is just a particular case of Theorem \ref{boundcoverimprovementdom}. 	
	We will now consider the case $n=1$. Hence, $\domdim_{A(\mi)} M(\mi)\geq \domdim_{(A, R)} M\geq 1$ for any $M\in \mathcal{R}_A$. Taking into account that $\domdim (A, R)\geq 2$ we obtain by Proposition \ref{faithfulcoverresiduefieldnext} that $(A, P)$ is a $(-1)$-$\mathcal{R}_A$ cover of $\End_A(V)$. Further, by Lemma \ref{unitschurfunctoranddominant}, Theorem \ref{Mullertheorem} and Proposition \ref{dominantgeqtwodcp}, we obtain that the unit map $\eta_M\colon M\rightarrow GFM$ is an $(A, R)$-monomorphism for every $M\in \mathcal{R}_A$. On the other hand, $K\otimes_R M\in \mathcal{R}_A(K\otimes_R A)$. Hence, $\eta_{K\otimes_R M}$ is an isomorphism by assumption. Thus, $K\otimes_R \eta_M$ is an isomorphism since $K$ is flat over $R$.
	
	Denote by $X$ the cokernel of $\eta_M$. By the flatness of $K$ and $K\otimes_R \eta_M$ being an isomorphism, it follows that $K\otimes_R X=0$. In particular, $X$ is a torsion $R$-module. Using the monomorphism
	\begin{align}
		GFM\rightarrow \Hom_R(V, FM)\in R\proj, 
	\end{align} we deduce that $GFM$ is a torsion-free $R$-module. By a result of Auslander-Buchsbaum (see Proposition 3.4 of \citep{zbMATH03151673}) if $X\neq 0$, then there exists a prime ideal of height one $\mathfrak{q}$ such that $X_\mathfrak{q}\neq 0$. But $\dim R_\mathfrak{q}=1$, so the localization $GFM_{\mathfrak{q}}$ is a projective $R_{\mathfrak{q}}$-module. Thus, using the fact that ${\eta_{M}}_\mathfrak{q}$ is an $(A_\mathfrak{q}, R_\mathfrak{q})$-monomorphism we obtain that $X_{\mathfrak{q}}$ is a projective $R_\mathfrak{q}$-module. By applying the tensor product $K\otimes_{R_{\mathfrak{q}}} -$ it follows that $X_\mathfrak{q}=0$. So, we must have $X=0$. Hence, $\eta_M$ is an isomorphism. This finishes the proof.
\end{proof}

\subsection{Hemmer-Nakano dimension of $A\proj$}\label{Hemmer-Nakano dimension of}

It is an easy consequence of $\domdim_{(A, R)} (M_1\oplus M_2) =\inf \{\domdim_{(A, R)} M_1, \domdim_{(A, R)} M_2\}$ (see \citep[Corollary 5.10]{CRUZ2022410}) that $$\domdim (A, R)=\inf \{\domdim_{(A, R)} M\colon M\in A\proj \}.$$
In particular,  given an  RQF3 algebra $(A, P, V)$ with $\domdim{(A,R)}\geq 2$, Theorem \ref{boundrelationcoversdom} gives that $(A, \Hom_{A^{op}}(V, A))$ is a cover of $B$ with \begin{align}
	\domdim (A, R)-2\leq \HN_{\Hom_A(\Hom_{A^{op}}(V, A), -)} A\proj\leq \domdim (A, R)-2+\dim R.
\end{align}
The idea of computing the Hemmer-Nakano dimension of $A\proj$ using the dominant dimension of the regular module goes back to \citep{zbMATH05871076}.

\subsection{Hemmer-Nakano dimension of $\mathcal{F}(\Stsim)$}\label{Hemmer-Nakano dimension of f}

We will now see what Theorem \ref{boundrelationcoversdom} gives for $\mathcal{F}(\Stsim)$-covers. 	The answer is based on  \citep[Corollary 3.7]{zbMATH05871076}.

\begin{Theorem}\label{domdimofmoduleswithfiltrationbystandard}
	Let $(A, \{\Delta(\lambda)_{\lambda\in \Lambda}\})$  be a split quasi-hereditary algebra over a commutative Noetherian ring and $(A, P, V)$ an
	RQF3 $R$-algebra. Let $T$ be a characteristic tilting module. Then
	\begin{align}
		\domdim_{(A, R)} T=\inf \{\domdim_{(A, R)}\St(\l)\colon \l\in\L \}=\inf
		\{\domdim_{(A, R)} M\colon M\in \mathcal{F}(\Stsim) \} . \label{eq64}
	\end{align}
\end{Theorem}
\begin{proof}	Let $M\in \mathcal{F}(\Stsim)$.
	By Lemma 5.12 of \cite{CRUZ2022410}, \begin{align}
		\domdim_{(A, R)} M\geq &\inf\{\domdim_{(A,R)} \St(\l)\otimes_R U_\l\colon \l\in \L, \ U_\l\in R\proj \}  \nonumber\\  \geq &\inf \{\domdim_{(A, R)}\St(\l)\colon \l\in\L \}, \label{eq65}
	\end{align} since $\St(\l)\otimes_R U_\l\in \add \St(\l)$.
If $\inf \{\domdim_{(A, R)}\St(\l)\colon \l\in\L \} = +\infty$, then (\ref{eq65}) implies that (\ref{eq64}) reads $+\infty=+\infty=+\infty$. Assume now that $\inf \{\domdim_{(A, R)}\St(\l)\colon \l\in\L \} $ is finite.
	Denote by $c$ the value $\inf\{\domdim_{(A,R)} \St(\l)\colon \l\in \L \}$ and $\l_0$ the index such that $\domdim_{(A, R)} \St(\l_0)=c$.
	
By (\ref{eq65}), $
		\domdim_{(A, R)} M\geq \inf\{\domdim_{(A,R)} \St(\l)\otimes_R U_\l\colon \l\in \L, \ U_\l\in R\proj \}=c,
$ for any $M\in \mathcal{F}(\Stsim)$. Consider an $(A, R)$-exact sequence
	\begin{align}
		0\rightarrow \St(\l_0)\rightarrow T(\l_0)\rightarrow X(\l_0)\rightarrow 0,
	\end{align} for example the one constructed in \citep[Proposition 4.5]{appendix}. 	By Lemma 5.12 of \cite{CRUZ2022410}, $$\domdim_{(A, R)} T(\l_0)\geq \inf\{\domdim_{(A, R)} \St(\l_0), \domdim_{(A, R)} X(\l_0) \}.$$ Since $X(\l_0)\in \mathcal{F}(\Stsim)$ we obtain $\domdim_{(A, R)} X(\l_0)\geq c$. Hence, $$\inf\{\domdim_{(A, R)} \St(\l_0), \domdim_{(A, R)} X(\l_0) \}=c.$$ Assume that $\domdim_{(A, R)} T(\l_0)>c$. Then \begin{align}
		\domdim_{(A, R)} X(\l_0)=\domdim_{(A, R)} \St(\l_0)-1=c-1.
	\end{align}This contradicts the minimality of $c$. Thus, \begin{align}
		\domdim_{(A, R)} T=\inf \{\domdim_{(A, R)} T(\l)\colon \l\in \L \}=c.
		\tag*{\qedhere}	\end{align}
\end{proof}

As a consequence, if $(A, \{\Delta(\lambda)_{\lambda\in \Lambda}\})$  is a split quasi-hereditary algebra over a commutative Noetherian ring and $(A, P, V)$ is an
RQF3 $R$-algebra satisfying $\domdim{(A, R)}\geq 2 $ and $\domdim_{(A, R)} T\geq 1$, then $(A, \Hom_{A^{op}}(V, A))$ is a $(\domdim_{(A, R)} T - 2)$-$\mathcal{F}(\Stsim)$ cover of $\End_A(V)$. In particular, \begin{align}
	\domdim_{(A, R)}T-2\leq \HN_{\Hom_A(\Hom_{A^{op}}(V, A), -)} \mathcal{F}(\Stsim)\leq \domdim_{(A, R)}T-2+\dim R.
\end{align} This case applied to finite-dimensional algebras was first observed in \cite{zbMATH05871076}.

Instead of requiring both conditions $\domdim{(A, R)} \geq 2$ and $\domdim_{(A, R)} T\geq 1$, when can we just ask for $\domdim{(A,R)} \geq 2$?  In \cite{zbMATH05871076}, a class of algebras that generalize classical Schur algebras of positive dominant dimension and block algebras of the BGG category $\mathcal{O}$ was introduced for which the dominant dimension of the characteristic tilting module is exactly half the dominant dimension of the regular module. This result can be generalized to the integral setup. 
For this, we need the concept of duality of a Noetherian algebra.
	\begin{Def}\label{dualitydef}
	Let $A$ be a free Noetherian $R$-algebra. Assume that $A$ admits a set of orthogonal idempotents $\textbf{e}:=\{e_1, \ldots, e_t \}$ such that for each maximal ideal $\mi$ of $R$ $\{e_1(\mi), \ldots, e_t(\mi) \}$ becomes a complete set of primitive orthogonal idempotents of $A(\mi)$. We say that $A$ has a \textbf{duality} $\iota\colon A\rightarrow A$ (with respect to $\textbf{e}$) if  $\iota$ is an anti-isomorphism with $\iota^2=\id_A$ fixing the set of orthogonal idempotents $\{e_1, \ldots, e_t \}$. 
\end{Def}

Given a free Noetherian $R$-algebra $A$, we say that $(A, \textbf{e})$ is a \textbf{split quasi-hereditary algebra with a duality} $\iota$ if $\iota$ is a duality of $A$ with respect to $\textbf{e}:=\{e_1, \ldots, e_t \}$ and $A$ is split quasi-hereditary with split heredity chain $	0\subset Ae_tA\subset \cdots \subset A(e_1+\cdots+ e_t)A=A $. In such a case, $A$ is also a cellular algebra (see \citep[Proposition 4.0.1]{cruz2021cellular}). Recall that a pair $(A, P)$ is an \textbf{relative gendo-symmetric algebra} over a commutative Noetherian ring if $(A, P, DP)$ is a RQF3 algebra so that $\domdim (A, R)\geq 2$ and $P\simeq DA\otimes_A P$  as $(A, \End_A(P)^{op})$-bimodules.

\begin{Theorem}\label{tiltingmoduledominantdimensioncover}
	Let $A$ be a free Noetherian $R$-algebra. Assume that the following holds.
	\begin{itemize}
		\item $(A, \textbf{e})$ is split quasi-hereditary algebra with a duality.
		\item $(A, Af)$ is a relative gendo-symmetric $R$-algebra for some idempotent $f\in \mathbb{Z}\langle \textbf{e}\rangle$ of $A$.
	\end{itemize}Let $T$ be a characteristic tilting module of $A$.
	Then \begin{align}
		\domdim{(A, R)}=2\domdim_{(A, R)} T.
	\end{align}
\end{Theorem}
\begin{proof}
 Let $T$ be a characteristic tilting module of $A$. Let $\mi$ be a maximal ideal of $R$. Then $A(\mi)$ is a split quasi-hereditary algebra over $R(\mi)$ with characteristic tilting module $T_{(\mi)}$ so that $\add_{A(\mi)}T(\mi)= \add_{A(\mi)} T_{(\mi)}$  (see for example Proposition \citep[Proposition 5.1]{appendix} and  \citep[Theorem 3.1.2]{cruz2021cellular} or \citep[Theorem 4.15, Proposition 4.30]{Rouquier2008}). Fix $V:=fA$. Let $\theta$ be the $(\End_A(V), A)$-bimodule isomorphism given by $V\rightarrow V\otimes_A DA$. Applying the functor $R(\mi)\otimes_R -$ to $\theta$ gives an $(\End_{A(\mi)}(V(\mi)), A(\mi))$-bimodule isomorphism between $V(\mi)$ and $V(\mi)\otimes_{A(\mi)} \Hom_{R(\mi)}(A(\mi), R(\mi))$. In particular, $\domdim A(\mi)\geq \domdim (A, R)\geq 2$. Hence, $A(\mi)$ is a gendo-symmetric algebra. By Theorem \citep[Theorem 4.3]{zbMATH05871076}, $\domdim A(\mi)=2\domdim_{A(\mi)} T(\mi)$. Hence, by Theorem 6.13 of \citep{CRUZ2022410} we obtain that $	\domdim (A, R) = 2\domdim_{(A, R)} T$.
\end{proof}

For the split quasi-hereditary algebras satisfying Theorem \ref{tiltingmoduledominantdimensioncover}, the Hemmer-Nakano dimension of $\mathcal{F}(\Stsim)$ (with respect to $\Hom_A(P, Ae)$) is at least $\frac{1}{2}\domdim{(A, R)}-2$.

We can extend to Noetherian rings the result given in \citep{zbMATH05278765} which indicates an upper bound of the faithfulness of a faithful split quasi-hereditary cover in terms of relative dominant dimension of the algebra. This result is particularly useful when we have no clear relation between the relative dominant dimension of a characteristic tilting module and the relative dominant dimension of the regular module while still wanting to have a lower bound for the Hemmer-Nakano dimension of $\mathcal{F}(\Stsim)$ using the relative dominant dimension of the regular module. Recall the length of $\l$ in the poset $\L$ (see Subsection \ref{mathcalFSt}).

\begin{Prop}\label{faithfulcoverboundbydominantdime}
	Let $(A, \{\Delta(\lambda)_{\lambda\in \Lambda}\})$ be a split quasi-hereditary algebra over a commutative Noetherian ring.  For any $\l\in \L$, $\domdim_{(A, R)} \St(\l)\geq \domdim{(A, R)}-d(\L, \l)$.
\end{Prop}
\begin{proof} 
	We shall prove this result by induction on $d(\L, \l)$. If $d(\L, \l)=0$, then $\l$ is maximal in $\L$. Thus, $\St(\l)$ is a projective $A$-module. By \citep[Lemma 5.12, Corollary 5.10]{CRUZ2022410}, $\domdim_{(A, R)} \St(\l)\geq \domdim (A, R)$. 
	
	Suppose now the claim holds for all $\mu\in \L$ with $d(\mu)<t$ for some $t>1$. Let $\l\in \L$ such that $d(\l)=t$. Consider the $(A, R)$-exact sequence
	\begin{align}
		0\rightarrow K(\l)\rightarrow P(\l)\rightarrow \St(\l)\rightarrow 0,
	\end{align} where $K(\l)\in \mathcal{F}(\Stsim_{\mu>\l})$ and $P(\l)\in A\proj$, $\Stsim_{\mu>\l}=\{\St(\mu)\otimes_R U\colon U\in R\proj, \ \mu>\l \}$. Comparing lengths, $d(\L, \mu)<d(\L, \l)=t$ for $\mu>\l$. By induction, $\domdim_{(A, R)} \St(\mu)\geq \domdim(A, R)-d(\L, \mu)>\domdim (A, R) -d(\L, \l)$. By \citep[Lemma 5.12]{CRUZ2022410}, $\domdim_{(A, R)} K(\l)> \domdim (A, R) -d(\L, \l)$. If $\domdim_{(A, R)} P(\l)> \domdim_{(A, R)} K(\l)$, then by \citep[Lemma 5.12]{CRUZ2022410}, we have \begin{align}
		\domdim_{(A, R)} \St(\l)=\domdim_{(A, R)} K(\l)-1\geq \domdim (A, R) -d(\L, \l).
	\end{align}
	If $\domdim_{(A, R)} P(\l)\leq \domdim_{(A, R)} K(\l)$, then by \citep[Lemma 5.12]{CRUZ2022410}, we have
	\begin{align}
		\domdim_{(A, R)} \St(\l)\geq \domdim_{(A, R)} P(\l)-1\geq \domdim(A, R)-1\geq \domdim(A, R)-d(\L, \l).\tag*{\qedhere}
	\end{align}
\end{proof}

\begin{Prop}\label{faithfulcoverupperbound}Let $(A, \{\Delta(\lambda)_{\lambda\in \Lambda}\})$ be a split quasi-hereditary algebra over a commutative Noetherian ring. 
	If $\domdim{(A, R)} \geq d(\L) + 2+ s$ for some $s\geq 1$, then $(A, \Hom_{A^{op}}(V, A))$ is an $s$-$\mathcal{F}(\Stsim)$ cover of $\End_A(V)$.
\end{Prop}
\begin{proof}
	By Proposition \ref{faithfulcoverboundbydominantdime}, \begin{align}
		\domdim_{(A, R)} \St(\l)\geq \domdim{(A, R)} -d(\St, \l)\geq d(\L)-d(\L, \l)+2+s,
	\end{align} for every $\l\in \L$.
	Hence, $\inf\{\domdim_{(A, R)} \St(\l)\colon \l\in \L \}\geq 2+s$. The result follows from Theorem \ref{boundrelationcoversdom}. 
\end{proof}

\subsection{Hemmer-Nakano dimension of contravariantly finite resolving subcategories}\label{Hemmer-Nakano dimension of contravariantly finite resolving subcategories}
Let $A$ be a finite-dimensional algebra (over a field) with finite global dimension. Let $\mathcal{A}$ be a subcategory of $A\m$. Recall that $\mathcal{A}$ is called \textbf{contravariantly finite} in $A\m$ if for every $X\in A\m$ there exists a map $f\colon M\rightarrow X$ so that $M\in \mathcal{A}$ and the map $\Hom_A(M', f)$ is surjective for every $M'\in \mathcal{A}$. Let $\mathcal{A}$ be a contravariantly finite resolving subcategory in $A\m$. By Theorem 2.2 of \cite{MR2384611}, there exists a tilting module $T$ so that $\mathcal{A}$ coincides with the full subcategory of $A\m$ whose modules $X$ fit into exact sequences 
$0\rightarrow X\rightarrow T_1 \rightarrow T_2\rightarrow \cdots \rightarrow T_t\rightarrow 0$, where all $T_i\in \add T$. 

\begin{Theorem}
	Let $A$ be a finite-dimensional algebra with finite global dimension and with dominant dimension greater than or equal to two.  Let $\mathcal{A}$ be a contravariantly finite resolving subcategory in $A\m$.  
	Then
	\begin{align}
		\domdim T = \inf \{\domdim M\colon M\in \mathcal{A} \},
	\end{align}where $T$ is the tilting module satisfying $\widecheck{\add T}=\mathcal{A}$.
\end{Theorem}
\begin{proof}
	If $\inf \{\domdim M\colon M\in \mathcal{A} \}$ is infinite, then there is nothing to show. 
	
	Assume that $\inf \{\domdim M\colon M\in \mathcal{A} \}$ is finite, and assume, without loss of generality, that \linebreak$\inf \{\domdim M\colon M\in \mathcal{A} \}= \domdim X =  t$, for some $X\in \mathcal{A}$. There exists an exact sequence $0\rightarrow X\rightarrow T_1\rightarrow \cdots \rightarrow T_d\rightarrow 0$, for some natural number $d$, where all $T_i\in \add T$. In particular,  $\domdim T_i\geq \domdim T$ (see for example  \citep[Corollary 5.10]{CRUZ2022410}). Let $X'$ be the cokernel of $X\rightarrow T_1$. So, $X'\in \widecheck{\add T}=\mathcal{A}$ and thus $\domdim X'\geq \domdim X$. Now, observe that if $\domdim T_1>\domdim X$, then $\domdim X'=\domdim X-1$ (see for example \citep[Lemma 5.12]{CRUZ2022410}) which leads to a contradiction with the choice of $X$. Hence, $\domdim T_1=\domdim X$. We conclude that $\domdim T=\domdim X$.
\end{proof}

It follows that if $A$ is a finite dimensional algebra with finite global dimension and dominant dimension greater than or equal to two with $V$ a faithful projective-injective right $A$-module, then $(A, \Hom_{A^{op}}(V, A))$ is an $(\domdim T-2)$-$\mathcal{A}$ cover of $\End_A(V)$, if $\domdim T\geq 1$.

\subsection{Hemmer-Nakano dimension of $\add DA\oplus A$}

We can also reformulate the Nakayama Conjecture in terms of Schur functors in a similar way as the cover property is defined.

\begin{Prop}\label{faithfulongeneratorcogeneratorequivalence}
	Let $A$ be a finite-dimensional algebra over a field. Let $V$ be a projective right $A$-module. Let $B=\End_A(V)=\End_A(\Hom_{A^{op}}(V, A))^{op}$. If the restriction of the Schur functor $F=\Hom_A(\Hom_{A^{op}}(V, A), -)$ to $\add DA\oplus A$ is  faithful, then $F$ is an equivalence of categories. 
\end{Prop}
\begin{proof}
	By assumption, the map induced by $F$, $\Hom_A(A, DA)\rightarrow \Hom_B(FA, F(DA))$ is injective. By Lemma \ref{considerationseta}, $\eta_{DA}\colon DA\rightarrow \Hom_B(FA, FDA)$ is a monomorphism. Note that $FA=V$ and $FDA=V\otimes_A DA$. Let $I_0$ be the injective hull of $V\otimes_A DA$. Since $\Hom_B(V, -)$ is left exact, the composition of maps $DA\rightarrow \Hom_B(V, V\otimes_A DA)\rightarrow \Hom_B(V, I_0)
	$ is a monomorphism. Observe that $\Hom_B(V, I_0)\in \add \Hom_B(V, DB)=\add DV$. Hence, $DA\in \add DV$, and consequently, $V$ is a right $A$-progenerator. By Morita theory, $\Hom_{A^{op}}(V, A)$ is a left $A$-progenerator. 	\end{proof}

So, we can rewrite the Nakayama Conjecture in the following way:

\begin{itemize}
	\item 	Let $A$ be a finite-dimensional algebra over a field.  If $\domdim A=+\infty$, then $A$ is a QF3 algebra with faithful projective-injective right $A$-module $V$ such that the restriction of the Schur functor $\Hom_A(\Hom_{A^{op}}(V, A), -)$ to $\add DA\oplus A$ is  faithful.
\end{itemize}

\section{Applications}

	In this section, we use the results obtained in \cite{CRUZ2022410} for Schur algebras and $q$-Schur algebras combined with the theory developed to Hemmer-Nakano dimensions to see in particular when the respective Schur functor induce exact equivalences between the exact categories of $S_R(d, d)$-modules having a finite filtration by direct summands of direct sums of Weyl modules (resp. of $S_{R, q}(d, d)$ having a $q$-Weyl filtration) and of  $RS_d$-modules  having a finite filtration by direct summands of direct sums of cell modules (resp. cell-filtered modules over an Iwahori-Hecke algebra). We study deformations of blocks of the BGG category $\mathcal{O}$ using the concept of the BGG category $\mathcal{O}$ over a commutative ring in the sense of Gabber and Joseph \cite{zbMATH03747378} and their properties like relative dominant dimension and their quasi-hereditary structure. We will present an integral version of Soergel's Struktursatz \cite{zbMATH00005018}. The computation of the Hemmer-Nakano dimension of the exact category of modules having a filtration by integral Verma modules will help us clarify the interconnections between relative dominant dimension and the Hemmer-Nakano dimension.

\subsection{Classical Schur algebras}\label{Classical Schur algebras}

The study of Schur algebras started in \citep{zbMATH02662157}. Schur used them to link  the polynomial representation theory of the complex general linear group with the representation theory of the symmetric group over the complex numbers. The latter was known at the time due to Frobenius \citep{Frobenius[s.a.]}. Schur algebras can be defined over any commutative ring and nowadays  this connection in positive characteristic is used in the opposite direction. A classical reference for the study of Schur algebras (over infinite fields) is \citep{zbMATH05080041}.

Let $R$ be a commutative ring with identity.	Fix natural numbers $n, d$. The symmetric group on $d$ letters $S_d$ acts  on the $d$-fold tensor product $(R^n)^{\otimes d}$ by place permutations. We will write $V_R^{\otimes d}$ instead of $(R^n)^{\otimes d}$ or simply $V^{\otimes d}$ when no confusion about the ground ring can arise.

\begin{Def} \citep{zbMATH03708660}
	The subalgebra $\End_{RS_d}\left( V^{\otimes d}\right)$ of the endomorphism algebra $\End_R\left( V^{\otimes d}\right)$ is called the \textbf{Schur algebra}. We will denote it by $S_R(n, d)$.
\end{Def}

Schur algebras admit nice properties. They have a base change property 
\begin{align}
	R\otimes_{\mathbb{Z}} S_{\mathbb{Z}}(n, d)\simeq S_R(n, d),
\end{align}and also $R\otimes_{\mathbb{Z}} V_{\mathbb{Z}}^{\otimes d}\simeq V_R^{\otimes d}$ as $S_R(n, d)$-modules. For any commutative Noetherian ring $R$, the Schur algebra $S_R(n, d)$ is a split quasi-hereditary $R$-algebra (see for example \citep[Theorem 4.1]{zbMATH04116809}, \citep[Theorem 11.5.2]{zbMATH04193959}, \citep[Theorem 3.7.2]{CLINE1990126}, \citep[1.2]{zbMATH04031957}, \citep[Theorem 5.0.1]{cruz2021cellular}).
The standard modules associated with this split heredity chain are called \textbf{Weyl modules}. In particular, the Weyl modules are indexed by the partitions of $d$ in at most $n$ parts. Also, the simple $S_K(n, d)$ modules are indexed by the partitions of $d$ in at most $n$ parts whenever $K$ is a field. As of the time of writing, determining the dimensions or the characters of simple modules of the Schur algebra remains still an open problem. Schur algebras over regular local Noetherian rings have finite global dimension (see for example \cite{zbMATH00966941} or \citep[Proposition 5.0.2]{cruz2021cellular}). Besides their quasi-hereditary structure, Schur algebras are also cellular algebras with a duality $\iota$ in the sense of Definition \ref{dualitydef} (see for example \cite[Section 5]{cruz2021cellular}). 

From now on we will assume that $n\geq d$. Then  $V^{\otimes d}$ is a projective $(S_R(n, d), R)$-injective $S_R(n, d)$-module and there exists an idempotent $e\in S_R(n, d)$ so that $\iota(e)=e$ and $V^{\otimes d}\simeq S_R(n, d)e$ as $S_R(n, d)$-modules. Moreover, we have the following:
\begin{Theorem}\citep[Theorem 7.12]{CRUZ2022410}\label{dominantdimensioofSchuralgebras}
	Let $R$ be a commutative Noetherian ring. If $n\geq d$ are natural numbers,  then $(S_R(n, d), V^{\otimes d})$ is a relative gendo-symmetric $R$-algebra and
	\begin{align}
		\domdim{(S_R(n, d), R)} =\inf\{2k\in \mathbb{N} \ | \ (k+1)\cdot 1_R\notin R^\times, \ k<d \}\geq 2.
	\end{align}
\end{Theorem}

This previous result is built upon \citep[Theorem 5.1]{zbMATH05871076}. Now, Theorem \ref{dominantdimensioofSchuralgebras} means that \linebreak$(S_R(n, d), V^{\otimes d})$ is a split quasi-hereditary cover of $RS_d$. By the \textbf{Schur functor} we mean the functor $$F_R=\Hom_{S_R(n, d)}(V^{\otimes d}, -)\colon S_R(n, d)\m\rightarrow RS_d\m$$ arising from this cover (we will write just $F$ when there is no confusion on the ground ring $R$).   Much of the representation theory of symmetric groups can be studied through the representation theory of Schur algebras using the Schur functor. For example, since $\iota(e)=e$ and $eS_R(n, d)e\simeq RS_d$, the Schur functor sends the Weyl modules to the cell modules of $RS_d$ making $RS_d$ a cellular algebra (see for example \citep[Proposition 2.2.11]{cruz2021cellular} or \cite{zbMATH00871761}). These cell modules are also known as dual Specht modules.

We now wish to determine the Hemmer-Nakano dimension of $\mathcal{F}(\Stsim)$, generalizing the results presented in \citep[Section 3]{hemmer_nakano_2004}  by completely determining the quality of the correspondence between Weyl filtrations and dual Specht filtrations. To do that observe the following:

\begin{Cor}\label{dominantdimensiontilting}
	Let $R$ be a commutative Noetherian ring and assume that $n\geq d$. Let $T$ be a characteristic tilting module of $S_R(n, d)$. Then\begin{align}
		\domdim_{(S_R(n, d), R)} T=\inf\{k\in \mathbb{N} \ | \ (k+1)\cdot 1_R\notin R^\times, \ k<d \}\geq 1.
	\end{align}
\end{Cor}
\begin{proof}
	The result follows from applying Theorem Theorem 7.12 of \cite{CRUZ2022410} and Theorem \ref{tiltingmoduledominantdimensioncover}. Alternatively, we could also reproduce the second part of the proof of Theorem 7.12 of \cite{CRUZ2022410} together with Theorem 4.3 of \citep{zbMATH05871076}.
\end{proof}

\subsubsection{Hemmer-Nakano dimension of $S_R(n, d)\proj$ and of $\mathcal{F}(\Stsim_{S_R(n,d)})$}

In relative Mueller characterization (see Theorem \ref{Mullertheorem}) we saw that the vanishing of certain $\Ext$ groups alone might not give the value of relative dominant dimension, only a lower bound dependent on the Krull dimension of the ground ring. Of course, when the ground ring is a field the dominant dimension can be determined using only $\Ext$ groups. Fortunately, for the Schur algebra over a local ring, we can completely determine the Hemmer-Nakano dimensions in terms of relative dominant dimensions. We can reduce the problem to local rings $R$ due to Propositions \ref{arbitraryfaithfulcoverflattwo} and \ref{arbitraryfaithfulcoverflat}. Essentially, the value of the Hemmer-Nakano dimension of $S_R(n, d)\proj$ in terms of relative dominant dimension divides into two separate cases. Assume that $R$ is a local regular Noetherian commutative ring with unique maximal ideal $\mi$. Either $R$ contains a field as a subring or not.

\paragraph{Case 1 - $R$ contains a field}\ \\
Recall that a regular local  commutative Noetherian ring is always an integral domain.
We can start by observing that a regular local commutative Noetherian ring $R$ contains a field as a subring if and only if every quotient ring of $R$ has the same characteristic. In fact, if $R$ contains a field $K$ as subring, the characteristics of $R$, $K$ and the residue field of $R$, $R/\mi$, are all equal since $K\hookrightarrow R\twoheadrightarrow R/\mi$ is injective.  Conversely, the canonical map $\mathbb{Z}\rightarrow R\twoheadrightarrow R/\mi$ can be extended to a homomorphism of rings $\mathbb{Q}\rightarrow R\twoheadrightarrow R/\mi$  if $\characteristic R=0$ or to a homomorphism of rings $\mathbb{Z}/(\characteristic R)\mathbb{Z}\rightarrow R\twoheadrightarrow R/\mi$  if the characteristic of $R$, $\characteristic R$, is positive.

\begin{Theorem}\label{schuralgebraHn}
	Let $R$ be a regular local commutative Noetherian ring containing a field $k$ as a subring. Assume that $n\geq d$. Let $T$ be a characteristic tilting module of $S_R(n, d)$.  Then
	\begin{align*}
		\HN_F(S_R(n, d)\proj)&=\domdim (S_R(n, d), R) - 2 = \inf\{2k\in \mathbb{N} \ | \ (k+1)\cdot 1_R\notin R^\times, \ k<d \}-2\geq 0.\\
		\HN_F(\mathcal{F}(\Stsim))&=\domdim_{(S_R(n, d), R)} T-2=\inf\{k\in \mathbb{N} \ | \ (k+1)\cdot 1_R\notin R^\times, \ k<d \}-2\geq -1.
	\end{align*}
\end{Theorem}
\begin{proof}
		Let $K$ be the quotient field of $R$. Then $\characteristic K=\characteristic R=\characteristic R(\mi)$, where $\mi$ is the unique maximal ideal of $R$. In particular,
	$\domdim S_K(n, d)=\domdim{(S_R(n, d), R)}$. By Theorem \ref{boundrelationcoversdom} and the flatness of $K$ over $R$,
	\begin{align}
		\HN_{F_R}(S_R(n, d\proj))&\geq \domdim (S_R(n, d), R)-2=\domdim S_K(n, d)-2 \\&=\HN_{F_K}(S_K(n, d)\proj)\geq \HN_{F_R}(S_R(n, d)\proj). \nonumber
	\end{align}By Theorem \ref{dominantdimensioofSchuralgebras}, the first line of equalities follow.

Let $T_{K}$ be the characteristic tilting module of $S_K(n, d)$. Note that  $\add T_K=\add  K\otimes_R T$. Again, by Theorems \ref{boundrelationcoversdom}, \ref{domdimofmoduleswithfiltrationbystandard} and the flatness of $K$, 
\begin{align}
	\HN_F(\mathcal{F}(\Stsim))\geq \domdim_{(S_R(n, d), R)} T -2 &= \domdim_{S_K(n, d)} K\otimes_R T -2 \\&=\HN_{F_K}(\mathcal{F}(K\otimes_R \St))\geq 	\HN_F(\mathcal{F}(\Stsim)).
\end{align}
By  Corollary \ref{dominantdimensiontilting}, the result follows.
\end{proof}

\begin{Remark}
	We should point out that there is a typo in Corollary 3.9.2 of \cite{hemmer_nakano_2004}. It should read $p-3$ instead of $p-2$. This typo is a repercussion of a typo in the use of spectral sequences in the published version \citep[2.3]{zbMATH01818928}. There we should read $0\leq i\leq t$ instead of $0\leq i\leq t+1$. The reader can check Lemma \ref{spectralsequencegrothendieck} for clarifications. The result \citep[2.3]{zbMATH01818928} was corrected on Kleshchev's homepage.
\end{Remark}

This result also shows that the existence of the flat condition in Theorem \ref{boundcoverimprovementdom} cannot be dropped.

\paragraph{Case 2 - $R$ does not contain a field}\ %
\begin{Theorem}\label{Hemmenakanodimofprojectives}
	Let $R$ be a local regular commutative Noetherian ring that does not contain a field as a subring. Assume that $n\geq d$. 	Let $T$ be a characteristic tilting module of $S_R(n, d)$. Then\begin{align*}
			\HN_F(S_R(n, d)\proj)&=\domdim{(S_R(n, d), R)} -1 =\inf\{2k\in \mathbb{N} \ | \ (k+1)\cdot 1_R\notin R^\times, \ k<d \}-1\geq 1. \\
		\HN_F(\mathcal{F}(\Stsim))&=\domdim_{(S_R(n, d), R)} T - 1 = \inf\{k\in \mathbb{N} \ | \ (k+1)\cdot 1_R\notin R^\times, \ k<d \}-1\geq 0. 
	\end{align*}
\end{Theorem}
\begin{proof}Of course, $R$ has Krull dimension greater than or equal to one. Let $K$ be the quotient field of $R$.
	By assumption, $\characteristic R=0$ and $\characteristic R(\mi)$ is a prime number $p>0$. In particular, $\characteristic K=0$  the quotient field of $R$ and $\add K\otimes_R T=\add T_K$, where $T_K$ is the characteristic tilting module of $S_K(n, d)$.  Hence, Theorem \ref{dominantdimensioofSchuralgebras} and Corollary \ref{dominantdimensiontilting} give that $\domdim_{S_K(n, d)} K\otimes_R T=\domdim S_K(n, d) = +\infty$, while $2\domdim_{(S_R(n, d), R)} T=\domdim (S_R(n, d), R)=2(p-1)\geq 2$. 
	By Theorems \ref{domdimofmoduleswithfiltrationbystandard} and \ref{boundcoverimprovementdomquotientfield}, it follows that $\HN_F(S_R(n, d)\proj)\geq \domdim{(S_R(n, d), R)}-1$ and $\HN_F(\mathcal{F}(\Stsim))\geq \domdim_{(S_R(n, d), R)} T-1$.  We can assume, without loss of generality, that $p\geq d$. Otherwise, all the values in the statement are equal to $+\infty$, and consequently, the equalities hold.
	 Since $R$ is a local regular ring, $R$ is a unique factorization domain. Therefore, we can write $p1_R=up_1\cdots p_n$ for some prime elements of $R$. So, $p1_R$ belongs to a prime ideal $\mathfrak{p}$ of height one. Hence, $\characteristic R/\mathfrak{p}$ is $p$.
 Let $Q(R/\mathfrak{p})$ be the quotient field of $R/\mathfrak{p}$. Then $\characteristic Q(R/\mathfrak{p})=p$ and by Theorem \ref{dominantdimensioofSchuralgebras} $\domdim S_{Q(R/\mathfrak{p})}(n, d) =2(p-1)= \domdim (S_R(n, d), R)$. Therefore,
	\begin{align}
		\HN_{F_{R/\pri}}(S_{R/\mathfrak{p}}(n, d)\proj )&\leq \HN_{F_{Q(R/\pri)}}(S_{Q(R/\mathfrak{p})}(n, d)\proj)= \domdim (S_R(n, d), R) -2,\\
		\HN_{F_{R/\pri}}(\mathcal{F}(R/\mathfrak{p}\otimes_R \Stsim))&\leq \HN_{F_{Q(R/\pri)}}(\mathcal{F}(Q(R/\mathfrak{p}\otimes_R \St))\\&=\domdim_{S_{Q(R/\mathfrak{p})}(n, d) } Q(R/\mathfrak{p})\otimes_R T -2 = p-3 = \domdim_{(S_R(n, d), R)} T-2. \nonumber
	\end{align}By Corollary \ref{coverheightprimeideal}, the Hemmer-Nakano dimension of $S_R(n, d)\proj$ cannot be higher than \linebreak $\domdim{(S_R(n, d), R)} -1$ while $\HN_F(\mathcal{F}(\Stsim))$ cannot be higher than $\domdim_{(S_R(n, d), R)} T - 1$.  The result now follows by Theorem \ref{dominantdimensioofSchuralgebras} and Corollary \ref{dominantdimensiontilting}.
\end{proof}

We note that the situation for $\mathbb{Z}$ is way better than for $\mathbb{F}_2$. In fact,
\begin{align}
	\HN_{F_{\mathbb{Z}}}(S_{\mathbb{Z}}(n, d)\proj)&=1, \quad \HN_{F_{\mathbb{F}_2}}(S_{\mathbb{F}_2}(n, d)\proj)=0\\
	\HN_{F_{\mathbb{Z}}}(\mathcal{F}(\Stsim_{\mathbb{Z}}))&=0, \quad \HN_{F_{\mathbb{F}_2}}(\mathcal{F}(\St_{\mathbb{F}_2}))=-1.
\end{align}
These results are compatible with the results of \citep{zbMATH00971625}. Moreover, these two particular cases were already known to them and they used this knowledge to define the Young modules and  Specht modules of the group algebra $\mathbb{F}_2S_d$ by defining first the Young and Specht modules for the integral group algebra $\mathbb{Z}S_d$ and then applying the functor $\mathbb{F}_2\otimes_{\mathbb{Z}}-$. This becomes more relevant for Weyl modules over fields of  characteristic two since they cannot be reconstructed from dual Specht modules. That is, 
the image of a dual Specht module under the adjoint functor of the Schur functor only contains, in general, a Weyl module. 

\subsubsection{Uniqueness of covers for $RS_d$}\label{Uniqueness of covers for RSd}

Considering the localization of $\mathbb{Z}$ away from $2$, $\mathbb{Z}[\frac{1}{2}]$, on Theorem \ref{Hemmenakanodimofprojectives} yields that $(S_{\mathbb{Z}[\frac{1}{2}]}(n, d), V_{\mathbb{Z}[\frac{1}{2}]}^{\otimes d})$ is a 1-$\mathcal{F}(\Stsim)$ cover of $\mathbb{Z}[\frac{1}{2}]S_d$. By Corollary \ref{equivalenceofoneuniqueness}, this Schur algebra is the unique cover of $RS_d$ which sends the standard modules (in this case the Weyl modules) to the Specht modules. We remark that this improves the situation for the fields of characteristic 3 since they are algebras over $\mathbb{Z}[\frac{1}{2}]$ and for characteristic 3 the Hemmer-Nakano dimension of $\mathcal{F}(\St)$ is only zero.

For the ring of integers, we can already conclude that there is no better cover than the Schur algebra to study the Specht modules over the symmetric group.

\begin{Theorem}\label{uniquenessofcoverspreservingspecht}
	Let $k$ be a field of characteristic two and $d\geq 2$. Let $\theta=\{\theta(\l)\colon \l\in \Lambda^+(d) \}$ be the cell modules of $kS_d$. Then $(kS_d, \theta)$ does not have a $0$-faithful split quasi-hereditary cover.
	Moreover, there are no $1$-faithful split quasi-hereditary covers of $\mathbb{Z}S_d$ satisfying $F\St(\l)=\theta_{\mathbb{Z}}(\l), \ \l\in \Lambda^+(d)$, where $F$ is the Schur functor associated with the cover of $\mathbb{Z}S_d$.
\end{Theorem}
\begin{proof}
	Assume, by contradiction, that $(A, P)$ is a $0$-faithful quasi-hereditary cover of $kS_d$ satisfying \linebreak $\Hom_A(P, \St(\l))=\theta(\l), \l\in \L^+(d):=\L^+(d, d)$.
	
	Let ${}^\natural(-)\colon kS_d\m\rightarrow kS_d\m$ be the simple preserving duality of the symmetric group. By Theorem 8.15 of \cite{zbMATH03609919},
	$
	{}^\natural\theta(1^d)\simeq \theta(d).$ 
	On the other hand, $\theta(1^d)$ is a simple module, so $\theta(d)\simeq \theta(1^d)$.
	This implies that
	\begin{align}
		\Hom_{A}(\St(d), \St(1^d))\simeq \Hom_{kS_d}(\theta(d), \theta(1^d))\simeq \Hom_{kS_d}(\theta(d), \theta(d))\neq 0.
	\end{align} This contradicts $A$ being split quasi-hereditary with the order on the partitions $d>1^d$. So, $kS_d$ has no such faithful quasi-hereditary cover.
	
	Assume that there exists a 1-faithful split quasi-hereditary cover of $\mathbb{Z}S_d$, say $(A, P)$ such that \linebreak$\Hom_A(P, \St(\l))=\theta(\l)$. By Theorem \ref{truncationcovers}, $(A(2\mathbb{Z}), P(2\mathbb{Z}))$ is a 0-faithful quasi-hereditary cover of $\mathbb{Z}/2\mathbb{Z}S_d=\mathbb{F}_2S_d$ satisfying \begin{align}
		\theta_{\mathbb{F}_2}(\l)=\mathbb{Z}/2\mathbb{Z}\otimes_{\mathbb{Z}} \theta(\l) =	\mathbb{Z}/2\mathbb{Z}\otimes_\mathbb{Z} \Hom_A(P, \St(\l))\simeq \Hom_{\mathbb{F}_2\otimes_\mathbb{Z} A}(P(2\mathbb{Z}) , \St(\l)(2\mathbb{Z})).
	\end{align} By the first part of our discussion, this cannot happen.
\end{proof}

\subsection{$q$-Schur algebras}\label{qSchuralgebras}

The $q$-Schur algebras were introduced in \cite{zbMATH00014844, zbMATH04168928}. A classical reference to $q$-Schur algebras is \citep{MR1707336}. They are used to link the representation theory of Iwahori-Hecke algebras with the representation theory of quantum general linear groups.

There are many equivalent ways to define Iwahori-Hecke algebras. In this work, we will follow the definition due to Parshall-Wang \citep{zbMATH04193959} (but we use $u$ instead of $q$ and $q$ instead of $h$).
Let $R$ be a commutative ring with identity. Fix natural numbers $n, d$. Let $u$ be an invertible element of $R$ and put $q=u^{-2}$. The \textbf{Iwahori-Hecke algebra} $H_{R, q}(d)$ is the $R$-algebra with basis $\{T_\sigma\colon \sigma\in S_d \}$ satisfying the relations
\begin{align}
	T_\sigma T_s =\begin{cases}
		T_{\sigma s}, \quad &\text{ if } l(\sigma s)=l(\sigma) +1\\
		(u-u^{-1})T_\sigma +T_{\sigma s}, \quad &\text{ if } l(\sigma s)=l(\sigma)-1,
	\end{cases}\label{eqex24}
\end{align}
where $s\in S:=\{(1,2), (2,3), \ldots, (d-1, d) \}$ and $l$ is the length function defined with respect to these basic transpositions. The length of a permutation $\omega$, denoted by $l(\omega)$, is the minimal number $t$ such that $w=s_{j_1}\cdots s_{j_t}$, where $s_{j_k}$ are basic transpositions for $k=1, \ldots, t$. 
 In particular,  the elements $T_s$, $s\in S$, generate $H_{R, q}(d)$ as algebra.

The Iwahori-Hecke algebra $H_{R, q}(d)$ admits a base change property.
\begin{align}
	H_{R, q}(d)\simeq R\otimes_{\mathbb{Z}[u, u^{-1}]} H_{\mathbb{Z}[u, u^{-1}], u^{-2}}(d) \label{heckechangering}
\end{align}
Under this isomorphism of $R$-algebras $1_R\otimes_{\mathbb{Z}[u, u^{-1}]} T_\sigma$ is mapped to $T_\sigma\in H_{R, q}(d)$. 

Let $I(n, d)$ be the set of maps $i\colon \{1, \ldots, d\}\rightarrow \{1, \ldots, n\}$. We write $i(a)=i_a$. We can associate to $I(n, d)$ a right $S_d$-action by place permutation, and extend it to an $S_d$-action on  $I(n, d)\times I(n, d)$. Let $\{e_1, \ldots, e_n\}$ be an $R$-basis of $V$.
We can regard $V^{\otimes d}$ as right $H_{R, q}(d)$-module by imposing to the $R$-basis $\{ e_i:=e_{i_1}\otimes\cdots \otimes e_{i_d}\ | \ i\in I(n, d)\}$ of $V^{\otimes d}$, 
\begin{align*}
	e_{i_1}\otimes\cdots \otimes e_{i_d}\cdot T_s = \begin{cases}
		e_{i_1}\otimes\cdots \otimes e_{i_d} \cdot s \quad &\text{ if } i_t<i_{t+1}\\
		u e_{i_1}\otimes\cdots \otimes e_{i_d} \quad &\text{ if } i_t=i_{t+1}\\
		(u-u^{-1}) e_{i_1}\otimes\cdots \otimes e_{i_d} + e_{i_1}\otimes\cdots \otimes e_{i_d}\cdot s \quad &\text{ if } i_t>i_{t+1}
	\end{cases}, \quad s=&(t, t+1)\in S, \\[-0.6cm] &1\leq t< d.\nonumber
\end{align*}
This action is a deformation of place permutation. In fact, $H_{R, 1}(d)$ is just the group algebra of the symmetric group $S_d$ and this action is the one given by place permutation.

\begin{Def}
	The subalgebra $\End_{H_{R, q}(d)}\left( V^{\otimes d}\right)$ of the endomorphism algebra $\End_R\left( V^{\otimes d}\right)$ is called the \mbox{\textbf{$q$-Schur algebra}}. We will denote it by $S_{R, q}(n, d)$.
\end{Def}

We can associate to $I(n, d)\times I(n, d)$ the lexicographical order. Each $S_d$-orbit of $I(n, d)\times I(n, d)$ has a representative $(i, j)$ satisfying $(i_1, j_1)\leq \cdots \leq (i_d, j_d)$. Denote by $\L$ the set of such representatives.
The $q$-Schur algebra admits an $R$-basis by elements $\xi_{j, i}$, $(i, j)\in \L$, satisfying 
\begin{align}
	\xi_{j, i}(e_f)=\sum_{g\in I(n, d)} p_{i, j}^{f, g}(u) e_g, \quad \forall f\in I(n, d),
\end{align} for some elements $p_{i, j}^{f, g}(u) \in \mathbb{Z}[u, u^{-1}]$ (see for example \citep[Proposition 7.16]{CRUZ2022410} or \cite{MR1707336}). In  \citep[Proposition 7.16]{CRUZ2022410} the strategy was to construct a basis for $V^{\otimes d} \otimes_{H_{R, q}(d)} DV^{\otimes d}$ and the dualize this basis to a basis of the $q$-Schur algebra. The advantage of such an approach is that it follows immediately that  the $q$-Schur algebra admits a base change property (see also \citep[2.18(ii)]{zbMATH04168928}), that is, for any commutative $R$-algebra $S$:
\begin{align}
	S_{R, q}(n, d)\simeq R\otimes_{\mathbb{Z}[u, u^{-1}]} S_{\mathbb{Z}[u, u^{-1}], u^{-2}}(n, d),
	\ \quad 	S_{S, q1_S}(n, d)\simeq S\otimes_R S_{R, q}(n, d).
\end{align}

Also, from \citep[Proposition 7.16]{CRUZ2022410}, for $f\in I(n,d)$, $p_{i, i}^{f, f}=1$ if $f\sim i$ and $p_{i, i}^{f, f}$ is zero otherwise. Therefore, for each $(i, i)\in \L$, $\xi_{i, i}$ is an idempotent. Further, we can index these idempotents by the compositions of $d$ in at most $n$ parts, by associating to each $i$ its weight. We can consider an  increasing bijection $\L^+(n, d)\rightarrow \{1, \ldots, t\}, \ \l^k\mapsto k$. Set $e^k$ to be the idempotent $\sum_{l\geq k} \xi_{\l^l}$ and define $J_k=S_{R, q}(n, d)e^kS_{R, q}(n, d)$. It follows that $S_{R, q}(n, d)$ is split quasi-hereditary. 

\begin{Theorem}
	For any commutative Noetherian ring $R$, the $q$-Schur algebra $S_{R, q}(n, d)$ is a split quasi-hereditary algebra over $R$ with split heredity chain $0\subset J_t\subset \cdots\subset J_2\subset J_1=S_R(n, d)$.
\end{Theorem}
\begin{proof}
	The statement for fields follows from \citep[Theorem 11.5.2]{zbMATH04193959}. The statement for Noetherian rings which are not fields follows from Theorem 3.7.2 of \citep{CLINE1990126}. An alternative proof for this statement without using Theorem 3.7.2 of \citep{CLINE1990126} is to apply Theorem 3.3.11 of \cite{cruz2021cellular}.
\end{proof}

In particular, $S_{R, q}(n, d)$ has finite global dimension whenever $R$ has finite global dimension. The standard modules associated with this split heredity chain of the $q$-Schur algebra are called $q$-Weyl modules, indexed by the partitions of $d$ in at most $n$ parts.  To define a cellular structure on the $q$-Schur algebra we can define the involution $\iota$ by assigning to each element basis $\xi_{j, i}$ $(i, j)\in \L$, the image in $S_{R, q}(n, d)$ of $(e_j\otimes_{H_{R, q}(d)} e_i^* )^*$.
Observe that $\iota(\xi_\l)=\xi_\l$ for every $\l\in \L^+(n, d)$. In particular, we can choose idempotents giving a split heredity chain of $S_{R, q}(n, d)$ which are all preserved by $\iota$. Hence, for any commutative Noetherian ring $R$, $S_{R, q}(n, d)$ is a cellular algebra (see for example \citep[Proposition 4.0.1]{cruz2021cellular}).

Analogously to the classical case,  we will assume, from now on, that $n\geq d$. Then $V^{\otimes d}$ is a projective $(S_{R, q}(n, d), R)$-injective $S_{R, q}(n, d)$-module and 	$V^{\otimes d}\simeq S_{R, q}(n, d)\xi_{(1, \ldots, d), (1, \ldots, d)}$ as $(S_{R, q}(n, d), R)$-modules and . Moreover, we have the following:

\begin{Theorem}\citep[Theorem 7.20]{CRUZ2022410}\label{dominantdimensionquantumschur}
	Let $R$ be a commutative Noetherian ring with invertible element $u\in R$. Put $q=u^{-2}$ and assume that $n\geq d$. Then $(S_{R, q}(n, d), V^{\otimes d})$ is a relative gendo-symmetric $R$-algebra and
	\begin{align}
		\domdim (S_{R, q}(n, d), R)=\inf \{ 2s\in \mathbb{N} \ | \ 1+q+\cdots+ q^s\notin R^\times, \ s<d \}.\label{eqex54}
	\end{align}
\end{Theorem}
The previous result is built upon \citep[Theorem 3.13]{zbMATH07050778} which uses quantum characteristic, in contrast to the classical case of Schur algebras, where the dominant dimension depends on the characteristic of the ground field. Theorem \ref{dominantdimensionquantumschur} implies that $(S_{R, q}(n, d), V^{\otimes d})$ is a split quasi-hereditary cover of $H_{R, q}(d)$. 
Parallel to the classical case, the Schur functor $$F_{R, q}=\Hom_{S_{R, q}(n, d)}(V^{\otimes d}, -)\colon S_{R, q}(n, d)\m\rightarrow H_{R, q}(d)\m$$ associated with this cover (we will write just $F_q$ when there is no confusion on the ground ring $R$) can be exploited to obtain information for the representation theory of Hecke algebras using  the representation theory of $q$-Schur algebras.
Since $\iota(\xi_{(1, \ldots, d), (1, \ldots, d)})=\xi_{(1, \ldots, d), (1, \ldots, d)}$ and
\begin{align}
	H_{R, q}(d)\simeq \End_{S_{R, q}(n, d)}(V^{\otimes d})\simeq \xi_{(1, \ldots, d), (1, \ldots, d)} S_{R, q}(n, d)\xi_{(1, \ldots, d), (1, \ldots, d)},
\end{align}multiplication by the idempotent $ \xi_{(1, \ldots, d), (1, \ldots, d)}$ sends the split heredity chain of $S_{\mathbb{Z}[u, u^{-1}], q}(n, d)$ to a cell chain of $H_{\mathbb{Z}[u, u^{-1}], q}(d)$. Also taking into account projective $\mathbb{Z}[u, u^{-1}]$-modules are free, this implies that $H_{\mathbb{Z}[u, u^{-1}, q}(d)$ and $H_{R, q}(d)$ are cellular algebras (see for example \citep[Proposition 2.2.11]{cruz2021cellular} and \cite{zbMATH00871761}).
In particular, the Schur functor sends the $q$-Weyl modules to the cell modules of $H_{R, q}(d)$. 
This motivates us to determine the connection between $q$-Weyl modules filtrations and cell filtrations. 

The starting point is to look at the values of relative dominant dimension of a characteristic tilting module.
\begin{Cor}\label{dominantdimensionquantumschurtilting}
	Let $R$ be a commutative Noetherian ring with an invertible element $u\in R$. Put $q=u^{-2}$ and assume that $n\geq d$.  Let $T$ be a characteristic tilting module of $S_{R, q}(n, d)$. Then
	\begin{align}
		\domdim_{(S_{R, q}(n, d), R)} T=\inf \{ s\in \mathbb{N} \ | \ 1+q+\cdots+ q^s\notin R^\times, \ s<d \}.\label{eqex55}
	\end{align}
\end{Cor}
\begin{proof}
	The result follows from applying Theorem \ref{dominantdimensionquantumschur} and Theorem \ref{tiltingmoduledominantdimensioncover}. 
\end{proof}

\subsubsection{Hemmer-Nakano dimension of $S_{R, q}(n, d)\proj$ and $\mathcal{F}(\Stsim_{S_{R, q}(n, d)})$}

For the Schur algebras, we saw that both the Hemmer-Nakano dimension and the relative dominant dimension are independent of the Krull dimension of the ground ring contrary to other homological invariants like the global dimension. For $q$-Schur algebras, we expect a similar behaviour. Further, a crucial fact for a better value of the Hemmer-Nakano dimension regarding the relative dominant dimension of $S_R(n, d)\proj$ was $R$ not being similar to a field. In particular, $R$ must have Krull dimension greater than or equal to one and it does not contain a field. So, a natural question that arises is 
\begin{itemize}
	\item For which rings $R$ does $S_{R, q}(n, d)\proj$ and $\mathcal{F}(\Stsim)$  have higher Hemmer-Nakano dimension than the respective resolving subcategories over its residue fields?
\end{itemize}

The following notion based on the work \citep[1.9]{stum2013quantum} gives us the answer to this question.

\begin{Def}Let $R$ be a commutative ring with invertible element $q$.
	We call $R$ a \textbf{$q$-divisible ring} (or \emph{quantum divisible ring}) if $1+q+\cdots+q^s\in R^\times$ whenever $1+q+\cdots + q^s\neq 0$ for any $s\in \mathbb{N}$. For a given natural number $d$, we call $R$ a \textbf{$d$-partially $q$-divisible ring} (or \emph{$d$-partially quantum divisible ring}) if $1+q+\cdots+q^s\in R^\times$ whenever $1+q+\cdots + q^s\neq 0$ for any $s<d$.
\end{Def}

For example, any field is a quantum divisible ring, and in particular, it is a $d$-partially quantum divisible ring for any $d$. The $q$-characteristic of $R$, denoted by $q\qcharacteristic R$, is the smallest positive number $s$ such that $1+q+\cdots+q^{s-1}=0$ if such $s$ exists, and zero otherwise.

Once again, we can assume that $R$ is a local regular (commutative Noetherian) ring for the computation of Hemmer-Nakano dimension of $S_{R, q}(n, d)\proj$ and $\mathcal{F}(\Stsim)$.

\paragraph{Case 1 - $R$ is a $d$-partially quantum divisible ring} \ 
\begin{Theorem}
	Let $R$ be a local regular $d$-partially $q$-divisible (commutative Noetherian) ring, where $q=u^{-2}$, $u\in R^\times$. Assume that $n\geq d$. Then
	\begin{align*}
		\HN_{F_{q}}(S_{R,q}(n, d)\proj)&=\domdim (S_{R, q}(n, d), R) - 2 = \inf \{ 2s\in \mathbb{N} \ | \ 1+q+\cdots+ q^s\notin R^\times, \ s<d \}-2 \\
		&=\inf \{ 2s\in \mathbb{N} \ | \ 1+q+\cdots+ q^s=0, \ s<d \}-2\geq 0.
	\end{align*}
	Moreover, if  $T$ is a characteristic tilting module of $S_{R, q}(n, d)$, \label{HNqschuralgebraI} then 
	\begin{align*}
		\HN_{F_q}(\mathcal{F}(\Stsim))=\domdim_{(S_{R, q}(n, d), R)} T-2 &= \inf \{ s\in \mathbb{N} \ | \ 1+q+\cdots+ q^s\notin R^\times, \ s<d \}-2\\
		&= \inf \{ s\in \mathbb{N} \ | \ 1+q+\cdots+ q^s=0, \ s<d \}-2\geq -1.
	\end{align*}
\end{Theorem}
\begin{proof}
	By Theorem \ref{boundrelationcoversdom}, Theorem \ref{dominantdimensionquantumschur}, Theorem \ref{domdimofmoduleswithfiltrationbystandard} and Corollary \ref{dominantdimensionquantumschurtilting},
	\begin{align*}
		\HN_{F_q}(S_{R, q}(n, d), R\proj)&\geq \domdim {(S_{R, q}(n, d), R)} - 2 \\&=  \inf \{ 2s\in \mathbb{N} \ | \ 1+q+\cdots+ q^s\notin R^\times, \ s<d \}-2,\\
		\HN_{F_q}(\mathcal{F}(\Stsim))&\geq \domdim_{(S_{R, q}(n, d), R)} T - 2 \\&= \inf \{ s\in \mathbb{N} \ | \ 1+q+\cdots+ q^s\notin R^\times, \ s<d \}-2.
	\end{align*}	If $ \inf \{ s\in \mathbb{N} \ | \ 1+q+\cdots+ q^s\notin R^\times, \ s<d \}=+\infty$, then we are done. 
	
	Assume that $\inf \{ s\in \mathbb{N} \ | \ 1+q+\cdots+ q^s\notin R^\times, \ s<d \}$ is finite.
	Let $K$ be the quotient field of $R$.
	Since $R$ is a $d$-partially $q$-divisible ring,
	\begin{align}
		\inf \{ s\in \mathbb{N} \ | \ 1+q+\cdots+ q^s\notin R^\times, \ s<d \} &= \inf \{ s\in \mathbb{N} \ | \ 1+q+\cdots+ q^s=0, \ s<d \}\\&=q\qcharacteristic R-1=q\qcharacteristic K-1>0.
	\end{align} 
	Therefore, 
	\begin{align}
		\HN_{F_q}(S_{R, q}(n, d)\proj)&\geq (\domdim S_{R, q}(n, d), R) -2 =2(q\qcharacteristic K-1)-2\\&=\domdim S_{K, q}(n, d) -2\\&=\HN_{F_{K, q}}(S_{K, q}(n, d)\proj)\geq \HN_{F_q}(S_{R, q}(n, d), R\proj). 
	\end{align} Furthermore,
	\begin{align} \hspace*{-0.15cm}
		\HN_{F_q}(\mathcal{F}(\Stsim))&\geq \domdim_{(S_{R, q}(n, d), R)} T -2 = q\qcharacteristic K- 3\\&=\domdim S_{K, q}(n, d)K\otimes_R T - 2 =\HN_{F_{K, q}}(\mathcal{F}(K\otimes_R \St))\geq \HN_{F_q}(\mathcal{F}(\Stsim)). \tag*{\qedhere}
	\end{align}
\end{proof}

\paragraph{Case 2 - $R$ is not a $d$-partially quantum divisible ring}
\begin{Theorem}Let $R$ be a local regular (commutative Noetherian) ring with an invertible element $u\in R$. Put $q=u^{-2}$. Assume that $R$ is not a $d$-partially $q$-divisible ring. Assume that $n\geq d$.  Then
	\begin{align*}
		\HN_{F_q}(S_{R,q}(n, d)\proj)=\domdim (S_{R, q}(n, d), R) - 1 = \inf \{ 2s\in \mathbb{N} \ | \ 1+q+\cdots+ q^s\notin R^\times, \ s<d \}-1.
	\end{align*}\label{HNqschuralgebraII}
	Moreover, if  $T$ is a characteristic tilting module of $S_{R, q}(n, d)$, then
	\begin{align*}
		\HN_{F_q}(\mathcal{F}(\Stsim))=\domdim_{(S_{R, q}(n, d), R)} T-1 = \inf \{ s\in \mathbb{N} \ | \ 1+q+\cdots+ q^s\notin R^\times, \ s<d \}-1\geq 0.
	\end{align*}
\end{Theorem}
\begin{proof}
	Since $R$ is not a $d$-partially $q$-divisible ring there exists a natural number $s$ smaller than $d$ such that the sum $0\neq 1+q+\cdots + q^s\notin R^\times$ is a non-zero invertible element of $R$. Let $s$ be the smallest natural number with such a property. Suppose that there exists a natural number $l$ smaller than $s$ satisfying $1+q+\cdots + q^l=0$.  Then\begin{align}
		0\neq q^{l+1}+\cdots+q^s=q^{l+1}(1+q+\cdots +q^{s-l-1})\notin R^\times.
	\end{align}As $q^{l+1}\in R^\times$ we obtain that $
	0\neq 1+\cdots+ q^{s-l-1}\notin R^\times.
	$ So, the existence of $l$ contradicts the minimality of $s$. Therefore, $
	\inf \{ t\in \mathbb{N}\colon 1+q+\cdots+ q^t=0, \ t<d \}>s. 
	$ It is clear that the Krull dimension of $R$ is at least one. Let $K$ be the quotient field of $R$. Hence, by Theorem \ref{boundrelationcoversdom}, Theorem \ref{dominantdimensionquantumschur}, Theorem \ref{domdimofmoduleswithfiltrationbystandard} and Corollary  \ref{dominantdimensionquantumschurtilting},
	\begin{align}
		\HN_{F_{K, q}}(S_{K, q}(n, d)\proj) &=\domdim S_{K, q}(n, d)-2 \\&= \inf\{2t\in \mathbb{N}\colon 1+q+\cdots+ q^t=0, \ t<d \}-2>2s-2,\label{eqex74}\\
		\HN_{F_{K, q}}(\mathcal{F}(K\otimes_R \St))&=\domdim_{S_{K, q}(n, d)}K\otimes_R T\\&=\inf\{t\in \mathbb{N}\colon 1+q+\cdots+ q^t=0, \ t<d \}-2>s-2.\label{eqex75}
	\end{align}So, all the assumptions of Theorem \ref{boundcoverimprovementdomquotientfield} are satisfied. Therefore, 
	\begin{align}
		\HN_{F_{q}}(S_{R, q}(n, d)\proj)&\geq \domdim (S_{R, q}(n, d), R)-1\\&=\inf \{ 2t\in \mathbb{N} \ | \ 1+q+\cdots+ q^t\notin R^\times, \ t<d \}-1=2s-1\geq 1\label{eqex77}\\
		\HN_{F_q}(\mathcal{F}(\Stsim))&\geq \domdim_{(S_{R, q}(n, d), R)} T-1\\&=\inf\{t\in \mathbb{N} \colon 1+q+\cdots + q^t\notin R^\times, \ t<d\}-1=s-1\geq 0.
	\end{align}
	On the other hand, $R$ is a unique factorization domain. So, we can write $1+q+\cdots+q^s=xy$ for some prime element $x\in R$. Thus, $Rx$ is a prime ideal of height one. Therefore, the image of $1+q+\cdots+q^s$ in $R/Rx$ is zero. Denote by $Q(R/Rx)$ the quotient field of $R/Rx$ and $q_x$ the image of $q$ in $R/Rx$. Then
	\begin{align}
		\inf \{ 2t\in \mathbb{N}\colon 1+q_x+\cdots+q_x^t=0, \ t<d \}\leq 2s
	\end{align} and so,
	\begin{align}
		\HN_{F_{R/Rx, q_x}}(S_{R/Rx, q_x}(n, d)\proj)\leq \HN_{F_{Q(R/Rx), q_x}}(S_{Q(R/Rx), q_x}(n, d)\proj)\leq 2s-2,\\
		HN_{F_{R/Rx, q_x}}(\mathcal{F}(R/Rx\otimes_R \Stsim))\leq \HN_{F_{Q(R/Rx), q_x}}(\mathcal{F})(Q(R/Rx)\otimes_R \St) \leq s-2.
	\end{align}By Corollary \ref{coverheightprimeideal}, we cannot have $\HN_{F_q}(S_{R, q}(n, d)\proj) >2s-1$. It follows that,  \linebreak$\HN_{F_q}(S_{R, q}(n, d)\proj) =2s-1$ and $\HN_{F_q}(\mathcal{F}(\Stsim))=s-1$.
\end{proof}

\begin{Observation}
	If $R$ is a regular integral domain with an invertible element $u\in R$ which is not a $d$-partially $u^{-2}$-divisible ring, then there exists a maximal ideal $\mi$ so that $0\neq 1+q+q+\cdots+q^s\in \mi$ for some $s<d$. If we pick $\mi\in \MaxSpec R$ so that $s$ is minimal, then $R_\mi$ is not a $d$-partially $q_\mi$-divisible ring and 
	$$ \domdim (S_{R, q}(n, d), R)=\domdim (S_{R_\mi, q_\mi}(n, d), R_\mi).$$
	Therefore, 
	\begin{align*}
		\HN_{F_{R, q}}(S_{R, q}(n, d)\proj )&\geq \domdim (S_{R, q}(n, d), R)-1=\HN_{F_{R_\mi, q_\mi}}(S_{R_\mi, q_\mi}(n, d)\proj)\\&\geq \HN_{F_{R, q}}(S_{R, q}(n, d)\proj).
	\end{align*} 
\end{Observation}

The ring $\mathbb{Z}[u, u^{-1}]$ is not a $d$-partially $q$-divisible ring for $d>2$.

Hence, the previous exposition generalizes many of the results discussed in \citep{zbMATH02182639}.

\subsection{Deformations of the BGG category $\mathcal{O}$}\label{Deformations of the BGG category O}

BGG category $\mathcal{O}$ was introduced in \citep{zbMATH03549206}.
For the study of the BGG category $\mathcal{O}$ of a complex semi-simple  Lie algebra we refer to \cite{zbMATH05309234}.

Following closely the work of Gabber and Joseph  \citep{zbMATH03747378} to study the Bernstein-Gelfand-Gelfand category $\mathcal{O}$ over a commutative ring, our aim is now to introduce projective Noetherian algebras that are deformations of the blocks of BGG category $\mathcal{O}$ of a complex semi-simple Lie algebra. 
We shall start by recalling some facts about root systems in complex semi-simple Lie algebras. The initial motivation to consider a category $\mathcal{O}$ over commutative rings was the study of the Kazhdan-Luzstig conjecture. At the time, this construction did not seem fruitful. However, we will find here that these deformations are very interesting to cover theory.

\subsubsection{Root systems}

Let $\mathfrak{g}$ be a (finite-dimensional) complex semi-simple Lie algebra with Cartan subalgebra $\mathfrak{h}$ and associated root system $\Phi\subset \mathfrak{h}^*$, where $\mathfrak{h}^*$ denotes the dual vector space of $\mathfrak{h}$.  In particular, $\mathfrak{g}$ admits a direct sum decomposition into weight spaces for $\mathfrak{h}$ of the form $\mathfrak{g}=\mathfrak{h}\oplus \bigoplus_{\alpha\in \Phi} \mathfrak{g}_\alpha$, where \linebreak${\mathfrak{g}_\alpha=\{x\in \mathfrak{g}\colon [h, x]=\alpha(h)x, \ \forall h\in \mathfrak{h} \}}$.
Let $\varPi$ be the set of simple roots of $\Phi$, and therefore it is a basis of the root system $\Phi$ (see  \citep[Definition 11.9]{zbMATH05031712}). It is also a basis of the vector space $\mathfrak{h}^*$.  Set $\Phi^+:=\mathbb{Z}^+_0\varPi\cap \Phi$, giving a direct sum decomposition $\mathfrak{g}=\mathfrak{n}^+\oplus \mathfrak{h}\oplus \mathfrak{n}^-$, where $\mathfrak{n}^{\pm}:=\bigoplus_{\alpha\in \Phi^\pm}\mathfrak{g}_\alpha$. The Lie algebra $\mathfrak{b}=\mathfrak{n}^+\oplus \mathfrak{h}$ is called the \textbf{Borel subalgebra }of $\mathfrak{g}$.

Let $E$ be the real span of $\Phi$ and $(-, -)$ be the symmetric bilinear form on $E$ induced by the Killing form associated with the adjoint representation of $\mathfrak{g}$. 
The \textbf{Weyl group} associated with the root system $\Phi$ which we denote by $W$ is the finite subgroup of $GL(E)$ generated by all reflections $s_\alpha$, $\alpha\in \Phi$, where $s_\alpha(\l)=\l-\frac{2(\l, \alpha)}{(\alpha, \alpha)}\alpha$, $\l\in \mathfrak{h}^*$.
For each root $\alpha\in \Phi$, we associate the coroot $\alpha^\vee:=\frac{2}{(\alpha, \alpha)}\alpha.$ Denote by $\Phi^\vee$ the set of all coroots. Hence, 
the bilinear form induces, in addition, the following map $\langle -, -\rangle\colon \Phi\times \Phi^\vee\rightarrow \mathbb{Z}$, given by $\langle \beta, \alpha^\vee\rangle:=\frac{2(\beta, \alpha)}{(\alpha, \alpha)}$. This map is called \textbf{Cartan invariant} in \cite{zbMATH05309234}.
We call $\mathbb{Z}\Phi$ the \textbf{root lattice}.  
\subsubsection{Integral semi-simple Lie algebras}

Let $\{h_\alpha\colon \alpha\in \varPi\}\cup \{x_\alpha\colon \alpha\in \Phi \}$ be a Chevalley basis of the semi-simple Lie algebra $\mathfrak{g}$, where $ \{h_\alpha\colon \alpha\in \varPi\}$ is a basis of $\mathfrak{h}$ and $x_\alpha\in \mathfrak{g}_\alpha$ for each root $\alpha\in \Phi$. In particular, $\alpha(h_\alpha)=2$ and $h_\alpha=[x_\alpha, x_{-\alpha}]$ for every $\alpha\in \Phi$. Also, $\langle \beta, \alpha^\vee\rangle=\beta(h_\alpha), $ $\alpha, \beta\in \Phi$. Let $\mathfrak{g}_\mathbb{Z}$ be the additive subgroup of $\mathfrak{g}$ with basis $\{h_\alpha\colon \alpha\in \varPi\}\cup \{x_\alpha\colon \alpha\in \Phi \}$. The restriction of the Lie bracket $[-, -]$ to $\mathfrak{g}_\mathbb{Z}\times \mathfrak{g}_\mathbb{Z}$ has image in $\mathfrak{g}_\mathbb{Z}$ making $\mathfrak{g}_\mathbb{Z}$ a Lie algebra.

For each commutative Noetherian ring with identity $R$, we define the Lie algebra $\mathfrak{g}_R:=R\otimes_\mathbb{Z} \mathfrak{g}_\mathbb{Z}$. By construction, \mbox{$\mathfrak{g}_\mathbb{C}=\mathbb{C}\otimes_{\mathbb{Z}} \mathfrak{g}_\mathbb{Z}\simeq \mathfrak{g}$. }
Using the Chevalley basis, we define the following integral Lie subalgebras of $\mathfrak{g}_\mathbb{Z}$: $\mathfrak{h}_\mathbb{Z}=\bigoplus_{\alpha\in \varPi}\mathbb{Z}h_\alpha$, $\mathfrak{n}_\mathbb{Z}^{\pm}=\bigoplus_{\alpha\in \Phi^+} \mathbb{Z}x_{\pm\alpha}$, $\mathfrak{b}_\mathbb{Z}=\mathfrak{n}_\mathbb{Z}^+\oplus \mathfrak{h}_\mathbb{Z}.$

Analogously, we define for each commutative Noetherian ring with identity $R$, $\mathfrak{h}_R=R\otimes_\mathbb{Z}\mathfrak{h}_\mathbb{Z}$, $\mathfrak{n}_R^\pm=R\otimes_\mathbb{Z} \mathfrak{n}_\mathbb{Z}^\pm$, $\mathfrak{b}_R=R\otimes_\mathbb{Z}\mathfrak{b}_\mathbb{Z}$. Since $\mathfrak{h}_R$ is free over $R$, its dual $\Hom_R(\mathfrak{h}_R, R)$ which we will denote by $\mathfrak{h}_R^*$ is free over $R$.

Observe that $\mathfrak{g}_\mathbb{Q}$ is again a semisimple Lie algebra, since otherwise, every solvable ideal of $\mathfrak{g}_\mathbb{Q}$ could be extended to a solvable ideal of $\mathfrak{g}_\mathbb{C}\simeq \mathfrak{g}$. Therefore, for any field extension $K \supset\mathbb{Q}$, the Lie algebra $\mathfrak{g}_K$ is semisimple.

Let $U(\mathfrak{g}_R)$ be the universal enveloping algebra of $\mathfrak{g}_R$, that is, $U(\mathfrak{g}_R)$ is the quotient $T(\mathfrak{g}_R)/I_R$ of the tensor algebra $T(\mathfrak{g}_R)=R\oplus \mathfrak{g}_R\oplus (\mathfrak{g}_R\otimes \mathfrak{g}_R)\oplus \cdots$, where $I_R$ is the two-sided ideal generated by the elements of the form \mbox{$x\otimes y-y\otimes x-[x, y]$,} $x, y\in \mathfrak{g}_R$.
We denote by $S(\mathfrak{g}_R)$ the symmetric algebra of $\mathfrak{g}_R$, that is, $S(\mathfrak{g}_R)$ is the quotient $T(\mathfrak{g}_R)/J_R$ of the tensor algebra and $J_R$ is the two-sided ideal generated by the elements of the form $x\otimes y-y\otimes x$, $x, y\in \mathfrak{g}_R$. The symmetric algebra $S(\mathfrak{g}_R)$ is isomorphic to the polynomial algebra $$R[\{1_R\otimes h_\alpha\colon \alpha\in \varPi\}, \{1_R\otimes x_\alpha\colon \alpha\in \Phi \}].$$ In particular, $R\otimes_\mathbb{Z} S(\mathfrak{g}_\mathbb{Z})\simeq S(\mathfrak{g}_R)$.  The enveloping algebra of $\mathfrak{g}_R$ also has the base change property. Since $\mathfrak{g}_R$ and $T(\mathfrak{g}_R)$ are free over $R$, with basis elements independent of $R$, we can identify $R\otimes_\mathbb{Z} T(\mathfrak{g}_\mathbb{Z})$ with $T(\mathfrak{g}_R)$ and $R\otimes_\mathbb{Z} I_\mathbb{Z}$ with $I_R$. Hence, we have a commutative diagram with exact rows
\begin{equation}
	\begin{tikzcd}
		R\otimes_\mathbb{Z} I_\mathbb{Z} \arrow[r, hookrightarrow] \arrow[d, "\simeq"]&R\otimes_\mathbb{Z} T(\mathfrak{g}_\mathbb{Z})\arrow[r, twoheadrightarrow] \arrow[d, "\simeq"]& R\otimes_\mathbb{Z} U(\mathfrak{g}_\mathbb{Z}) \arrow[d] \\
		I_R\arrow[r, hookrightarrow] & T(\mathfrak{g}_R)\arrow[r, twoheadrightarrow] & U(\mathfrak{g}_R)
	\end{tikzcd}.
\end{equation}
Therefore, we obtain:

\begin{Lemma}
	Let $R$ be a commutative Noetherian ring with identity. 
	
	Then $U(\mathfrak{g}_R)\simeq R\otimes_\mathbb{Z} U(\mathfrak{g}_\mathbb{Z})$ and \mbox{$S(\mathfrak{g}_R)\simeq R\otimes_\mathbb{Z} S(\mathfrak{g}_\mathbb{Z})$}.
\end{Lemma}

Fix a total order in the Chevalley basis of $\mathfrak{g}_R$.
Since $\mathfrak{g}_R$ is free over $R$, the PBW theorem (see for example \citep[17.3]{zbMATH03699133}) gives the $R$-isomorphism \begin{align}
	U(\mathfrak{g}_R)\simeq U(\mathfrak{n}_R^-)\otimes_R U(\mathfrak{h}_R)\otimes_R U(\mathfrak{n}_R^+),
\end{align}and $U(\mathfrak{g}_R)$ has, as an $R$-module, a monomial basis over the basis elements of $\mathfrak{g}_R$. We call \textbf{PBW monomials} such monomials forming a PBW basis of $U(\mathfrak{g}_R)$. Further, it follows that the enveloping algebra of a free Lie algebra is a Noetherian ring (see \citep[7.4]{zbMATH04049807}).

Since both $U(\mathfrak{n}_R^+)$ and $U(\mathfrak{n}_R^-)$ are free over $R$, the PBW theorem allows us to view $U(\mathfrak{h}_R)$ as an $R$-summand of $U(\mathfrak{g}_R)$. Further, denote by $\pi_R$ the projection $U(\mathfrak{g}_R)\twoheadrightarrow U(\mathfrak{h}_R)$ which sends all PBW monomials with factors either in $\mathfrak{n}^+_R$ or in $\mathfrak{n}_R^-$ to zero.

Let $Z(\mathfrak{g}_R)$ be the center of the universal enveloping algebra $U(\mathfrak{g}_R)$. The restriction of $\pi_R$ to the center $Z(\mathfrak{g}_R)$ is called the \textbf{Harish-Chandra homomorphism}. For details on why this map is an $R$-algebra homomorphism see for example \citep[1.3.2]{zbMATH03747378}.
For each $\l\in \mathfrak{h}_R^*$, the \textbf{central character} associated with $\l$ is the $R$-algebra homomorphism $\chi_\l\colon Z(\mathfrak{g}_R)\rightarrow R$ , given by $\chi_\l(z)=\l(\pi(z))$, $z\in Z(\mathfrak{g}_R)$. For a given semisimple Lie algebra over a splitting field $K$, the Harish-Chandra theorem (see \citep[1.10]{zbMATH05309234}) says that all $K$-algebra homomorphisms are of the form $\chi_\l$ for some $\l\in \mathfrak{h}^*$.

\subsubsection{BGG category $\mathcal{O}$ over commutative rings}

Assume in the remainder of this section, unless stated otherwise, that $R$ is a commutative Noetherian ring and a $\mathbb{Q}$-algebra. In particular, $R$ has characteristic zero and there exists an injective homomorphism of rings $\mathbb{Q}\rightarrow R$, $q\mapsto q1_R$. We can extend the map $\langle -, -\rangle$ to $\mathfrak{h}_R^*\times \Phi^\vee\rightarrow R$.  Recall that $\{(1\otimes h_\alpha)^*\colon \alpha\in \varPi\}$ is a basis of $\mathfrak{h}_R^*$. We define $\langle \l, \alpha^\vee\rangle_R:=\sum_{\beta\in \varPi} t_\beta \langle \beta, \alpha^\vee\rangle$ for $\l=\sum_{\beta\in \varPi} t_\beta (1\otimes h_\beta)^*\in \mathfrak{h}_R^*$.

We call the set of integral weights $\L_R:=\{\l\in \mathfrak{h}_R^*\colon \langle \l, \alpha^\vee\rangle_R \in \mathbb{Z}, \ \forall \alpha \in \Phi \}$ the \textbf{integral weight lattice} associated with $\Phi$ with respect to $R$. For each $M\in U(\mathfrak{g}_R)\m$ and each $\l\in \mathfrak{h}_R^*$, we define the weight space \mbox{$M_\l:=\{m\in M\colon h\cdot m=\l(h)m, \ \forall h\in \mathfrak{h}_R \}$}.

For each $\l\in \mathfrak{h}_R^*$, we will denote by $[\l]$ the set of elements of $\mathfrak{h}_R^*$, $\mu$, that satisfy $\mu-\l\in \L_R$.
We define an ordering in $[\l]$ by imposing $\mu_1\leq \mu_2$ if and only if $\mu_2-\mu_1\in \mathbb{Z}^+_0\Phi^+\subset \L_R$.

\begin{Def}\label{BGGcategoryO}
	The BGG category $\mathcal{O}$ (or just the category $\mathcal{O}$) of a semi-simple Lie algebra $\mathfrak{g}$ over a splitting field of characteristic zero is the full subcategory of $U(\mathfrak{g})\M$ whose modules satisfy the following conditions:
	\begin{enumerate}[(i)]
		\item $M\in U(\mathfrak{g})\m$.
		\item $M$ is semi-simple over $\mathfrak{h}$, that is, $M=\bigoplus_{\l\in \mathfrak{h}^*} M_\l$.
		\item $M$ is locally $\mathfrak{n}^+$-finite, that is, for each $m\in M$ the subspace $U(\mathfrak{n}^+)m$ of $M$ is finite-dimensional.
	\end{enumerate}
\end{Def}

In a naive look, one could think that the category $\mathcal{O}$ is too large to be considered under the techniques that we studied here for projective Noetherian $R$-algebras. Especially, since there are an infinite number of Verma modules (which are the standard modules making the category $\mathcal{O}$ a split highest weight category) and even these Verma modules have infinite vector space dimension. So, instead of generalizing right away Definition \ref{BGGcategoryO} to the integral setup we can decompose $\mathcal{O}$ into smaller subcategories. 
In fact, for any $\l\in \mathfrak{h}^*$, there is a "block" associated with $\l$. In the following, we will identify $\L\subset \mathfrak{h}^*$ with $\L_\mathbb{C}$ and consider $[\l]\subset \mathfrak{h}^*$.

\begin{Lemma}\label{lemmaelemedecomO}
	Let $M\in \mathcal{O}$. For each $\l\in \mathfrak{h}^*$, define the vector space $M^{[\l]}:=\bigoplus_{\mu\in [\l]} M_\mu$, where $\mu\in [\l]$ if and only if $\mu-\l\in \L$.
	Then $M^{[\l]}\in U(\mathfrak{g})\m$ and $M=\bigoplus_{[\l]\in \mathfrak{h}^*/\sim} M^{[\l]}$, where $\sim$ denotes the equivalence relation given by $\mu-\l\in \L$.
\end{Lemma}
\begin{proof}See \citep[p.15]{zbMATH05309234}.
\end{proof}

\begin{Def}\citep[1.4]{zbMATH03747378}
	Let $R$ be a commutative Noetherian ring which is a $\mathbb{Q}$-algebra.	Let $\l\in \mathfrak{h}_R^*$.
	\begin{itemize}
		\item We define $\mathcal{O}_{[\l], (II), R}$ to be the full subcategory of  $U(\mathfrak{g}_R)\M$ consisting of modules $M$ satisfying $M=\sum_{\mu\in [\l]} M_\mu$.
		\item We define $\mathcal{O}_{[\l], (I), R}$ to be the full subcategory of $\mathcal{O}_{[\l], (II), R}$ whose modules $M$ are $U(\mathfrak{n}_R^+)$-locally finite, that is, $U(\mathfrak{n}_R^+)m\in R\m$ for every $m\in M$.
		\item We define $\mathcal{O}_{[\l], R}$ to be the full subcategory of $\mathcal{O}_{[\l], (I), R}$ whose modules are finitely generated over $U(\mathfrak{g}_R)$. 
	\end{itemize} 
\end{Def}

In particular, Lemma \ref{lemmaelemedecomO} says that we can reduce the study of the category $\mathcal{O}$ to the categories $\mathcal{O}_{[\l], \mathbb{C}}$, where $\l\in \mathfrak{h}^*$.  Moreover, by a BGG category $\mathcal{O}$ over a commutative ring $R$ we will mean a category $\mathcal{O}_{[\l], R}$ for some $\l\in \mathfrak{h}_R^*$.

It comes as no surprise that Verma modules can be defined over any ground ring. Let $\mu\in [\l]$ and $R_\mu$ be the free $R$-module with rank one together with the $U(\mathfrak{h}_R)$-action $h1_R=\mu(h)1_R$, $h\in \mathfrak{h}_R$. We can extend $R_\mu$ to be an $U(\mathfrak{b}_R)$-module by letting $1_R \otimes x_\alpha$, $\alpha\in \Phi^+$,  act on $R_\mu$ identically as zero. The \textbf{Verma module} $\St(\mu)$ (associated with $\mu$) is defined to be the $U(\mathfrak{g}_R)$-module $\St(\mu):=U(\mathfrak{g}_R)\otimes_{U(\mathfrak{b}_R)}R_\mu$. 

\begin{Lemma}Let $\l\in \mathfrak{h}_R^*$.
	If $\mu\in [\l]$, then $\St(\mu)\in \mathcal{O}_{[\l], R}$ and $\St(\mu)$ is free as $U(\mathfrak{n}^-_R)$-module.
\end{Lemma}
\begin{proof}
The reasoning is analogous to \citep[1.3]{zbMATH05309234}. See also \citep[1.4.1]{zbMATH03747378}. 
\end{proof}

We observe that  $\St(\mu)$ is not finitely generated over $R$ and the set of weights of $\St(\mu)$ is contained in $\mu-\mathbb{Z}_0^+\varPi$. In the following, we state some known facts about homomorphisms between Verma modules.

\begin{Lemma}\label{homomorphismbetweenvermamodules}
	Let $\l\in \mathfrak{h}_R^*$.  Then:
	\begin{enumerate}[(i)]
		\item For every $\mu, \omega\in [\l]$, if $\Hom_{\mathcal{O}_{[\l]}, R}(\St(\mu), \St(\omega))\neq 0$, then $\mu\leq \omega$.
		\item $\End_{\mathcal{O}_{[\l], R}}(\St(\mu))\simeq R$ for every $\mu\in [\l]$.
		\item For every $\mu, \omega\in [\l]$, any non-zero map in $\Hom_{\mathcal{O}_{[\l]}, R}(\St(\mu), \St(\omega))$ is injective.
	\end{enumerate}
\end{Lemma}
\begin{proof}
	Let $\mu, \omega\in [\l]$ such that $\Hom_{\mathcal{O}_{[\l]}, R}(\St(\mu), \St(\omega))\neq 0$.
	By Tensor-Hom adjunction, $$\Hom_{\mathcal{O}_{[\l]}, R}(\St(\mu), \St(\omega))\simeq \Hom_{U(\mathfrak{b}_R)}(R_\mu, \St(\omega))\subset \St(\omega)_\mu.$$ By assumption, $\mu$ is a weight of $\St(\omega)$. Hence, $\mu\in \omega-\mathbb{Z}_0^+\varPi$.  So, $\mu\leq \omega$.
	
	If $\mu=\omega$, then for any homomorphism $f\in\Hom_{U(\mathfrak{b}_R)}(R_\mu, \St(\mu))$, $f(1_R)\in \St(\mu)_\mu=R$ and it is annihilated by $\mathfrak{n}_R^+$. Further, for every element $r\in R$, we can define $g\in \Hom_{U(\mathfrak{b}_R)}(R_\mu, \St(\mu))$, by imposing $g(1_R)=r(1_{U(\mathfrak{g}_R)}\otimes_{U(\mathfrak{b}_R)} 1_R)$. This shows that $\End_{\mathcal{O}_{[\l], R}}(\St(\mu))\simeq R$.
	
	For (iii), we can apply the same idea as in the classical case (see \citep[4.2]{zbMATH05309234}). In fact, for every $f\in \Hom_{\mathcal{O}_{[\l]}, R}(\St(\mu), \St(\omega))$ we can write $f(1_{U(\mathfrak{g}_R)}\otimes_{U(\mathfrak{b}_R)} 1_{R_\mu})=u1_{U(\mathfrak{g}_R)}\otimes_{U(\mathfrak{b}_R)} 1_{R_\omega}$ for some $u\in U(\mathfrak{n}_R^-)$. Using the PBW theorem we can see that $U(\mathfrak{n}_R^-)$ is an integral domain (see \citep[7.4]{zbMATH04049807}). By identifying $f$ with an endomorphism of $U(\mathfrak{n}_R^-)$ given by $a\mapsto au$, $U(\mathfrak{n}_R^-)$ being an integral domain implies that $f$ is injective.
\end{proof}

\subsubsection{Properties of (classical) BGG category $\mathcal{O}$ }

Before we proceed any further, we should recall some properties of the category $\mathcal{O}$ for a given semisimple Lie algebra $\mathfrak{g}$ over a splitting field $K$ of characteristic zero.

The category $\mathcal{O}$ can be decomposed in finer blocks than the ones described in Lemma \ref{lemmaelemedecomO} and these can be completely determined by the orbits under the dot action of the Weyl group on the space $\mathfrak{h}^*$. 
In fact, for any $M\in \mathcal{O}$, $M=\bigoplus_{\chi} M^\chi$, as $\chi$ runs over the central characters $Z(\mathfrak{g})\rightarrow K$ and \begin{align}
	M^\chi:=\{m\in M\colon \forall z\in Z(\mathfrak{g}) \ \exists n\in \mathbb{N}, \ (z-\chi(z))^nm=0 \}.
\end{align}is a module in $\mathcal{O}$. The argument provided in \citep[1.12]{zbMATH05309234} requires $K$ to be an algebraically closed field, but we do not need such a condition. We could use instead  Gabber and Joseph techniques (see \citep[1.4.2]{zbMATH03747378}) together with the Harish-Chandra theorem stating that $\chi_\l=\chi_\mu$ if $\mu$ and $\l$ are linked by a certain Weyl group and taking into account that the category $\mathcal{O}$ is both Artinian and Noetherian (see \citep[1.11]{zbMATH05309234}). To see that this is a finite direct sum is also required to observe that $\St(\l)^{\chi_\l}=\St(\l)$ and $M\mapsto M^{\chi_\l}$ is an exact functor $\mathcal{O}\rightarrow \mathcal{O}$ for every $\l\in \mathfrak{h}^*$.
For each central character $\chi$, denote by $\mathcal{O}_\chi$ the full subcategory of $\mathcal{O}$ whose objects are the modules $M$ satisfying $M=M^{\chi}$.

The dot action of the Weyl group $W$ is defined as \mbox{$w\cdot \l:=w(\l+\rho)-\rho$}, where $\rho$ is the half-sum of all positive roots.
With this, for each $\l\in \mathfrak{h}$, one can define another Weyl group $W_{[\l]}$ associated with a root system that views $\l$ as an integral weight lattice. Explicitly, $W_{[\l]}:=\{w\in W\colon w\cdot\l-\l\in \mathbb{Z}\Phi\}$.

\begin{Theorem}The following results hold for the category $\mathcal{O}$ of a semisimple Lie algebra over a splitting field of characteristic zero.
	\begin{enumerate}[(a)]
		\item 	For each $\l\in \mathfrak{h}^*$, the Verma module $\St(\l)$ has a unique simple quotient and every simple module in $\mathcal{O}$ is isomorphic to the simple quotient of some Verma module $\St(\l)$ which we will denote by $L(\l)$.
		\item The simple module $L(\l)$ in $\mathcal{O}$ is finite-dimensional if and only if $\langle \l, \alpha^\vee\rangle\in \mathbb{Z}_0^+$ for every $\alpha\in \Phi^+$. Such weights are known as integral dominant weights. In such a case, $L(\l)\simeq U(\mathfrak{g})/J$. Here, $J$ is the left ideal of $U(\mathfrak{g})$ generated by the elements $x_\alpha$, ($\alpha\in \Phi^+$), $h-\l(h)1$, ($h\in \mathfrak{h}$) and $x_{\beta}^{n_\beta+1}$, ($\beta\in \varPi$), where $n_\beta=\langle \l, \beta^\vee\rangle\in \mathbb{Z}_0^+$.
		\item The Verma module $\St(\l)$ is simple in $\mathcal{O}$ if and only if $\l$ is antidominant, that is, $\langle \l+\rho, \alpha^\vee\rangle \notin \mathbb{N}$ for all $\alpha\in \Phi^+$. In particular, $\l$ is minimal and the unique antidominant weight in its $W_{[\l]}$-orbit.
		\item The Verma module $\St(\l)$ is projective in $\mathcal{O}$ if and only if $\l$ is dominant, that is, $\langle \l+\rho, \alpha^\vee\rangle \notin \mathbb{Z}^-$ for all $\alpha\in \Phi^+$. In particular, $\l$ is maximal and the unique dominant weight in its $W_{[\l]}$-orbit.
		\item $\mathcal{O}$ has enough projectives, and the projective cover of $\St(\l)$ (which exists) is injective if and only if $\l$ is antidominant.
		\item The category $\mathcal{O}$ is the direct sum of the subcategories $\mathcal{O}_{\chi_\l}$ consisting of modules whose composition factors all have highest weights linked by $W_{[\l]}$, as $\l$ runs over all antidominant weights (or alternatively over all dominant weights). In particular, $\chi_\l=\chi_\mu$ if $\mu$ and $\l$ belong to the same orbit under the Weyl group $W_{[\l]}$.
		\item The blocks $\mathcal{O}_{\chi_\l}$ with the Verma modules being the standard modules are split highest weight categories with a finite number of standard modules. Here, the ordering is given by $\mu_1\leq \mu_2$ if and only if $\mu_2-\mu_1\in \mathbb{Z}_0^+\varPi $. 
	\end{enumerate}\label{Propertiesofsemisimplecomplexliealg}
\end{Theorem}
\begin{proof}
	For (a) see \cite[p.18]{zbMATH05309234}. For (b) see \cite[p.21, p.44]{zbMATH05309234}. For (c) see \cite[p.55,p.77]{zbMATH05309234}. For (d) see \cite[p.55, p.60]{zbMATH05309234}. For (e) see \cite[p.60-61, p.149-151]{zbMATH05309234} For (f) see \cite[p.83]{zbMATH05309234}. For (g) see \cite[p.64-65, p.68]{zbMATH05309234}.
\end{proof}
Observe that if a weight $\l$ is both antidominant and dominant, then $\St(\l)$ is projective and simple. So, the block $\mathcal{O}_{\chi_\l}$ is semisimple if and only if $\l$ is both antidominant and dominant.

\subsubsection{Properties of BGG category $\mathcal{O}$ over commutative rings}

An important property of the category $\mathcal{O}$ is that ultimately it can be viewed as a direct sum of module categories over finite-dimensional algebras. This is what we will explore for the BGG category $\mathcal{O}$ over a commutative ring. As we will see later on, we must impose that $R$ is also local so that the classical category $\mathcal{O}$ is obtained as a specialization of a direct sum of module categories of projective Noetherian $R$-algebras. But for now, assume just that $R$ is a commutative Noetherian ring which is a $\mathbb{Q}$-algebra.

\begin{Def}\label{decompositiontechnique}
	Let $S$ be any commutative ring and $\{J_i\}_{i\in I}$ be a family of two-sided ideals of $S$ such that $J_i+J_j=S$ whenever $i\neq j$. Let $\mathcal{J}$ be the category of $S$-modules $M=\sum_{i\in I} M_i$ where $$M_i=\{m\in M\colon \forall x\in J_i \ \exists n\in \mathbb{N}, \ x^nm=0 \}.$$
\end{Def}Note that $m_1+m_2\in M_i$ whenever $m_1, m_2\in M_i$ since for each $x\in J_i$ we can choose  the higher value $n_1$ and $n_2$ and then $x^{\max \{n_1, n_2\}}$ kills $m_1+m_2$. 
Since $S$ is commutative $M_i$ becomes an $S$-module.

\begin{Lemma}\label{propertiesoffamilyidealslikeidempotents}
	Let $S$ be any commutative ring and $\{J_i\}_{i\in I}$ be a family of two-sided ideals of $S$ such that $J_i+J_j=S$ whenever $i\neq j$. The following assertions hold.
	\begin{enumerate}[(i)]
		\item If $M\in \mathcal{J}$, then $M= \bigoplus_{i\in I}M_i$.
		\item The category $\mathcal{J}$ is closed under submodules, quotients and direct sums.
		\item The functor $\mathcal{J}\rightarrow \mathcal{J}$, given by $M\mapsto M_i$, is an exact functor. 
	\end{enumerate}
\end{Lemma}
\begin{proof} See \citep[1.4.3]{zbMATH03747378}.
	Due to \citep[1.4.2]{zbMATH03747378}, $\mathcal{J}$ is also closed under direct sums. For each element of the quotient of $M\in \mathcal{J}$ and for every $x_i\in J_i$ we can pick the $n_i\in \mathbb{N}$ which satisfies the condition in \citep[1.4.2]{zbMATH03747378} for its preimage. Hence, $\mathcal{J}$ is also closed under quotients.
\end{proof}

To get an idea of what Lemma \ref{propertiesoffamilyidealslikeidempotents} is doing we can think about central idempotents. 
For a set of central orthogonal  idempotents, $\{e_1, \ldots, e_n\}$ of a commutative ring $S$, define $J_i=S\sum_{j=1, j\neq i}^n e_j$. Then $J_i+J_j=S$ whenever $i\neq j$ and $M_i=e_iM$. Hence, Lemma \ref{propertiesoffamilyidealslikeidempotents} is a generalization of the process of decomposing a module in terms of orthogonal idempotents over a commutative ring.

We will now apply  Lemma \ref{propertiesoffamilyidealslikeidempotents} to the symmetric algebra of the Cartan algebra  $S=U(\mathfrak{h}_R)=S(\mathfrak{h}_R)$ and the category $\mathcal{O}_{[\l], (II), R}$ taking the role of $\mathcal{J}$. For each $\l\in \mathfrak{h}_R^*$, define the $R$-algebra homomorphism $p_\l\colon S(\mathfrak{h}_R)\rightarrow R$, given by $h\mapsto \l(h), \ h\in \mathfrak{h}_R$.

\begin{Lemma} Fix $\l\in \mathfrak{h}_R^*$.
	Consider the family of ideals $J_\mu:=\ker p_\mu$, $\mu\in [\l]$, of the symmetric algebra $S(\mathfrak{h}_R)$. Then $J_{\mu_1}+J_{\mu_2}=S(\mathfrak{h}_R)$ whenever $\mu_1\neq \mu_2$.\label{symmetricalgebradecomposition}
\end{Lemma}
\begin{proof}
The result follows from $R$ being a $\mathbb{Q}$-algebra. See \citep[1.4.4]{zbMATH03747378}.
\end{proof}

\begin{Cor}\label{corpropertiesOtypeII}
	For every $M\in \mathcal{O}_{[\l], (II), R}$, the following assertions hold:
	\begin{enumerate}
		\item $M= \bigoplus_{\mu\in [\l]} M_\mu$.
		\item The assignment $M\mapsto M_\mu$ is an exact functor on $\mathcal{O}_{[\l], (II), R}$.
		\item The category  $\mathcal{O}_{[\l], (II), R}$ is closed under quotients, submodules and direct sums.
	\end{enumerate}
\end{Cor}
\begin{proof}
	It follows combining Lemma \ref{propertiesoffamilyidealslikeidempotents} with Lemma \ref{symmetricalgebradecomposition}. See  \citep[1.4.4]{zbMATH03747378}.
\end{proof}

\begin{Cor}
\label{allmodulesOarefginnegatives}	The categories $\mathcal{O}_{[\l], (I), R}$ and $\mathcal{O}_{[\l], R}$ are closed under submodules, quotients and direct sums. Furthermore, if $M\in \mathcal{O}_{[\l], R}$, then $M\in U(\mathfrak{n}_R^-)\m$.
\end{Cor}
\begin{proof}
See \citep[1.4.7]{zbMATH03747378} and \citep[1.4.8]{zbMATH03747378}.
\end{proof}

Taking into account that the Verma modules are free of rank one over $U(\mathfrak{n}_R^-)$, the second part of Corollary \ref{allmodulesOarefginnegatives} can be interpreted as saying that Verma modules are in some sense the building blocks of the category $\mathcal{O}_{[\l], R}$ taking the place of projective indecomposable modules. Note once more that for non-local rings the category $\mathcal{O}_{[\l], R}$ is very far from being Krull-Schmidt.  To make this statement about Verma modules more precise, the following equivalent construction of Verma modules is useful.

For each $\l\in \mathfrak{h}_R^*$, the Verma module $\St(\l)$ is generated by $1_{U(\mathfrak{g}_R)}\otimes_{U(\mathfrak{b}_R)} 1_R$ as $U(\mathfrak{g}_R)$-module. Moreover, for every $\alpha\in \Phi^+$, $1_R\otimes x_\alpha$ acts as zero and each $h\in \mathfrak{h}_R$ acts as $\l(h)$. Also $\St(\l)$ is free as $U(\mathfrak{n}_R^-)$-module, therefore the surjective map  $U(\mathfrak{g}_R)\rightarrow \St(\l)$ given by $1_{U(\mathfrak{g}_R)}\mapsto 1_{U(\mathfrak{g}_R)}\otimes_{U(\mathfrak{b}_R)} 1_R$ has kernel $I_\l$ where $I_\l$ is the ideal generated by $1_R\otimes x_\alpha$, $\alpha\in \Phi^+$ and $h-\l(h)1_R$, $h\in \mathfrak{h}_R$. Hence, $\St(\l)\simeq U(\mathfrak{g}_R)/I_\l$.

\begin{Lemma}\label{filtrationquotientsofVerma} Let $M\in \mathcal{O}_{[\l], R}$. The following assertions hold.
	\begin{enumerate}[(a)]
		\item $M\in \mathcal{O}_{[\l], R}$ has a finite filtration with quotients isomorphic to quotients of $\St(\mu)$, $\mu\in [\l]$.
		\item For each $\mu\in [\l]$, $M_\mu\in R\m$.
		\item For every $N\in \mathcal{O}_{[\l], R}$, $\Hom_{\mathcal{O}_{[\l], R}}(M, N)\in R\m$.
	\end{enumerate} 
\end{Lemma}
\begin{proof}
	For (a) see \citep[1.4.9]{zbMATH03747378} and (b) is  \citep[1.4.10]{zbMATH03747378}.
\end{proof}

\begin{Prop}
	Let $M, N\in \mathcal{O}_{[\l], R}$. Then $\Hom_{\mathcal{O}_{[\l], R}}(M, N)\in R\m$.
\end{Prop}
\begin{proof}
	We will proceed by induction on the length of $M$ and $N$ by quotients of Verma modules given in Lemma \ref{filtrationquotientsofVerma}. Let $Q(\mu)$ be a quotient of $\St(\mu)$ and $Q(\omega)$ be a quotient of $\St(\omega)$, $\mu, \omega\in [\l]$. Applying $\Hom_{\mathcal{O}_{[\l], R}}(-, Q(\mu))$ we obtain the monomorphism $\Hom_{\mathcal{O}_{[\l], R}}(Q(\omega), Q(\mu))\rightarrow \Hom_{\mathcal{O}_{[\l], R}}(\St(\omega), Q(\mu))\subset Q(\mu)_\omega$. By (b), $Q(\mu)_\omega\in R\m$. Since $R$ is a Noetherian ring we obtain that $\Hom_{\mathcal{O}_{[\l], R}}(Q(\omega), Q(\mu))\in R\m$. Assume now that there exists an exact sequence $0\rightarrow M'\rightarrow M\rightarrow Q(\omega)\rightarrow 0$. Again, applying $\Hom_{\mathcal{O}_{[\l], R}}(-, Q(\mu))$ we obtain the exact sequence
	\begin{align}
		0\rightarrow \Hom_{\mathcal{O}_{[\l], R}}(Q(\omega), Q(\mu))\rightarrow \Hom_{\mathcal{O}_{[\l], R}}(M, Q(\mu))\rightarrow X\rightarrow 0,
	\end{align}where $X$ is a submodule of $\Hom_{\mathcal{O}_{[\l], R}}(M', Q(\mu))\in R\m$ by induction. 
	
	Thus, $X\in R\m$ and $\Hom_{\mathcal{O}_{[\l], R}}(M, Q(\mu))\in R\m$. Now, using exact sequences $0\rightarrow N'\rightarrow N\rightarrow Q(\mu)\rightarrow 0$ and applying $\Hom_{\mathcal{O}_{[\l], R}}(M, -))$ we obtain by induction that $\Hom_{\mathcal{O}_{[\l], R}}(M, N)\in R\m$.
\end{proof}

\subsubsection{Verma flags}

For this subsection, we will require in addition that $R$ is a local commutative Noetherian ring which is a $\mathbb{Q}$-algebra.
We will now discuss modules having filtrations by Verma modules. As usual, we will denote by $\mathcal{F}(\St_{[\l]})$ the full subcategory of $\mathcal{O}_{[\l], R}$ having a filtration by modules $\St(\mu)$, $\mu\in [\l]$, where $\St_{[\l]}$ denotes the set $\{\St(\mu)\colon \mu\in [\l] \}$.  In the literature, these filtrations are known as \textbf{Verma flags}.

To construct Verma flags, we need to discuss the tensor product of modules in $\mathcal{O}_{[\l], R}$. 

Let $\mathfrak{a}_R$ be any Lie algebra with finite rank over $R$. Recall that $L\otimes_R M\in U(\mathfrak{a}_R)\M$ whenever $L, M\in U(\mathfrak{a}_R)$ with action $a\cdot (l\otimes m)=(al)\otimes m+l\otimes (am)$, $a\in \mathfrak{a}_R$, and any left $U(\mathfrak{a}_R)$-module $L$ can be regarded as right $U(\mathfrak{a}_R)$-module by taking $l\cdot a:=-al$, $a\in \mathfrak{a}_R$. This is a consequence of the universal enveloping algebra of a Lie algebra being a Hopf algebra. For each left $U(\mathfrak{a}_R)$-module, by $L^*$ we mean the left $U(\mathfrak{a}_R)$-module, $\Hom_R(L, R)$, which inherits the left action by  regarding $L$ as a right $U(\mathfrak{a}_R)$-module.
So, for every $L, M\in U(\mathfrak{a}_R)\M$ we can also regard $\Hom_R(L, M)$ as an $U(\mathfrak{a}_R)$-module by taking $(a\cdot f)(l):=af(l)-f(al)$, $f\in \Hom_R(L, M)$, $a\in \mathfrak{a}_R$, $l\in L$. In particular, the $\mathfrak{a}_R$-invariants of $\Hom_R(L, M)$ are the elements $f\in \Hom_R(L, M)$ satisfying $a\cdot f=0$. Therefore, they coincide with $\Hom_{U(\mathfrak{a}_R)}(L, M)$. 

So, the following Tensor Identity (see \citep[3.6]{zbMATH05309234}) also holds for Lie algebras over commutative rings. 

\begin{Lemma}
	For $L\in U(\mathfrak{b}_R)\M$, $M\in U(\mathfrak{g}_R)\M$, we have the isomorphism\begin{align}
		U(\mathfrak{g}_R)\otimes_{U(\mathfrak{b}_R)} (L\otimes_R M)\simeq (U(\mathfrak{g}_R)\otimes_{U(\mathfrak{b}_R)} L)\otimes_R M.
	\end{align}
\end{Lemma}
\begin{proof}
	For any $X\in U(\mathfrak{g}_R)\M$, we can write
	\begin{align}
		\Hom_{U(\mathfrak{g}_R)}(U(\mathfrak{g}_R)\otimes_{U(\mathfrak{b}_R)} (L\otimes_R M), X)&\simeq \Hom_{U(\mathfrak{b}_R)}(L\otimes_R M, X)\\
		&\simeq \Hom_{U(\mathfrak{b}_R)}(L, \Hom_R(M, X))\\
		&\simeq \Hom_{U(\mathfrak{b}_R)}(L,   \Hom_{U(\mathfrak{g}_R)}(U(\mathfrak{g}_R), \Hom_R(M, X))   \\ &\simeq \Hom_{U(\mathfrak{g}_R)} ( (U(\mathfrak{g}_R)\otimes_{U(\mathfrak{b}_R)} L)\otimes_R M, X).
	\end{align}
	The first isomorphism is obtained by Tensor-Hom adjunction, the second by Tensor-Hom adjunction and taking on both sides $\mathfrak{b}_R$-invariants and the other ones are again obtained by Tensor-Hom adjunction. So, this provides an isomorphism between these two Hom functors. By taking the image of the identity on $(U(\mathfrak{g}_R)\otimes_{U(\mathfrak{b}_R)} L)\otimes_R M$ under the unit of the isomorphism of functors we obtain the desired isomorphism as $U(\mathfrak{g}_R)$-modules.
\end{proof}

\begin{Remark}
	The functor $U(\mathfrak{g}_R)\otimes_{U(\mathfrak{b}_R)} - \colon U(\mathfrak{b}_R)\M\cap R\Proj\rightarrow U(\mathfrak{g}_R)\M$ is exact. In fact, $U(\mathfrak{g}_R)\simeq U(\mathfrak{n}_R^-)\otimes_R U(\mathfrak{b}_R)\in U(\mathfrak{b}_R)\Proj$ when regarded as $U(\mathfrak{b}_R)$-module.
\end{Remark}

The following is the generalization of \citep[3.6]{zbMATH05309234}. 

\begin{Prop}	\label{tensoringbyVermahasvermaflag}Assume that $R$ is a local commutative Noetherian ring which is a $\mathbb{Q}$-algebra.
	Let $M\in \mathcal{O}_{[\l], R}$ which is free over the ground ring $R$. Then $\St(\omega)\otimes_R M\in \mathcal{F}(\{\St(\mu+\omega)\colon  \mu\in [\l], \ M_\mu\neq 0\})$.
\end{Prop}
\begin{proof}
	The module $M$ is free of finite rank, and so each $M_\mu$ is also free of finite rank. Hence, the basis of $M$ can be picked among the weight vectors of $M$. So, the arguments used in \citep[3.6]{zbMATH05309234} are still valid when the ground ring is $R$.
\end{proof}

Now, we show that the results in \citep[3.7]{zbMATH05309234} also hold in this setup.

\begin{Prop}Assume that $R$ is a local commutative Noetherian ring which is a $\mathbb{Q}$-algebra.
	Let $M\in \mathcal{F}(\St_{[\l]})$. The following assertions hold.
	\begin{enumerate}[(a)]
		\item If $\mu$ is a maximal weight in $M$, then $\St(\mu)\subset M$ and $M/\St(\mu)\in \mathcal{F}(\St_{[\l]})$.
		\item The category $\mathcal{F}(\St_{[\l]})$ is closed under direct summands.
		\item $M$ is free as $U(\mathfrak{n}_R^-)$-module.
	\end{enumerate}
\end{Prop}
\begin{proof}
	The argument for \citep[3.7(a)]{zbMATH05309234} remains valid in our setup.
	Assume that $M=M_1\oplus M_2$ has a filtration by Verma modules. If $M$ is a Verma module, then there is nothing to prove since $R$ is local the Verma modules are indecomposable modules. We shall proceed by induction on the size of the filtration of $M$. Let $\mu$ be a maximal weight of $M$. We have $M_\mu= (M_1)_\mu\oplus (M_2)_\mu$. Assume that $(M_1)_\mu\neq 0$. By (a), $\St(\mu)\subset M_1\subset M$ and $M/\St(\mu)\simeq M_1/\St(\mu)\oplus M_2\in \mathcal{F}(\St_{[\l]})$. By induction, $M_1/\St(\mu)$ has a filtration by Verma modules. So, $M_1\in\mathcal{F}(\St_{[\l]})$. 
	
	By proceeding on induction on the filtration of $M$ and since each Verma module is free as $U(\mathfrak{n}_R^-)$-module we obtain that $M$ is also free as $U(\mathfrak{n}_R^-)$-module.
\end{proof}

\subsubsection{Duality in BGG categories over commutative rings}

The classical category $\mathcal{O}$ admits a simple preserving duality functor. However, since the most interesting modules in the category $\mathcal{O}$ are not finite-dimensional we cannot use the usual standard duality. But, the weight spaces are finite-dimensional, and so this property could be used to define a duality in $\mathcal{O}$ using the standard duality "locally".
However, there is another problem in this case. For a general BGG category $\mathcal{O}$ over a commutative local Noetherian ring $R$ which is a $\mathbb{Q}$-algebra  we cannot define a duality, even locally, for all modules. We have to focus our attention only on those modules which are free over $R$. In addition, we have to impose that $R$ is an integral domain.

Define $M^\vee=\bigoplus_{\mu\in[\l]} DM_\mu$ for $M\in \mathcal{O}_{[\l], R}\cap R\Proj$. This becomes an $U(\mathfrak{g}_R)$-module by imposing $(g\cdot f)(v)=f(\tau(g)v)$, where $\tau\colon U(\mathfrak{g}_R)\rightarrow U(\mathfrak{g}_R)$ is the involution map that fixes $\mathfrak{h}_R$ and sends $x_\alpha$ to $x_{-\alpha}$ for every $\alpha\in \Phi$.  Using the fact that $R$ is an integral domain one sees that his action identifies $(DM)_\mu$ with $D(M_\l)$ justifying why we changed the action. In fact, any $f\in DM$ with weight $\mu$ satisfies \begin{align}
	f(\mu(h)m)=\mu(h)f(m)=(h\cdot f)(m)=f(\tau(h)m)=f(hm)=f(\omega(h)m), \ \forall m\in M_\omega.
\end{align}Now using that $R$ is an integral domain we would obtain that $f(m)=0$ for all $m\in M_\omega$ whenever $\omega\neq \mu$. Hence, $M^\vee\in \mathcal{O}_{[\l], (II), R}$ whenever $M\in \mathcal{O}_{[l], R}$. 

Observe also that for every $\alpha\in \Phi^+$, $f\in (DM)_\mu$, we have
\begin{align}
	h(x_\alpha f)=[h, x_\alpha]f+x_\alpha h f=\alpha(h)x_\alpha f+\mu(h)x_\alpha f=(\alpha+\mu)(h)x_\alpha f, \ \forall h\in \mathfrak{h}_R.
\end{align}Hence, $x_\alpha f \in (DM)_{\alpha+\mu}\simeq D(M_{\alpha+\mu})$. So, $\mathfrak{n}_R^+f\in R\m$ and consequently, $M^\vee\in \mathcal{O}_{[\l], (I), R}$ whenever $M\in \mathcal{O}_{[\l], R}$. The problem lies in deciding if $M^\vee\in \mathcal{O}_{[\l], R}$, that is if $M^\vee$ is finitely generated as $U(\mathfrak{g}_R)$-module. In the classical case, this is achieved by exploiting the simple modules and the composition series of the modules in $\mathcal{O}$. 

Observe that for $M\in \mathcal{O}_{[\l], (I), R}\cap R\Proj$, $(M^\vee)^\vee=\bigoplus_{\mu\in [\l]} DDM_\mu\simeq\bigoplus_{\mu\in [\l]} M_\mu\simeq M$.

So, in short, we obtained a contravariant exact functor $(-)^\vee\colon \mathcal{O}_{[\l], (I), R}\cap R\Proj\rightarrow \mathcal{O}_{[\l], (I), R}\cap R\Proj$ which is self-dual. 
In particular, it is fully faithful.

\subsubsection{Change of rings}

It is at this point that our approach will start to diverge from Gabber and Joseph. As the reader may see we are closer to seeing that $\mathcal{O}_{[\l], R}$ is a split highest weight category.  
But, for that, we require further techniques and constructions. In particular, how $\mathcal{O}_{[\l], R}$ behaves under change of ground ring. 

Concerning Verma modules, we can see that they remain Verma under change of ring. In fact, for every commutative $R$-algebra $S$ which is a Noetherian ring, and any $\l\in \mathfrak{h}_R^*$, \begin{align*}
	S\otimes_R \St(\l)&=S\otimes_R U(\mathfrak{g}_R)\otimes_{U(\mathfrak{b}_R)} R_\l
	\simeq S\otimes_R U(\mathfrak{g}_R)\otimes_{ S\otimes_RU(\mathfrak{b}_R)} S\otimes_R R_\l
	\simeq U(\mathfrak{g}_S)\otimes_{ U(\mathfrak{b}_S)} S_{1_S\otimes_R \l}\\&=\St(1_S\otimes_R \l).
\end{align*}

More generally, we can say the following.

\begin{Lemma}Let $R$ be a commutative Noetherian ring which is a $\mathbb{Q}$-algebra and let $\l\in \mathfrak{h}_R^*$.
	For any commutative Noetherian ring $S$ which is an $R$-algebra, the functor $S\otimes_R -\colon \mathcal{O}_{[\l], R}\rightarrow \mathcal{O}_{[1_S\otimes_R \l], S}$ is well defined and $S\otimes_R \St(\mu)\simeq \St(1_S\otimes_R \mu)$  for every $\mu\in [\l]$. 
	Moreover, $S\otimes_R M_\mu=(S\otimes_R M)_{1_S\otimes \mu}$ for every $\mu\in [\l]$. \label{changeringscategoryO}
\end{Lemma}
\begin{proof}Observe that $S$ is also a $\mathbb{Q}$-algebra, by imposing $q\cdot 1_S=(q1_R)\cdot 1_S$.
	Let $M\in \mathcal{O}_{[\l], R}$. By Lemma \ref{corpropertiesOtypeII}, $M=\bigoplus_{\mu\in [\l]} M_\mu$. Thus, $S\otimes_R M=\sum_{\mu\in [\l]} S\otimes_R M_\mu$ and $S\otimes_R M\in U(\mathfrak{g}_S)\m$, by identifying $U(\mathfrak{g}_S)$ with $S\otimes_R U(\mathfrak{g}_R)$.
	By assumption, for all $m\in M$,
	\begin{align}
		U(\mathfrak{n}_S^+)(1_S\otimes m)\simeq S\otimes_R U(\mathfrak{n}_R^+)(1_S\otimes m)\simeq S\otimes_R U(\mathfrak{n}_R^+)m\in S\m.
	\end{align}Since the elements $1_S\otimes m$, $m\in M$, generate $S\otimes_R M$ we obtain that $S\otimes_R M\in \mathcal{O}_{[1_S\otimes_R \l], S}$. It remains to show that $S\otimes_R M_\mu=(S\otimes_R M)_\mu$ for every $\mu\in [\l]$. Any element of $S\otimes_R M_\mu$ has weight $1_S\otimes_R \mu$. So, $S\otimes_R M_\mu\subset (S\otimes_R M)_{1_S\otimes \mu}$.  But, for each $\mu\in [\l]$, \begin{align}
		(S\otimes_R M)_{1\otimes \mu}\subset S\otimes_R M = \sum_{\theta\in [\l]} S\otimes_R M_\theta\subset \sum_{\theta\in [\l]} (S\otimes_R M)_{1\otimes \theta}.
	\end{align}
	Since $S\otimes_R M\in \mathcal{O}_{[1_S\otimes_R \l], S}$, we can write $S\otimes_R M=\bigoplus_{\omega\in [1_S\otimes_R \l]} (S\otimes_R M)_\omega$ and consequently the result follows.
\end{proof}

We observe that we cannot apply right way Theorem 4.15 of \citep{Rouquier2008} (see also \citep[Theorem 3.1.2]{cruz2021cellular}) since $\mathcal{O}_{[\l], R}$ is still too big and contains an infinite number of Verma modules. Instead, we will construct projective objects and decompose  $\mathcal{O}_{[\l], R}$  into smaller subcategories which will allow us to construct projective Noetherian $R$-algebras with module categories being deformations of the blocks of the category $\mathcal{O}$.
To obtain such a statement $R$ being local is crucial. In fact, Gabber and Joseph \citep[1.7]{zbMATH03747378} proved that all simple modules are quotients of Verma modules and the number of simple modules for deformations of the category $\mathcal{O}$ that appear as a quotient of a Verma module depends on the number of maximal ideals of the ground ring. So, outside local rings, we cannot expect $\mathcal{O}_{[\l], R}$ to decompose in the desired way.

As in the classical case, the first step is to see that the center of the universal enveloping algebra behaves well under change of ground ring.

\begin{Lemma}\label{centerofLiealgebra}
	Let $R$ be a commutative Noetherian ring which is a $\mathbb{Q}$-algebra and $S$ a commutative Noetherian ring which is an  $R$-algebra. Then $S\otimes_R Z(\mathfrak{g}_R)\simeq Z(\mathfrak{g}_S)$.
\end{Lemma}
\begin{proof}
	Actually, we just need to observe that
	$ 	R\otimes_{\mathbb{Q}} Z(\mathfrak{g}_\mathbb{Q})= Z(R\otimes_\mathbb{Q}\mathfrak{g}_\mathbb{Q}). 
$ Assume for the moment that it holds. Then \begin{align*}
		S\otimes_R Z(U(\mathfrak{g}_R))\simeq S\otimes_R Z(R\otimes_\mathbb{Q} U(\mathfrak{g}_\mathbb{Q}))\simeq S\otimes_R R \otimes_\mathbb{Q} Z(U(\mathfrak{g}_\mathbb{Q}))\simeq Z(S\otimes_\mathbb{Q} U(\mathfrak{g}_\mathbb{Q}))\simeq Z(U(\mathfrak{g}_S)).
	\end{align*}
	The result for $R=\mathbb{Q}$ is straightforward. The inclusion $R\otimes_\mathbb{Q} Z(\mathfrak{g}_\mathbb{Q})\subset Z(R\otimes_\mathbb{Q} \mathfrak{g}_\mathbb{Q})$ is clear.  Since any element of $Z(R\otimes_\mathbb{Q} U(\mathfrak{g}_\mathbb{Q}))$ can be written in the form $\sum_{i\in F} r_i\otimes a_i$ with $\{r_i\colon i\in F \}$ being a linearly independent set over $\mathbb{Q}$, it follows that all $a_i$ belong to $Z(\mathfrak{g}_\mathbb{Q})$.
\end{proof}

In light of Lemma \ref{centerofLiealgebra}, the next natural question is to know what happens to the central characters under change of ring.

Let $\pi_R$ denote the projection $U(\mathfrak{g}_R)\twoheadrightarrow U(\mathfrak{h}_R)$ and $\l\in \mathfrak{h}_R^*$. 
By the PBW theorem, for each $\l\in \mathfrak{h}_\mathbb{R}^*$, and each commutative Noetherian ring $S$ which is an $R$-algebra, we obtain the commutative diagrams
\begin{equation}
	\begin{tikzcd}
		S\otimes_R U(\mathfrak{g}_R)\arrow[r, "S\otimes_R \pi_R", outer sep=0.75ex] \arrow[d, "\simeq"] & S\otimes_R U(\mathfrak{h}_R)\arrow[d, "\simeq"] & & S\otimes_R U(\mathfrak{h}_R)\arrow[r, "1_S\otimes_R \l"] \arrow[d, "\simeq"]& S\otimes_R R\arrow[d, "\simeq"]\\
		U(\mathfrak{g}_S)\arrow[r, "\pi_S"] & U(\mathfrak{h}_S) & & U(\mathfrak{h}_S)\arrow[r, "1_S\l"] & S
	\end{tikzcd}.\label{eqexemploseq130}
\end{equation}
Here, $1_S\l$ denotes the homomorphism of $R$-algebras given by $(1_S\l)(1_S\otimes_\mathbb{Z} h_\alpha)=1_S\l(1_R\otimes_\mathbb{Z} h_\alpha)\in S$ for each $\alpha\in \varPi$. In particular, $1_S\l\in \mathfrak{h}_S^*$.

By Lemma \ref{centerofLiealgebra}, and combining all these diagrams we obtain the following commutative diagrams
\begin{equation}
	\begin{tikzcd}
		S\otimes_R Z(\mathfrak{g}_R)\arrow[r] \arrow[d, "\simeq"] & S\otimes_R U(\mathfrak{h}_R)\arrow[d, "\simeq"] & & S\otimes_R Z(\mathfrak{g}_R)\arrow[r, "S\otimes_R \chi_\l", outer sep=0.75ex] \arrow[d, "\simeq"]& S\otimes_R R\arrow[d, "\simeq"]\\
		Z(\mathfrak{g}_S)\arrow[r] & U(\mathfrak{h}_S) & & Z(\mathfrak{h}_S) \arrow[r, "\chi_{1_S\l}"] & S
	\end{tikzcd}.\label{eqexemplos131}
\end{equation}

If $I$ is an ideal of $R$, there is one more commutative diagram of interest:
\begin{equation}
	\begin{tikzcd}
		R/I\otimes_R U(\mathfrak{h}_R)\arrow[r, "1_{R/I}\otimes_R \l"] \arrow[d, "\simeq"]& R/I\otimes_R R\arrow[d, "\simeq"]\\
		U(\mathfrak{h}_R)/IU(\mathfrak{h}_R)\arrow[r] & R/I
	\end{tikzcd},
\end{equation} where the bottom map is given by $1_R\otimes_\mathbb{Z} h_\alpha+IU(\mathfrak{h}_R)\mapsto \l(1_R\otimes_\mathbb{Z} h_\alpha)+I$, $\alpha\in \varPi$. In other words, this is the image of $\l\in \mathfrak{h}_R^*$ in $\mathfrak{h}_R^*/I\mathfrak{h}_R^*$.

\subsubsection{Decomposition of $\mathcal{O}_{[\l], R}$ into blocks}

Assume in the remainder of this section that $R$ is a local commutative Noetherian ring which is a $\mathbb{Q}$-algebra.
To simplify notation, we shall denote by $\overline{\l}$ the image of $\l\in \mathfrak{h}_R^*$ in $\mathfrak{h}_R^*/\mi \mathfrak{h}_R^*$, where $\mi$ is the maximal ideal of the local ring $R$, and denote by $\overline{r}$ the image of $r\in R$ in the quotient $R/\mi$. We will also denote by $\overline{z}$ the image of $z\in Z(\mathfrak{g}_R)$ in $Z(\mathfrak{g}_R)/\mi Z(\mathfrak{g}_R)$. Recall that $W$ is the Weyl group associated with the root system $\Phi$. Explicitly, each reflection $s_\alpha$ acts in the following way: $s_\alpha \l=\l-\langle \l, \alpha^\vee\rangle_R \alpha$, where $1_R\alpha$ can be seen as the element in $ \mathfrak{h}_R^*$ satisfying $1_R\alpha(1\otimes h_\alpha)=2$.  So, the Weyl group $W$ acts on $\mathfrak{h}_{R(\mi)}^*\simeq R(\mi)\otimes_R \mathfrak{h}_R^*\simeq \mathfrak{h}_R^*/\mi\mathfrak{h}_R^*$. In particular, for any $w\in W$ and $\l\in \mathfrak{h}_R^*$ we have $\overline{w\cdot \l}=w\cdot \overline{\l}$, under the dot action (see  \citep[1.8.1]{zbMATH03747378}).

\begin{Def}
	Let $K$ be a field of characteristic zero. 	For each $\mu\in \mathfrak{h}_K^*$, we can consider the root system making the weight $\mu$ integral, that is, $\Phi_{\mu}:=\{\alpha\in \Phi\colon \langle \mu, \alpha^\vee\rangle_{K}\in \mathbb{Z} \}$ and its associated Weyl group \mbox{$W_{\mu}:=\{ w\in W\colon w\cdot \mu-\mu\in \L_K\}$.}
\end{Def}

\begin{Def}
	Let $\l\in \mathfrak{h}_R^*$. We call $\mathcal{D}\subset [\l]$ a \textbf{block} of $[\l]$ if $\{\overline{\mu}\colon \mu\in  \mathcal{D} \}$ is an orbit under the dot action of the Weyl group $W$.
\end{Def}

\begin{Remark}
	\begin{enumerate}
		\item Every orbit under the Weyl group is a finite set, so a block is always finite.
		\item If $\mu_1, \mu_2\in \mathcal{D}$, then $\mu_1-\mu_2\in \L_R$ and since all non-zero integers are invertible in $R$, we also obtain \mbox{$\overline{\mu_1}-\overline{\mu_2}\in \L_{R(\mi)}$.} Further, $\{\overline{\mu}\colon \mu\in \mathcal{D} \}$ is a $W$-orbit and also an orbit under the subgroup $W_{\overline{\mu_1}}$.\qedhere
	\end{enumerate}
\end{Remark}

\begin{Lemma}\label{decompositionblocksweights}
	Let $\l\in \mathfrak{h}_R^*$ and let $\mathcal{D}\subset [\l]$ be a block. Then there exists $\mu\in \mathfrak{h}_R^*$, $\nu\in \mi \mathfrak{h}_R^*$ satisfying
	\begin{enumerate}[(i)]
		\item $s_\alpha \mu-\mu\in \mathbb{Z} \alpha$ for all $\alpha\in \Phi_{\overline{\mu}}$.
		\item $\mathcal{D}=W_{\overline{\mu}}\cdot\mu +\nu$.
	\end{enumerate}
\end{Lemma}
\begin{proof}
	See  \citep[1.8.2]{zbMATH03747378}.
\end{proof}

It follows that $[\l]$ is a disjoint union of its blocks.
Now knowing how to decompose $[\l]$, we shall proceed to decompose $\mathcal{O}_{[\l], R}$. The idea is analogous to the argument showing that any module decomposes into its weight modules as $U(\mathfrak{h}_R)$-modules using now the commutative algebra $Z(\mathfrak{g}_R)$. 

\begin{Lemma} Let $\l\in \mathfrak{h}_R^*$.\label{kernelcentralcharactermaximal}
If $\mu, \omega\in [\l]$ belong to distinct blocks, then \mbox{$\ker \chi_\omega+\ker \chi_\mu=Z(\mathfrak{g}_R)$.}
\end{Lemma}
\begin{proof}
See \citep[1.8.3]{zbMATH03747378}.
\end{proof}

As a consequence, it follows that all central characters which are non-zero $\chi_\mu\colon Z(\mathfrak{g}_R)\rightarrow R$ are surjective for $\mu\in [\l]$ since we can always find a weight which belongs to a different block than the one that contains $\mu$.

\begin{Lemma}\label{centralcharactersameblock}
	Let $\l\in \mathfrak{h}_R^*$. Suppose that $\mu, \omega\in [\l]$ belong to the same block. Then 	$\ker \chi_\omega=\ker \chi_\mu$.
\end{Lemma}
\begin{proof}
	Let $\mathcal{D}=W_{\overline{\mu}}\cdot\mu+\nu$ be a block of $[\l]$. By the commutative diagram (\ref{eqexemploseq130}) and Theorem \ref{Propertiesofsemisimplecomplexliealg}, the surjective map $\chi_{w\cdot\mu-\mu}=\chi_{w\cdot \mu+\nu}-\chi_{\mu+\nu}$ becomes zero under $R(\mi)\otimes_R -$ for any $w\in W_{\overline{\mu}}$. So, the image of $\chi_{w\cdot\mu-\mu}$ is contained in $\mi$ and the central character is not surjective. But, this can only happen if $\chi_{w\cdot \mu-\mu}$ is the zero map. Now, assume that $x\in \ker \chi_{\mu+\nu}.$ Then for every $w\in W_{\overline{\mu}}$,
	\begin{align}
		\chi_{w\cdot \mu +\nu}(x)=w\cdot\mu(\pi_R(x))+\nu(\pi_R(x))=w\cdot \mu(\pi_R(x))-\mu(\pi_R(x))=\chi_{w\cdot\mu-\mu}(x)=0.
	\end{align} So, $x$ is also an element of $\ker 	\chi_{w\cdot \mu +\nu}$. 
\end{proof}

\begin{Prop}
	Let $R$ be a local commutative Noetherian ring which is a $\mathbb{Q}$-algebra. Let $\l\in \mathfrak{h}_R^*$. 
	For every $M\in \mathcal{O}_{[\l], R}$, $M=\bigoplus_{\mathcal{D}} M^{\mathcal{D}}$, where $\mathcal{D}$ runs over all blocks of $[\l]$ and
	$$M^{\mathcal{D}}:=\{m\in M\colon \forall x\in \ker \chi_\mu, \ \mu\in \mathcal{D}, \ \exists n\in \mathbb{N} \ x^nm=0  \}. $$
	Moreover, the following assertions hold:
	\begin{enumerate}[(a)]
		\item $M^{\mathcal{D}}$ is non-zero only for a finite number of blocks $\mathcal{D}$ of $[\l]$.
		\item $\St(\mu)^\mathcal{D}=\St(\mu)$ if and only if $\mu\in \mathcal{D}$, otherwise it is zero.
		\item $M\mapsto M^{\mathcal{D}}$ is an exact endofunctor on $\mathcal{O}_{[\l], R}$.
	\end{enumerate}
\end{Prop}
\begin{proof}The result follows from Lemma \ref{propertiesoffamilyidealslikeidempotents} together with Lemma \ref{kernelcentralcharactermaximal} and \ref{centralcharactersameblock}.
 (b) follows by Lemma \ref{centralcharactersameblock} observing what is the action of PBW monomials on the generator of $\St(\mu)$ and the exactness in (c).  See also \citep[1.8.4, 1.8.5, 1.8.6]{zbMATH03747378}.
\end{proof}

\begin{Def}
	Let $R$ be a local commutative Noetherian ring which is a $\mathbb{Q}$-algebra. Let $\l\in \mathfrak{h}_R^*$. For a block $\mathcal{D}\subset [\l]$, define $\mathcal{O}_{\mathcal{D}}$ the full subcategory of $\mathcal{O}_{[\l], R}$ whose objects satisfy $M=M^{\mathcal{D}}$.
\end{Def}

In the classical case, the blocks of the category $\mathcal{O}$ are in a one-to-one correspondence with the antidominant weights. We can generalize the notion of dominant and antidominant weight to this setup since these notions will help us study the structure of the category $\mathcal{O}_{[\l], R}$.
We will call $\mu\in \mathfrak{h}_R^*$ a \textbf{dominant weight} if $\overline{\mu}$ is a dominant weight. Analogously, we will call $\mu$ an \textbf{antidominant weight} if $\overline{\mu}$ is an antidominant weight. We call $\mu\in \mathfrak{h}_R^*$ an \textbf{integral dominant weight} if $\overline{\mu}$ is an integral dominant weight.

We should remark that the blocks $\mathcal{D}$ of $[\l]$ are constructed with the ring $R$ in mind. So, after change of rings these blocks can be refined even further. Moreover, the interested reader can see that typically the Weyl groups associated with elements $\overline{1_S\otimes \mu}$ are subgroups of the Weyl groups associated with elements $\overline{\mu}$. This is the phenomenon that we will exploit in this deformation of the category $\mathcal{O}$, although we will not explore it under the current formulation.

\begin{Lemma}Let $\l\in \mathfrak{h}_R^*$.
	Let $\mathcal{D}$ be a block of $[\l]$. Then $\mathcal{D}$ admits a unique (resp. antidominant) dominant weight $\mu$. In addition, $\mu$ is (resp. minimal) maximal in $\mathcal{D}$. 
\end{Lemma}
\begin{proof}Since every non-zero integer is invertible in $R$, the uniqueness and existence of dominant weights in $W_{\overline{\theta}}\overline{\theta}$ implies the uniqueness and existence of dominant weights in  $W_{\overline{\theta}}\theta +\nu$ for $\nu\in \mi\mathfrak{h}_R^*$ and $\theta\in \mathfrak{h}_R^*$. The result now follows from Lemma \ref{decompositionblocksweights}.
\end{proof}

It is a natural question to know whether extension of scalars $S\otimes_R -$ preserves dominant (resp. antidominant) weights.

\begin{Lemma}
	Let $R$ be a local Noetherian integral domain which is a $\mathbb{Q}$-algebra. Assume that $S$ is isomorphic to one of the following rings:
	\begin{itemize}
		\item a localization $R_\pri$ of $R$ at some prime ideal $\pri$ of $R$;
		\item a quotient ring $R/I$ of $R$ for some ideal $I$.
	\end{itemize}
	If $\l\in \mathfrak{h}_R^*$ is a dominant weight, then $1_S\otimes_R \l$ is a dominant weight. If $\l\in \mathfrak{h}_R^*$ is an antidominant weight, then $1_S\otimes_R \l$ is an antidominant weight.
\end{Lemma}
\begin{proof}
	We will prove the assertion for dominant weights. The other case is analogous.
	By assumption, $\overline{\l}$ is a dominant weight. That is, $\langle \overline{\l}+1_{R(\mi)}\rho, \alpha^\vee\rangle_{R(\mi)} \notin \mathbb{Z}^-$ for every $\alpha\in \Phi^+$. So, $\langle \l+\rho, \alpha^\vee \rangle_R \notin \mathbb{Z}^- +\mi$ for every $\alpha\in \Phi^+$. Assume that $S=R_\pri$ for some prime ideal $\pri$ of $R$ and assume, by contradiction, that $1_{R_\pri}\otimes \l$ is not a dominant weight. Hence, there exists $\alpha\in \Phi^+$ such that $\langle 1_S\otimes \l + 1_S\rho , 1_S\alpha^\vee\rangle_S \in \mathbb{Z}^-+\pri_\pri. $ Further, there exists $t\in \mathbb{Z}^-$ and $s\in R\backslash \pri$ so that $s(\langle \l+\rho, \alpha^\vee\rangle_R -t )\in \pri$. Thus, $\langle \l+\rho, \alpha^\vee\rangle_R -t\in \pri\subset \mi$ for some $\alpha\in \Phi^+$. The existence of such $t$ contradicts $\overline{\l}$ being a dominant weight. So, $1_S\otimes_R \l$ is a dominant weight. 
	
	Assume now that $S=R/I$ for some ideal $I$. In particular, $\mi/I$ is the unique maximal ideal of $S$. Assume, by contradiction, that $1_S\otimes_R \l$ is not a dominant weight. Then there exists $\alpha\in \Phi^+$, $t\in \mathbb{Z}^-$, $x\in \mi$ so that $\langle \l+\rho, \alpha^\vee\rangle_R -t -x\in I\subset \mi$. Hence, $\langle \l+\rho, \alpha^\vee\rangle_R -t\in \mi$ which contradicts $\l$ being a dominant weight.
\end{proof}

We can now see that there are no homomorphisms between Verma modules that belong to distinct blocks.

\begin{Lemma}	Let $R$ be a local commutative Noetherian ring which is a $\mathbb{Q}$-algebra. Let $\l\in \mathfrak{h}_R^*$. 
	Then $\Hom_{\mathcal{O}_{[\l], R}}(M, N)=0$ if $M\in \mathcal{O}_{\mathcal{D}_1}\cap \mathcal{F}(\St_{[\l]})$ and $N\in \mathcal{O}_{\mathcal{D}_2}\cap \mathcal{F}(\St_{[\l]})$ for distinct blocks $\mathcal{D}_1\neq \mathcal{D}_2$ of $[\l]$.\label{Homomorphismbetweenblocks}
\end{Lemma}
\begin{proof}
	Assume that $\mathcal{D}_1=W_{\overline{\mu}}\cdot \mu+\nu_1$ and $\mathcal{D}_2=W_{\overline{\omega}}\cdot \omega +\nu_2$. As usual, we will start with the Verma modules. Suppose that $\Hom_{\mathcal{O}_{[\l], R}}(\St(w_1\cdot \mu+\nu_1), \St(w_2\cdot \omega +\nu_2))\neq 0$ for $w_1\in W_{\overline{\mu}}$, $w_2\in W_{\overline{\omega}}$. Then
	\begin{align}
		0&\neq \Hom_{U(\mathfrak{g}_R)}(U(\mathfrak{g}_R)\otimes_{U(\mathfrak{b}_R)} R_{w_1\cdot \mu+\nu_1}, \St(w_2\cdot \omega+\nu_2))\\&\simeq \Hom_{U(\mathfrak{b}_R)}(R_{w_1\cdot\mu +\nu_1}, \St(w_2\cdot \omega +\nu_2)) \subset \St(w_2\cdot\omega +\nu_2)_{w_1\cdot\mu+\nu_1}.
	\end{align}
	It follows that $w_2\cdot\omega +\nu_2-w_1\cdot\mu-\nu_1\in \mathbb{Z}_0^+\varPi$. By Lemma \ref{changeringscategoryO}, we obtain that $\St(w_2\cdot \overline{\omega})_{w_1\cdot \overline{\mu}}\neq 0$. So, also $\St(w_2\cdot \omega)_{w_1\cdot \mu}(\mi)\simeq \St(w_2\cdot \overline{\omega})_{w_1\cdot\overline{\mu}}\neq 0$. Since $\St(w_2\cdot \omega)_{w_1\cdot \mu}\in R\proj$ it must be non-zero. Hence, $w_2\cdot \omega -w_1\cdot \mu \in \mathbb{Z}_0^+\varPi$. Therefore, $\nu_2-\nu_1\in \mathbb{Z}_0^+\varPi$. But, this forces $\nu_2=\nu_1$ since all non-zero integers are invertible in $R$. By assumption, 
	$\St(w_2\omega +\nu_2)_{w_1\cdot \mu+\nu_2}\neq 0$ and it contains an element which is killed by $\mathfrak{n}_R^+$. Since this module is free, one of its elements basis of the form $x_{-\alpha_1}^{t_1}\cdots x_{-\alpha_d}^{t_d}(1_{U(\mathfrak{g}_R)\otimes_{U(\mathfrak{b}_R)}}1_{R_{w_2\cdot \omega +\nu_2}})$, $\alpha_i\in \varPi$ which does not belong to $\mi \St(w_2\cdot \omega+\nu_2)$ is killed by $\mathfrak{n}_R^+$. Therefore, there exists a non-zero map between $\St(w_1\cdot \overline{\mu})$ and $\St(w_2\cdot\overline{\omega})$.
	Therefore, both Verma modules $\St(w_1\cdot\overline{\mu})$ and $\St(w_2\cdot\overline{\omega})$ have a common simple module as composition factor, and so they belong to the same block.  By Theorem \ref{Propertiesofsemisimplecomplexliealg} (f), $W_{\overline{\mu}}=W_{\overline{\omega}}$, and $\overline{\omega}=w\cdot \overline{\mu}$ for some $w\in W_{\overline{\mu}}$. Hence, $\omega-w\cdot\mu =\nu$ for some $\nu\in \mi \mathfrak{h}_R^*$. As we have seen, $w_2\cdot\omega-w_1\cdot \mu=w_2w\cdot \mu +w_2\cdot \nu -w_1\cdot \mu\in \mathbb{Z}_0^+\varPi$. So, also $w_2\cdot \nu \in \mathbb{Z}_0^+\varPi$. Therefore, $\nu=0$. This shows that the blocks $\mathcal{D}_1$ and $\mathcal{D}_2$ coincide.

	Now, the claim follows using the (finite) filtrations of $M$ and $N$ by quotients of Verma modules. 
\end{proof}

\subsubsection{Projective objects in $\mathcal{O}_{[\l], R}$}

At this point, it is difficult to know whether there is information getting out of the blocks of $\mathcal{O}_{[\l], R}$, that is, if there are non-zero homomorphisms between modules belonging to distinct blocks. For modules with a Verma filtration, we saw that  such a situation is not possible. But, since we do not know if this is the case for general modules, the classical arguments for the construction of projective objects do not carry over to this more general setup. In particular, not knowing if the previous situation might or might not happen makes it difficult to deduce whether the Verma module associated with a dominant weight is a projective object or not. Instead, we will take the advantage of knowing projective objects in $\mathcal{O}_{[\l],(II),  R}$ to construct projective objects in $\mathcal{O}_{[\l], R}\cap R\Proj$, using  change of rings techniques. 

For each $\mu\in [\l]$, define $Q(\mu):=U(\mathfrak{g}_R)\otimes_{U(\mathfrak{h}_R)} R_\mu\in \mathcal{O}_{[\l], (II), R}$. These modules are a sort of linearisation of the projective module $U(\mathfrak{g}_R)$. Note that by the PBW theorem, \begin{align*}
	Q(\mu)\simeq U(\mathfrak{n}_R^-)\otimes_R U(\mathfrak{h}_R)\otimes_R U(\mathfrak{n}_R^+)\otimes_{U(\mathfrak{h}_R)} R_\l\simeq U(\mathfrak{n}_R^+)\otimes_R U(\mathfrak{n}_R^-)\in R\Proj.
\end{align*} Also, for every commutative $R$-algebra which is a Noetherian ring, $S\otimes_R Q(\mu)\simeq Q(1_S\otimes_R \mu)$. The main difference between $Q(\mu)$ and the Verma modules is that $Q(\mu)$ is not annihilated by $\mathfrak{n}_R^+$. But, this feature allows $Q(\mu)$ to detect more information outside $\mathcal{O}$. For instance, for every $M\in U(\mathfrak{g}_R)\M$, $\Hom_{U(\mathfrak{g}_R)}(Q(\mu), M)\simeq M_\mu$, and thus the functor $\Hom_{U(\mathfrak{g}_R)}(Q(\mu), -)\colon \mathcal{O}_{[\l], (II), R}\rightarrow  \mathcal{O}_{[\l], (II), R}$ is exactly the functor $M\mapsto M_\mu$ which is exact by Corollary \ref{corpropertiesOtypeII}. Therefore, $Q(\mu)$ is a projective object in $\mathcal{O}_{[\l], (II), R}$. As Gabber and Joseph pointed out, $\mathcal{O}_{[\l], (II), R}$ is closed under arbitrary direct sums, hence each weight module $M_\mu$ is a quotient of an arbitrary direct sum of copies of $Q(\mu)$, and so each $M\in \mathcal{O}_{[\l], (II), R}$ is a quotient of an arbitrary direct sum of copies of modules of the form $Q(\mu)$, where $\mu$ runs over all weights in $[\l]$. Hence, $\mathcal{O}_{[\l], (II), R}$ has enough projectives, and so we can use homological algebra techniques on $\mathcal{O}_{[\l], (II), R}$.  We obtained so far, the following:

\begin{Lemma}
	Let $R$ be a local commutative Noetherian ring which is a $\mathbb{Q}$-algebra. Let $\l\in \mathfrak{h}_R^*$. The following assertions hold.
	\begin{enumerate}[(a)]
		\item The modules $Q(\mu)=U(\mathfrak{g}_R)\otimes_{U(\mathfrak{h}_R)}R_\mu \in \mathcal{O}_{[\l], (II), R}\cap R\Proj$ are projective objects in $\mathcal{O}_{[\l], (II), R}$. 
		\item The module $\bigoplus_{\mu\in [\l]} Q(\mu)$ is a projective generator of $\mathcal{O}_{[\l], (II), R}$ and $\mathcal{O}_{[\l], (II), R}$ has enough projectives.
		\item For every commutative $R$-algebra $S$ which is a Noetherian ring, $S\otimes_R Q(\mu)\simeq Q(1_S\otimes_R \mu)$, $\mu\in [\l]$.
	\end{enumerate}
\end{Lemma}
Now, observe the following:
given an exact sequence $0\rightarrow Q\rightarrow X\rightarrow P\rightarrow 0\in\Ext_{\mathcal{O}_{[\l], (II), R}}^1(P, Q),$ if $P, Q\in \mathcal{O}_{[\l], R}$, then also $X\in U(\mathfrak{g}_R)\m$ and some power of $\mathfrak{n}_R^+$ annihilates $X$. Hence, $X\in \mathcal{O}_{[\l], R}$. Since $\mathcal{O}_{[\l], R}$ is a full subcategory of $\mathcal{O}_{[\l], (II), R}$, such exact sequence is an exact sequence in $\mathcal{O}_{[\l], R}$.

\begin{Lemma}\label{transicaoprojectivesO}
	Let $R$ be a local regular commutative Noetherian ring which is a $\mathbb{Q}$-algebra with unique maximal ideal $\mi$. Let $\l\in \mathfrak{h}_R^*$. The following assertions hold.\begin{enumerate}[(a)]
		\item For each $P\in \mathcal{O}_{[\l], (I),  R}\cap R\Proj$, $\Ext_{\mathcal{O}_{[\l], (II), R}}^1(P, X)=0$ for every $X\in \mathcal{O}_{[\l], (I), R}\cap R\Proj$ if and only if $P$ is a projective object in $\mathcal{O}_{[\l], (I), R}\cap R\Proj$.
		\item For each $P\in \mathcal{O}_{[\l], R}\cap R\Proj$, $\Ext_{\mathcal{O}_{[\l], (II), R}}^1(P, X)=0$ for every $X\in \mathcal{O}_{[\l], R}\cap R\Proj$ if and only if $P$ is a projective object in $\mathcal{O}_{[\l], R}\cap R\Proj$.
		\item If $P\in \mathcal{O}_{[\l], R}\cap R\Proj$ so that $P(\mi)$ is a projective object in $\mathcal{O}_{[\overline{\l}], R(\mi)}$, then $P$ is a projective object in $\mathcal{O}_{[\l], R}\cap R\Proj$.
	\end{enumerate}
\end{Lemma}
\begin{proof}
	The assertions (a) and (b) follow immediately from the above discussion. 
	To prove (c) we want to apply (b) together with Corollary \ref{Kunnethdeformationresult}. In order to do that, we have to proceed by induction on the Krull dimension of $R$. If $\dim R$ is zero, then $R=R(\mi)$ and there is nothing to prove.
	Let $x\in \mi/\mi^2$. Fix $S=R/Rx$, so $\dim S=\dim R-1$ and for any module $X\in U(\mathfrak{g}_R)\M$, we can write $X(\mi)\simeq (S\otimes_R X)(\mi_S)$, where $\mi_S$ denotes $\mi/Rx$. Hence, by assumption, $S\otimes_R P\in \mathcal{O}_{[1_S\otimes_R \l], S}\cap S\Proj$ so that $P(\mi_S)\simeq P(\mi)$ is a projective object in $\mathcal{O}_{[\overline{\l}, S(\mi_S)]}$, where $S(\mi_S)=R(\mi)$. By induction, $S\otimes_R P$ is a projective object in $\mathcal{O}_{[1_S\otimes_R \l], S}\cap S\Proj$.
	
	Consider a projective resolution of $P$ by direct sums of $\bigoplus_{\mu\in [\l]} Q(\mu)$  and denote the respective deleted projective resolution by $Q^{\bullet}$. Since $P\in R\Proj$, each $Q(\mu)\in R\Proj$ and the tensor product commutes with arbitrary direct sums we obtain that $S\otimes_R Q^{\bullet}$ is a deleted projective resolution of $S\otimes_R P$ in $\mathcal{O}_{[1_S\otimes_R \l], (II), S}$. Now, for each $X\in \mathcal{O}_{[\l], R}\cap R\Proj$, \begin{align}
		\Hom_{\mathcal{O}_{[\l], (II), R}}(\bigoplus_{\mu\in [\l]} Q(\mu), X)\simeq \prod_{\mu\in [\l]} \Hom_{\mathcal{O}_{[\l], (II), R}} ( Q(\mu), X)\simeq \prod_{\mu\in [\l]} X_\mu.
	\end{align}Each $X_\mu$ is a flat module, and since every Noetherian ring is coherent, so the arbitrary direct product of flat modules is flat. Hence, the complex $\Hom_{\mathcal{O}_{[\l], (II), R}}(Q^{\bullet}, X)$ satisfies the hypothesis of Corollary \ref{Kunnethdeformationresult}. Further, 
	since $R$ is Noetherian and applying Lemma \ref{changeringscategoryO} we obtain
	\begin{align}
		S\otimes_R \Hom_{\mathcal{O}_{[\l], (II), R}}(\bigoplus_{\mu\in [\l]} Q(\mu), X) \simeq S\otimes_R \prod_{\mu\in [\l]} X_\mu\simeq \prod_{\mu\in [\l]} S\otimes_R X_\mu \simeq \prod_{\mu\in [\l]} X_{1_S\otimes_R \mu} \\ \simeq \Hom_{\mathcal{O}_{[1_S\otimes_R \l], (II), S}}(\bigoplus_{{1_S\otimes_R \mu}\in [1_S\otimes_R \l]} Q(1_S\otimes_R \mu),S\otimes_R X).
	\end{align}
	Therefore, for each integer $i>0$,
	\begin{align*}
		H^i(\Hom_{\mathcal{O}_{[\l], (II), R}}(Q^{\bullet}, X))&=\Ext^i_{\mathcal{O}_{[\l], (II), R}}(P, X), \\ H^i(S\otimes_R \Hom_{\mathcal{O}_{[\l], (II), R}}(Q^{\bullet}, X))&=\Ext^i_{\mathcal{O}_{[1_S\otimes_R \l], (II), S}}(S\otimes_R P, S\otimes_R X).
	\end{align*}
	By Corollary \ref{Kunnethdeformationresult} and $S\otimes_R P$ being projective in $\mathcal{O}_{[1_S\otimes_R \l], (II), S}$ we obtain that \linebreak${\Ext^1_{\mathcal{O}_{[\l], (II), R}}(P, X)\otimes_R R/Rx}=0$. Using the surjective map $R/Rx\rightarrow R/\mi$ we obtain that \linebreak${\Ext^1_{\mathcal{O}_{[\l], (II), R}}(P, X)\otimes_R R/\mi}=0$. Observe that $P$ is finitely generated as $U(\mathfrak{g}_R)$-module (for which such generator set can be chosen to be a set of weight vectors). Hence, we can choose only a finite set of weights $F$ so that $\bigoplus_{\mu\in F}Q(\mu)\rightarrow P$ is a surjective map. Since $R$ is Noetherian and each weight module of $X$ is finitely generated as $R$-module, $\Ext^1_{\mathcal{O}_{[\l], (II), R}}(P, X)$ is a quotient of a finitely generated $R$-module, and so it is finitely generated. By Nakayama's Lemma, $\Ext^1_{\mathcal{O}_{[\l], (II), R}}(P, X)=0$. By (b), $P$ is a projective object in $\mathcal{O}_{[\l], R}\cap R\Proj$.
\end{proof}

The construction of projective objects in $\mathcal{O}$ is based on tensoring Verma projective modules with simple modules of finite vector space dimension. These are the simple modules indexed by an integral dominant weight. Their deformations in $\mathcal{O}_{[\l], R}$ are similarly obtained. We are expecting them to be free as $U(\mathfrak{n}_R^-)$-modules and consequently free as $R$-modules. So, the modules taking the place of simple modules should be free over $R$. For deformations, these modules are not the simple modules in $\mathcal{O}_{[\l], R}.$  The reason for this is that Gabber and Joseph showed that the simple modules in $\mathcal{O}_{[\l], R}$, where $R$ is a local commutative Noetherian $\mathbb{Q}$-algebra, are of the form $\St(\mu)/N$, where $\mi \St(\mu)\subset N$, $\mu\in [\l]$ and $\mi$ is the unique maximal ideal of $R$. Thus, $\mi L(\mu)=0$, and so $L(\mu)$ would be free over the ground ring if and only if $\mi N=\mi \St(\mu)$. The latter condition is something that we just do not know at this point. So, we must consider a different approach. 

As we discussed we will try to obtain an integral version of the simple modules indexed by integral dominant weights. Let $\mu\in \mathfrak{h}_R^*$ be an integral dominant weight. We define $J_R$ be the left ideal of $U(\mathfrak{g}_R)$ generated by the set of elements
\begin{align}
	\{x_\alpha\colon \alpha\in \Phi^+ \}\cup \{h_\alpha-\mu(h_\alpha)1_R\colon \alpha\in \varPi \}\cup \{x_{-\alpha}^{n_\alpha+1}\colon \alpha\in \varPi \},
\end{align} where $n_\alpha=\langle \overline{\mu}, \alpha^\vee\rangle_{R(\mi)}\in \mathbb{Z}_0^+$.

By the PBW theorem, the monomials generated by this set of elements are linearly independent and also PBW monomials making $J_R$ a free $R$-module. Moreover, the basis of $J_R$  can be extended to a basis of $U(\mathfrak{g}_R)$, so the canonical inclusion of $J_R$ into $U(\mathfrak{g}_R)$ is an $(U(\mathfrak{g}_R), R)$-monomorphism. Also for any commutative $R$-algebra $S$, $S\otimes_R J_R$ is isomorphic to $J_S$. Let $E(\mu)$ denote the quotient $U(\mathfrak{g}_R)/J_R$. Since $0\rightarrow J_R \rightarrow U(\mathfrak{g}_R)\rightarrow E(\mu)\rightarrow 0$ remains exact under $R(\mi)\otimes_R - $ and $U(\mathfrak{g}_R)$ is free over $R$ we obtain $\Tor_1^R(E(\mu), R(\mi))=0$. Further, $S\otimes_R E(\mu)\simeq E(1_S\otimes_R \mu)$ for every commutative $R$-algebra $S$ which is a Noetherian ring making $1_S\otimes_R \mu$ an integral dominant weight in $\mathfrak{h}_S^*$. We can also see that $E(\mu)$ is a quotient of $\St(\mu)$. Therefore, $E(\mu)\in \mathcal{O}_{[\mu], R}$. In addition, $R(\mi)\otimes_R E(\mu)\simeq L(1\otimes_R \mu)$ and $1\otimes_R \mu$ is an integral dominant weight in $\mathfrak{h}^*_{R(\mi)}$. Therefore, it is finite-dimensional. By Nakayama's Lemma, $E(\mu)$ is finitely generated over $R$. Hence, $E(\mu)$ is free over $R$ with finite rank (see for example \citep[Lemma 8.53g Theorem 4.58]{Rotman2009a}). Observe that for each $n\in \mathbb{N}$, by Lemma \ref{changeringscategoryO}, the weights of $E(n\rho)$ are weights ranging from $-n\rho$ to $n\rho$. Moreover, the weight modules associated with $-n\rho$ and $n\rho$ are free with rank one.

For each $\mu\in \mathfrak{h}_R^*$, if $\overline{\mu}$ is not a dominant weight, then there is some $\alpha\in \Phi^+$ so that $\langle \overline{\mu}+\rho, \alpha^\vee\rangle_{R(\mi)} \in \mathbb{Z}^-$. Since $\langle \rho, \alpha^\vee\rangle_{R(\mi)}=1$, there exists $n\in \mathbb{N}$ so that $\overline{\mu}+n\rho$ is a dominant weight.  

\begin{Def}
	Let $\l\in \mathfrak{h}_R^*$ and $\mu\in [\l]$. If $\mu$ is a dominant weight, define $P(\mu):=\St(\mu)$. Otherwise, pick $n\in \mathbb{N}$ minimal so that $\mu+n\rho\in [\l]$ is a dominant weight and define $P(\mu):=(\St(\mu+n\rho)\otimes_R E(n\rho))^{\mathcal{D}_\mu}$, where $\mathcal{D}_\mu$ is the block of $[\l]$ that contains $\mu$.
\end{Def}

\begin{Theorem}
	Let $R$ be a local regular commutative Noetherian ring which is a $\mathbb{Q}$-algebra with unique maximal ideal $\mi$. Let $\l\in \mathfrak{h}_R^*$. 
	The following assertions hold. 
	\begin{enumerate}[(a)]
		\item If $\mu\in [\l]$ is a dominant weight, then $\St(\mu)$ is projective in $\mathcal{O}_{[\l], R}\cap R\Proj$.
		\item The modules $P(\mu)\in \mathcal{O}_{\mathcal{D}_\mu}$ are projective objects in $\mathcal{O}_{[\l], R}\cap R\Proj$, where $\mathcal{D}_\mu$ is the block of $[\l]$ that contains $\mu$.
		\item For each $\mu\in [\l]$, there exists an exact sequence in $\mathcal{O}_{\mathcal{D}_\mu}$ \begin{align}
			0\rightarrow C(\mu)\rightarrow P(\mu)\rightarrow \St(\mu)\rightarrow 0,
		\end{align}where $C(\mu)\in \mathcal{F}(\St_{[\l]_{\omega>\mu}})$ and $\mathcal{D}_\mu$ is the block that contains $\mu$.
		\item Fix $\displaystyle P=\bigoplus_{\mu\in [\l]} P(\mu)$. For each $Q\in \add_{\mathcal{O}_{[\l], R}} P$, $\Hom_{\mathcal{O}_{[\l], R}}(Q, M)\in R\proj$ for every $M\in \mathcal{F}(\St_{[\l]})$.
		\item Assume that $S$ is isomorphic to one of the following rings:	\begin{itemize}
			\item a localization $R_\pri$ of $R$ at some prime ideal $\pri$ of $R$.
			\item a quotient ring $R/I$ of $R$ for some ideal $I$.
		\end{itemize} Then for each $\omega\in [\l]$, and $M\in \mathcal{F}(\St_{[\l]})$, the canonical map $$S\otimes_R \Hom_{\mathcal{O}_{[\l]}, R}(P(\omega), M)\rightarrow \Hom_{\mathcal{O}_{[1_S\otimes_R \l]}, S}(S\otimes_R P(\omega), S\otimes_R M)$$ is an isomorphism.
	\end{enumerate}\label{changeofringsconstructionprojectives}
\end{Theorem}
\begin{proof}
	If $\mu$ is a dominant weight, then $\overline{\mu}$ is dominant. By Theorem \ref{Propertiesofsemisimplecomplexliealg}(d), $\St(\overline{\mu})$ is projective in $\mathcal{O}_{[\overline{\l}], R(\mi)}$. By Lemma \ref{transicaoprojectivesO}(b), $\St(\mu)$ is projective in $\mathcal{O}_{[\l], R}\cap R\Proj$.
	
	Since $E(n\rho)\in R\proj$ it is clear that $P(\mu)\in R\Proj$. By Theorem \ref{Propertiesofsemisimplecomplexliealg}(e), \linebreak${R(\mi)\otimes_R \St(\mu+n\rho)\otimes_R E(n\rho)}\simeq \St(\overline{\mu}+n1_{R(\mi)}\rho)\otimes_{R(\mi)} L(1_{R(\mi)}n\rho)$ is a projective object in $\mathcal{O}_{[\overline{\l}], R(\mi)}$. By Lemma \ref{transicaoprojectivesO}, $\St(\mu+n\rho)\otimes_R E(n\rho)$ is a projective object in $\mathcal{O}_{[\l], R}\cap R\Proj$. As $P(\mu)$ is a summand of $\St(\mu+n\rho)\otimes_R E(n\rho)$ it is also a projective object in $\mathcal{O}_{[\l], R}\cap R\Proj$ and also in $\mathcal{O}_{\mathcal{D}_\mu}\cap R\Proj$.
	
	If $\mu$ is dominant, the exact sequence on (c) is just the identity map on $\St(\mu)$.	Assume that $\mu$ is not dominant. By Proposition \ref{tensoringbyVermahasvermaflag}, $\St(\mu+n\rho)\otimes_R E(n\rho)\in \mathcal{F}(\{ \St(\mu+n\rho+\omega)\colon  \omega\in [\l], \ E(n\rho)_\omega\neq 0  \})$. So, the lowest weight in the filtration of $\St(\mu+n\rho)\otimes_R E(n\rho)$ which occurs only once is $\mu+n\rho - n\rho=\mu$ which again by Proposition \ref{tensoringbyVermahasvermaflag} appears at the top of the filtration. We obtained in this way an exact sequence \begin{align}
		0\rightarrow C_1(\mu)\rightarrow \St(\mu+n\rho)\otimes_R E(n\rho) \rightarrow \St(\mu)\rightarrow 0. \label{eqexce143}
	\end{align}The remaining weights are of the form $\mu+n\rho +\gamma$, where $\gamma\in \mathbb{N}\varPi$ is not smaller than $-n\rho$. Hence, $\mu+n\rho+\gamma-\mu\in \mathbb{N}\varPi$. So, all weights  of $C_1(\mu)$ are greater than $\mu$. Applying $\mathcal{D}_\mu$ to (\ref{eqexce143}) we obtain (c).
	
	By the above discussion for (b), for each $n\in \mathbb{N}$ and each $\omega\in [\l]$, the functor \linebreak${\Hom_{\mathcal{O}_{[\l], R}}(\St(\omega)\otimes_R E(n\rho), -)}$ is exact on $\mathcal{F}(\St_{[\l]})$. Therefore, $\Hom_{\mathcal{O}_{[\l], R}}(\St(\omega)\otimes_R E(n\rho), M)\in R\proj$ for every $M\in \mathcal{F}(\St_{[\l]})$ if and only if $\Hom_{\mathcal{O}_{[\l], R}}(\St(\omega)\otimes_R E(n\rho), \St(\mu))\in R\proj$ for every $\mu\in [\l]$. 
	By Lemma \ref{Homomorphismbetweenblocks},  \begin{align}
		\Hom_{\mathcal{O}_{[\l], R}}((\St(\omega)\otimes_R E(n\rho))^{\mathcal{D}}, \St(\mu)^{\mathcal{D}})\simeq \Hom_{\mathcal{O}_{[\l], R}}((\St(\omega)\otimes_R E(n\rho))^{\mathcal{D}}, \St(\mu)),
	\end{align}which is zero unless $\mu\in \mathcal{D}$.  Assuming that $\mu\in \mathcal{D}$, again by Lemma \ref{Homomorphismbetweenblocks},
	\begin{align*}
		\Hom_{\mathcal{O}_{[\l], R}}((\St(\omega)\otimes_R E(n\rho))^{\mathcal{D}}, \St(\mu))&\simeq \Hom_{\mathcal{O}_{[\l], R}}((\St(\omega)\otimes_R E(n\rho)), \St(\mu))\\&\simeq \Hom_{\mathcal{O}_{[\l], R}}(\St(\omega), E(n\rho)^*\otimes_R \St(\mu)).
	\end{align*}Since $\omega$ is dominant $\St(\omega)$ is a projective object in $\mathcal{F}(\St_{[\l]})$. Hence, if $\St(\l)$ appears as a factor in a Verma filtration of an arbitrary $M\in \mathcal{F}(\St_{[\l]})$, then we can assume that all its occurrences appear at the bottom of the filtration. Moreover, all its occurrences can be encoded in a direct sum of copies of $\St(\omega)$. Thanks to $\omega$ being dominant, by Lemma \ref{homomorphismbetweenvermamodules} and Lemma \ref{Homomorphismbetweenblocks}, homomorphisms from $\St(\omega)$ to another Verma module $\St(\omega_1)$ are only non-zero if $\omega=\omega_1$. 
	
	Therefore, if we fix $M:=E(n\rho)^*\otimes_R \St(\mu)\in \mathcal{F}(\St_{[\l]})$ by Proposition \ref{tensoringbyVermahasvermaflag}, we obtain
	\begin{align}
		\Hom_{\mathcal{O}_{[\l], R}}(\St(\omega), M)\simeq 	\Hom_{\mathcal{O}_{[\l], R}}(\St(\omega), \St(\omega)^j)\simeq R^j, \label{p58eq152}
	\end{align}where $j$ denotes the number of occurrences of $\St(\omega)$ in $M$.
	This shows that $\Hom_{\mathcal{O}_{[\l], R}}(P(\omega), \St(\mu))\in R\proj$ for all $\mu\in [\l]$. By induction on the size of filtrations by Verma modules, we obtain (d). Indeed, if $M\in \mathcal{F}(\St_{[\l]})$, then there exists an exact sequence $0\rightarrow M'\rightarrow M\rightarrow \St(\mu)\rightarrow 0$ for some $\mu\in [\l]$ and $M'\in \mathcal{F}(\St_{[\l]})$ having a filtration by Verma modules with lesser length than a filtration by Verma modules of $M$. By induction, $\Hom_{\mathcal{O}_{[l], R}}(P(\omega), M')\in R\proj$. Since $\Hom_{\mathcal{O}_{[l], R}}(P(\omega), -)$ is exact on $\mathcal{F}(\St_{[\l]})$ our claim follows. 
	
	Now, we will proceed to prove (e). Let $S$ be a local commutative Noetherian ring which is a $\mathbb{Q}$-algebra. Since any $M\in \mathcal{F}(\St_{[\l]})$ is free as $R$-module (of infinite rank), the filtrations in $\mathcal{F}(\St_{[\l]})$ remain exact under $S\otimes_R -$. In particular, assuming that $1_S\otimes_R \omega$ is a dominant weight, (\ref{p58eq152}) gives
	\begin{align*}S\otimes_R \Hom_{\mathcal{O}_{[\l], R}}((\St(\omega)\otimes_R E(n\rho))^{\mathcal{D}}, \St(\mu))&\simeq S\otimes_R R^j\simeq 
		S^j\\
		&\simeq \Hom_{\mathcal{O}_{[1_S\otimes_R \l]}, S}(\St(1_S\otimes_R \omega), E(n1_S\rho)^*\otimes_S \St(1_S\otimes_R \mu))\\
		&\simeq \Hom_{\mathcal{O}_{[1_S\otimes_R \l]}, S}(\St(1_S\otimes_R \omega)\otimes_S E(n1_S\rho), \St(1_S\otimes_R \mu) )\\
		&\simeq \Hom_{\mathcal{O}_{[1_S\otimes_R \l]}, S} (S\otimes_R \St(\omega)\otimes_R E(n\rho), \St(1\otimes_R \mu))\\
		&\simeq \Hom_{\mathcal{O}_{[1_S\otimes_R \l]}, S}( S\otimes_R P(\omega-n\rho), \St(1\otimes_R \mu)).
	\end{align*}Since all these isomorphisms are functorial, we obtain that the canonical map $$S\otimes_R \Hom_{\mathcal{O}_{[\l]}, R}(P(\omega-n\rho), \St(\mu))\rightarrow \Hom_{\mathcal{O}_{[1_S\otimes_R \l]}, S}(S\otimes_R P(\omega-n\rho), S\otimes_R \St(\mu))$$ is an isomorphism for every $\mu\in [\l]$. Since $P(\omega-n\rho)$ is a projective object in $\mathcal{O}_{[\l], R}\cap R\Proj$ by using the previous statement there is for every $M\in \mathcal{F}(\St_{[\l]})$ a commutative diagram with exact columns
	\begin{equation}
		\begin{tikzcd}
			S\otimes_R \Hom_{\mathcal{O}_{[\l]}, R}(P(\omega-n\rho), M')\arrow[r, "\simeq"] \arrow[d, hookrightarrow]& \Hom_{\mathcal{O}_{[1_S\otimes_R \l]}, S}(S\otimes_R P(\omega-n\rho), S\otimes_R M') \arrow[d, hookrightarrow]\\
			S\otimes_R \Hom_{\mathcal{O}_{[\l]}, R}(P(\omega-n\rho), M)\arrow[r, ] \arrow[d, twoheadrightarrow] &\Hom_{\mathcal{O}_{[1_S\otimes_R \l]}, S}(S\otimes_R P(\omega-n\rho), S\otimes_R M) \arrow[d]\\
			S\otimes_R \Hom_{\mathcal{O}_{[\l]}, R}(P(\omega-n\rho), \St(\mu))\arrow[r, "\simeq"] & \Hom_{\mathcal{O}_{[1_S\otimes_R \l]}, S}(S\otimes_R P(\omega-n\rho), S\otimes_R \St(\mu))
		\end{tikzcd},
	\end{equation}where $M'\in \mathcal{F}(\St_{[\l]})$ which together with $\St(\mu)$ gives a Verma filtration to $M$. Hence, the upper row is obtained by induction. By Snake Lemma, the middle map is also an isomorphism. 
\end{proof}

\begin{Remark}
	Using the same argument as in the classical theory of the category $\mathcal{O}$, we see that the Verma modules associated with dominant weights are projective objects in their blocks.
\end{Remark}

\subsubsection{Noetherian algebra $A_\mathcal{D}$ associated with a category $\mathcal{O}_{[\l], R}$}

Let $R$ be a local regular commutative Noetherian ring which is a $\mathbb{Q}$-algebra. Let $\mathcal{D}$ be a block of $[\l]$ for some $\l\in \mathfrak{h}_R^*$. Define $P_{\mathcal{D}}:=\bigoplus_{\mu\in \mathcal{D}} P(\mu)$. 

\begin{Def}\label{noetherianalgebracategoryOblock}
	Let $\l\in \mathfrak{h}_R^*$ and $\mathcal{D}$ a block of $[\l]$. We define the $R$-algebra $A_{\mathcal{D}}$ to be the endomorphism algebra $\End_{\mathcal{O}_{[\l], R}}\left( P_{\mathcal{D}}  \right)^{op} $. 
\end{Def}
By Theorem \ref{changeofringsconstructionprojectives}, $A_{\mathcal{D}}$ is a projective Noetherian $R$-algebra. By Lemma \ref{filtrationquotientsofVerma}, we can see that $P_{\mathcal{D}}$ is a generator of $\mathcal{O}_{\mathcal{D}}$. Moreover, since the filtrations involved in Lemma \ref{filtrationquotientsofVerma} are finite for each object $X\in \mathcal{O}_{\mathcal{D}}$ there exists an exact sequence
$P_{\mathcal{D}}^s\rightarrow P_{\mathcal{D}}^t\rightarrow X\rightarrow 0$. Unfortunately, our methods in Theorem \ref{changeofringsconstructionprojectives} do not allow us to state already that $P_{\mathcal{D}}$ is a projective generator. However, we can see that the functor $H:=\Hom_{\mathcal{O}_{\mathcal{D}}}(P_{\mathcal{D}}, -)\colon \mathcal{O}_{\mathcal{D}}\rightarrow A_{\mathcal{D}}\m$ is fully faithful since $P_{\mathcal{D}}$ is a generator of $\mathcal{O}_{\mathcal{D}}$. It is an equivalence of categories whenever $R$ is a field. Further, the restriction of $  H$ to $\mathcal{F}(\St_{[\l]})$ is an exact fully faithful functor. This reduces the study of the category $\mathcal{O}_{[\l], R}$ to the study of module categories of projective Noetherian $R$-algebras and its subcategories.

As we have been mentioning throughout this section, the algebras $A_{\mathcal{D}}$ are split quasi-hereditary.

\begin{Theorem}\label{Oissplitqh}
	Let $R$ be a local regular commutative Noetherian ring which is a $\mathbb{Q}$-algebra. Let $\mathcal{D}$ be a block of $[\l]$ for some $\l\in \mathfrak{h}_R^*$. The algebra $A_{\mathcal{D}}$ is split quasi-hereditary with standard modules $\St_A(\mu):=H\St(\mu)$, $\mu\in \mathcal{D}$.
	The set $\mathcal{D}$ is a poset with the partial order defined as follows: $\mu_1<\mu_2$ if and only if $\mu_2-\mu_1\in \mathbb{N}\varPi$. \end{Theorem}
\begin{proof}
	By Theorem \ref{changeofringsconstructionprojectives} (d), $\St_{A}(\mu)\in R\proj$ for all $\mu\in \mathcal{D}$. By Lemma \ref{homomorphismbetweenvermamodules}(i) and (ii) and together with $H$ being fully faithful we obtain $\End_{A_{\mathcal{D}}}(\St_A(\mu))\simeq R$ and if $\Hom_{A_{\mathcal{D}}}(\St_A(\mu_1), \St_A(\mu_2))\neq 0$, then $\mu_1\leq \mu_2$. Denote by $P_A(\mu)$ the projective $A$-modules $\Hom_{\mathcal{O}_{\mathcal{D}}}(P_{\mathcal{D}}, P(\mu))$.
	By Theorem \ref{changeofringsconstructionprojectives} (c),(b), and $H$ being fully faithful we obtain, for each $\mu\in [\l]$, an exact sequence 
	\begin{align}
		0\rightarrow C_A(\mu)\rightarrow P_A(\mu)\rightarrow \St_A(\mu)\rightarrow 0,
	\end{align}where $C_A(\mu)\in \mathcal{F}(\St_{A_{\omega>\mu}})$. Here, $\St_A$ denotes the set $\{\St_A(\mu)\colon \mu\in \mathcal{D}\}$.
	Further, \begin{align}
		\bigoplus_{\mu\in \mathcal{D}}P_A(\mu)= \bigoplus_{\mu\in \mathcal{D}} \Hom_{\mathcal{O}_{\mathcal{D}}}(P_{\mathcal{D}}, P(\mu))\simeq \Hom_{\mathcal{O}_{\mathcal{D}}}(P_{\mathcal{D}},  \bigoplus_{\mu\in \mathcal{D}}  P(\mu))\simeq {}_{{A_{\mathcal{D}}}} {A_{\mathcal{D}}}.
	\end{align}Hence, this direct sum is a progenerator of $A_{\mathcal{D}}$. 
\end{proof}

We could wonder given the definition of $P(\mu)$ if there could be other projectives taking its role of mapping surjectively to $\St(\mu)$. By Proposition 3.2.1 of \cite{cruz2021cellular}, we see that $P_A(\mu)$ is the right choice and $P_A(\mu)(\mi)$ is actually the projective cover of $\St_A(\overline{\mu})$. Hence, the idempotents $$e_\mu:=P_{\mathcal{D}}\twoheadrightarrow P(\mu)\hookrightarrow P_{\mathcal{D}}, \quad \mu\in \mathcal{D},$$ in $\End_{\mathcal{O}_{\mathcal{D}}}(P_{\mathcal{D}})^{op}=A_{\mathcal{D}}$ form a set of orthogonal idempotents such that their image under $R(\mi)$, according to Theorem \ref{changeofringsconstructionprojectives}(e), form a complete set of primitive orthogonal idempotents of $A_{\mathcal{D}}(\mi)$. In particular, by Theorem 3.3.11 of \citep{cruz2021cellular}, $$0\subset A_{\mathcal{D}}e_{\omega}A_{\mathcal{D}}\subset \cdots \subset A_{\mathcal{D}}(\sum_{\mu\in \mathcal{D}}e_\mu)A_{\mathcal{D}}$$ is a split heredity chain of $A_{\mathcal{D}}$. Here, $\omega$ is the dominant weight of $\mathcal{D}$.

\begin{Cor}
	Let $R$ be a local regular commutative Noetherian ring which is a $\mathbb{Q}$-algebra. Let $\mathcal{D}$ be a block of $[\l]$ for some $\l\in \mathfrak{h}_R^*$.  
	\begin{enumerate}[(a)]
		\item The algebra $A_{\mathcal{D}}$ is semi-perfect and $A_{\mathcal{D}}\proj$ is a Krull-Schmidt category.
		\item The algebra $A_{\mathcal{D}}$ has finite global dimension.
	\end{enumerate}
\end{Cor}
\begin{proof}By \citep[Theorem 3.4.1]{cruz2021cellular}, (a) follows.
	$R$ is a regular local ring, so $\gldim R$ is finite. For (b) see for example \citep[3.5.2]{cruz2021cellular}. 
\end{proof}

The category $\mathcal{O}$ has a simple preserving duality, so we expect the algebra $A_{\mathcal{D}}$ to be a cellular algebra as well. 
We can use the duality functor restricted to the block $(-)^\vee\colon \mathcal{O}_{\mathcal{D}}\cap R\Proj\rightarrow \mathcal{O}_{[\l], (I), R}\cap R\Proj$ to construct the relative injective modules of $A_{\mathcal{D}}$.

\begin{Lemma}
	Let $R$ be a local regular commutative Noetherian ring which is a $\mathbb{Q}$-algebra. Let $\mathcal{D}$ be a block of $[\l]$ for some $\l\in \mathfrak{h}_R^*$. For each $\mu\in \mathcal{D}$,  the module $\Hom_{\mathcal{O}_{[\l], (I), R}}(P_{\mathcal{D}}, P(\mu)^\vee)$ is projective over $R$ and $(A_{\mathcal{D}}, R)$-injective. 
\end{Lemma}
\begin{proof}
	Since $(-)^\vee$ is exact the module $P(\mu)^\vee$ belongs to $\mathcal{F}(\{ \St(\omega)^\vee\colon \ \omega\in \mathcal{D} \})$ for each $\mu\in \mathcal{D}$. As we saw, by the construction of the duality functor $\St(\omega)^\vee\in \mathcal{O}_{[\l], (I), R}\cap R\Proj$ and each weight module of $\St(\omega)^\vee$ is finitely generated as $R$-module. Using the same arguments as in Lemma \ref{transicaoprojectivesO}, replacing $P$ by $\St(\omega_1)$ and $X$ by $\St(\omega)^\vee$ and knowing that in the classical case $\St(\omega)^\vee$ are the costandard modules making $\mathcal{O}$ a split highest weight category we obtain $\Ext_{\mathcal{O}_{[\l], (I), R}}^1(\St(\omega_1), \St(\omega)^\vee)=0$ for all $\omega_1, \omega\in \mathcal{D}$. Hence, using induction on finite filtration by Verma modules $\St$ and on finite filtration by dual Verma modules $\St^\vee$ we can reduce the problem of $\Hom_{\mathcal{O}{[\l], (I), R}}(P_{\mathcal{D}}, X)$ being projective over $R$, with $X\in \mathcal{F}(\{ \St(\omega)^\vee\colon \ \omega\in \mathcal{D} \})$, to showing that $\Hom_{\mathcal{O}_{[\l], (I), R}}(\St(\omega), \St(\theta)^\vee)\in R\proj$ for all weights $\omega, \theta\in \mathcal{D}$.
	
	Observe that $\Hom_{\mathcal{O}_{[\l], (I), R}}(\St(\omega), \St(\theta)^\vee)\subset (\St(\theta)^\vee)_\omega=D(\St(\theta))_\omega$. So, if the homomorphism group is non-zero, then $\omega\leq \theta$. In addition,
$
		0\neq\Hom_{\mathcal{O}_{[\l], (I), R}}(\St(\omega), \St(\theta)^\vee)\simeq \Hom_{\mathcal{O}_{[\l], (I), R}}(\St(\theta), \St(\omega)^\vee).
$ So, also $\theta \leq \omega$. Therefore, $\Hom_{\mathcal{O}_{[\l], (I), R}}(\St(\omega), \St(\theta)^\vee)=0$ unless $\theta=\omega$. In case, $\theta=\omega$ we obtain \begin{align}
		\Hom_{\mathcal{O}_{[\l], (I), R}}(\St(\omega), \St(\omega)^\vee)\simeq \Hom_{U(\mathfrak{b}_R}(R_\omega, \St(\omega)^\vee)\simeq \St(\omega)^\vee_\omega\simeq D\St(\omega)_\omega\simeq R\in R\proj.
	\end{align}
	
	Since $\Ext_{\mathcal{O}_{[\l], (I), R}}^1(\St(\omega_1), \St(\omega)^\vee)=0$ for all $\omega_1, \omega\in \mathcal{D}$ we can apply the same argument in \citep[Proposition 5.5 and Corollary 5.6]{appendix}  to deduce that the homomorphisms between modules with Verma filtrations and modules with dual Verma filtrations commute with the functor $R(\mi)\otimes_R -$. Since $P(\overline{\mu})^\vee$ is injective and $R(\mi)\otimes_R H$ is an equivalence we obtain that  $\Hom_{\mathcal{O}_{[\l], (I), R}}(P_{\mathcal{D}}, P(\mu)^\vee)$ is $(A_{\mathcal{D}}, R)$-injective (see for example \citep[Corollary 2.12]{CRUZ2022410}).
\end{proof}

\begin{Remark}
	The reader can observe that the modules $\Hom_{\mathcal{O}_{[\l], (I), R}}(P_{\mathcal{D}}, \St(\mu)^\vee)$, $\mu\in \mathcal{D}$, are the costandard modules of $A_{\mathcal{D}}$.
\end{Remark}

It follows that $\Hom_{\mathcal{O}_{[\l], (I), R}}(P_{\mathcal{D}}, P_{\mathcal{D}}^\vee)\simeq \bigoplus_{\mu\in \mathcal{D}} \Hom_{\mathcal{O}_{[\l], (I), R}}(P_{\mathcal{D}}, P(\mu)^\vee)\simeq DA_{\mathcal{D}}$. Using this we can deduce a duality map on $A_{\mathcal{D}}$. In fact, as $R$-algebras, we have the following commutative diagram
$$
	\hspace*{-1.5em}	\begin{tikzcd}[row sep=1.2em, column sep=-1.2em, scale cd=0.8]
		A_{\mathcal{D}}(\mi) \arrow[rrr, "\simeq"]   \arrow[ddd, "\simeq"]&&&
		\End_{\mathcal{O}_{\mathcal{D}}}(P_{\mathcal{D}})^{op}(\mi) \arrow[rrr, "\simeq"] \arrow[ddd, "\simeq"]&&&
		\End_{\mathcal{O}_{[\l], (I), R}}(P_{\mathcal{D}}^\vee )(\mi)  \arrow[ddrr] \arrow[ddd] 
		\\
		& A_{\mathcal{D}}^{op}(\mi) && 
		&& 
		\\
		&& \End_{A_{\mathcal{D}}^{op}}(A_{\mathcal{D}})^{op}(\mi) \arrow[ul, "\simeq"] &&& 
		\End_{A_{\mathcal{D}}}(DA_{\mathcal{D}}) (\mi) \arrow[lll,crossing over, "\simeq"]  &&&
		\End_{A_{\mathcal{D}}}( \Hom_{\mathcal{O}_{[\l], (I), R}}(P_{\mathcal{D}}, P_{\mathcal{D}}^\vee) )(\mi) \arrow[lll,crossing over, "\simeq" near end] \arrow[ddd,"\simeq"]\\
		A(\mi)_{\mathcal{D}} \arrow[rrr, "\simeq"]   &&&
		\End_{\mathcal{O}}(P_{\mathcal{D}}(\mi))^{op}\arrow[rrr, "\simeq"]&&&
		\End_{\mathcal{O}}(P_{\mathcal{D}}(\mi)^\vee )  \arrow[ddrr,swap,"\simeq"] 
		\\
		& A(\mi)_{\mathcal{D}}^{op} \arrow[uuu,<-,crossing over,"\simeq" near end]  && 
		&& 
		\\
		&& \End_{A(\mi)_{\mathcal{D}}^{op}}(A(\mi)_{\mathcal{D}})^{op} \arrow[uuu,<-,crossing over, "\simeq"] \arrow[ul, "\simeq"]&&& 
		\End_{A(\mi)_{\mathcal{D}}}(DA_{\mathcal{D}}(\mi))  \arrow[lll,crossing over,"\simeq" near start]  \arrow[uuu,<-,crossing over, "\simeq"]&&& 
		\End_{A(\mi)_{\mathcal{D}}}( \Hom_{\mathcal{O}}(P_{\mathcal{D}}(\mi), P_{\mathcal{D}}(\mi)^\vee) )	 \arrow[lll,crossing over,"\simeq" near start]\\
	\end{tikzcd} 
$$
The isomorphisms in the diagram are marked with $\simeq$. We required this approach since without it we do not know if the Hom functor on the generator $P_{\mathcal{D}}$ is fully faithful on the additive closure of its dual $P_{\mathcal{D}}^\vee$. Since $A_{\mathcal{D}}\in R\proj$ by Nakayama's Lemma we obtain that the following composition of maps is an isomorphism of $R$-algebras
\begin{equation*}
	\begin{tikzcd}[row sep=0.8em, column sep=-1.8em]
		A_{\mathcal{D}} \arrow[rrr]   &&&
		\End_{\mathcal{O}_{\mathcal{D}}}(P_{\mathcal{D}})^{op}\arrow[rrr] &&&
		\End_{\mathcal{O}_{[\l], (I), R}}(P_{\mathcal{D}}^\vee ) \arrow[ddrr] 
		\\
		& A_{\mathcal{D}}^{op}&& 
		&& 
		\\
		&& \End_{A_{\mathcal{D}}^{op}}(A_{\mathcal{D}})^{op}\arrow[ul] &&& 
		\End_{A_{\mathcal{D}}}(DA_{\mathcal{D}}) \arrow[lll,crossing over]  &&&
		\End_{A_{\mathcal{D}}}( \Hom_{\mathcal{O}_{[\l], (I), R}}(P_{\mathcal{D}}, P_{\mathcal{D}}^\vee) )\arrow[lll,crossing over] 
	\end{tikzcd}.
\end{equation*}
Observe that under this composition of maps, for each $\mu\in \mathcal{D}$, \begin{align*}
	e_\mu&\mapsto P_{\mathcal{D}} \twoheadrightarrow P(\mu)\hookrightarrow P_{\mathcal{D}}\mapsto P_{\mathcal{D}}^\vee\twoheadrightarrow P(\mu)^\vee \hookrightarrow P_{\mathcal{D}}^\vee\\
	& \mapsto \Hom_{\mathcal{O}_{[\l], (I), R}}(P_{\mathcal{D}}, P_{\mathcal{D}}^\vee) \twoheadrightarrow \Hom_{\mathcal{O}_{[\l], (I), R}}(P_{\mathcal{D}}, P(\mu)^\vee)\hookrightarrow \Hom_{\mathcal{O}_{[\l], (I), R}}(P_{\mathcal{D}}, P_{\mathcal{D}}^\vee) \\
	&\mapsto DA_{\mathcal{D}}\twoheadrightarrow I(\mu)=D(e_\mu A_{\mathcal{D}})\hookrightarrow DA_{\mathcal{D}} \mapsto A_{\mathcal{D}}\twoheadrightarrow e_\mu A_{\mathcal{D}} \hookrightarrow A_{\mathcal{D}}\mapsto e_\mu.
\end{align*}
Hence, this gives an involution on $A_{\mathcal{D}}$, denoted by $\iota$, which fixes the set of orthogonal idempotents $\{e_\mu\colon \mu\in \mathcal{D} \}$.
In addition, we can assign a new duality functor on $A_{\mathcal{D}}$ using the duality $\iota$. For each $M\in A_{\mathcal{D}}\m$, define the right $A_{\mathcal{D}}$-module $M^{\iota}\in $ by imposing $m\cdot_{\iota}a:=\iota(a)m, \ m\in M$. The assignment $M\mapsto DM^{\iota}$ is a duality functor on $A_{\mathcal{D}}\m \cap R\proj$, which we will denote by $(-)^\natural$.

\begin{Theorem}\label{Oiscellular}
	Let $R$ be a local regular commutative Noetherian ring which is a $\mathbb{Q}$-algebra. Let $\mathcal{D}$ be a block of $[\l]$ for some $\l\in \mathfrak{h}_R^*$. The algebra $A_{\mathcal{D}}$ is a cellular algebra with involution $\iota$ and cell chain $$0\subset A_{\mathcal{D}}e_{\omega}A_{\mathcal{D}}\subset \cdots \subset A_{\mathcal{D}}(\sum_{\mu\in \mathcal{D}}e_\mu)A_{\mathcal{D}}=A_{\mathcal{D}},$$ where $\omega$ is the dominant weight of $\mathcal{D}$.
\end{Theorem}
\begin{proof}
	The result follows by Proposition 4.0.1 of \citep{cruz2021cellular}.
\end{proof}

\subsubsection{The algebra $A_{\mathcal{D}}$ is a relative gendo-symmetric algebra}

Our aim now is to compute the relative dominant dimension of the algebra $A_{\mathcal{D}}$ and to prove that it is a relative gendo-symmetric algebra.

Koenig, Slungård and Xi gave a lower bound for the dominant dimension of the blocks of the classical category $\mathcal{O}$ in \citep[Theorem 3.2]{Koenig2001}. Later, Fang proved in \citep[Proposition 4.5]{zbMATH05278765} that this lower bound was indeed the value of the dominant dimension. Mainly, the dominant dimension sees two cases. Either the algebra associated with a block is semi-simple which obviously has infinite dominant dimension or the algebra associated with a non semi-simple block has dominant dimension two. The main reason for this situation is that the blocks of the category $\mathcal{O}$ only have one projective-injective module.
We will now generalize these results to the Noetherian algebras $A_{\mathcal{D}}$.

\begin{Theorem}\label{dominantdimensionO}
	Let $R$ be a local regular commutative Noetherian ring which is a $\mathbb{Q}$-algebra with unique maximal ideal $\mi$. 
	Let $\mathcal{D}$ be a block of $[\l]$ for some $\l\in \mathfrak{h}_R^*$. 
	For the unique antidominant weight $\mu\in \mathcal{D}$,  $(A_{\mathcal{D}}, P_A(\mu), DP_A(\mu))$ is a relative QF3 $R$-algebra and
	$$ 
	\domdim (A_{\mathcal{D}}, R)=\begin{cases}
		+\infty, \quad &\text{ if } \ |\mathcal{D}|=1, \\ 2, \quad &\text{ otherwise}.
	\end{cases}
	$$
\end{Theorem}
\begin{proof}
	By Lemma \ref{decompositionblocksweights}, $\mathcal{D}$ is of the form $W_{\overline{\mu}}\cdot \mu +\nu$ for some $\nu \in\mathfrak{h}_R^*$ and $w\cdot \mu -\mu\in \mathbb{Z}\Phi$ for every $w\in W_{\overline{\mu}}$.
	
	By Theorem \ref{Propertiesofsemisimplecomplexliealg}(e), for  an antidominant weight $\omega\in \mathcal{D}$, $P(\omega)$ is the unique projective-injective in $\mathcal{O}_{W_{\overline{\mu}}\cdot \overline{\mu}}$. By Theorem \ref{Oissplitqh} and \ref{changeofringsconstructionprojectives}, $P_A(\omega)(\mi)\simeq P_{A(\mi)}(\overline{\omega}) $ is the unique projective-injective of $A_{W_{\overline{\mu}}\cdot \overline{\mu}}\simeq A_{\mathcal{D}}(\mi)$. There are now two distinct cases. Assume that $A_{W_{\overline{\mu}}\cdot \overline{\mu}}$ is semi-simple. In particular, $\St(\overline{\mu})$ is projective-injective. By Theorem \ref{Propertiesofsemisimplecomplexliealg}, $\overline{\mu}$ is a dominant and antidominant weight. 
	On the other hand, if there exists a weight in $W_{\overline{\mu}}\cdot \overline{\mu}$ which is both dominant and antidominant, then
	it is both a maximal and minimal element in $W_{\overline{\mu}}\cdot \overline{\mu}$. Thus, $W_{\overline{\mu}}\cdot \overline{\mu}=\{\overline{\mu} \}$. In such a case, $A_{W_{\overline{\mu}}\cdot \overline{\mu}}\simeq R(\mi)$. Further, for any two elements $w_1, w_2\in W_{\overline{\mu}}$ we have $w_1\cdot \mu-w_2\cdot \mu\in \mi\mathfrak{h}_R^*$. By construction of $\mu$, we also have  $w_1\cdot \mu-w_2\cdot \mu\in \mathbb{Z}\Phi$. So, we deduce that $\domdim A_{W_{\overline{\mu}}}=+\infty$ if and only if the cardinality of $\mathcal{D}$ is one. 
	
	Assume now that the cardinality of $\mathcal{D}$ is greater than one. By \citep[Proposition 4.5]{zbMATH05278765}, we obtain that \mbox{$ \domdim A_{W_{\overline{\mu}}\cdot \overline{\mu}}=2$} with faithful projective-injective $P_A(\omega)(\mi)\simeq P_{A(\mi)}(\overline{\omega})$. By \citep[Proposition 6.3, Theorem 6.13]{CRUZ2022410}, the result follows. 
\end{proof}

It follows from Theorem \ref{dominantdimensionO} and Theorem \ref{Mullertheorem}, the analogue of integral Schur--Weyl duality for the blocks of the category $\mathcal{O}$:
There is a double centralizer property
$$C:=\End_{A_{\mathcal{D}}}(P_A(\omega))^{op}, \quad A_{\mathcal{D}}\simeq \End_C(P_A(\omega))$$
where $\omega$ is the antidominant weight of $\mathcal{D}$. Here, $C(\mi)$ is the so-called \textbf{coinvariant algebra} $S(\mathfrak{h}_{R(\mi)})/I$ whenever $|W_{\overline{\mu}}\cdot \overline{\mu}|=|W_{\overline{\mu}}|$ for the block $\mathcal{D}=W_{\overline{\mu}}\mu+\nu$. 
Here, $I$ denotes the ideal of the symmetric algebra of $\mathfrak{h}_{R(\mi)}^*$ generated by the polynomials which are invariants, under the Weyl group linear action, of positive degree with respect to the grading of the symmetric algebra of $\mathfrak{h}_{R(\mi)}^*$. In the other cases, where the stabilizer of $\overline{\mu}$ under the Weyl group $W_{\overline{\mu}}$ is non-trivial, the algebra $C(\mi)$ is a subalgebra of invariants of the coinvariant algebra under the elements of the stabilizer of $\overline{\mu}$. In particular, $C(\mi)$ is a commutative algebra (see \citep[Endomorphismensatz]{zbMATH00005018}). 

Taking advantage of the previous double centralizer property, we can define the Schur functor
$$\mathbb{V}_{\mathcal{D}}= \Hom_{A_{\mathcal{D}}}(P_A(\omega), -) \colon A_{\mathcal{D}}\m\rightarrow C\m.$$

In the literature, this functor is known as \textbf{Soergel's combinatorial functor}. The famous result known as \textbf{Struktursatz} \citep[Struktursatz 9]{zbMATH00005018} states that $\mathbb{V}_{\overline{\mathcal{D}}}\colon A_{\mathcal{D}}(\mi)\m\rightarrow C(\mi)\m$ is fully faithful on projectives. To see that, in this more general setup, we start by observing that since $P_A(\omega)$ is projective-injective, it is a (partial) tilting module, so it is self-dual with respect to $(-)^\natural$, that is, $P_A(\omega)^\natural\simeq P_A(\omega)$. So, it follows by  \citep[Proposition 2.4 and Lemma 3.2]{zbMATH05871076} that $C(\mi)\simeq DC(\mi)$. We will briefly explain the idea: the arguments are based on bookkeeping the twisted actions and realizing that $P_A(\omega)(\mi)$ being self dual implies that the isomorphism of $P_A(\omega)(\mi)$ to its dual is also an isomorphism of $C(\mi)$ under the twisted action. Then applying $\Hom_{R(\mi)}(-, R(\mi))$ one would obtain an isomorphism between $D(A_{\mathcal{D}}e_\omega)(\mi)$ and $e_\omega A_{\mathcal{D}}(\mi)$ as left $C(\mi)$-modules under the usual action. Now, applying the Schur functor we would obtain the desired isomorphism. 
Since $C\in C\proj$ the homomorphism
$
	{C\twoheadrightarrow C(\mi)\xrightarrow{\simeq} DC(\mi)}
$ can be lifted to a $C$-homomorphism $f\colon C\rightarrow DC$. Moreover, $f(\mi)$ is an isomorphism. Since $C, DC\in R\proj$ we obtain  that $f$ is an isomorphism as $C$-modules. This shows that $C$ is a relative self-injective $R$-algebra. Now, using that $C$ is a commutative $R$-algebra $f$ yields in addition that $C$ is a relative symmetric $R$-algebra. To see this observe that the action of the center of the enveloping algebra $Z(\mathfrak{g}_R)$ on $P(\omega)$ ($\omega$ being the antidominant weight of $\mathcal{D}$) yields a homomorphism of $R$-algebras $Z(\mathfrak{g}_R)\rightarrow \End_{A_{\mathcal{D}}}(HP(\omega))$. Further, we have a commutative diagram
\begin{equation}
	\begin{tikzcd}
		Z(\mathfrak{g}_R)(\mi)\arrow[r] \arrow[d, "\simeq"]& \End_{A_{\mathcal{D}}}(HP(\omega))(\mi)\arrow[d, "\simeq"]\\
		Z(\mathfrak{g}_{R(\mi)})\arrow[r, twoheadrightarrow] & \End_{A(\mi)_{\mathcal{D}}}(\Hom_{\mathcal{O}}(P_{\mathcal{D}}(\mi), P(\overline{\omega}))
	\end{tikzcd}.
\end{equation}
Here, the bottom row is surjective due to Soergel \citep[Lemma 5]{zbMATH00005018}, the left map is an isomorphism by Lemma \ref{centerofLiealgebra} while the right map is an isomorphism by Theorem \ref{changeofringsconstructionprojectives}. It follows that the upper map is also a surjective map. Denote by $X$ the cokernel (as $R$-homomorphisms) of the homomorphism $Z(\mathfrak{g}_R)\rightarrow \End_{A_{\mathcal{D}}}(HP(\omega))$. Thus, $X(\mi)=0$. Since $\End_{A_{\mathcal{D}}}(HP(\omega))\in R\proj$, we obtain $X\in R\m$ and by Nakayama's Lemma $X=0$. Hence, $Z(\mathfrak{g}_R)\rightarrow C$ is surjective, and therefore $C$ is a commutative $R$-algebra.

To sum up, we obtained:

\begin{Theorem}\label{integralstruktursatz}
	Let $R$ be a local regular commutative Noetherian ring which is a $\mathbb{Q}$-algebra with unique maximal ideal $\mi$. 
	Let $\mathcal{D}$ be a block of $[\l]$ for some $\l\in \mathfrak{h}_R^*$. Suppose that $\omega$ is the antidominant weight of $\mathcal{D}$. The following assertions hold.
	\begin{enumerate}[(a)]
		\item $(A_{\mathcal{D}}, P_A(\omega))$ is a relative gendo-symmetric $R$-algebra.
		\item $A_{\mathcal{D}}$ is split quasi-hereditary over $R$ with standard modules $\St_A(\mu)$, $\mu\in \mathcal{D}$.
		\item $A_{\mathcal{D}}$ is a cellular $R$-algebra with cell modules $\St_A(\mu)$, $\mu\in \mathcal{D}$, with respect to the duality map $\iota$.
		\item \textbf{(Integral Struktursatz) }$(A_{\mathcal{D}}, P_A(\omega))$ is a split quasi-hereditary cover of the commutative $R$-algebra $C$.
		\item $C$ is a cellular $R$-algebra with cell modules $\mathbb{V}_{\mathcal{D}}\St_A(\mu)$, $\mu\in \mathcal{D}$, with respect to the duality map $\iota_{|_{e_{\omega}A_{\mathcal{D}}e_{\omega} }}$.
		\item If $T$ is a characteristic tilting module of $A_{\mathcal{D}}$, then $2\domdim_{(A_{\mathcal{D}}, R)}T=\domdim (A_{\mathcal{D}}, R)$.
	\end{enumerate}
\end{Theorem}
\begin{proof}
	Statements (b) and (c) are Theorem \ref{Oissplitqh} and \ref{Oiscellular}, respectively.
	By the relative Morita-Tachikawa correspondence (see \citep[Theorem 4.1]{CRUZ2022410}) and \ref{dominantdimensionO}, $P(\omega)$ is a generator as $C$-module and satisfies $P_A(\omega)\otimes_C DP_A(\omega)\in R\proj$.	By Theorem 7.8 of \citep{CRUZ2022410} and the discussion above showing that $C$ is a relative symmetric $R$-algebra (a) follows. 
	
	By (a), $\Hom_{A_{\mathcal{D}}}(P_A(\omega), A_{\mathcal{D}})\simeq DP_A(\omega)$. 
	This fact, together with Theorem \ref{dominantdimensionO} implies the existence of a double centralizer property on $\mathbb{V} A_{\mathcal{D}}$. This shows (d). 
	By Proposition 2.2.11 of \cite{cruz2021cellular}, (e) follows.
	By Theorem \ref{tiltingmoduledominantdimensioncover}, (f) follows.
\end{proof}

\subsubsection{Hemmer-Nakano dimension under $\mathbb{V}_{\mathcal{D}}$}

Given that $(A_{\mathcal{D}}, P_A(\omega))$ is a cover of the algebra $C$, the Schur functor $\mathbb{V}_{\mathcal{D}}$ is fully faithful on projectives. There are two natural questions.
Is the above fully faithfulness a precursor of an identification of $\Ext$-groups under the restriction of the Schur functor to projectives?
 Since $A_{\mathcal{D}}$ is split quasi-hereditary, the same question can be posed involving the Verma modules. We shall start by discussing the classical case of complex semi-simple Lie algebras.
In that case, the Schur functor $\mathbb{V}$ (on a non semi-simple block) restricted to the projective modules cannot induce a bijection on the first $\Ext$ groups since otherwise the fact that $A_{W_{\overline{\mu}}\cdot \overline{\mu}}$ is a gendo-symmetric algebra would imply an increase in the dominant dimension to at least three. 

Now, regarding the Verma modules, the situation in the classical case is not very promising. Indeed, the vector space dimension of $\mathbb{V} \St(w\cdot\overline{\mu})$ is equal to the multiplicity of the simple module $\St(\overline{\omega})$ in the standard module $\St(w\cdot\overline{\mu})$ for every $w\in W_{\overline{\mu}}$, where $\overline{\omega}$ is the unique antidominant weight in the orbit. Since the non-zero homomorphisms between Verma modules are always injective, $\St(\overline{\omega})$ only occurs in the socle of $\St(w\cdot\overline{\mu})$. Therefore, $\dim_{R(\mi)} \mathbb{V} \St(w\cdot \overline{\mu})=1$. Since the Schur functor $\mathbb{V}$ kills all simple modules which are not in the top of the projective module $P_A(\omega)$ then $\mathbb{V}$  sends all standard modules to the same module  with dimension one over $C$. Therefore, $\mathbb{V}$ is not even fully faithful on Verma modules. It is only faithful on Verma modules.
This is the major difference between the classical case and the Noetherian algebras $A_{\mathcal{D}}$ as we will see now. 

\begin{Theorem}\label{HNforcategoryO}
	Fix $t$ a natural number. Let $R$ be the localization of $\mathbb{C}[X_1, \ldots, X_t]$ at the maximal ideal $(X_1, \ldots, X_t)$. Denote by $\mi$ the unique maximal ideal of $R$. Pick $\theta\in \mathfrak{h}_{R(\mi)}^*\simeq \mathfrak{h}_R^*/\mi \mathfrak{h}_R^*$ to be an antidominant weight which is not dominant. Define $\mu\in \mathfrak{h}_R^*$ to be a preimage of $\theta$ without coefficients in $\mi$ in its unique linear combination of simple roots. Fix $s$ to be a natural number satisfying $1\leq s\leq \rank_R \mathfrak{h}_R^*$ and $s\leq t$. Consider the block $\mathcal{D}=W_{\overline{\mu}}\cdot \mu +\frac{X_1}{1}\alpha_1+\cdots +\frac{X_s}{1}\alpha_s$, where $\alpha_i\in \varPi$ are distinct simple roots, $i=1, \ldots, s$ and by $\frac{f}{1}$ we mean the image of $f\in \mathbb{C}[X_1, \ldots, X_t]$ in $R$.
	Then
	\begin{enumerate}[(i)]
		\item $\HN_{\mathbb{V}_{\mathcal{D}}} A_{\mathcal{D}}\proj =s$.
		\item $\HN_{\mathbb{V}_{\mathcal{D}}} \mathcal{F}(\St_A)=s-1$.
	\end{enumerate}
\end{Theorem}

Hence, Theorem \ref{HNforcategoryO} implies that for the algebras $A_{\mathcal{D}}$ not only the Schur functor is fully faithful on Verma modules but also the Schur functor behaves quite well on $\Ext$ groups of Verma modules. This fact alone justifies studying the category $\mathcal{O}$ under other rings than the complex numbers.

It follows by Theorem \ref{HNforcategoryO} that the assumptions on Theorem \ref{improvingcoverwithspectrum} are optimal.
\begin{Remark}
	The non-zero homomorphisms between distinct Verma modules are injective maps but they are not $(A_{\mathcal{D}}, R)$-monomorphisms, in general. Otherwise, we would obtain a $(C, R)$-monomorphism \linebreak$\mathbb{V}\St(\omega_1)\rightarrow \mathbb{V}(\St(\omega_2))$ which remains injective under $R(\mi)\otimes_R -$. $\mathbb{V}\St(\omega_2)(\mi)$ is a  simple module, hence the mentioned map must be an isomorphism and by Nakayama's Lemma, so is the map $\mathbb{V}\St(\omega_1)\rightarrow \mathbb{V}(\St(\omega_2))$. Theorem \ref{HNforcategoryO}(ii) illustrates that there are rings $R$ and $R$-algebras $A_{\mathcal{D}}$ for which such a situation cannot happen.
\end{Remark}

\begin{Remark}
	Note that all results for the BGG category $\mathcal{O}$ of a complex semi-simple Lie algebra $\mathfrak{g}$ remain valid for $\mathfrak{g}_K$ for all fields $K$ of characteristic zero (since they are valid for a given a semisimple Lie algebra $\mathfrak{g}$ over a splitting field $K$ of characteristic zero.) In particular, Theorem \ref{HNforcategoryO} remains valid replacing $\mathbb{C}$ with $K$.
\end{Remark}

\begin{Remark}\label{dominantdimensionusingExtclarification}
	Much focus on \cite{CRUZ2022410} was given to illustrate that the relative dominant dimension over projective Noetherian algebras should be measured using $\Tor$ groups instead of $\Ext$ groups. 
	The reason for this was that the Krull dimension of regular local rings is an obstruction to deducing information on vanishing of $\Ext$ groups. More precisely we can ask how sharp is the lower bound presented in Theorem \ref{Mullertheorem} (e). Theorem \ref{HNforcategoryO} clarifies this situation by showing that under the assumptions of Theorem \ref{Mullertheorem} (e) the relative dominant dimension can attain any value between $n-\dim R$ and $n$.
\end{Remark}

\begin{proof}[Proof of Theorem \ref{HNforcategoryO}]	Denote by $\mi$ the maximal ideal of $R$. Assume that $\alpha_1, \ldots, \alpha_n\in \varPi$ are the simple roots of $\Phi$ and $\mu=\sum_{i=1}^n c_i \alpha_i$, where $c_i$ is the image of some complex number in $R$. Denote by $\nu$ the weight $\frac{X_1}{1}\alpha_1+\cdots +\frac{X_s}{1}\alpha_s$. By Theorem \ref{integralstruktursatz}, $(A_{\mathcal{D}}, P_A(\mu))$ is a split quasi-hereditary cover of $C$ and it is a relative gendo-symmetric $R$-algebra.
	
	We will start by showing that $s$ and $s-1$ are upper bounds for the Hemmer-Nakano dimension of $A_{\mathcal{D}}\proj$ and $\mathcal{F}(\St_A)$, respectively, under the Schur functor $\mathbb{V}_{\mathcal{D}}$. Let $T$ be a characteristic tilting module of $A_{\mathcal{D}}$.

	Choose $\pri$ the prime ideal of $R$ generated by the monomials $\frac{X_i}{1}$, with $i=1, \ldots, s$. In particular, $\pri$ has height $s$. Further, $R/\pri\otimes_R \mu$ is an antidominant weight which is not dominant and it has no coefficients belonging to the maximal ideal of $R/\pri$ in its unique linear combination of simple roots and $R/\pri \otimes_R \nu=0$.  Therefore, $Q(R/\pri)\otimes_R \mathcal{D}$ contains $Q(R/\pri)\otimes_R \mu$ which is an antidominant but is not a dominant weight, where $Q(R/\pri)$ denotes the quotient field of $R/\pri$. 
	Therefore, $Q(R/\pri)\otimes_R A_{\mathcal{D}}$ contains as direct product the algebra $A_{W_{Q(R/\pri)\otimes_R \mu}\cdot Q(R/\pri)\otimes_R \mu}$ which is not semi-simple. By Theorem \ref{dominantdimensionO}, \ref{boundrelationcoversdom} and the flatness of $Q(R/\pri)$ over $R/\pri$,
	\begin{align}
		-1 &=	\HN_{\mathbb{V}_{W_{Q(R/\pri)\otimes_R \mu}\cdot Q(R/\pri)\otimes_R \mu} } \mathcal{F}(Q(R/\pri)\otimes_R \St_A)\geq \HN_{Q(R/\pri)\otimes_R \mathbb{V}_{\mathcal{D}}} \mathcal{F}(Q(R/\pri)\otimes_R \St_A)\\&\geq \HN_{R/\pri\otimes_R \mathbb{V}_{\mathcal{D}}}\mathcal{F}(R/\pri\otimes_R \St_A).
		\\
		0&=\HN_{\mathbb{V}_{W_{Q(R/\pri)\otimes_R \mu}\cdot Q(R/\pri)\otimes_R \mu} } A_{W_{Q(R/\pri)\otimes_R \mu}\cdot Q(R/\pri)\otimes_R \mu}\proj
	\\	&\geq \HN_{Q(R/\pri)\otimes_R \mathbb{V}_{\mathcal{D}}} Q(R/\pri)\otimes_R A_{\mathcal{D}}\proj 
		\geq \HN_{R/\pri\otimes_R \mathbb{V}_{\mathcal{D}}} R/\pri\otimes_R A_{\mathcal{D}}\proj.
	\end{align}
	As a consequence of Corollary \ref{coverheightprimeideal} and thanks to $\height(\pri)=s$ we obtain
	\begin{align}
		\HN_{\mathbb{V}_{\mathcal{D}} }\mathcal{F}(\St_A)&\leq \HN_{R/\pri\otimes_R \mathbb{V}_{\mathcal{D}}}\mathcal{F}(R/\pri\otimes_R \St_A)+\height(\pri)=-1+s\\
		\HN_{\mathbb{V}_{\mathcal{D}}} A_{\mathcal{D}}\proj &\leq \HN_{R/\pri\otimes_R \mathbb{V}_{\mathcal{D}}} R/\pri\otimes_R A_{\mathcal{D}}\proj+\height(\pri)=s.
	\end{align}
	
	We claim that this inequality is actually an equality. To show that we will proceed by induction on the coheight of prime ideals $\pri$ of $R$, that is, on $\dim R-\height(\pri)$ with induction basis step $t-s$ to show that
	\begin{align*}
		\HN_{R/\pri\otimes_R \mathbb{V}} \mathcal{F}(R/\pri\otimes_R \St_A)\geq -1+s-\height(\pri), \quad 
		\HN_{R/\pri\otimes_R \mathbb{V}_{\mathcal{D}}} R/\pri\otimes_R A_{\mathcal{D}}\proj\geq s-\height(\pri).
	\end{align*}

	Let $\pri$ be a prime ideal of $R$ with coheight $t-s$, then it has height $s$. Since $R/\pri$ has maximal ideal $\mi/\pri$ with residue field $R(\mi)$ the claim follows by Theorem \ref{boundrelationcoversdom}, Theorem \ref{integralstruktursatz}(f) and Theorem \ref{dominantdimensionO}.
	
	Now assume that $\pri$ is a prime ideal of $R$ with coheight greater than $t-s$. Then $\pri$ has height smaller than $s$. In particular, $\pri$ cannot contain any prime ideal with height $s$. Consequently, some monomial $\frac{X_i}{1}$ has non-zero image in $R/\pri$. Moreover, $\nu$ has non-zero image in $R/\pri$ and its image belongs to $\mi/\pri$. Therefore, any weight in $R/\pri\otimes_R \mathcal{D}$ when viewed as weight in the quotient field $\mathfrak{h}_{Q(R/\pri)}^*$ does not belong to the integral weight lattice. Thus, all weights belonging to $R/\pri\otimes_R \mathcal{D}$ viewed as weights in $\mathfrak{h}_{Q(R/\pri)}^*$ are both dominant and antidominant. By the discussion in Theorem \ref{dominantdimensionO}, we obtain that \begin{align}
		\domdim Q(R/\pri)\otimes_R A_{\mathcal{D}}=\domdim_{ Q(R/\pri)\otimes_R A_{\mathcal{D}}} Q(R/\pri)\otimes_R T=+\infty. \label{eqexemplios159}
	\end{align} By Theorem \ref{boundcoverimprovementdomquotientfield}, we obtain that the claim holds for prime ideals with coheight $t-s+1$.

	Upon these considerations, assume the induction claim known for some prime ideal with coheight $t-s+r$ with $r\geq 1$.
	Let $\pri$ be a prime ideal of coheight $t-s+r+1$. Then $\height(\pri)=t-t+s-1=s-r-1<s$ and (\ref{eqexemplios159}) holds. By Theorem \ref{boundcoverimprovementdomquotientfield} the assumptions of Theorem \ref{improvingcoverwithspectrum} for $R/\pri\otimes_R A_{\mathcal{D}}$ and the resolving subcategories $\mathcal{F}(R/\pri\otimes_R \St_A)$ and $R/\pri\otimes_R A_{\mathcal{D}}$ are satisfied. Also, condition (i) of Theorem \ref{improvingcoverwithspectrum} is also satisfied. It remains to consider (ii). Let $\mathfrak{q}$ be a prime ideal of $R/\pri$ of height one. Then we can write $\mathfrak{q}=\mathfrak{q'}/\pri$ for some prime ideal $\mathfrak{q'}$ of $R$. Furthermore,
	\begin{align}
		1=\height(\mathfrak{q'}/\pri)=\dim(R/\pri)-\coheight(\mathfrak{q'}/\pri)=\coheight(\pri)-\coheight(\mathfrak{q'}),
	\end{align}  Hence, $\coheight(\mathfrak{q'})=\coheight(\pri)-1=t-s+r$. 
	By induction,  
	\begin{align}
		\HN_{R/\mathfrak{q'}\otimes_R \mathbb{V}} \mathcal{F}(R/\mathfrak{q'}\otimes_R \St_A)\geq -1+s-\height(\mathfrak{q'})=-1+r\\
		\HN_{R/\mathfrak{q'}\otimes_R \mathbb{V}_{\mathcal{D}}} R/\mathfrak{q'}\otimes_R A_{\mathcal{D}}\proj\geq s-\height(\mathfrak{q'})=r. 
	\end{align}
	Because of $R/\mathfrak{q'}\simeq (R/\pri)/(\mathfrak{q'}/\pri)=(R/\pri)/\mathfrak{q}$ Theorem \ref{improvingcoverwithspectrum} yields 
	\begin{align}
		\HN_{R/\mathfrak{p}\otimes_R \mathbb{V}} \mathcal{F}(R/\mathfrak{p}\otimes_R \St_A)\geq r\\
		\HN_{R/\mathfrak{p}\otimes_R \mathbb{V}_{\mathcal{D}}} R/\mathfrak{p}\otimes_R A_{\mathcal{D}}\proj\geq r+1 
	\end{align}This finishes the proof of the claim. 
	Now considering the prime ideal zero which has height zero, the result follows.
\end{proof}

\section*{Acknowledgments}
Most of the results of this paper are contained in the author's PhD thesis \citep{thesis}, financially supported by \textit{Studienstiftung des Deutschen Volkes}. The author would like to thank Steffen Koenig for all the conversations on these topics, his comments and suggestions towards improving this manuscript. I would also like to thank the anonymous referee for carefully reading this work and for all the comments given.

\bibliographystyle{alphaurl}
\bibliography{bibarticle}

\Address
\end{document}